\pdfoutput=1   
\documentclass[oneside]{book}

\usepackage{microtype}

\usepackage{makeidx}
\usepackage{mathtools}
\mathtoolsset{mathic=true}
\usepackage{hyperref}
\usepackage{amsthm,amsmath,amsfonts,amssymb}
\usepackage{srcltx}                
\usepackage{xspace}                
\usepackage{tikz}
\usetikzlibrary{decorations.pathreplacing}
\usetikzlibrary{patterns}
\tikzset{>=latex}
\usepackage{bookmark}              

\usepackage{subcaption}                 
\usepackage{chngcntr}
\counterwithin{figure}{chapter}

\newcommand{\cev}[1]{\reflectbox{\ensuremath{\vec{\reflectbox{\ensuremath{#1}}}}}}

\newtheorem{theorem}{Theorem}[chapter]
\newtheorem*{theorem-nn}{Theorem}
\newtheorem{corollary}[theorem]{Corollary}
\newtheorem{lemma}[theorem]{Lemma}
\newtheorem{proposition}[theorem]{Proposition}
\newtheorem{question}[theorem]{Question}
\newtheorem*{question-nn}{Question}
\newtheorem{conjecture}[theorem]{Conjecture}
\newtheorem*{conjecture-nn}{Conjecture}
\theoremstyle{definition}
\newtheorem{definition}[theorem]{Definition}
\newtheorem*{definition-nn}{Definition}
\theoremstyle{remark}
\newtheorem*{claim*}{Claim}
\newtheorem{claim}{Claim}
\newtheorem{remark}[theorem]{Remark}
\newtheorem{example}[theorem]{Example}
\newtheorem*{example-nn}{Example}
\newtheorem*{ack}{Aknowledgments}
\newtheorem*{funding}{Funding}

\numberwithin{section}{chapter}
\numberwithin{equation}{chapter}

\newenvironment{cproof}{\begin{proof}[Proof of the claim]}{\end{proof}}

\newcommand{\derived}{\mathfrak{D}}
\newcommand{\fgr}[1]{[#1\mkern1.5mu]}
\newcommand{\lfgr}[1]{[#1\mkern1.5mu]_{1}}

\newcommand{\pfgr}{\mathrm{PF}}
\newcommand{\pafgr}[1]{[\mkern-5mu[#1\mkern1.5mu]\mkern-4.5mu]}

\newcommand{\isom}[1]{\mathrm{Iso}^{\star}_{#1}}

\newcommand{\Aut}{\mathrm{Aut}}
\newcommand{\Autm}[1]{\mathrm{Aut}^{\star}_{#1}}

\newcommand{\ind}{\mathcal{I}}
\newcommand{\einv}{\mathrm{EINV}}
\newcommand{\comm}{\mathfrak{C}}
\newcommand{\acts}{\curvearrowright}

\newcommand{\gap}{\mathrm{gap}}
\DeclareMathOperator*{\supp}{supp}
\newcommand{\proj}{\mathrm{proj}}
\newcommand{\ismph}{\simeq}

\newcommand{\inv}{^{-1}}
\newcommand{\abs}[1]{\left\lvert #1\right\rvert}
\newcommand{\norm}[1]{\left\lVert #1\right\rVert}
\newcommand{\snorm}[1]{\lVert #1\rVert}
\newcommand{\winv}{\mathrm{inv}}
\newcommand{\reor}[1]{\tilde{#1}}

\newcommand{\MAlg}{\mathrm{MAlg}}
\newcommand{\LL}{\mathrm{L}}
\newcommand{\R}{\mathbb{R}}
\newcommand{\Z}{\mathbb{Z}}
\newcommand{\Q}{\mathbb{Q}}
\newcommand{\N}{\mathbb{N}}

\newcommand{\tprank}{\mathrm{rk}}
\newcommand{\eqr}{\mathcal{R}}
\newcommand{\ceqr}{\mathcal{R}} 
\newcommand{\copious}{\xi}

\DeclareMathOperator{\dom}{\mathrm{dom}}
\DeclareMathOperator{\rng}{\mathrm{rng}}
\DeclareMathOperator*{\esssup}{ess\,sup}

\title{\( \mathbf{L}^{\mathbf{1}} \) Full Groups of Flows}

\author{Fran\c{c}ois Le Ma\^\i{}tre and Konstantin Slutsky}

\date{}

\makeindex

\begin{document}

\maketitle

\frontmatter

\chapter*{Abstract}
  We introduce the concept of an \( \LL^{1} \) full group associated with a measure-preserving
  action of a Polish normed group on a standard probability space. These groups carry a natural
  Polish group topology induced by an \( \LL^{1} \) norm.  Our construction generalizes
  \( \LL^{1} \) full groups of actions of discrete groups, which have been studied recently by the
  first author.

  We show that under minor assumptions on the actions, topological derived subgroups of
  \( \LL^{1} \) full groups are topologically simple and --- when the acting group is locally
  compact and amenable --- are whirly amenable and generically two-generated.  \(\LL^{1}\) full
  groups of actions of compactly generated locally compact Polish groups are shown to remember the
  \(\LL^{1}\) orbit equivalence class of the action.

  For measure-preserving actions of the real line (also often called measure-preserving flows), the
  topological derived subgroup of the \( \LL^{1} \) full group is shown to coincide with the kernel
  of the index map, which implies that \( \LL^{1} \) full groups of free measure-preserving flows
  are topologically finitely generated if and only if the flow admits finitely many ergodic
  components. The latter is in striking contrast to the case of \( \mathbb{Z} \)-actions, where the
  number of topological generators is controlled by the entropy of the action. We also prove a
  reconstruction-type result: the \(\LL^{1}\) full group completely characterizes the associated
  ergodic flow up to flip Kakutani equivalence.

  Finally, we study the coarse geometry of \(\LL^{1}\) full groups.  The \(\LL^{1}\) norm on the
  derived subgroup of the \(\LL^{1}\) full group of an aperiodic action of a locally compact
  amenable group is proved to be maximal in the sense of C.~Rosendal.  For measure-preserving flows,
  this holds for the \(\LL^{1}\) norm on all of the \(\LL^{1}\) full group.


\begin{funding}
  François Le Maître's research was partially supported by the ANR project AGRUME
  (ANR-17-CE40-0026), the ANR Project AODynG (ANR-19-CE40-0008) and the IdEx Université de Paris
  (ANR-18-IDEX-0001).  Konstantin Slutsky's research was partially supported by the ANR project
  AGRUME (ANR-17-CE40-0026) and NSF Grant DMS-2153981.
\end{funding}


\tableofcontents

\mainmatter


\chapter{Introduction}
\label{chap:introduction}

Full groups were introduced by H.~Dye~\cite{dyeGroupsMeasurePreserving1959} in the framework of
measure-preserving actions of countable groups as measurable analogues of unitary groups of von
Neumann algebras, by mimicking the fact that the latter are stable under countable cutting and
pasting of partial isometries. These Polish groups have since been recognized as important
invariants as they encode the induced partition of the space into orbits. A similar viewpoint
applies in the setup of minimal homeomorphisms on the Cantor space~\cite{MR1710743}, where likewise
the full groups are responsible for the orbit equivalence class of the action.

Full groups are defined to consist of transformations which act by a permutation on each orbit. When
the action is free, one can associate with an element \( h \) of the full group a cocycle defined by
the equation \( h(x) = \rho_h(x)\cdot x \). From the point of view of topological dynamics, it is
natural to consider the subgroup of those \( h \) for which the cocycle map is continuous, which is
the defining condition for the so-called topological full groups. The latter has a much tighter
control of the action, and encodes minimal homeomorphisms of the Cantor space up to flip-conjugacy
(see~\cite{MR1710743}).

A celebrated result of H.~Dye states that all ergodic \( \mathbb{Z} \)-actions produce the same
partition up to isomorphism, and hence the associated full groups are all isomorphic. The first
named author has been motivated by the above to seek for the analog of topological full groups in
the context of ergodic theory, which was achieved in~\cite{MR3810253} by imposing integrability
conditions on the cocycle. In particular, he introduced \( \LL^1 \) full groups of
measure-preserving ergodic transformations, and showed based on the result of
R.~M.~Belinskaja~\cite{MR0245756} that they also determine the action up to flip-conjugacy. Unlike
in the context of Cantor dynamics, these \( \LL^1 \) full groups are uncountable, but they carry a
natural Polish topology.

In this work, we widen the concept of an \( \LL^{1} \) full group and associate such an object with
any measure-preserving Borel action of a Polish normed group (the reader may consult
Appendix~\ref{chap:normed-groups} for a concise reminder about group norms). Quasi-isometric
compatible norms will result in the same \( \LL^{1} \) full groups, so actions of Polish boundedly
generated groups have canonical \( \LL^{1} \) full groups associated with them based on to the work
of C.~Rosendal~\cite{MR4327092}. Our study also parallels the generalization of the full group
construction introduced by A.~Carderi and the first named author in~\cite{MR3464151}, where full
groups were defined for Borel measure-preserving actions of Polish groups.

\section{Main results}
\label{sec:main-results}

Let \( G \) be a Polish group with a compatible norm \( \snorm{\cdot} \) and consider a Borel
measure-preserving action \( G \acts X \) on a standard probability space \( (X, \mu) \). The group
action defines an orbit equivalence relation \( \eqr_{G} \) by declaring points
\( x_{1}, x_{2} \in X \) equivalent whenever \( G \cdot x_{1} = G \cdot x_{2} \). The norm induces a
metric onto each \( \eqr_{G} \)-class via
\( D(x_{1}, x_{2}) = \inf_{g \in G} \{ \snorm{g} : gx_{1} = x_{2}\} \). Following~\cite{MR3464151},
the \textbf{full group} of the action is denoted by \( \fgr{\eqr_{G}} \) and is defined as the
collection of all measure-preserving \( T \in \Aut(X, \mu) \) that satisfy \( x \eqr_{G} Tx \) for
all \( x \in X \). The \( \LL^{1} \) \textbf{full group} \( \lfgr{G \acts X} \) is given by those
\( T \in \fgr{\eqr_{G}} \) for which the map \( X \ni x \mapsto D(x, Tx) \) is integrable. This
defines a subgroup of \( \fgr{\eqr_{G}} \), and we show in
Theorem~\ref{thm:l1-full-groups-are-Polish} that these groups are Polish in the topology of the norm
\( \snorm{T} = \int_{X}D(x, Tx)\, d\mu(x) \). The strategy of establishing this statement is analogous
to that of~\cite{MR3748570}, where the Polish topology for full groups \( \fgr{\eqr_{G}} \) was
defined.

Understanding of the structure of various types and variants of full groups often hinges on
examining their derived subgroups (also called commutator subgroups). This holds true for our setup
as well. Since we're dealing with full groups equipped with non-discrete topologies, we focus on
their \emph{topological} derived subgroup, defined as the \emph{closure} of the subgroup generated
by commutators.

\begin{theorem}\label{thm:derived-gen-by-invol}
  The topological derived subgroup of any aperiodic \( \LL^1 \) full group is equal to the closed
  subgroup generated by involutions.
\end{theorem}

The argument needed for Theorem~\ref{thm:derived-gen-by-invol} is quite robust. We extract the
idea used in~\cite{MR3810253}, isolate the class of finitely full groups, and show that under mild
assumptions on the action, Theorem~\ref{thm:derived-gen-by-invol} holds for such groups. We
provide these arguments in Section~\ref{sec:app-fin-ful} and in
Corollary~\ref{cor:all-subgroups-are-equal} in particular. Alongside we mention
Corollary~\ref{cor:l1-topological-simplicity} which implies that \( \LL^{1} \) full groups of
ergodic actions are topologically simple.

For the rest of our results we narrow down the generality of the acting groups, and consider
\emph{locally compact} Polish normed groups.  In Chapter~\ref{chap:l1-locally-compact}, we show that
if \( H < G \) is a dense subgroup of a locally compact Polish normed group \( G \) then
\( \lfgr{H \acts X} \) is dense in \( \lfgr{G \acts X} \).  In fact, we prove a considerably
stronger statement by showing that for each \( T \in \fgr{G \acts X} \) and \( \epsilon > 0 \) there
is \( S \in \fgr{H \acts X} \) such that \( \esssup_{x \in X} D(Tx, Sx) < \epsilon \).

Recall that a topological group is \textbf{amenable} if
all of its continuous actions on compact spaces preserve some Radon probability measure, and that it
is \textbf{whirly amenable} if it is amenable and moreover every invariant Radon measure is
supported on the set of fixed points. The following is a combination of
Theorem~\ref{thm:derived-whirly-amenable} and Corollary~\ref{cor:amenability-equivalences}.

\begin{theorem}
  Let \( G \acts X \) be a measure-preserving action of a locally compact Polish normed group.
  Consider the following three statements:
  \begin{enumerate}
    \item\label{item:intro-G-amenable} \( G \) is amenable;
    \item\label{item:intro-derived-L1-full-group-whirly-amenable} the (topological) derived subgroup
          \( \derived(\lfgr{G\acts X}) \) is whirly amenable;
    \item\label{item:intro-L1-full-group-amenable} the \( \LL^1 \) full group \( \lfgr{G\acts X} \) is
          amenable.
  \end{enumerate}
  The  implications~\eqref{item:intro-G-amenable}\( \implies
  \)~\eqref{item:intro-derived-L1-full-group-whirly-amenable}\( \implies
  \)~\eqref{item:intro-L1-full-group-amenable}
  always hold. If \( G \) is unimodular and the action is free, then the three statements above are
  all equivalent.
\end{theorem}

When the acting group is amenable and the orbits of the action are uncountable, we are able to
compute the
\textbf{topological rank}\index{Polish group!topological rank}\index{Topological rank}
of the derived \( \LL^1 \) full groups --- that is, the minimal number
of elements required to generate a dense subgroup of the closure of the commutator subgroup.
Theorem~\ref{thm:generically-2-generated-derived} provides a stronger version of the following.

\begin{theorem}\label{thm:topo-rank-derived}
  Let \( G \acts X \) be a measure-preserving action of an amenable locally compact Polish normed
  group on a standard probability space \( (X, \mu) \). If all orbits of the action are uncountable,
  then the topological rank of the derived \( \LL^{1} \) full group \( \derived(\lfgr{G \acts X}) \) is
  equal to \( 2 \).
\end{theorem}

It is instructive to contrast the situation with the actions of finitely generated groups, where
finiteness of the topological rank of the derived \( \LL^1 \) full group is equivalent to finiteness
of the Rokhlin entropy of the action~\cite{MR4398251}.

Our most refined understanding of \( \LL^{1} \) full groups is achieved for free actions of
\( \mathbb{R} \), which are known as \textbf{flows}. All the results we described so far are valid
for all compatible norms on the acting group. When it comes to the actions of \( \mathbb{R} \),
however, we consider only the standard Euclidean norm on it. Just like the actions of
\( \mathbb{Z} \), flows give rise to an important homomorphism, known as the \textbf{index map}.
Assuming the flow is ergodic, the index map can be described most easily as
\( \lfgr{\mathbb{R} \acts X} \ni T \mapsto \int_{X} |\rho_{T}|\, d\mu \), where \( \rho_{T} \) is
the cocycle of \( T \). Chapter~\ref{chap:index-map} is devoted to the analysis of the index map for
general \( \mathbb{R} \)-flows.

The most technically challenging result of our work is summarized in
Theorem~\ref{thm:index-kernel-is-derived-subgroup}, which identifies the derived \( \LL^{1} \) full
group of a flow with the kernel of the index map, and describes the abelianization of
\( \lfgr{\mathbb{R} \acts X} \).

\begin{theorem}\label{thm:kernel-index-map}
  Let \( \mathcal F \) be a measure-preserving flow on \( (X,\mu) \). The kernel of the index map is
  equal to the derived \( \LL^1 \) full group of the flow, and the topological abelianization of
  \( \lfgr{\mathcal{F}} \) is \( \mathbb{R} \).
\end{theorem}

Theorem~\ref{thm:kernel-index-map} parallels the known results for
\( \mathbb{Z} \)-actions from~\cite{MR3810253}. The structure of its proof, however, has an
important difference. We rely crucially on the fact that each element of the full group
acts in a measure-preserving manner on each orbit. This allows us
to use Hopf decomposition (described in
Appendix~\ref{sec:hopf-decomposition-appendix}) in order to separate any given element
\( T \in \lfgr{\mathbb{R} \acts X} \) into two parts --- recurrent and dissipative. If the acting
group were discrete, the recurrent part would reduce to periodic orbits only.
This is not at all the case
for non-discrete groups, hence we need a new machinery to understand non-periodic recurrent
transformations. To cope with this, we introduce the concept of an \textbf{intermitted}
transformation, which plays the central role in Chapter~\ref{chap:intermitted-transformations}, and
which we hope will find other applications.

Theorems~\ref{thm:topo-rank-derived} and~\ref{thm:kernel-index-map} can be combined to obtain
estimates for the topological rank of the whole \( \LL^1 \) full groups of flows, which is the
content of Proposition~\ref{prop:rank-inequality-l1-full-group}.

\begin{theorem}
  Let \( \mathcal{F} \) be a free measure-preserving flow on a standard probability space
  \( (X, \mu) \). The topological rank \( \tprank(\lfgr{\mathcal{F}}) \) is finite if and only if
  the flow has finitely many ergodic components. Moreover, if \( \mathcal{F} \) has exactly \( n \)
  ergodic components then
  \[ n+1 \le \tprank(\lfgr{\mathcal{F}}) \le n+3. \]
\end{theorem}

In particular, the topological rank of the \( \LL^1 \) full group of an ergodic flow is equal to
either \( 2 \), \( 3 \) or \( 4 \). We conjecture that it is always equal to \( 2 \), and more
generally that the topological rank of the \( \LL^1 \) full group of any measure-preserving flow is
equal to \( n+1 \) where \( n \) is the number of ergodic components.

Our work connects to the notion of \(\LL^1\) orbit equivalence, an intermediate notion between orbit
equivalence and conjugacy.  It goes back to the work of R.~M.~Belinskaja~\cite{MR0245756}
but recently attracted more attention.  Stated in our framework, two flows are \(\LL^1\) orbit
equivalent if they can be conjugated so that the first flow is contained in the \(\LL^1\) full group
of the second and vice versa.  A symmetric version of Belinskaja's theorem is that ergodic
\(\Z\)-actions are \(\LL^1\) orbit equivalent if and only if they are flip conjugate.  It is very
natural to wonder whether this amazing result has a version for flows.  Our
Theorem~\ref{thm:flip-Kakutani-equivalence-same-orbits} implies the following.

\begin{theorem}\label{thm:l1oe-implies-flip-kak}
  If two measure-preserving ergodic flows are \(\LL^1\) orbit equivalent, then they admit some
  cross-sections whose induced transformations\footnote{We refer the reader to
    Definition~\ref{def:flip-kak-eq} and the paragraph that follows it for details on the
    measure-preserving transformation one associates to a cross-section.} are flip-conjugate.
\end{theorem}

We do not know whether the above result is optimal, that is, whether having flip-conjugate
cross-sections implies \(\LL^1\) orbit equivalence, but it seems unlikely.  It is tempting to think
that the correct analogue of Belinskaja's theorem would be a positive answer to the following
question.

\begin{question}
  Let \( \mathcal{F}_{1} \) and \( \mathcal{F}_{2} \) be free ergodic measure-preserving flows which
  are \(\LL^1\) orbit equivalent.  Is it true that there is \( \alpha\in\R^* \) such that
  \( \mathcal{F}_{1} \) and \( \mathcal{F}_{2}\circ m_\alpha \) are isomorphic, where \( m_\alpha \)
  denotes the multiplication by \( \alpha \)?
\end{question}

Let us also mention that Theorem~\ref{thm:l1oe-implies-flip-kak} implies that there are
uncountably many \(\LL^1\) full groups of ergodic free measure-preserving flows up to (topological)
group isomorphism (see Corollary~\ref{cor:flip-kakutani-equivalence-same-l1} and the paragraph right
after its proof).

Finally, we also investigate the coarse geometry of the \(\LL^{1}\) full groups.  We establish that
the \(\LL^{1}\) norm is maximal (in the sense of C.~Rosendal~\cite{MR4327092}, see also
Appendix~\ref{sec:maximal-norms}) on the derived subgroup of an \(\LL^{1}\) full group of an
aperiodic measure-preserving action of any locally compact amenable Polish group
(Theorem~\ref{thm:lc-amenable-derived-maximal}).  For the measure-preserving flows, the \(\LL^{1}\)
norm is, in fact, maximal on the whole full group (Theorem~\ref{thm:l1-norm-maximal}).

\begin{ack}
We are deeply thankful to the referee for their careful reading and numerous
detailed remarks, which led to many improvements of the present monograph.
\end{ack}

\section{Preliminaries}
\label{sec:preliminaries}

We assume the reader is familiar with the fundamentals of real analysis and measure theory, as presented in standard
textbooks such as~\cite{Folland,Royden}. For the necessary results in descriptive set theory, we
primarily rely on~\cite{kechris_classical_1995}.

\subsection{Ergodic theory}
Our work belongs to the field of ergodic theory, which means that all the constructions are defined
and results are proven up to null sets. We occasionally phrase our results as holding ``for all
\( x \)'' when strictly speaking they hold only ``for almost all \( x \)''. The only part where
certain care needs to be exercised in this regard appears in Chapter~\ref{chap:full-groups-metric},
where we define \( \LL^{1} \) full groups for Borel measure-preserving actions of Polish normed
groups.  As in~\cite{MR3464151}, these definitions require genuine actions rather than \emph{boolean
  actions}---a distinction that we clarify at the end of this section.  This technicality vanishes
when considering the more restrictive setting of measure-preserving actions of locally compact
Polish groups.

By a \textbf{standard probability space}, we mean a unique (up to isomorphism) separable atomless
measure space \( (X, \mu) \) with \( \mu(X) = 1 \), i.e., the unit interval \( [0,1] \) equipped
with the Lebesgue measure. Occasionally, in Chapter~\ref{chap:derived-l1-full-group} and
Appendices~\ref{sec:disintegration-measure} and~\ref{chap:appendix-actions-lcsc}, we refer to a
\textbf{standard Lebesgue space}, by which we mean a separable finite measure space,
\( \mu(X) < \infty \).  Unlike a standard probability space, this concept allows for the presence of
atoms and does not require normalization.

Throughout, we frequently work with spaces of measurable functions identified up to sets of measure
zero. To simplify notation, we omit explicit references to the underlying measure \( \mu \). For
example, we write \( \LL^1(X, \R) \) instead of \( \LL^1(X, \mu, \R) \) to denote the Banach space
of \( \mu \)-integrable functions \( X \to \R \).

We denote by \(\Aut(X,\mu)\) the group of all measure-preserving bijections of \((X,\mu)\) up to
measure zero.  This is a Polish group when equipped with the \textbf{weak topology}, defined by
\(T_n\to T\) if and only if for all \(A\subseteq X\) Borel, \(\mu(T_n(A)\bigtriangleup
T(A))\to0\). The weak topology is a Polish group topology, see~\cite[Sec.~1]{MR2583950}.  Given
\(T\in\Aut(X,\mu)\), its \textbf{support} is the set
\[ \supp T=\{x\in X : T(x)\neq x\}.\]
We often refer to measure-preserving bijections as (measure-preserving)
\textbf{transformations},\index{Transformation} although they could more precisely be called
\emph{invertible transformations}. Since this work does not involve non-invertible transformations,
this terminology should not cause confusion.  We also consider (measure-preserving)
\textbf{partial transformations}\index{Partial transformation}\index{Transformation!partial}, which are Borel
bijections \(T:A\to B\) between Borel subsets \(A\), \(B\) of \(X\) satisfying
\(\mu(T^{-1}(C))=\mu(C)\) for all Borel \(C\subseteq B\). The set \(A\) is called the domain of
\(T\), denoted \(\dom T\), and \(B\) is its range, denoted \(\rng T\).

A measure-preserving bijection \(T\) is called
\textbf{periodic}\index{Transformation!periodic}\index{Periodic transformation} when almost all of
its orbits are finite. The cardinalities of the finite orbits of \(T\) are called the
\textbf{periods}\index{Period of a transformation} of \(T\).  Periodicity implies the existence of a
\textbf{fundamental domain} \(A\) for \(T\), which is a measurable set that intersects every
\(T\)-orbit at exactly one point.  Since the ambient measure \(\mu\) is finite, the existence of a
fundamental domain actually characterizes periodicity.  We recall that a transformation \(T\) is
called \textbf{aperiodic}\index{Transformation!aperiodic}\index{Aperiodic!transformation} if it has
no periodic points, meaning that all of its orbits are infinite.

When considering actions of full group elements on orbits, we also need to deal with bijections that
preserve only the measure class on a possibly infinite \(\sigma\)-finite standard measured space.
Such bijections are referred to as \textbf{non-singular transformations}\index{Transformation!non-singular}.

As explained at the beginning of this section, full groups are constructed for (Borel)
\textbf{measure-preserving actions}\index{Measure-preserving action} of a given Polish group \(G\)
on a standard probability space \((X,\mu)\).  These actions, called \emph{spatial actions}, are
Borel maps \(\alpha:G\times X\to X\) such that for each \(g\in G\), the transformation \(\alpha(g)\)
is measure-preserving.  A related notions is that of \emph{boolean actions}, which are continuous
group homomorphisms \(G \to \Aut(X,\mu)\). Unlike spatial actions, boolean actions identify
transformations up to null sets. Consequently, a boolean action can a priori be lifted to an action
map \(\alpha\) such that, given \(g, h \in G\), \(\alpha(gh)x=\alpha(g)\alpha(h)x\) holds merely for
\emph{almost} every \(x\in X\).  As discovered by E.~Glasner, B.~Tsirelson and B.~Weiss, boolean
actions (also called near actions) of Polish groups do not admit Borel realizations in general, and
even when they do, it could happen that different realizations yield different full groups. This
subtlety disappears once we shift our attention to locally compact group actions, which is the case
for Chapter~\ref{chap:l1-locally-compact} and onwards. All boolean actions of locally compact Polish
groups admit Borel realizations which are all conjugate up to measure zero (and hence have
isomorphic full groups), so null sets can be neglected just as they always are in ergodic theory.
We refer the reader
to~\cite{glasnerautomorphismgroupGaussian2005,glasnerSpatialNonspatialActions2005} for more
information on this topic.

\subsection{Orbit equivalence relations}

Any group action \( G \acts X \) induces the orbit equivalence relation \( \eqr_{G \acts X} \),
where two points \( x, y \in X \) are \( \eqr_{G \acts X} \)-equivalent whenever
\( G \cdot x = G \cdot y \). We will usually write this equivalence relation simply as \( \eqr_G \)
for brevity. For the actions \( \mathbb{Z} \acts X \) generated by an automorphism
\( T \in \Aut(X, \mu) \), we denote the corresponding orbit equivalence relation by \( \eqr_{T} \).
For clarity, we may sometimes want to name a measure-preserving action as \(\alpha\) and write
\(G\overset{\alpha}{\acts} X\). Then for all \(g\in G\) we denote by \(\alpha(g)\) the
measure-preserving transformation of \((X,\mu)\) induced by the action of \(g\).

We encounter various equivalence relations throughout this monograph. An equivalence class of a
point \( x \in X \) under the relation \( \eqr \) is denoted by \( [x]_{\eqr} \) and the
\textbf{saturation}\index{Set!saturation} of a set \( A \subseteq X \) is denoted by \( [A]_{\eqr} \) and is defined to be
the union of \( \eqr \)-equivalence classes of the elements of \( A \):
\( [A]_{\eqr} = \bigcup_{x \in A} [x]_{\eqr} \). In particular, \( [x]_{\eqr_{T}} \) is the orbit of
\( x \) under the action of \( T \). The reader may notice that the notation for a saturation
\( [A]_{\eqr} \) resembles that for the full group of an action \( \fgr{G \acts X} \) (see
Chapter~\ref{chap:full-groups-metric}). Both notations are standard, and we hope that confusion will
not arise, as it applies to objects of different nature --- sets and actions, respectively.

\subsection{Actions of locally compact groups}

Consider a measure-preserving action of a locally compact Polish (equivalently, second-countable)
group \( G \) on a standard Lebesgue space \( (X,\mu) \). A \textbf{complete section} for the action
is a measurable set \( \mathcal{C} \subseteq X \) that intersects almost every orbit, i.e.,
\( \mu(X \setminus G \cdot \mathcal{C}) = 0 \). A \textbf{cross-section} is a complete section
\( \mathcal{C} \subseteq X \) such that for some non-empty neighborhood of the identity
\( U \subseteq G \) we have \( Uc \cap Uc' = \varnothing \) whenever \( c, c' \in \mathcal{C} \) are
distinct. When the need to mention such a neighborhood \( U \) explicitly arises, we say that
\( \mathcal{C} \) is a \textbf{\( U \)-lacunary} cross-section.

With any cross-section \( \mathcal{C} \) one associates a decomposition of the phase space known as
the Voronoi tessellation. Slightly more generally, Appendix~\ref{sec:tessellations} defines the
concept of a tessellation over a cross-section, which corresponds to a set
\( \mathcal{W} \subseteq \mathcal{C} \times X \) for which the fibers
\( \mathcal{W}_{c} = \{ x \in X : (c, x) \in \mathcal{W}\} \), \( c \in \mathcal{C} \), partition
the phase space. Every tessellation \( \mathcal{W} \) gives rise to an equivalence relation
\( \eqr_{\mathcal{W}} \), where points \( x, y \in X \) are deemed equivalent whenever they belong
to the same fiber \( \mathcal{W}_{c} \).  The projection map
\( \pi_{\mathcal{W}} : X \to \mathcal{C} \) associates with each \( x \in X \) the unique
\( c \in \mathcal{C} \) satisfying \( x \in \mathcal{W}_{c} \), and it is therefore defined by the
condition \( (\pi_{\mathcal{W}}(x), x) \in \mathcal{W} \) for all \( x \in X \).

When the action \( G \acts X \) is free, each orbit \( G \cdot x \) can be identified with the
acting group. Such a correspondence \( g \mapsto g\cdot x \) depends on the choice of the anchor
point~\( x \) within the orbit, but suffices to transfer structures invariant under the right
translations from the group \( G \) onto the orbits of the action. For instance, if the acting group
is locally compact, then a right-invariant Haar measure \( \lambda \) can be pushed onto orbits by
setting \( \lambda_{x}(A) = \{g \in G: g\cdot x \in A \} \) as discussed in
Section~\ref{sec:orbit-transf}.  Freeness of the action \( G \acts X \) gives rise to the
\textbf{cocycle map} \( \rho : \eqr_{G\acts X} \to G \), which is well-defined by the condition
\( \rho(x, y) \cdot x = y \). Elements of the full group \( \fgr{G \acts X} \) are characterized as
measure-preserving transformations \( T\in\Aut(X,\mu) \) such that \( (T(x),x)\in\eqr_{G\acts X} \)
for all \( x\in X \). With each \( T \in \fgr{G \acts X} \) one may therefore associate the map
\( \rho_{T} : X \to G \), also known as the \textbf{cocycle map}, and defined by
\( \rho_{T}(x) = \rho(x, Tx) \). Both the context and the notation will clarify which cocycle map is
being referred to.

\subsection{Measure-preserving free flows}\label{sec:prelim-flows}

All the previous concepts appear prominently in the chapters that deal with free measure-preserving
\emph{flows}, namely free measure-preserving actions of \( \mathbb{R} \) on standard probability
spaces. We use the additive notation for such actions:
\( \mathbb{R} \times X \ni (r,x) \mapsto x + r \in X \). The group \( \mathbb{R} \) carries a
natural linear order that is invariant under the group operation and can therefore be transferred
onto orbits. More specifically, given a free measure-preserving flow \( \mathbb{R} \acts X \), we use
the notation \( x < y \) whenever \( x \) and \( y \) belong to the same orbit and \( y = x + r \)
for some \( r > 0 \). Every cross-section \( \mathcal{C} \) of a free flow intersects each orbit in
a bi-infinite fashion --- each \( c \in \mathcal{C} \) has a unique successor and a unique
predecessor in the order of the orbit. One therefore has a bijection
\( \sigma_{\mathcal{C}} : \mathcal{C} \to \mathcal{C} \), called the \textbf{first return map} or
the \textbf{induced map}\index{Induced map}, which sends \( c \in \mathcal{C} \) to the next element of the
cross-section within the same orbit. We also make use of the gap function that measures the lengths
of intervals of the cross-section, i.e.,
\( \gap_{\mathcal{C}}(c) = \rho(c, \sigma_{\mathcal{C}}(c)) \), and of the projection function
\( \pi_{\mathcal{C}}: X\to \mathcal{C} \) which takes every \(x\in X\) to the largest
\( c\in\mathcal{C} \) such that \( c \leq x \).

There is a canonical tessellation associated with a cross-section \( \mathcal{C} \) which
partitions each orbit into intervals between adjacent points of \( \mathcal{C} \) and is given by
\[ \mathcal{W}_{\mathcal{C}} = \{(c, x) \in \mathcal{C} \times X : c \le x < \sigma_{\mathcal{C}}(c)
  \}.\]
The associated equivalence relation \( \eqr_{\mathcal{W}_{\mathcal{C}}} \) is denoted simply by
\( \eqr_{\mathcal{C}} \). It groups points \( (x, y) \in \eqr_{\mathbb{R} \acts X} \) which belong
to the same interval of the tessellation, \( \pi_{\mathcal{C}}(x) = \pi_{\mathcal{C}}(y) \). The
\( \eqr_{\mathcal{C}} \)-equivalence class of \( x \in X \) is equal to
\( [x]_{\eqr_{\mathcal{C}}} = \pi_{\mathcal{C}}(x) + \bigl[0,
\gap_{\mathcal{C}}(\pi_{\mathcal{C}}(x))\bigr) \).

Often enough we need to restrict sets and functions to an \( \eqr_{\mathcal{C}} \)-class. Since such
a need arises very frequently, especially in Chapter~\ref{chap:peri-appr-monot}, we adopt the
following shorthand notations. Given a set \( A \subseteq X \) and a point \( c \in \mathcal{C} \),
the intersection \( A \cap [c]_{\eqr_{\mathcal{C}}} \) is denoted simply by \( A(c) \).
Likewise,
\( \lambda_{c}^{\mathcal{C}}(A) \) stands for
\( \lambda(\{t \in \mathbb{R} : c + t \in A \cap [c]_{\eqr_{\mathcal{C}}}\}) \) and corresponds to
the Lebesgue measure of the set \( A \cap [c]_{\eqr_{\mathcal{C}}} \).

The phase space \(X\) can be identified with the subset
\(Z_{\mathcal{C}} \subseteq \mathcal{C} \times \mathbb{R}\),
\[Z_{\mathcal{C}} = \{(c,t) : 0\leq t < \gap_{\mathcal{C}}(c)\},\]
via the map \(\Psi : Z_{\mathcal{C}} \to X\) given by \(\Psi(c, t) = c + t\).  Through this
identification, \(\lambda_c\) corresponds to the Lebesgue measure on the fiber of
\(Z_{\mathcal{C}}\) over \(c\).  Moreover, uniqueness of the Lebesgue measure implies uniqueness of
a finite measure \(\nu\) on \(\mathcal{C}\) such that \(\mu\) is the push-forward by \(\Psi\) of the
restriction of \(\nu \times \lambda\) to \(Z_{\mathcal{C}}\) (see, for
instance,~\cite[Prop.~4.3]{kyedBettiNumbersLocally2015}).  The natural disintegration of
\((\mathcal{C} \times \mathbb{R}, \nu \otimes \lambda)\) corresponds to the disintegration of
\(\mu\) of the form \( \mu(A) = \int_{\mathcal{C}} \lambda_{c}^{\mathcal{C}}(A)\, d\nu(c) \) (see
Appendix~\ref{sec:disintegration-measure} for an overview of measure disintegration).



\chapter{\texorpdfstring{\( \LL^{1} \)}{L1} full groups of Polish group actions}
\label{chap:full-groups-metric}

We begin by introducing the central concept of this work, namely the \( \LL^{1} \) full groups of
Borel measure-preserving actions of Polish normed groups on a standard probability space. While our
primary focus will be on actions of locally compact groups, especially flows, the notion of an
\( \LL^{1} \) full group can be introduced for actions of arbitrary Polish normed groups. In
Section~\ref{sec:l1-full-groups}, we present the definitions in this general setting.

In Section~\ref{sec:l1-full-groups-bound}, we examine the case when there is a natural choice
(up to quasi-isometry) of a norm on \(G\), allowing us to
speak of \emph{the} \(\LL^1\) full group of a \(G\)-action.
This framework is applicable to \(G=\R\), yielding the definition of \(\LL^1\) full
groups of measure-preserving flows.

Returning to the general setting, the \(\LL^1\) full group of an action of a Polish normed group is
itself equipped with a natural norm. In Section~\ref{sec:L1-embedding}, we demonstrate that the
resulting metric space is remarkably large: it contains an isometric copy of the
infinite-dimensional Banach space \( \LL^1(X,\R) \).

We conclude this chapter in Section~\ref{sec:stab-under-first} by establishing the closure of the
\(\LL^1\) full group under the operation of taking induced transformations. This result will play a pivotal
role in the subsequent chapters.

\section{\texorpdfstring{\( \LL^{1} \)}{L1} full groups of Polish normed group actions}
\label{sec:l1-full-groups}

A \textbf{Polish normed group}\index{Polish group!normed} is a Polish group equipped with a
compatible norm (see Appendix~\ref{sec:norms-on-groups}). Let \( (G, \norm{\cdot}) \) be a Polish
normed group, and let \( G \acts X \) be a Borel measure-preserving action on a standard probability
space \( (X, \mu) \). Using the norm, we define a metric \(D\) on \(X\), which
may take the value \(+\infty\), as follows:
\begin{equation}\label{eq:def-D}
  D(x,y) = \inf_{u \in G} \{\, \norm{u} : ux = y \,\} \text{ for } (x,y) \in X\times X.
\end{equation}
\begin{remark}
  By definition, the infimum of the empty set is \(+\infty\). Thus, \(D(x,y)=+\infty\) if and only
  if \(x\) and \(y\) belong to distinct \(G\)-orbits.
\end{remark}

The fact that \(D\) is a metric is straightforward except, possibly, for the implication
\( D(x,y) = 0 \implies x = y \). To justify the latter, let \( u_{n} \in G \),
\( n \in \mathbb{N} \), be a sequence such that \( u_{n} \to e \) and \( u_{n}x = y \). The elements
\( u^{-1}_{n}u_{0} \), \( n \in \mathbb{N} \), belong to the stabilizer of \( x \). By Miller's
theorem~\cite{MR440519}, stabilizers of all points are closed, whence
\( u_{0} = \lim_{n}u_{n}^{-1}u_{0} \) fixes \( x \). Thus, \( u_{0}x = x \), and \( x = y \) as
claimed.

\begin{remark}
  When the \(G\)-orbit of \(x_0 \in X\) is identified with the homogeneous space \(G/H\), where
  \(H\) is the stabilizer of \(x_0\), the restriction of the metric \(D\) to the \(G\)-orbit of
  \(x_0\) corresponds to the quotient metric on \(G/H\) induced by the right-invariant metric
  associated with the given norm on \(G\).
\end{remark}

We can then use \(D\) to define a norm (which may take the value \(+\infty\)) on \(\Aut(X,\mu)\),
leading to the definition of \(\LL^1\) full groups.

\begin{definition}
  \label{def:l1-full-group}
  Let \( G \acts X \) be a Borel measure-preserving action of a Polish normed group
  \( (G, \norm{\cdot}) \) on a standard probability space \( (X,\mu) \), and let
  \( D : X \times X \to \mathbb{R}^{\ge 0} \cup \{+\infty\} \) be the associated metric on
  \(X\). The \textbf{\(\LL^{1}\)-norm}\index{L1@\(\LL^{1}\)!norm}\index{Norm!L1@\(\LL^{1}\)} of an
  automorphism \( T \in \Aut(X,\mu) \) is denoted by \( \norm{T}_{1} \) and is defined as the
  integral
  \[
    \norm{T}_{1} = \int_{X} D(x, Tx)\, d\mu(x).
  \]
  In general, many elements of \(\Aut(X,\mu)\) may have an infinite norm. The \textbf{\(\LL^{1}\)
    full group}\index{L1@\(\LL^{1}\)!full group} of the action consists of those automorphisms for
  which the norm is finite:
  \[
    \lfgr{G \acts X} = \{T \in \Aut(X,\mu) : \norm{T}_{1} < +\infty \}.
  \]
\end{definition}

Elements of \( \lfgr{G \acts X} \) form a group under composition, as can be readily verified using
the triangle inequality for \( D \) and the fact that the transformations in the \(\LL^{1}\) full
group are measure-preserving. Moreover, it is straightforward to check that \( \norm{\cdot}_1 \)
defines a norm on \( \lfgr{G \acts X} \). Our primary objective is to prove that this norm
\( \norm{\cdot}_{1} \) induces a Polish group topology on \( \lfgr{G \acts X} \). Before doing so,
however, we establish a connection to the \textbf{full group}\index{Full group} of
the action \( G \acts X \), defined by
\[ [G \acts X] = \bigl\{ T \in \Aut(X, \mu) : (x, T(x)) \in \mathcal{R}_G \textrm{
  for almost all \(x \in X\)} \bigr\}.
\]
Note that since the \(\LL^1\)-norm is given by \(\norm{T}_1 = \int_{X} D(x, Tx) \, d\mu(x)\), if
\(\norm{T}_1 < +\infty\), then \( D(x, Tx) < +\infty \) holds almost surely, implying
\( T \in \fgr{G \acts X} \). Thus, the \(\LL^1\) full group \( \lfgr{G \acts X} \) is a subgroup of
the full group \( \fgr{G \acts X} \).

The full group of the action was shown to be a Polish group by A.~Carderi and the first-named
author in \cite{MR3464151} (where it is referred to as the \emph{orbit} full group). Furthermore,
when the fixed group norm \(\norm{\,\cdot\,}\) is bounded, so is \( D \), which implies that
\( \lfgr{G \acts X} = \fgr{G \acts X} \). Consequently, \(\LL^1\) full groups encompass the full
groups studied in~\cite{MR3464151}. To demonstrate that \(\LL^1\) full groups are Polish, we will
adopt the same approach as in~\cite{MR3464151}. We first provide an alternative definition of the
\(\LL^1\) full group, from which the Polishness of the topology will follow directly, and then show
that the two definitions yield isometrically isomorphic structures. This alternative definition
relies on understanding the cocycles associated with elements of the \(\LL^1\) full group.

Let us introduce some notation from Appendix~\ref{sec:spaces-and-groups-of-measurable-maps}. For a
standard Borel space \(Y\), we denote by \(\LL^0(X,Y)\) the space of measurable maps \(X \to
Y\). When \(Y\) is a Polish normed group \((G, \norm{\cdot})\), we define \(\LL^1(X,G)\) to be the
set of all \(f \in \LL^0(X,G)\) satisfying
\[
  \int_X \norm{f(x)} \, d\mu(x) < +\infty.
\]
Furthermore, given a Borel measure-preserving action \(G \acts X\), we define the map
\(\Phi: \LL^0(X,G) \to \LL^0(X,X)\) by \(\Phi(f)(x) = f(x) \cdot x\) for \(f \in \LL^0(X,G)\) and
\(x \in X\).

\begin{definition}
  Let \( G \acts X \) be a Borel measure-preserving action of a Polish group \( G \) on a standard
  probability space \( (X,\mu) \). A function \( c \in \LL^0(X, G) \) is called a
  \textbf{cocycle}\index{Cocycle} of \( T \in \fgr{G \acts X} \) if \(T(x) = c(x) \cdot x\) holds
  for almost every \( x \in X \).
\end{definition}

Note that if we identify \(\Aut(X, \mu)\) with a subset of \(\LL^0(X, X)\), then \( c \) is a
cocycle of \( T \in [G \acts X] \) precisely when \( T = \Phi(c) \).

\begin{definition}
  Consider a Borel measure-preserving action \( G \acts X \) of a Polish normed group
  \( (G, \norm{\cdot}) \) on a standard probability space \( (X, \mu) \). The \textbf{\(\LL^1\)
    pre-full group}\index{Pre-full group} \( \pfgr^{1} \) is defined as
  \[
    \pfgr^{1} = \Phi^{-1}(\Aut(X, \mu)) \cap \LL^1(X, G).
  \]
\end{definition}

In other words, the \(\LL^1\) pre-full group consists of all \emph{integrable} cocycles.

\begin{remark}
  When the group norm \(\norm{\cdot}\) on \( G \) is bounded, the integrability condition becomes
  trivial.  In this case, \( \pfgr^{1} = \Phi^{-1}(\Aut(X, \mu)) \) coincides with the pre-full
  group \( \pfgr \) as defined in~\cite[p.~91]{MR3464151}. The latter was shown to be Polish in the
  topology of convergence in measure induced by \(\LL^0(X,G)\), and our next result encompasses this
  in view of Proposition~\ref{prop:polish-topo-cv-measure}.
\end{remark}

We equip the \(\LL^1\) pre-full group with the topology induced by \(\LL^1(X,G)\), which arises from
the norm
\[
  \norm{f}_{1}^{\LL^{1}(X, G)} = \int_X \norm{f(x)} \, d\mu(x).
\]
By Proposition~\ref{prop:l1-is-a-Polish-group}, \(\LL^1(X,G)\) is a Polish normed group under
pointwise multiplication.

Following the approach in~\cite[p.~91]{MR3464151}, we lift the composition law from \(\Aut(X,\mu)\)
to cocycles as follows: for \( f, g \in \pfgr^1 \) and \( x \in X \), define the multiplication by
\[
  (f * g)(x) = f(\Phi(g)(x))g(x),
\]
and the inverse\footnote{The symbol \( f^{-1} \) is already reserved for the pointwise inverse on
  \(\LL^{1}(X, G)\). To avoid confusion, we introduce a distinct notation for this new operation.}
by
\[
  \winv(f)(x) = f(\Phi(f)^{-1}(x))^{-1}.
\]

\begin{proposition}
  \label{prop:pf-is-Polish}
  \( \pfgr^{1} \) is a Polish group with the multiplication \( (f, g) \mapsto (f * g) \) and the
  inverse \( f \mapsto \winv(f) \). The function \( f \mapsto \norm{f}_{1}^{\LL^{1}(X, G)} \) is a
  compatible group norm on \( \pfgr^{1} \), and
  \( \Phi \restriction_{\pfgr^{1}} : \pfgr^{1} \to \Aut(X, \mu) \) is a continuous homomorphism.
\end{proposition}

\begin{proof}
  First of all, we need to show that these operations are well-defined in the sense that the
  functions \( f * g \) and \( \winv(f) \) belong to \( \LL^{1}(X, G) \) whenever \( f \) and
  \( g \) do. To this end, observe that for \( f, g \in \pfgr^{1} \),
  \begin{displaymath}
    \begin{aligned}
      \snorm{f * g}_{1}^{\LL^{1}(X, G)}
      &= \int_{X} \norm{f(\Phi(g)(x))g(x)}\, d\mu(x) \\
      &\le \int_{X} \norm{f(\Phi(g)(x))}\, d\mu(x) + \int_{X} \norm{g(x)}\, d\mu(x).
    \end{aligned}
  \end{displaymath}
  Now note that since \( \Phi(g) \) is measure-preserving, we have
  \[
    \int_{X} \norm{f(\Phi(g)(x))}\, d\mu(x) = \int_{X} \norm{f(x)}\, d\mu(x),
  \]
  and therefore
  \[
    \norm{f * g}_{1}^{\LL^{1}(X, G)} \le \int_{X} \norm{f(x)}\, d\mu(x) + \int_{X} \norm{g(x)}\,
    d\mu(x) = \norm{f}_{1}^{\LL^{1}(X, G)} + \norm{g}_{1}^{\LL^{1}(X, G)}.
  \]
  In particular, \( f * g \in \LL^{1}(X, G) \), and thus \( \pfgr^{1} \) is closed under
  multiplication. Similarly, \( \Phi(f) \in \Aut(X, \mu) \) implies
  \begin{displaymath}
    \begin{aligned}
      \norm{\winv(f)}^{\LL^{1}(X, G)}_{1}
      &= \int_{X} \snorm{f(\Phi(f)^{-1}(x))^{-1}}\, d\mu(x) \\
      &= \int_{X} \snorm{f(x)^{-1}}\, d\mu(x) = \norm{f}^{\LL^{1}(X, G)}_{1}.
    \end{aligned}
  \end{displaymath}
  Thus \( \pfgr^{1} \) is closed under taking inverses. We leave it to the reader to verify that the
  operation \(*\) endows \(\pfgr^1\) with a group structure, where the inverse of an element
  \(f \in \pfgr^1\) is given by \(\winv(f)\). Additionally, we have shown that
  \( \snorm{\cdot}^{\LL^{1}(X, G)}_{1} \) defines a group norm on \( \pfgr^{1} \).
  It remains to establish the following properties:
  \begin{itemize}
  \item The topology induced by \(\LL^1(X,G)\) on \(\pfgr^1\) is Polish.
  \item The topology induced by \(\LL^1(X,G)\) on \(\pfgr^1\) is a group topology.
  \item The restriction of the norm \( \norm{\cdot}^{\LL^{1}(X, G)}_{1} \) to \( \pfgr^{1} \) is
    compatible with this group topology.
  \end{itemize}
  Note that the third property follows directly from the first two, as the norm
  \(\norm{\cdot}^{\LL^{1}(X, G)}_{1}\) determines the topology of \(\LL^1(X,G)\), even though the
  latter is equipped with distinct group operations.

  We are thus left with verifying that the topology induced by \(\LL^1(X,G)\) on \(\pfgr^{1}\) is a
  Polish group topology. To achieve this, we equip \(X\) with a Polish topology that induces its
  standard Borel structure and ensures the continuity of the \(G\)-action on \(X\). Such a topology
  is guaranteed by a well-known result of H.~Becker and
  A.~S.~Kechris~\cite[Thm.~5.2.1]{MR1425877}. With this topology in place, \(\LL^0(X,X)\) can be
  endowed with the topology of convergence in measure, and the evaluation map
  \[
    \Phi : \LL^{0}(X, G) \to \LL^{0}(X, X), \quad \Phi(f)(x) = f(x) \cdot x,
  \]
  becomes continuous by Lemma~\ref{lem:continuity-on-L0}.

  Since the topology of \(\LL^1(X,G)\) refines that of \(\LL^0(X,G)\), the restriction of \(\Phi\)
  to \(\LL^1(X,G)\) remains continuous under its Polish group topology. The continuity of the group
  operations now follows directly from the continuity of \(\Phi\), combined with the continuity of
  the \(\Aut(X,\mu)\)-action on \(\LL^1(X,G)\) (Proposition~\ref{prop:aut-acts-continuously}) and
  the fact that \(\LL^1(X,G)\) forms a topological group under pointwise multiplication
  (Proposition~\ref{prop:l1-is-a-Polish-group}).

  We now establish that the induced topology on \(\pfgr^{1}\) is Polish by invoking Alexandrov's
  theorem. This theorem states that a subspace of a Polish space is Polish under the induced
  topology if and only if it is a \(G_\delta\) set
  (see~\cite[Thm.~3.9]{kechris_classical_1995}). The forward direction of Alexandrov's theorem,
  together with Proposition~\ref{prop:Aut-embeds}, implies that the Polish group \(\Aut(X,\mu)\) is
  a \(G_\delta\) subset of \(\LL^0(X,X)\). By the continuity of \(\Phi\), it follows that
  \(\pfgr^{1} = \Phi^{-1}(\Aut(X,\mu)) \cap \LL^1(X,G)\) is a \(G_\delta\) subset of
  \(\LL^1(X,G)\). Consequently, the reverse implication of Alexandrov's theorem ensures that
  \(\pfgr^{1}\) is Polish.
\end{proof}

\begin{remark}
  Our arguments rely fundamentally on the results of H.~Becker and A.~S.~Kechris~\cite{MR1425877},
  which allow us to transform Borel actions into continuous ones. We note that the application
  of~\cite[Thm.~5.2.1]{MR1425877} in the proof of Proposition~\ref{prop:pf-is-Polish} can be
  replaced with the more straightforward result~\cite[Thm.~2.6.6]{MR1425877}: every Borel
  \(G\)-action admits a Borel embedding into a continuous \(G\)-action on a compact Polish space. By
  equipping this space with the push-forward measure, we can proceed with our analysis, as the Borel
  embedding ensures that the actions are isomorphic up to a null set.
\end{remark}

Let \( K \trianglelefteq \pfgr^{1} \) denote the kernel of \( \Phi \restriction_{\pfgr^{1}} \), and
let \( \norm{\cdot}^{\pfgr^{1}/K}_{1} \) denote the quotient norm induced by
\( \norm{\cdot}^{\LL^{1}(X, G)}_{1} \) (see Proposition~\ref{prop:quotient-norm} regarding the
properties of the quotient norm). The factor group
\( (\pfgr^{1}/K, \norm{ \cdot }^{\pfgr^{1}/K}_{1}) \) is evidently a Polish normed group, and it
turns out to be isometrically isomorphic to the \( \LL^{1} \) full group introduced in
Definition~\ref{def:l1-full-group}, as we will now see. Let
\( \tilde{\Phi} : \pfgr^{1}/K \to \Aut(X, \mu) \) denote the homomorphism induced by
\( \Phi \restriction_{\pfgr^{1}} \) onto the factor group.

\begin{proposition}
  \label{prop:pf-over-K-is-l1-full-group}
  The homomorphism \( \tilde{\Phi} : \pfgr^{1}/K \to \Aut(X, \mu) \) establishes an isometric
  isomorphism between \( (\pfgr^{1}/K, \norm{ \cdot }^{\pfgr^{1}/K}_{1}) \) and
  \( (\lfgr{G \acts X}, \norm{\cdot}_{1}) \).
\end{proposition}

\begin{proof}
  We begin by showing that \( \snorm{gK}^{\pfgr/K}_{1} = \snorm{\tilde{\Phi}(gK)}_{1} \) holds for
  any \( gK \in \pfgr^{1}/K \). By the definition of the quotient norm,
  \[ \norm{gK}^{\mathrm{PF^{1}/K}}_{1} = \inf_{k \in K} \int_{X} \norm{g(x)k(x)}\, d\mu(x). \]
  For any fixed \( k \in K \), we have \( g(x)k(x) \cdot x = g(x) \cdot x \), and therefore
  \[ D(x, g(x) \cdot x) \le \norm{g(x)k(x)} \quad \textrm{for almost every \( x \in X \)}. \]
  This readily implies the inequality
  \( \snorm{\tilde{\Phi}(gK)}_{1} \le \norm{gK}^{\pfgr^{1}/K}_{1} \). For the other direction, let
  \( \epsilon > 0 \) and consider the set
  \[ \{(x,u) \in X \times G : g(x) \cdot x = u \cdot x \textrm{ and } \norm{u} \le D(x, g(x) \cdot
    x) + \epsilon\}. \]
  Using the Jankov--von Neumann uniformization theorem, one may pick a measurable map
  \( g_{0} : X \to G \) that satisfies \( g_{0}(x) \cdot x = g(x) \cdot x \) and
  \( \norm{g_{0}(x)} \le D(x, g(x) \cdot x) + \epsilon \) for almost all \( x \in X \). Since
  \( x \mapsto g(x)^{-1}g_{0}(x) \in K \), we have
  \begin{displaymath}
    \begin{aligned}
      \snorm{\tilde{\Phi}(gK)}_{1}
      &= \int_{X}D(x,g(x) \cdot x)\, d\mu(x) \\
      &\ge \int_{X}\norm{g(x)g(x)^{-1}g_{0}(x)}\, d\mu(x) - \epsilon \\
      &\ge \norm{gK}_{1}^{\pfgr^{1}/K} - \epsilon.
    \end{aligned}
  \end{displaymath}
  As \( \epsilon \) is an arbitrary positive real, we conclude that
  \( \norm{gK}_{1}^{\pfgr^{1}/K} = \snorm{\tilde{\Phi}(gK)}_{1} \).

  It remains to check that \( \tilde{\Phi} \) is surjective. For an automorphism
  \( T \in \lfgr{G \acts X} \), consider the set
  \[ \{(x,u) \in X \times G : Tx = u \cdot x \textrm{ and } \norm{u} \le D(x, Tx) + 1\}. \]
  Applying the Jankov--von Neumann uniformization theorem once again, we get a map
  \( g \in \LL^{0}(X, G) \) such that \( \Phi(g) = T \) and \( \norm{g(x)} \le D(x,Tx) + 1 \). The
  latter inequality, together with the assumption that \( T \in \lfgr{G \acts X} \), easily implies
  that \( g \in \LL^{1}(X, G) \), and thus \( gK \in \pfgr^{1}/K \) is the preimage of \( T \) under
  \( \tilde{\Phi} \).
\end{proof}

Results discussed thus far can be summarized as follows.

\begin{theorem}
  \label{thm:l1-full-groups-are-Polish}
  Let \( G \acts X \) be a Borel measure-preserving action of a Polish normed group
  \( (G, \norm{\cdot}) \) on a standard probability space. The \( \LL^{1} \) full group
  \( \lfgr{G \acts X} \) is a Polish normed group relative to the norm
  \( \norm{T}_{1} = \int_{X} D(x, Tx)\, d\mu(x) \).
\end{theorem}

\begin{remark}
  When the acting group is finitely generated and equipped with the word length metric with respect
  to a finite generating set, it can be shown that the left-invariant metric induced by the norm
  on the \(\LL^1\) full group is complete (see~\cite[Prop.~3.4 and 3.5]{MR3810253} and the remark
  thereafter for a more general statement). Nevertheless, \(\LL^1\) full groups generally do not
  admit compatible complete left-invariant metrics, i.e., they are not necessarily CLI groups. For
  instance, if \(G = \R\) is acting by rotation on the circle, the \(\LL^1\) full group of the
  action is all of \(\Aut(\mathbb S^1,\lambda)\), which is not CLI.
\end{remark}

Let us point out a possibility to generalize our framework. Given a standard probability space
\((X,\mu)\), consider an extended Borel metric \(D\) on \(X\), i.e., a Borel metric that is allowed
to take the value \(+\infty\) (Eq.~\eqref{eq:def-D} provides such an example). Note that the
relation \(D(x,y)<+\infty\) is an equivalence relation. One can now define the \(\LL^1\) full group
of \(D\) in complete analogy with Definition~\ref{def:l1-full-group} as the group of all
\(T\in\Aut(X,\mu)\) whose norm \(\norm{T}_{D} = \int_X D(x,T(x))\, d\mu(x) \) is finite.

\begin{question}
  Suppose that \(D\) restricts to a complete separable metric on each equivalence class
  \(\{y\in X : D(x,y)<+\infty\}\), \(x \in X\). Is the \(\LL^1\) full group of \(D\) Polish in the
  topology of the norm \(\norm{\cdot}_{D}\)?
\end{question}

\section{\texorpdfstring{\( \LL^{1} \)}{L1} full groups and quasi-metric structures}
\label{sec:l1-full-groups-bound}

When viewed as a normed group, the \(\LL^{1}\) full group \(\lfgr{G \acts X}\) depends on the choice
of a compatible norm on \(G\). The topological structure on \(\lfgr{G \acts X}\), however, depends
only on the quasi-metric structure of the acting group.  Recall that two norms \( \snorm{\cdot} \)
and \( \snorm{\cdot}' \) on a Polish group \( G \) are
\textbf{quasi-isometric}\index{Norm!quasi-isometric} if there exists a constant \( C>0 \) such that
for all \( g\in G \),
\[
  \frac{1}{C} \snorm{g} - C \le \snorm{g}' \le C \snorm{g} + C.
\]
\begin{lemma}\label{lem:quasi-metric-structure-same-L1}
  Let \( \snorm\cdot \) and \( \snorm{\cdot}' \) be two quasi-isometric compatible norms on a Polish
  group \( G \), and let \( G\acts (X,\mu) \) be a Borel measure-preserving action on a standard
  probability space. The \( \LL^1 \) full groups associated with the two norms are equal as
  topological groups.
\end{lemma}

\begin{proof}
  The quasi-isometry condition implies that a function \( f: X \to G \) satisfies
  \( \int_X \snorm{f(x)}\, d\mu(x) < +\infty \) if and only if
  \( \int_X \snorm{f(x)}' \, d\mu(x) < +\infty \). In particular, the \( \LL^1 \) full groups
  associated with these norms are equal as abstract groups.

  Both topologies make the inclusion of \(\lfgr{G \acts X}\) into \(\Aut(X,\mu)\) continuous by
  Proposition~\ref{prop:pf-is-Polish}, and, in particular, the inclusion map is Borel.  Since
  injective images of Borel sets by Borel maps are Borel (see, for
  example,~\cite[Thm.~15.1]{kechris_classical_1995}), it follows that both topologies induce the
  same Borel structure on \( \lfgr{G\acts X} \), which also coincides with the one induced by the
  weak topology on \(\Aut(X,\mu)\).  A standard automatic continuity result (originally due to
  S.~Banach~\cite[Thm.~4 p.~23]{banachTheorieOperationsLineaires1932}) then yields equality of the
  two topologies (see also the second paragraph following~\cite[Lem.~1.2.6]{MR1425877}).
\end{proof}

When a Polish group \(G\) admits a canonical choice of the quasi-metric structure, \(\LL^{1}\) full
groups \(\lfgr{G \acts X}\) are unambiguously defined as topological groups without the need to
choose any particular norm on \(G\).  This is the case for
\emph{boundedly generated}\index{Polish group!boundedly generated} Polish groups---the class of
groups identified and studied by C.~Rosendal in his treatise~\cite{MR4327092}.
Appendix~\ref{sec:maximal-norms} provides a succinct review of the concept of maximal norms on
boundedly generated Polish groups.

An example of this situation is given by \( G = \mathbb{R} \), where the usual Euclidean norm is
maximal in the sense of Definition~\ref{def:maximal-norm}.

\begin{remark}
  We will see in the last chapter that the natural \(\LL^1\) norm on the \(\LL^1\) full groups of
  \(\R\)-actions is \emph{maximal}\index{Norm!maximal} so that it defines a quasi-metric structure
  which is canonically associated with the topological group structure.
\end{remark}

\section{Embedding \texorpdfstring{\(\LL^1\)}{L1} isometrically in \texorpdfstring{\(\LL^1\)}{L1}
  full groups}\label{sec:L1-embedding}

We now present a general result on the geometry of \(\LL^{1}\) full groups equipped with the
\(\LL^{1}\) norm \( \norm{\cdot}_{1} \), demonstrating that these groups are quite large.

Given a \(\sigma\)-finite measured space \((X, \mathcal{B}, \lambda)\), let
\(\MAlg_{f}(X, \lambda)\)\index{MAlg@\(\MAlg\)} denote the space of all finite-measure subsets
\(B \in \mathcal{B}\), identified up to measure zero.  Endow \(\MAlg_{f}(X, \lambda)\) with the
metric \(d_{\lambda}(B_{1}, B_{2}) = \lambda(B_{1} \bigtriangleup B_{2})\), where \(\bigtriangleup\)
denotes the symmetric difference.

\begin{proposition}
  \label{prop:embedding-L1}
  Let \( G \acts X \) be a Borel measure-preserving action of a Polish normed group
  \( (G, \norm{ \cdot }) \).  If \[\lfgr{G\acts X}\neq [G\acts X],\]
  then the metric space \((\MAlg_f(\R, \lambda),d_\lambda)\) embeds isometrically into the \(\LL^1\)
  full group of \(G\acts X\) endowed with its \(\LL^1\) metric, and hence so does
  \(\LL^1(X,\mu,\R)\).
\end{proposition}

\begin{proof}
  \([G\acts X]\) is a full group, so any of its elements can be written as a product of three
  involutions belonging to \([G\acts X]\)
  by~\cite{ryzhikovRepresentationTransformationsPreserving1985}.  By assumption,
  \(\lfgr{G\acts X} \neq [G\acts X]\), so there must be an involution \(U \in [G\acts X]\) which
  does not belong to \(\lfgr{G\acts X}\).  Denote by \(\mathcal{B}_{U}\) the \(\sigma\)-algebra on
  \(\supp U\) consisting of \(U\)-invariant sets, endowed with the measure given by
  \(\lambda_{U}(A)=\norm{U_{A}}_{1}\), where \(U_{A}(x)=U(x)\) if \(x \in A\) and \(U_{A}(x)=x\)
  otherwise.  Since \(\supp U=\bigcup_n\{x\in \supp U: D(x,U(x))\leq n\}\), the measure
  \(\lambda_U\) is \(\sigma\)-finite. Also, \(\lambda_U\) is non-atomic, because so is \(\mu\), and
  infinite, because \(U\not\in\lfgr{G\acts X}\).  There is only one \(\sigma\)-finite standard
  atomless infinite measured space up to isomorphism (namely \((\R,\mathcal B(\R),\lambda)\)), so we
  conclude that \((\MAlg_f(\supp U, \lambda_{U}),d_{\lambda_{U}})\) is isometric to
  \((\MAlg_f(\mathbb{R}, \lambda),d_{\lambda})\).  Composing this isometry with \(A\mapsto U_A\), we
  get the desired isometric embedding \((\MAlg_f(\R, \lambda),d_\lambda)\to \lfgr{G\acts X}\).

  Finally, we observe that \(\LL^1(X,\mu,\R)\) can be embedded into
  \(\MAlg_f(X\times\R,\mu\otimes\lambda)\) by taking a function \(f\) to its epigraph, namely the
  set of all \((x,y)\in X\times \R\) such that \(f(x)\leq y\leq 0\) or \(0\leq y\leq f(x)\).  Since
  there is again only one infinite \(\sigma\)-finite standard atomless measured space and
  \((X\times\R,\mu\otimes\lambda)\) is such a space, we get an isometric embedding
  \(\LL^1(X,\mu,\R)\to \MAlg_f(\R,\lambda)\) as wanted.
\end{proof}

\begin{remark}
  Full groups of actions of Polish groups are always coarsely
  bounded\index{Polish group!coarsely bounded}.  In fact, they are coarsely
  bounded even as discrete groups\footnote{Being coarsely bounded as a discrete group is also called
    the Bergman property.}, which is a result due to M.~Droste, W.~C.~Holland and
  G.~Ulbrich~\cite{drosteFullGroupsMeasurepreserving2008} (see
  also~\cite[Section~I.8]{millerFullGroupsClassification2004} for a more general statement which
  encompasses the non-ergodic case). In particular, the above result is actually a sharp dichotomy:
  every \(\LL^1\) full group of a Polish normed group action is either coarsely bounded, or it
  contains an isometric copy of \(\LL^1(X,\mu,\R)\).
\end{remark}

\begin{remark}
  Proposition~\ref{prop:embedding-L1} significantly improves~\cite[Prop.~6.9]{MR4398251}, since
  \(\R^n\), endowed with the \(\ell^1\) norm, embeds isometrically into \(\LL^1(X,\mu,\R)\).
\end{remark}

\section{Stability under taking induced transformations}\label{sec:stab-under-first}

Some of the basic properties of \(\LL^{1}\) full groups are discussed---in the wider generality of
induction friendly finitely full groups---in Chapter~\ref{sec:app-fin-ful}.  The often-used
fundamental fact is the closure of \(\LL^{1}\) full groups under taking the induced transformations,
which is a generalization of~\cite[Prop.~3.6]{MR3810253}.  We formulate this in
Proposition~\ref{prop:stability-under-first-return-map}.

Let \( T \in \Aut(X, \mu) \) be a measure-preserving transformation.  Recall that for a measurable
subset \(A \subseteq X\), the \textbf{induced transformation}\index{Transformation!induced}\index{Induced transformation}
\( T_{A} \) is supported on \(A\)
and is defined to be \( T^n(x) \) for \( x\in A \) where \( n \geq 1 \) is the smallest
integer such that \( T^n(x)\in A \).  By the Poincaré recurrence theorem, such a map yields a
well-defined measure-preserving transformation.

\begin{proposition}
  \label{prop:stability-under-first-return-map}
  Let \( G \acts X \) be a Borel measure-preserving action of a Polish normed group
  \( (G, \norm{ \cdot }) \).  For any element \( T \in \lfgr{G \acts X} \) and any measurable set
  \( A \subseteq X \), the induced transformation \( T_{A} \) belongs to \( \lfgr{G \acts X} \) and
  moreover \( \norm{T_{A}}_{1} \le \norm{T}_{1}\).
\end{proposition}

\begin{proof}
  For \( n \geq 1 \), let \( A_n \) be the set of elements of \( A \) whose return time is equal to
  \( n \); note that \( X = \bigsqcup_{n\geq 1}\bigsqcup_{i=0}^{n-1}T^i(A_n) \). Let as before
  \( D : \mathcal{R}_{G} \to \mathbb{R}^{\ge 0} \) be the metric induced by the group norm
  \( \norm{\cdot} \) on the orbits of the action. To estimate the value of \( \norm{T_{A}}_{1} \),
  observe that
  \begin{displaymath}
    \begin{aligned}
      \norm{T_{A}}_{1}
      &= \int_{X}D(x,T_{A}x)\, d\mu(x)
        = \sum_{n=1}^{\infty}\int_{A_{n}}\mkern-10mu D(x, T_{A}x)\, d\mu(x)\\
      &= \sum_{n=1}^{\infty}\int_{A_{n}}\mkern-10mu D(x, T^{n}x)\, d\mu(x).
    \end{aligned}
  \end{displaymath}
  Using the triangle inequality, we get
  \begin{displaymath}
    \begin{aligned}
      \norm{T_{A}}_{1}
      &\le \sum_{n=1}^{\infty} \sum_{i=0}^{n-1}\int_{A_{n}}\mkern-10mu D(T^{i}x,
        T^{i+1}x)\, d\mu(x)\\
      &= \sum_{n=1}^{\infty} \sum_{i=0}^{n-1}\int_{T^{i}(A_{n})}\mkern-35mu D(x, Tx)\,
        d(\mu \circ T^{-i})(x) \\
      \because\ T \textrm{ preserves } \mu
      &= \sum_{n=1}^{\infty}\sum_{i=0}^{n-1}\int_{T^{i}(A_{n})}\mkern-35mu D(x,Tx)\, d\mu(x)
        = \int_{X}D(x, Tx)\, d\mu(x) = \norm{T}_{1}.
    \end{aligned}
  \end{displaymath}
  Thus \( T_{A} \in \lfgr{G \acts X} \) and \( \norm{T_{A}}_{1} \le \norm{T}_{1} \) as claimed.
\end{proof}


\chapter{Polish finitely full groups}\label{sec:app-fin-ful}

The primary focus of this work is the study of \(\LL^{1}\) full groups of Borel measure-preserving
actions of Polish normed groups. However, several results are valid in the more general context of
what we call Polish finitely full groups. This generalization encompasses \(\LL^1\) full groups and
provides a unified and extended context for addressing topics such as topological simplicity and
maximal norms, building on the results of~\cite{MR3810253,MR4398251}.

Beginning with a Polish finitely full group as defined in Section~\ref{sec:full-and-finfull}, we
construct in Section~\ref{sec:derived-equals-symmetric} a natural closed subgroup, termed the
symmetric subgroup, which is analogous to V.~Nekrashevych's symmetric and alternating topological
full groups~\cite{MR3904185}. We show that this closed subgroup coincides with the topological
derived subgroup under a mild hypothesis, satisfied by \(\LL^1\) full groups, which we call
induction friendliness.  Section~\ref{sec:topol-simpl-symm} explores closed normal subgroups of the
symmetric subgroup: we establish their correspondence to invariant sets---a fact that easily yields
topological simplicity when the ambient Polish finitely full group is ergodic.  Finally, in
Section~\ref{sec:maxim-metr-deriv}, we provide a condition on a \emph{normed} induction friendly
Polish finitely full group that guarantees the maximality of its norm's restriction to the symmetric
subgroup. Maximality is understood in the sense of C.~Rosendal, and Appendix~\ref{sec:maximal-norms}
contains a brief reminder of the relevant notions.

\section{Polish full and finitely full groups}\label{sec:full-and-finfull}

H.~Dye defined a subgroup \(\mathbb G\leq \Aut(X,\mu)\) as being
\textbf{full}\index{Transformation group!full} when it is stable under the cutting and pasting of
its elements along a \emph{countable partition}: given any partition \( (A_n)_{n} \) of \(X\) and
any sequence \( (g_n)_{n}\) such that the family \( (g_n(A_n))_{n}\) also partitions \(X\), the
element \( T \in \Aut(X,\mu) \) obtained as the reunion over \( n \in \N \) of the restrictions
\( g_{n}\restriction_{A_n} \) belongs to \( \mathbb G \). In particular, the group \(\Aut(X,\mu)\)
itself is full.

Given any \(\mathbb G\leq\Aut(X,\mu)\), the group obtained by cutting and pasting elements of
\(\mathbb G\) along countable partitions is the smallest full subgroup containing \(\mathbb G\).  We
denote it by \([\mathbb G]\) and call it the \textbf{full group generated} by \(\mathbb G\).

Recall that the uniform topology on \(\Aut(X,\mu)\) is the topology induced by the uniform metric
\(d_u\) defined by
\[
d_u(T_1,T_2) = \mu(\{x\in X : T_1x\neq T_2x\}).
\]
The following can essentially be traced back to
H.~Dye~\cite[Lem.~5.4]{dyeGroupsMeasurePreserving1959}.

\begin{proposition}
  The metric \(d_u\) is complete on any full group \(\mathbb G\), and it is separable if and only if
  the full group is generated by a countable group.
\end{proposition}
\begin{proof}
  Suppose that \((T_n)_{n}\) is a Cauchy sequence in the full group \(\mathbb G\).  Taking a
  subsequence, we may assume that \(d_u(T_n,T_{n+1})<2^{-n}\) for all \(n\). By the Borel--Cantelli
  lemma, for almost every \(x\in X\) there is some \(N\in\N\) such that \(T_nx=T_Nx\) for all
  \(n\geq N\).  Let \(Tx=T_Nx\) for such \(N = N(x)\), and note that \(T\) is a measure-preserving
  bijection\footnote{ This also follows from the fact due to
    P.~Halmos~\cite{halmosLecturesErgodicTheory2017} that \(\Aut(X,\mu)\) is \(d_u\)-complete.} and
  \(d_u(T_n,T)\leq 2^{-n+1}\).  By construction, \(T\) is obtained by cutting and pasting the
  elements \(T_n\) of \(\mathbb G\) along a countable partition so \(T \in \mathbb G\), since
  \(\mathbb{G}\) is full.

  Suppose \(\mathbb G\) is separable and let \(\Gamma\) be a countable dense subgroup. The group
  \([\Gamma]\) is a countably generated full group which is dense in \(\mathbb G\), so
  \(\mathbb G=[\Gamma]\) by completeness.  The converse is obtained by noting that if \(\Gamma\)
  generates \(\mathbb G\), then one can view \(\mathbb G\) as the full group of the equivalence
  relation generated by a realization of the action of \(\Gamma\) on \((X,\mu)\), which is
  \(d_u\)-separable by~\cite[Prop.~3.2]{MR2583950}.
\end{proof}

The \(\LL^1\) full groups that we are considering are not full in the sense of H.~Dye unless the
norm on the acting Polish group is bounded, a case which was considered earlier in~\cite{MR3464151}.
They nevertheless satisfy the following weaker property.

\begin{definition}
  A group \(\mathbb G\leq \Aut(X,\mu)\) of measure-preserving transformations is \textbf{finitely
    full}\index{Transformation group!finitely full}\index{Finitely full group} if for any partition
  \( X = A_{1} \sqcup \cdots \sqcup A_{n} \) and all elements \(g_{1}, \ldots, g_{n}\in \mathbb G\)
  such that the sets \(g_{1}A_{1}, \ldots, g_{n}A_{n}\) also partition \(X\), the transformation
  \( T \in \Aut(X,\mu) \), obtained as the reunion over \( i \in \{1, \ldots, n\} \) of the
  restrictions \( g_{i}\restriction_{A_i} \), belongs to \( \mathbb G \).
\end{definition}

We have the following useful relationship between fullness and finite fullness.

\begin{proposition}\label{prop:du-closure-of-fin-full-is-full}
  The \(d_u\)-closure of any finitely full group \(\mathbb G\) is equal to the full group
  \([\mathbb G]\) generated by \(\mathbb G\). Moreover, every element \(T\in[\mathbb G]\)
  is a \(d_u\)-limit
  of elements of \(\mathbb G\) whose support is contained in the support of \(T\).
\end{proposition}

\begin{proof}
  Since full groups are \(d_u\)-closed and using the definition of fullness, it suffices to show
  that every element \(T\in[\mathbb G]\) is a limit of elements of \(\mathbb G\) that belong to the
  full group generated by \(T\).

  Since every \(T\in[\mathbb G]\) is a product of three involutions in \([T]\)\footnote{In fact, we
    only need the much easier fact that every element is a limit of products of two involutions from
    its full group, which follows by combining Theorem 3.3 and Sublemma 4.3 from~\cite{MR2583950}.
  }~\cite{ryzhikovRepresentationTransformationsPreserving1985}, it suffices to show that every
  involution in \([\mathbb G]\) is a limit of elements of \(\mathbb G\) whose support is contained
  in the support of that involution.  Let \(U\) be such an involution, let \((A_n)_{n}\) be a
  partition of \(X\), and let \((g_n)_{n}\) in \(\mathbb G\) be such that \(Ux=g_nx\) for all
  \(x\in A_n\).  Pick a fundamental domain \(B\) for \(U\), i.e., \(B \cap U(B) = \varnothing\) and
  \(\supp U = B \cup U(B)\).  If \(B_n=A_n\cap B\), then \(Ux=g_nx\) for all \(x\in B_n\), and,
  since \(U\) is an involution, \(Ux=g_n\inv x\) for all \(x\in U(B_n)\).  Let
  \[ U_nx=\left\{
      \begin{array}{cl}
        Ux & \text{ if }x\in\bigcup_{m\leq n} \left(B_m\cup U(B_m)\right),\\
        x & \text{ otherwise.}
      \end{array}
    \right.
  \]
  Clearly \(U_{n} \in \mathbb{G}\), since \(\mathbb{G}\) is finitely full.
  Furthermore, \(U_n\xrightarrow{d_u} U\)
   and \(\supp U_n\subseteq \supp U\) by construction, which finishes the proof.
\end{proof}

Consider a finitely full group \( \mathbb G \) which is a Borel subset of \(\Aut(X,\mu)\) and
therefore inherits the structure of a standard Borel space.  If \(\mathbb{G}\) is Polishable, i.e.,
if it admits a Polish group topology compatible with the Borel structure, then such topology is
necessarily unique and must refine the weak topology inherited from \(\Aut(X,\mu)\) (standard
automatic continuity results can be found, for instance, in~\cite[Sec.~1.6]{MR1425877}).  We refer
to such Polishable groups \( \mathbb G \) endowed with their unique Polish group topology refining
the weak topology as
\textbf{Polish finitely full groups}\index{Polish group!finitely full}.
In this monograph, our motivating example for introducing this class is of course \(\LL^1\) full groups.

\begin{remark}
  For clarity, we adopt the notation \(T_n\xrightarrow{\mathbb{G}} T\) to mean convergence of the
  sequence \((T_n)\) to \(T\) in the Polish topology of \(\mathbb{G}\).
\end{remark}

For any subgroup \( G\leq \Aut(X,\mu) \), there is the smallest finitely full group containing
\( G \). Note that if \( H \le \Aut(X,\mu) \) is a finite group, then the finitely full group it
generates coincides with the full group it generates. This, in particular, applies to the group
generated by a periodic transformation with bounded periods.

\begin{proposition}\label{prop:unifor-topology-on-periodic-bounded}
  Suppose \( \mathbb G \) is a Polish finitely full group, and \( U \in \mathbb{G} \) is a
  periodic transformation with bounded periods.  The topology induced by \( \mathbb G \) on the full
  group of \( U \) is equal to the uniform topology.
\end{proposition}
\begin{proof}
  The weak and uniform topologies on \(\fgr{U}\) coincide because \(U\) is periodic. We have already
  mentioned that the topology of \(\mathbb G\) refines the weak topology.  Since \(\fgr{U}\) is
  Polish in the uniform topology, by the automatic continuity result~\cite[Thm.~1.2.6]{MR1425877},
  the topology induced by \(\mathbb G\) on the full group of \(U\) is refined by the uniform
  topology. Consequently, the uniform topology and the topology induced from \(\mathbb{G}\) onto
  \(\fgr{U}\) must coincide.
\end{proof}

We conclude this preliminary discussion with a definition of aperiodicity, which applies to
arbitrary subgroups of \(\Aut(X,\mu)\).  Such a notion was already worked out by
H.~Dye~\cite[Sec.~2]{dyeGroupsMeasurePreserving1959} when he introduced type II subgroups.  An
equivalent version, which suffices for our purposes, is as follows.

\begin{definition}\label{def:general-aperiodicity}
  A subgroup \(G\leq \Aut(X,\mu)\) is \textbf{aperiodic}\index{Transformation group!aperiodic}\index{Aperiodic!subgroup of \(\Aut(X,\mu)\)} if it
  contains a countable subgroup whose action on \((X,\mu)\) has no finite orbits.
\end{definition}

Since the weak topology on \(\Aut(X,\mu)\) is separable and metrizable, every group
\(G\leq \Aut(X,\mu)\) contains a countable weakly dense subgroup.  Therefore, every aperiodic \(G\)
contains a countable \emph{weakly dense} subgroup whose action on \((X,\mu)\) has no finite orbits.
Further discussion of aperiodicity can be found in Appendix~\ref{sec:aperiodicity}.

\section{Derived subgroup and symmetric subgroup}\label{sec:derived-equals-symmetric}

Recall that the \emph{algebraic} derived subgroup of a group \(G\) is the subgroup generated by all
commutators. If \(G\) is additionally equipped with a group topology, the \textbf{topological
  derived subgroup}\index{Topological derived subgroup} is defined as the
closure of the algebraic derived subgroup. In this work, we do not consider algebraic derived
subgroups and use the term \textbf{derived subgroup}\index{Derived subgroup} exclusively to refer to
the \emph{topological} derived subgroup.

Our goal in this section is to determine when the derived subgroup of a Polish finitely full group
is topologically generated by involutions---that is, when involutions generate a dense subgroup of
the derived subgroup. We begin by noting that aperiodic finitely full groups admit many involutions
in the sense of~\cite[p.~384]{MR2459668}.

\begin{lemma}\label{lem:many-involutions}
  Let \(\mathbb G\) be a finitely full aperiodic group. For every measurable nontrivial
  \(A\subseteq X\), there is a nontrivial involution \(g\in \mathbb G\) whose support is contained
  in \(A\).
\end{lemma}
\begin{proof}
  By Lemma~\ref{lem:aperiodic-action-has-many-involutions}, there is an involution
  \(T\in[\mathbb G]\) whose support is equal to \(A\).  By the moreover part of
  Proposition~\ref{prop:du-closure-of-fin-full-is-full}, \(T\) is the \(d_u\)-limit of
  \(g_n\in\mathbb G\) supported in \(A\). In particular, one of the \(g_n\)'s is nontrivial and
  \(g=g_n\) satisfies the statement of the lemma.
\end{proof}

The first and second items of the following definition constitute analogues of V.~Nekrashevych's
symmetric and alternating topological full groups~\cite{MR3904185}, respectively.  In the setup of
\(\LL^{1}\) full groups, however, these notions coincide, as we will see shortly.

\begin{definition}\label{def:three-algebraic-groups}
  Given a Polish finitely full group \(\mathbb G\), we let
  \begin{itemize}
  \item \(\mathfrak S(\mathbb G)\) be the closed subgroup of \(\mathbb G\) generated by involutions,
    which we call the \textbf{symmetric subgroup of \(\mathbb G\)}\index{Symmetric subgroup}.
  \item \(\mathfrak A(\mathbb G)\) be the closed subgroup of \(\mathbb G\) generated by
    \(3\)-cycles, i.e., generated by periodic transformations whose non-trivial orbits have size
    \(3\).
  \item \(\derived(\mathbb{G})\) be the closed subgroup generated by commutators, called
    the \textbf{derived subgroup}.
  \end{itemize}
\end{definition}

All these groups are closed normal subgroups of \( \mathbb{G} \), and
\( \mathfrak A(\mathbb G)\leq\mathfrak S(\mathbb G) \cap \derived(\mathbb{G}) \) because every \(3\)-cycle
is a commutator of two involutions from its full group.

\begin{proposition}\label{prop:symmetric-equals-alternating}
  \(\mathfrak A(\mathbb G)=\mathfrak S(\mathbb G)\) for any aperiodic finitely full group
  \( \mathbb{G} \).
\end{proposition}
\begin{proof}
  We need to show that every involution is a limit of products of \( 3 \)-cycles. Let
  \(U\in\mathbb G\) be an involution, and let \(D\) denote a fundamental domain of \(U\);
  thus \(\supp U=D \sqcup U(D)\).
  By Lemma~\ref{lem:aperiodic-action-has-many-involutions}, one can find
  an involution \(V\in \fgr{\mathbb{G}}\) whose support is equal to \(D\). Since \( \mathbb{G} \) is
  finitely full, we may write \( D \) as an increasing union \(D=\bigcup_n D_n\),
  \(D_n\subseteq D_{n+1}\), where each \(D_n\) is \(V\)-invariant, and for every
  \( n \in \mathbb{N} \) the transformation \(V_n\) induced by \(V\) on \(D_n\) belongs to the group
  \(\mathbb G\) itself. Let \( U_{n}= U_{D_n\sqcup U(D_n)}\)
  denote the transformation induced by \( U \) onto \(D_n\sqcup U(D_n)\)
  and note that \(U_{n}\to U\) in the uniform topology, and hence also in the topology of
  \( \mathbb{G} \) by Proposition~\ref{prop:unifor-topology-on-periodic-bounded}. Our plan is to use
  the following permutation identity
  \begin{equation}
    \label{eq:involution-product-of-3-cycles}
    (12)(34)=(12)(23)(24)(23)=(123)(423),
  \end{equation}
  where \(U_{n}\) corresponds to \((12)(34)\), \(V_n\) to \((13)\), and \(U_{n}V_{n}U_{n}\)
  corresponds to \((24)\). To this end, let \(C_{n}\) be a fundamental domain for \(V_{n}\), put
  \(W_{n}=U_{C_{n}\sqcup U(C_{n})}\) (which corresponds to the involution \((12)\)),
  and, at last, set \(S_{n}=W_{n}V_{n}W_{n}\) (corresponding to \((23)=(12)(13)(12)\)).
  Figure~\ref{fig:involution-product-3-cycles} illustrates the relations between these sets and
  transformations.

  \begin{figure}[htb]
          \centering
          \begin{tikzpicture}
          \draw (0,0) rectangle (2.0,2.0);
          \draw (1,0.7) circle (3mm) node {1};
          \draw (1,1.4) node {\( C_{n} \)};
          \draw[<->] (1,2) -- (1,3.5) node[pos=0.5, anchor=east] {\( W_{n} \)};
          \draw (0,3.5) rectangle (2.0,5.5);
          \draw (1,4.2) circle (3mm) node {2};
          \draw (1,4.9) node {\( U_{n}(C_{n}) \)};
          \draw[<->] (2,1) -- (3.5,1) node[pos=0.5, anchor=north] {\( V_{n} \)};
          \draw (3.5,0) rectangle (5.5,2.0);
          \draw (4.5,0.7) circle (3mm) node {3};
          \draw (4.5,1.4) node {\( V_{n}(C_{n}) \)};
          \draw[<->] (2,4.5) -- (3.5,4.5) node[pos=0.5, anchor=south] {\( U_{n}V_{n}U_{n} \)};
          \draw (3.5,3.5) rectangle (5.5,5.5);
          \draw (4.5,4.2) circle (3mm) node {4};
          \draw (4.5,4.9) node {\( U_{n}V_{n}(C_{n}) \)};
          \draw[<->] (2,3.5) -- (3.5,2) node[pos=0.5, anchor=south,xshift=2mm, yshift=-1mm]
          {\( S_{n} \)};
          \draw [decorate,decoration={brace,amplitude=6pt},yshift=5pt]
          (0,5.5) -- (5.5, 5.5) node [pos=0.5, anchor=south, yshift=2mm] {\( U_{n}(D_{n}) \)};
          \draw [decorate,decoration={brace,amplitude=6pt,mirror},yshift=-5pt]
          (0,0) -- (5.5, 0) node [pos=0.5, anchor=north, yshift=-2mm] {\( D_{n} \)};
          \end{tikzpicture}
          \caption{The involution \( U_{n} \) is a products of \( 3 \)-cycles via
                  \( (12)(34) = (123)(234) \).}
          \label{fig:involution-product-3-cycles}
  \end{figure}

  Eq.~\eqref{eq:involution-product-of-3-cycles} translates into
  \( U_{n} = \bigl(W_n S_{n}\bigr) \bigl((U_{n}V_{n}U_{n})S_{n}\bigr) \), so \(U_{n}\) is a product
  of two \(3\)-cycles, hence it belongs to \(\mathfrak A(\mathbb G)\). Since
  \(U_{n}\xrightarrow{\mathbb G} U\), we conclude that \( U \in \mathfrak{A}(\mathbb{G}) \).
\end{proof}

We do not know whether \(\mathfrak A(\mathbb G)=\derived(\mathbb{G})\) holds for all finitely full
groups, but here is a convenient sufficient condition.

\begin{definition}\label{def:induction-friendly}
  A Polish finitely full group \(\mathbb G\) is called
  \textbf{induction friendly}\index{Induction friendly group}\index{Transformation group!induction friendly} if it is stable
  under taking induced transformations and, furthermore, whenever \(T\in\mathbb G\) and
  \((A_n)_{n}\) is an increasing sequence of \(T\)-invariant sets such that \(\bigcup_n A_n= A\),
  then \(T_{A_n}\xrightarrow{\mathbb G} T_A\).
\end{definition}

In the above definition, we require stability under taking the induced transformations, and so
\(T_{A_{n}}\) always belongs to \(\mathbb{G}\).  However, for \(T\)-invariant \(A_{n}\), the
assertion \(T_{A_{n}} \in \mathbb{G}\) is already a consequence of \(\mathbb{G}\) being finitely
full.

Observe that \(\LL^1\) full groups of measure-preserving actions of Polish normed groups are
finitely full and also induction friendly. Indeed, finite fullness follows from a straightforward
computation, while induction friendliness is a direct consequence of
Proposition~\ref{prop:stability-under-first-return-map} and the Lebesgue dominated convergence
theorem.

\begin{lemma}\label{lem:periodic-elements-are-in-S(G)}
  In an induction friendly Polish finitely full group \(\mathbb G\), every periodic element belongs
  to \(\mathfrak S(\mathbb G)\).
\end{lemma}
\begin{proof}
  Suppose \(T\) is periodic. For \(n\in\N\), let \(A_n\) be the set of \(x\in X\) whose \(T\)-orbit
  has cardinality at most \(n\). Each \(A_n\) is \(T\)-invariant and \(\bigcup_n A_n=X\).  Moreover,
  \(T_{A_n}\) is periodic, so it can be written as a product of two involutions from its full group
  (see e.g.,~\cite[Sublem.~4.3]{MR2583950}).  Since \(\mathbb G\) is finitely full and the periods
  of \(T_{A_n}\) are bounded, these two involutions belong to \(\mathbb G\), hence
  \(T_{A_n}\in \mathfrak S(\mathbb G)\).  By induction friendliness,
  \(T_{A_n}\xrightarrow{\mathbb G} T\), which finishes the proof since \(\mathfrak S(\mathbb G)\) is
  closed in \(\mathbb G\).
\end{proof}

\begin{lemma}\label{lem:equality-in-the-factor-by-symmetric}
  Let \(\mathbb{G}\) be an induction friendly Polish finitely full group.  Let \(T \in \mathbb{G}\)
  and \(F \subseteq X\) be the aperiodic part of \(T\), i.e.,
  \[F = \{x \in X : T^{k}x \ne x \textrm{ for all } k \ne 0\}.\]
  For any \(A \subseteq X\) such that \(F \subseteq \bigcup_{k \in \mathbb{Z}}T^{k}(A)\) one has
  \(T_{A}\mathfrak{S}(\mathbb{G}) = T\mathfrak{S}(\mathbb{G})\).
\end{lemma}
\begin{proof}
  Since \(F \subseteq \bigcup_{k \in \mathbb{Z}}T^{k}(A)\), the transformation \(T^{-1}T_{A}\) is
  periodic and therefore belongs to \(\mathfrak{S}(\mathbb{G})\) by
  Lemma~\ref{lem:periodic-elements-are-in-S(G)}. Hence
  \[T\mathfrak{S}(\mathbb{G}) = T T^{-1} T_{A} \mathfrak{S}(\mathbb{G}) =
    T_{A}\mathfrak{S}(\mathbb{G}).\qedhere\]
\end{proof}

\begin{remark}
  The usefulness of the above lemma stems from the following simple observation.  If \(T,T', U, U'\)
  satisfy \(T\mathfrak{S}(\mathbb{G}) = T'\mathfrak{S}(\mathbb{G})\) and
  \(U\mathfrak{S}(\mathbb{G}) = U' \mathfrak{S}(\mathbb{G})\), then
  \([T,U] \in \mathfrak{S}(\mathbb{G})\) if and only if \([T', U'] \in \mathfrak{S}(\mathbb{G})\).
  In particular, for \(A\) as in Lemma~\ref{lem:equality-in-the-factor-by-symmetric},
  \([T,U] \in \mathfrak{S}(\mathbb{G})\) whenever \([T_{A},U] \in \mathfrak{S}(\mathbb{G})\).  This
  fact is used in the proof of the next lemma.\label{rem:commutator-in-symmetric-group}
\end{remark}

\begin{lemma}\label{lem:aperiodic-commutator-symmetric-group}
  Suppose \(\mathbb{G}\) is an induction friendly Polish finitely full group.  If
  \(T, U \in \mathbb{G}\) are aperiodic on their supports, then
  \([T,U] \in \mathfrak{S}(\mathbb{G})\).
\end{lemma}

\begin{proof}
  Let \( C \) be a cross-section for the restriction of \(\eqr_{T}\) onto \(\supp T\).  In other
  words, \(C \subseteq X\) is a measurable set satisfying
  \(\bigcup_{i \in \mathbb{Z}}T^{i}(C) = \supp T\).  The induced transformation
  \( U_{X \setminus C} \) commutes with \( T_{C} \), since their supports are disjoint. We would be
  done if \( \supp U \subseteq \bigcup_{i \in \mathbb{Z}}U^{i}(X \setminus C) \). Indeed, in this
  case \( T \mathfrak{S}(\mathbb{G}) =T_{C} \mathfrak{S}(\mathbb{G}) \),
  \( U \mathfrak{S}(\mathbb{G}) = U_{X \setminus C} \mathfrak{S}(\mathbb{G}) \) by
  Lemma~\ref{lem:equality-in-the-factor-by-symmetric} and \( [T_{C},U_{X \setminus C}] \) is
  trivial, hence \( [T,U] \in \mathfrak{S}(\mathbb{G}) \).

  Motivated by this observation, we argue as follows. Pick a vanishing nested sequence
  \( (C_{n})_{n \in \mathbb{N}} \) of cross-sections for \(\eqr_{T}\restriction_{\supp T}\), i.e.,
  \( C_{n} \supseteq C_{n+1} \), \( \bigcup_{k\in \mathbb{Z}}T^{k}(C_{n}) = \supp T \) for all
  \( n \in \mathbb{N} \), and \( \bigcap_{n \in \mathbb{N}} C_{n} = \varnothing \) (see also
  Lemma~\ref{lem:marker-lemma}). Such a sequence of cross-sections exists since \( T \) is assumed
  to be aperiodic on its support. Define inductively sets \( B'_{n} \), \( n \in \mathbb{N} \), by
  setting \( B'_{0} = X \setminus C_{0} \), and letting \(B'_{n}\) be the part of
  \(X \setminus C_{n}\) that does not belong to the \(U\)-saturation of any \(B'_{k}\), \(k < n\),
  \[ B'_{n} = (X \setminus C_{n}) \setminus \bigcup_{k < n} \bigcup_{i \in \mathbb{Z}}
    U^{i}(B'_{k}). \]
  By construction, saturations under \( U \) of the sets \( B'_{n} \) are pairwise disjoint, and the
  saturation of their union is the whole space,
  \( \bigcup_{i \in \mathbb{Z}}U^{i}\bigl(\bigcup_{n \in \mathbb{N}} B'_{n}\bigr) = X \), because
  sets \( C_{n} \) vanish.

  Let \( B_{n} = \bigsqcup_{k < n} B'_{k} \), \( B = \bigsqcup_{k \in \mathbb{N}} B'_{k} \), and
  note that \( U_{B_{n}}, U_{B} \in \mathbb{G} \), and \( U_{B_{k}} \xrightarrow{\mathbb G} U_{B} \)
  by the induction friendliness of \( \mathbb{G} \). By construction, the transformations
  \( T_{C_{n}} \) and \( U_{B_{n}} \) have disjoint supports for each \(n\) and, therefore,
  commute. Since all sets \( C_{n} \) are cross-sections for \(\eqr_{T}\restriction_{\supp T}\), one
  has \( [T,U_{B_{n}}] \in \mathfrak{S}(\mathbb{G}) \) by
  Lemma~\ref{lem:equality-in-the-factor-by-symmetric} and
  Remark~\ref{rem:commutator-in-symmetric-group}. Taking the limit as \( n \to \infty \), this
  yields \( [T, U_{B}] \in \mathfrak{S}(\mathbb{G}) \).  Finally, the \(U\)-saturation of \( B \) is
  all of \(X\), and we use Lemma~\ref{lem:equality-in-the-factor-by-symmetric} and
  Remark~\ref{rem:commutator-in-symmetric-group} once again to conclude that
  \( [T, U] \in \mathfrak{S}(\mathbb{G}) \), as claimed.
\end{proof}

\begin{proposition}\label{prop:induction-friendly-implies-groups-coincide}
  If \(\mathbb G\) is an aperiodic induction friendly Polish finitely full group, then
  \(\mathfrak{S}(\mathbb G)=\derived(\mathbb{G})\).
\end{proposition}
\begin{proof}
  The inclusion \( \mathfrak{A}(\mathbb{G}) \le \derived(\mathbb{G}) \) holds for any Polish
  finitely full group, and Proposition~\ref{prop:symmetric-equals-alternating} gives
  \( \mathfrak{S}(\mathbb{G}) \le \derived(\mathbb{G}) \).  We therefore concentrate on proving the
  reverse inclusion: given \(T, U \in \mathbb{G}\), we need to check that
  \([T,U] \in \mathfrak{S}(\mathbb{G})\).  Let \(F_{T}\) and \(F_{U}\) be the aperiodic parts of
  \(T\) and \(U\) respectively, so that
  \(T\mathfrak{S}(\mathbb{G}) = T_{F_{T}}\mathfrak{S}(\mathbb{G})\),
  \(U\mathfrak{S}(\mathbb{G}) = U_{F_{U}}\mathfrak{S}(\mathbb{G})\) by
  Lemma~\ref{lem:equality-in-the-factor-by-symmetric}.  By construction, \(T_{F_{T}}\) and
  \(U_{F_{U}}\) are aperiodic on their supports and therefore
  \([T_{F_{T}}, U_{F_{U}}] \in \mathfrak{S}(\mathbb{G})\) by
  Lemma~\ref{lem:aperiodic-commutator-symmetric-group}.  It remains to use
  Remark~\ref{rem:commutator-in-symmetric-group} to conclude that necessarily
  \([T,U] \in \mathfrak{S}(\mathbb{G})\), as needed.
\end{proof}

\begin{corollary}\label{cor:all-subgroups-are-equal}
  Let \(G\) be a Polish normed group, and let \(G \acts X\) be an aperiodic Borel measure-preserving
  action on a standard probability space \( (X, \mu) \). The three subgroups of
  \( \lfgr{G \acts X} \) introduced in Definition~\ref{def:three-algebraic-groups} coincide:
  \[ \derived(\lfgr{G\acts X})=\mathfrak A(\lfgr{G\acts X})=\mathfrak S(\lfgr{G\acts X}). \]
  Moreover, they are all equal to the closure of the group generated by periodic elements of
  \(\lfgr{G\acts X}\).
\end{corollary}
\begin{proof}
  The equality
  \( \derived(\lfgr{G\acts X})=\mathfrak A(\lfgr{G\acts X})=\mathfrak S(\lfgr{G\acts X}) \) follows
  immediately from Propositions~\ref{prop:symmetric-equals-alternating}
  and~\ref{prop:induction-friendly-implies-groups-coincide}, since \(\lfgr{G\acts X}\) is both
  finitely full and induction friendly. All these groups are equal to the closure of the group
  generated by periodic elements of \(\lfgr{G\acts X}\) in view of
  Lemma~\ref{lem:periodic-elements-are-in-S(G)} and the fact that this group obviously contains
  \(\mathfrak S(\lfgr{G\acts X})\).
\end{proof}

\section{Topological simplicity of the symmetric group}\label{sec:topol-simpl-symm}

We now move on to showing that symmetric subgroups of ergodic Polish finitely full groups are always
topologically simple.  More generally, we describe the closed normal subgroups of symmetric
subgroups of aperiodic Polish finitely full groups.  Our argument abstracts
from~\cite[Sec.~3.4]{MR3810253}.  In particular, we rely on conditional measures associated with
subgroups of \(\Aut(X,\mu)\), whose construction and basic properties are recalled in
Appendix~\ref{chap:conditional-measures-appendix}.  We begin with two lemmas on involutions.

\begin{lemma}\label{lem:approxconj}
  Let \(\mathbb G\) be an aperiodic Polish finitely full group, let \(U, V\in\mathbb G\) be two
  involutions whose supports are disjoint and have the same \(\mathbb G\)-conditional measure.  Then
  \(U\) and \(V\) are approximately conjugate in \(\mathfrak S(\mathbb G)\), i.e., there are
  \(T_{n} \in \mathfrak{S}(\mathbb{G})\) such that \(T_{n}UT_{n}^{-1} \xrightarrow{\mathbb G} V\).
\end{lemma}

\begin{proof}
  Let \(A\) (resp.\ \(B\)) be a fundamental domain of the restriction of \(U\) (resp.\ \(V\)) to its
  support. Then \(\mu_{\mathbb G}(A)=\mu_{\mathbb G}(B)\), and there is an involution
  \(T\in [\mathbb G]\) such that \(T(A)=B\).

  Since \(\mathbb G\) is finitely full, there is an increasing sequence \((A_n)_{n}\) of subsets of
  \(A\) such that the involutions \(T'_n\) induced by \(T\) on \(A_n\cup U(A_n)\) belong to
  \(\mathbb G\), and \(\bigcup_n A_n=A\). Let \(B_n=T(A_n)=T'_n(A_n)\) and define involutions
  \(T_n\in\mathbb G\) which almost conjugate \(U\) to \(V\) as follows. For \(x\in X\), let
  \[
    T_nx= \left\{
      \begin{array}{cl}
        Tx & \text{if }x\in A_n\sqcup B_n \\
        VTUx & \text{if } x\in U(A_n) \\
        UTVx&\text{if }x\in V(B_n) \\x & \text{otherwise.}\end{array}\right.
  \]
  For all \(n\in\N\) and all \(x\in X\), an easy calculation yields that:
  \begin{itemize}
  \item if \(x\in (A\cup U(A))\setminus (A_n\cup U(A_n))\), then \(T_nUT_nx=Ux\);
  \item if \(x\in B_n\cup V(B_n)\), then \(T_nUT_nx=Vx\);
  \item and \(T_nUT_nx=x\) in all other cases.
  \end{itemize}
  In particular, \(d_u(T_nUT_n,V)\to 0\) and
  Proposition~\ref{prop:unifor-topology-on-periodic-bounded}, applied to the full group of the
  involution \(UV\) (which contains both \(U\) and \(V\)), guarantees that
  \(T_nUT_n\xrightarrow{\mathbb G} V\).
\end{proof}

\begin{lemma}\label{lem:commutator-trick}
  Let \(\mathbb G\) be an aperiodic Polish finitely full group, let \(U\in \mathbb G\) be an
  involution, and let \(A\) be a \(U\)-invariant subset contained in \(\supp U\).  Suppose that
  there exists an involution \(V\in\mathbb G\) such that \(V(A)\) is disjoint from \(\supp U\).
  Then for all \(\mathbb G\)-invariant functions \(f\leq 2 \mu_{\mathbb G}(A)\), there is an
  involution \(W\in \mathbb G\) such that \(UWUW\) is an involution whose support has
  \(\mathbb G\)-conditional measure \(f\).
\end{lemma}
\begin{proof}
  Let \(B\subseteq A\) be a fundamental domain for the restriction of \(U\) to \(A\) and note that
  \(\mu_{\mathbb G}(B)=\mu_\mathbb G(A)/2\).  By Maharam's lemma (Theorem~\ref{thm:maharam-lemma}),
  there is \(C\subseteq B\) such that \(\mu_{\mathbb G}(C)= f/4\).  The set \(D= C\sqcup U(C)\) is
  \(U\)-invariant and satisfies \(\mu_{\mathbb G}(D)= f/2\).  Consider the involution
  \(W\in\mathbb G\) defined by
  \[
    Wx= \left\{
      \begin{array}{cl}
        Vx & \text{if }x\in D\sqcup V(D) \\
        x & \text{otherwise.}\end{array}\right.
  \]
  A straightforward computation shows that \(UWUW\) is an involution that coincides with \(U\) on
  \(D\), with \(VUV\) on \(V(D)\), and is trivial elsewhere. Hence, the support of \(UWUW\) is equal
  to \(D\sqcup V(D)\) and has \(\mathbb G\)-conditional measure \(f\).
\end{proof}

Given a subgroup \(G \le \Aut(X, \mu)\) and a \(G\)-invariant set \(A\), we let \(G_{A}\) stand for
the subgroup \(\{T \in G : \supp T \subseteq A\}\).  Note that \(G_{A}\) is a normal subgroup of
\(G\).  Our focus is on the case \(G = \mathfrak{S}(\mathbb{G})\), where \(\mathbb{G}\) is an
aperiodic Polish finitely full group.  Every subgroup \(\mathfrak{S}(\mathbb{G})_{A}\) is
necessarily closed, because the topology of \(\mathbb G\) refines the weak topology.  We show in
Theorem~\ref{thm:normal-subgroups-description} that all closed normal subgroups of
\(\mathfrak S(\mathbb G)\) arise in this way.

\begin{proposition}\label{prop:S-conjugates-of-T-contain-S}
  Let \(\mathbb G\) be an aperiodic Polish finitely full group, let \(T\in\mathbb G\), and let \(A\)
  denote the \(\mathbb G\)-saturation of \(\supp T\).  Then the closed subgroup of \(\mathbb G\)
  generated by the \(\mathfrak S(\mathbb G)\)-conjugates of \(T\) contains
  \(\mathfrak S(\mathbb G)_A\).
\end{proposition}

\begin{proof}
  Let the closed subgroup of \(\mathbb G\) generated by the \(\mathfrak S(\mathbb G)\)-conjugates of
  \(T\) be denoted by \(G\). By~\cite[Lem.~7.2]{friedmanIntroductionErgodicTheory1970}, we can find
  a set \(B\subseteq \supp T\) whose \(T\)-translates cover \(\supp T\) and which satisfies
  \(B\cap T(B)=\varnothing\).  Since \(T\)-translates of \(B\) cover \(\supp T\), we conclude that
  the \(\mathbb G\)-translates of \(B\) cover \(A\), and so \(\mu_{\mathbb G}(B)(x)>0\) for all
  \(x\in A\).  By Maharam's lemma (Theorem~\ref{thm:maharam-lemma}), we can find \(C\subseteq B\)
  whose \(\mathbb G\)-conditional measure is everywhere at most \(1/4\), and is strictly positive on
  \(A\).  Take \(V\in \fgr{\mathbb G}\) to be an involution such that \(V(C\sqcup T(C))\) is
  disjoint from \(C\sqcup T(C)\).  Such an involution exists because
  \(\mu_{\mathbb G}(C\sqcup T(C))\leq 1/2\), and so, by Maharam's lemma, we can find a subset
  contained in \(X\setminus \left(C\sqcup T(C)\right)\) having the same conditional measure as
  \(C\sqcup T(C)\).  We can then apply the last item from
  Proposition~\ref{prop:construct-involutions}.

  Let \(W\in [\mathbb G]\) be an involution such that \(\supp W= C\), whose existence is guaranteed
  by Lemma~\ref{lem:aperiodic-action-has-many-involutions}.  Using the facts that \(\mathbb G\) is
  finitely full, that \(T\in\mathbb G\) and that \(V,W\in[\mathbb G]\), one can find an increasing
  sequence \((C_{n})_{n}\) of \(W\)-invariant subsets of \(C\) such that \(\bigcup_n C_n=C\) and for
  each \(n\in\N\) both \(W_{C_n}\in \mathbb G\) and
  \(V_{C_n\sqcup T(C_n)\sqcup V(C_n\sqcup T(C_n))}\in\mathbb G\).  The transformations
  \(W_{C_n}TW_{C_n}T\inv\) belong to \(G\), and are, in fact, involutions whose support is equal to
  \(C_n\sqcup T(C_n)\) and has conditional measure at most \(2\mu_{\mathbb G}(C)\leq 1/2\). Let us
  define for brevity
  \[
    \tilde U_n=W_{C_n}TW_{C_n}T\inv \in G \quad \text{ and } \quad \tilde V_n=V_{C_n\sqcup T(C_n)\sqcup
      V(C_n\sqcup T(C_n))}\in \mathbb G.
  \]

  For every \(n\in\N\), let \(A_n\) denote the \(\mathbb G\)-saturation of \(C_n\).  Note that
  \(A=\bigcup_n A_n\) and the union is increasing.  Every involution supported on \(A\) is thus the
  uniform limit of the involutions it induces on \(A_n\)'s. By
  Proposition~\ref{prop:unifor-topology-on-periodic-bounded}, it therefore suffices to show that
  \(G\) contains all the involutions which are supported on some \(A_n\).

  Let \(U\) be an involution supported on some \(A_n\). Let \(D\) be a fundamental domain for
  the restriction of \(U\) to its support. Using Maharam's lemma repeatedly, we can partition \(D\)
  into a countable family \((D_k)_{k}\) such that
  \begin{equation}\label{ineq:measure-Dn}
    \mu_{\mathbb G}(D_k)\leq \mu_{\mathbb G}(\supp \tilde U_n)/2 \quad \textrm{ for all \(k \in \mathbb{N}\)}.
  \end{equation}
  If we let \(E_k=D_k\sqcup U(D_k)\), the sequence \((E_k)_{k}\) forms a partition of \(\supp U\)
  into \(U\)-invariant sets. In particular, \(U=\lim_{k} \prod_{i=0}^k U_{E_k}\) in the uniform
  topology and therefore in the topology of \(\mathbb G\) as well by
  Proposition~\ref{prop:unifor-topology-on-periodic-bounded}. Moreover, the support of \(U_{E_k}\)
  has \(\mathbb G\)-conditional measure at most \(\mu_{\mathbb G}(\supp \tilde U_n)\) by
  Eq.~\eqref{ineq:measure-Dn}.  The set \(\tilde V_n(\supp \tilde U_n)\) is disjoint from
  \(\supp\tilde U_n\) by construction.  Lemma~\ref{lem:commutator-trick} applies and provides an
  involution in \(G\) whose support has the same conditional measure as that of \(U_{E_k}\).
  Lemma~\ref{lem:approxconj} shows that each \(U_{E_k}\) belongs to \(G\) and therefore also
  \(U\in G\), as needed.
\end{proof}

\begin{theorem}\label{thm:normal-subgroups-description}
  Let \( \mathbb{G} \le \Aut(X, \mu) \) be an aperiodic Polish finitely full group. For any closed
  normal subgroup \( N \le \mathfrak{S}(\mathbb{G}) \), there is a unique \( \mathbb{G} \)-invariant
  set \( A \) such that \( N = \mathfrak{S}(\mathbb{G})_A \).
\end{theorem}

\begin{proof}
  First, observe that for \(\mathbb{G}\)-invariant \(A_1\) and \(A_2\), any involution
  \(U\in\mathbb G\) supported in \(A_1\cup A_2\) decomposes into the product of one involution
  supported in \(A_1\), and one supported in \(A_2\).  It follows that the closed group generated by
  \(\mathfrak S(\mathbb G)_{A_1}\cup\mathfrak S(\mathbb G)_{A_2}\) is equal to
  \(\mathfrak S(\mathbb G)_{A_1\cup A_2}\).  Also, by
  Proposition~\ref{prop:unifor-topology-on-periodic-bounded}, whenever \((A_n)_{n}\) is an
  increasing sequence of \(\mathbb G\)-invariant sets, one has
  \[\overline{\bigcup_n \mathfrak S(\mathbb G)_{A_n}}=\mathfrak S(\mathbb G)_{\bigcup_n A_n}.\]
  The set
  \(\{A\in\MAlg(X,\mu) : A\text{ is }\mathbb G\text{-invariant and }\mathfrak S(\mathbb G)_A\leq
  N\}\) is thus directed and is closed under the countable unions. It therefore admits a unique
  maximum element, which is the set \(A\) we seek.  Indeed, \(\mathfrak S(\mathbb G)_{A}\leq N\),
  and the reverse inclusion is a direct consequence of
  Proposition~\ref{prop:S-conjugates-of-T-contain-S}.

  It remains to argue that the set \(A\) satisfying \(N = \mathfrak{S}(\mathbb{G})_{A}\) is
  unique. Suppose towards a contradiction that
  \(\mathfrak S(\mathbb G)_{A_1}=\mathfrak S(\mathbb G)_{A_2}\)for \(A_1\neq A_2\). By symmetry, we
  may assume that \(\mu(A_1\setminus A_2)>0\).  Lemma~\ref{lem:many-involutions} provides an
  involution \(V\in\mathbb G\) whose support is nontrivial and is contained in \(A_1\setminus A_2\),
  thus \(V\in \mathfrak S(\mathbb G)_{A_1}\) but \(V\not \in\mathfrak S(\mathbb G)_{A_2}\),
  contradicting \(\mathfrak S(\mathbb G)_{A_1}=\mathfrak S(\mathbb G)_{A_2}\).
\end{proof}

\begin{corollary}\label{cor:S-topologically-simple-action-ergodic}
  Let \(\mathbb G\leq \Aut(X,\mu)\) be an aperiodic Polish finitely full group. The group
  \(\mathfrak S(\mathbb G)\) is topologically simple if and only if \(\mathbb G\) is ergodic.
\end{corollary}

\begin{proof}
  If \(\mathbb G\) is ergodic, then \(\mathfrak S(\mathbb G)\) is topologically simple by
  Theorem~\ref{thm:normal-subgroups-description}.  Conversely, suppose that \(\mathbb G\) is not
  ergodic and let \(A\subseteq X\) be a \(\mathbb G\)-invariant set with
  \(\mu(A)\not\in\{0,1\}\). Then \(\mathfrak S(\mathbb G)_A\) is a normal subgroup of \(\mathbb G\)
  which is neither trivial nor equal to \(\mathfrak S(\mathbb G)\) as a consequence of
  Lemma~\ref{lem:many-involutions} applied to \(A\) and its complement.
\end{proof}

Specifying the corollary above to \(\LL^1\) full groups and using
Corollary~\ref{cor:all-subgroups-are-equal}, we obtain the following result.

\begin{corollary}\label{cor:l1-topological-simplicity}
  Let \(G\) be a Polish normed group, and let \(G \acts X\) be an aperiodic Borel measure-preserving
  action on a standard probability space \( (X, \mu) \). The topological derived subgroup of the
  \(\LL^1\) full group of the action is topologically simple if and only if the action is ergodic.
\end{corollary}

\section{Maximal norms on the derived subgroup}\label{sec:maxim-metr-deriv}

The purpose of this section is to establish sufficient conditions for a norm on the derived subgroup
of an induction friendly Polish finitely full group to be maximal in the sense of
Section~\ref{sec:l1-full-groups-bound}.  Our argument follows closely the one given
in~\cite[Sec.~6.2]{MR4398251} for amenable graphings.  The main application of
Proposition~\ref{prop:induction-friendly-maximal-norm} will be given in
Theorem~\ref{thm:lc-amenable-derived-maximal}, but we hope that the setup of this section can be
useful in other contexts, such as \(\varphi\)-integrable full groups~\cite{2201.06662}.

\begin{definition}\label{def:additive-norm}
  A norm \(\norm{\cdot}\) on a subgroup \(\mathbb{G} \le \Aut(X,\mu)\) is
  \textbf{additive}\index{Norm!additive} if \(\norm{TS} = \norm{T} + \norm{S}\) for all
  \(T,S \in \mathbb{G}\) with disjoint supports.
\end{definition}

The following lemma parallels~\cite[Lem.~6.4]{MR4398251} and is the key to showing that the norm on
the derived subgroup is both coarsely proper and large-scale geodesic.

\begin{lemma}\label{lem:product-of-small}
  Let \(\mathbb{G} \le \Aut(X,\mu)\) be a finitely full Polish group, and suppose that
  \(\norm{\cdot}\) is a compatible additive norm on \(\mathbb{G}\).  For any periodic
  \(U \in \mathbb{G}\) with bounded periods and for every \( n \in \N \), there are periodic
  elements \( U_1, \ldots ,U_n\in \mathbb{G} \) such that \[U = U_1\cdots U_n\text{ and }
  \norm{U_i}=\frac{\norm{U}}{n}\text{ for every } 1 \le i \le n.\]
\end{lemma}

\begin{proof}
  Let \( M = \norm{U} \) and \(A \subseteq X\) be a fundamental domain for \(U\).  We may identify
  \(A\) with the interval \([0, \mu(A)]\) endowed with the Lebesgue measure. Put
  \(A_{t} = [0,t] \cap A\), \(0 \le t \le \mu(A)\), and let
  \(B_{t} = \bigcup_{n \in \mathbb{Z}} U^{n}(A_{t})\) be the \(U\)-saturation of \(A_{t}\).  Note
  that \(U_{B_{t}} \in \mathbb{G}\) for all \(t \in [0, \mu(A)]\) since \(B_{t}\) is \(U\)-invariant
  and \(\mathbb{G}\) is finitely full, and that \(t\mapsto B_t\) is continuous.

  The map \( [0, \mu(A)] \ni t \mapsto U_{B_{t}} \in \fgr{U} \subseteq \mathbb{G}\) is thus continuous
  with respect to the uniform topology on \(\fgr{U}\), and therefore also with respect to the
  topology of \(\mathbb{G}\) by Proposition~\ref{prop:unifor-topology-on-periodic-bounded}.  Whence
  the function \( \psi : [0, \mu(A)] \to \mathbb{R} \) given by \(\psi(t) = \norm{U_{B_{t}}}\) is
  also continuous.

  We have \(\psi(0)=0\) and \(\psi(\mu(A)) = M\), so the intermediate value theorem yields existence
  of reals \(0 = t_0 < t_1< \cdots < t_{n-1} < t_n = \mu(A)\) such that \(\psi(t_i)=\frac{iM}n\) for
  all \(i \in \{0, \ldots ,n\} \).  Set \(C_i=B_{t_i}\setminus B_{t_{i-1}}\) for \(i\in\{1, \ldots ,n\}\).  By
  construction, each \(C_i\) is \(U\)-invariant and \(X = \bigsqcup_{i=1}^n C_i\).  Putting
  \(U_i=U_{A_i}\), we get \(U=\prod_{i=1}^n U_i\).  Finally for each \(i\in\{1, \ldots ,n\}\) the
  equality \(C_i=B_{t_i}\setminus B_{t_{i-1}}\) and additivity of the norm gives
  \[\psi(t_{i}) = \snorm{U_{B_{t_{i}}}} = \snorm{U_{i}U_{B_{t_{i-1}}}} = \snorm{U_{i}} +
    \snorm{U_{B_{t_{i-1}}}} = \snorm{U_{i}} + \psi(t_{i-1}),\]
  hence \(\norm{U_{i}} = \frac{\norm{U}}{n}\) for all \(i \le n\), as needed.
\end{proof}

\begin{proposition}\label{prop:induction-friendly-maximal-norm}
  Let \(\mathbb{G} \le \Aut(X,\mu)\) be an induction friendly Polish finitely full group and let
  \(\norm{\cdot}\) be a compatible additive norm on it. If the set of periodic elements is dense in
  \(\derived(\mathbb{G})\), then \(\norm{\cdot}\) is a maximal norm on \(\derived(\mathbb{G})\).
\end{proposition}

\begin{proof}
  In view of Proposition~\ref{prop:chara-maximal-norm}, it suffices to show that \(\norm{\cdot}\) is
  both large-scale geodesic\index{Norm!large-scale geodesic} (see
  Definition~\ref{def:large-scale-geodesic}) and coarsely proper\index{Norm!coarsely proper} (see
  Definition~\ref{def:coarsely-proper}).  Note that induction friendliness yields density in
  \(\derived(\mathbb{G})\) of periodic automorphisms with bounded periods.

  To see that \(\norm\cdot\) is large-scale geodesic (with constant \(K=2\)), let us take a
  non-trivial \( T\in \derived(\mathbb{G}) \) and pick a periodic \(U \in \derived(\mathbb{G}) \) with bounded
  periods such that \(\norm{TU^{-1}} < \min\{2,\norm{T}/2\}\).  Note that
  \begin{equation}
    \label{eq:U-norm-esitmate}
    \snorm{U} = \snorm{U^{-1}} = \snorm{T^{-1}TU^{-1}} \le \snorm{T^{-1}} +
    \snorm{TU^{-1}} \le 3\snorm{T}/2
  \end{equation}

  Fix \(n\in\N\) large enough to ensure \(\frac{3\norm T}{2n}\leq 2\).  By
  Lemma~\ref{lem:product-of-small}, we may decompose \(U\) into a product of \(n\) elements
  \(U_1,\dots,U_n\) each of norm at most \(\frac{3\norm{T}}{2n}\leq 2\).  Therefore
  \[T=(TU^{-1}) \cdot U_1\cdots U_n,\]
  where \(TU^{-1}\) and each of \(U_i\), \(1 \le i \le n\), has norm at most \(2\) and, in view of
  Eq.~\eqref{eq:U-norm-esitmate},
  \[
    \snorm{TU^{-1}} + \sum_{i=1}^n\norm{U_i}\leq \frac{\norm T}{2} + \snorm{U} \le 2\snorm{T},
  \]
  thus concluding the proof that \(\norm\cdot\) is large-scale geodesic.

  We now show that \(\norm\cdot\) is coarsely proper. Fix \(\epsilon>0\) and \(R>0\).  Let
  \(n\in\N\) be so large that \(n\epsilon \ge R+\epsilon\).  Then every element
  \(T \in \derived(\mathbb{G})\) of norm at most \(R\) is a product of \(n+1\) elements of norm at
  most \(\epsilon\), namely one element \(TU^{-1}\) of norm at most \(\epsilon \), where \(U\) is
  periodic with bounded periods as provided by density, and \(U = U_1\cdots U_n\), where each
  \(U_i\) has norm at most \(\frac{R+\epsilon} n\leq \epsilon\) as per
  Lemma~\ref{lem:product-of-small}.  Thus \(\norm{\cdot}\) is both coarsely proper and large-scale
  geodesic, and hence is maximal by Proposition~\ref{prop:chara-maximal-norm}.
\end{proof}

\begin{remark}
  We do not have an example of an induction friendly Polish finitely full group \(\mathbb G\) for
  which the periodic elements are not dense in \(\derived(\mathbb{G})\). A potential candidate might
  be the \(\LL^1\) full group of a free action of the free group on \(2\) generators, endowed with
  the norm given by the word length with respect to the canonical generating set.
\end{remark}



\chapter[Locally compact group actions]{Full groups of locally compact group actions}
\label{chap:l1-locally-compact}

In this chapter, we narrow down the generality of the narrative and focus on actions of
\emph{locally compact} Polish groups, or equivalently, of locally compact second-countable
groups. Such restrictions enlarge our toolbox in a number of ways. For instance, all locally compact
Polish group actions admit cross-sections to which the so-called Voronoi tessellations can be
associated. We use this to show in Section~\ref{sec:dense-subgr-full-groups} a natural density
result for subsets of \(\LL^1\) full groups defined from dense subsets of the acting group
(Theorem~\ref{thm:approximation-by-cocycle-in-dense-set} and
Corollary~\ref{cor:l1-approximation-by-cocycle-in-dense-set}).  For the reader's convenience,
Appendix~\ref{sec:tessellations} contains a concise reminder of the needed facts about
tessellations.

Another key property of free\footnote{Motivated by our focus on \(\mathbb{R}\)-flows, this monograph
  primarily concentrates on \emph{free} actions.  We note, however, that each orbit of a Borel
  action of a locally compact Polish group is a homogeneous space, since point stabilizers are
  necessarily closed.  In particular, orbits can be endowed with the Haar measure, even without the
  freeness assumption.} actions of locally compact groups is the existence of a Haar measure on each
individual orbit. As we discuss in Section~\ref{sec:orbit-transf}, elements of the full group act by
non-singular transformations and, in particular, admit the Hopf decomposition (see
Appendix~\ref{sec:hopf-decomposition-appendix}).  Section~\ref{sec:hopf-decomp-elem} explains how
these orbitwise decompositions can be understood globally, yielding a natural generalization of the
periodic/aperiodic partition for elements of the full group of a measure-preserving action of a
\emph{discrete} group.  The periodic part in the latter case corresponds to the conservative piece
of the Hopf decomposition, which generally exhibits a much more complicated dynamical behavior.  We
return to this in Chapters~\ref{chap:example-transf} and~\ref{chap:intermitted-transformations}.

In the final Section~\ref{sec:texorpdfstr-full-gro}, we connect \(\LL^1\) full groups to the notion
of \(\LL^1\) orbit equivalence for actions of locally compact \emph{compactly generated} Polish
groups.

\section{Dense subgroups in \texorpdfstring{\(\LL^1\)}{L1} full groups}
\label{sec:dense-subgr-full-groups}

Our goal in this section is to prove that any element of the full group \(\fgr{G \acts X}\) can be
approximated arbitrarily well by an automorphism that piecewise acts by elements of a given dense
subset of \(G\).

\begin{definition}
  \label{def:H-decomposable}
  A measure-preserving transformation \( T : A \to B \) between two measurable sets
  \( A, B \subseteq X \) is said to be
  \textbf{\( H \)-decomposable}\index{Transformation!decomposable@$H$-decomposable}, where
  \(H \subseteq \Aut(X,\mu)\), if there exist a measurable partition
  \( A = \bigsqcup_{k \in \mathbb{N}} A_{k} \) and elements \( h_{k} \in H \) such that
  \( T \restriction_{A_{k}} = h_{k}\restriction_{A_{k}} \) for all \( k \in \mathbb{N} \).
\end{definition}

The property of being \(H\)-decomposable is similar to being an element of the full group generated
by \(H\), except that we do not require the transformation to be defined on all of \(X\).

\begin{theorem}
  \label{thm:approximation-by-cocycle-in-dense-set}
  Let \( G \acts X \) be a measure-preserving action of a locally compact Polish group.  Let
  \( \norm{ \cdot } \) be a compatible norm on \( G \) with the associated metric on the orbits
  \( D : \eqr_{G} \to \mathbb{R}^{\ge 0} \), and let \( H \subseteq G \) be a dense set. For any
  \( T \in \fgr{G \acts X} \) and any \( \epsilon > 0 \), there exists an \( H \)-decomposable
  transformation \( S \in \fgr{G \acts X} \) such that \( \esssup_{x \in X} D(Tx, Sx) < \epsilon \).
\end{theorem}

Theorem~\ref{thm:approximation-by-cocycle-in-dense-set} establishes the density of
\(H\)-decomposable transformations in the very strong uniform topology given by \(\esssup\). In
particular, this result also applies to the \(\LL^{1}\) topology.

\begin{corollary}
  \label{cor:l1-approximation-by-cocycle-in-dense-set}
  Let \( G \acts X \) be a measure-preserving action of a locally compact Polish group, let
  \( \snorm{ \cdot } \) be a compatible norm on \( G \), and let \( H \subseteq G \) be a dense
  subgroup. The \(\LL^{1}\) full group \( \lfgr{ H \acts X } \) is dense in
  \( \lfgr{ G \acts X } \).
\end{corollary}

\begin{remark}
  \label{rem:density-in-full-groups}
  Theorem~\ref{thm:approximation-by-cocycle-in-dense-set} is an improvement upon the conclusion
  of~\cite[Thm.~2.1]{MR3748570}, which shows that \( \fgr{H \acts X} \) is dense in
  \( \fgr{G \acts X} \) whenever \( H \) is a dense subgroup of \( G \). While the proof, which we
  present below, establishes density in a much stronger topology through more elementary means, we
  note that, as already mentioned in~\cite[Thm.~2.3]{MR3748570}, their methods apply to all
  \emph{suitable} (in the sense of~\cite{MR2995370}; see also
  Definition~\ref{def:suitable-action}) actions of Polish groups, whereas our approach here
  crucially uses local compactness of the acting group to guarantee existence of various
  cross-sections.
\end{remark}

Let \( \mathcal{C} \) be a cross-section for a measure-preserving action \( G \acts X \), and let
\( \mathcal{W} \) be a tessellation over \( \mathcal{C} \) (in the sense of
Appendix~\ref{sec:tessellations}).  Let \( \nu_{\mathcal{W}}\) be the push-forward measure
\( (\pi_{\mathcal{W}})_{*}\mu \) on the cross-section, and let \( (\mu_{c})_{c \in \mathcal{C}} \)
be the disintegration of \( \mu \) over \( (\pi_{\mathcal{W}}, \nu_{\mathcal{W}}) \) (see
Appendix~\ref{sec:disintegration-measure} and Theorem~\ref{thm:disintegration-of-measure},
specifically).  Without loss of generality, we assume, whenever convenient, that the set \(H\) in
the statement of Theorem~\ref{thm:approximation-by-cocycle-in-dense-set} is countable.

\begin{definition}
  \label{def:proportionate-and-equimeasurable}
  Two Borel sets \( A, B \subseteq X \) are said to be
  \begin{itemize}
    \item \textbf{\( \mathcal{W} \)-proportionate} if the equivalence
          \( \mu_{c}(A) = 0 \iff \mu_{c}(B) = 0 \) holds for
          \( \nu_{\mathcal{W}} \)-almost all \( c \in \mathcal{C} \);
    \item \textbf{\( \mathcal{W} \)-equimeasurable} if \( \mu_{c}(A) = \mu_{c}(B) \) for
          \( \nu_{\mathcal{W}} \)-almost all \( c \in \mathcal{C} \).
  \end{itemize}

\end{definition}

For the context of Lemmas~\ref{lem:propornionate-containment}
through~\ref{lem:equimeasurable-match}, we let \(N\) denote an open \emph{symmetric} neighborhood of
the identity of \(G\), and \(\mathcal{W}\) stands for an \(N\)-lacunary tessellation.  The following
lemma relies on the key fact that for any two \(\mathcal{W}\)-proportionate Borel sets
\(A, B \subseteq N \cdot \mathcal{C}\), the equivalence
\(A \cap (N \cdot c) \neq \varnothing \iff B \cap (N \cdot c) \neq \varnothing\) holds for all
\(c \in \mathcal{C}\) after changing \(A\) and \(B\) on a null set.

\begin{lemma}
  \label{lem:propornionate-containment}
  If \( A, B \subseteq N \cdot \mathcal{C} \) are \( \mathcal{W} \)-proportionate Borel sets then
  \[ \mu(B \setminus N^{2} \cdot A) = 0. \]
\end{lemma}

\begin{proof}
  By the defining property of the disintegration,
  \[\mu(B\setminus N^{2} \cdot A) = \int_{\mathcal{C}} \mu_{c}(B \setminus N^{2} \cdot A)\,
    d\nu_{\mathcal{W}}(c),\]
  and so we need to check that \(\mu_{c}(B \setminus N^{2} \cdot A) = 0\) for
  \(\nu_{\mathcal{W}}\)-almost all \(c\).  Since \(A\) and \(B\) are \(\mathcal{W}\)-proportionate,
  it suffices to show that \(\mu_{c}(B \setminus N^{2} \cdot A) = 0\) whenever \(\mu_{c}(A) \ne 0\).
  For any \(c \in \mathcal{C}\) satisfying the latter, one necessarily has \(c \in N \cdot A\)
  (because \(A \subseteq N \cdot \mathcal{C}\) and \(\mathcal{W}\) is \(N\)-lacunary, by
  assumption), and thus \(N \cdot c \subseteq N^{2} \cdot A\).  In particular,
  \((B \setminus N^{2} \cdot A) \cap N \cdot c = \varnothing\). It remains to use the inclusion
  \(B \subseteq N \cdot \mathcal{C}\), which, together with the \(N\)-lacunarity of \(\mathcal{W}\),
  guarantees that
  \[\mu_{c}(B \setminus N^{2} \cdot A)
    = \mu_{c}((B \setminus N^{2} \cdot A) \cap N\cdot c) = 0. \qedhere\]
\end{proof}

For the proof of the next lemma, we need the notion of a suitable action, introduced by
H.~Becker~\cite[Def.~1.2.7]{MR2995370}.

\begin{definition}
  \label{def:suitable-action}
  A measure-preserving Borel action \( G \acts X \) of a Polish group~\(G\) is \textbf{suitable}
  \index{Suitable action}\index{Measure-preserving action!suitable} if for
  all Borel sets \( A, B \subseteq X \) one of the following two options holds:
  \begin{enumerate}
  \item\label{item:suitable-intersect} for any open neighborhood of the identity \( M \subseteq G \)
    there exists \( g \in M \) such that \( \mu(gA \cap B) > 0 \);
  \item\label{item:suitable-disjoint} there exist Borel sets \( A' \subseteq A \),
    \( B' \subseteq B \) such that \( \mu(A \setminus A') = 0 = \mu(B \setminus B') \) and an open
    neighborhood of the identity \( M \subseteq G \) such that
    \( M \cdot A' \cap B' = \varnothing \).
  \end{enumerate}
\end{definition}

All measure-preserving actions of locally compact Polish groups are known to be suitable
(see~\cite[Thm.~1.2.9]{MR2995370}).

\begin{lemma}
  \label{lem:shift-proportionate-sets-open}
  For all non-negligible \( \mathcal{W} \)-proportionate Borel sets
  \( A, B \subseteq N\cdot \mathcal{C} \), there exists an open set \( U \subseteq N^{3} \)
  such that \( \mu(gA \cap B) > 0 \) for all \( g \in U \).
\end{lemma}

\begin{proof}
  Let \(H_1=\{h_n :  n \in \N\}\) be a countable dense subset of \( N^{2} = NN^{-1}\), and put
  \( A_{1} = H_1 \cdot A \).  We apply the dichotomy in the definition of a suitable action to the
  sets \( A_{1}, B \) and show that item~\eqref{item:suitable-disjoint} cannot hold.

  Indeed, suppose there exist \( A_{1}' \subseteq A_{1} \), \( B' \subseteq B \) satisfying
  \[ \mu(A_{1} \setminus A'_{1}) = 0 = \mu(B\setminus B'), \]
  and an open neighborhood of the identity \( M \subseteq G \) such that
  \( (M \cdot A'_{1}) \cap B' = \varnothing \). Set
  \( A' = \bigcap_{n} (h_{n}^{-1}A'_{1} \cap A) \), and note that
  \( \mu(A \setminus A') = 0 \) and \( (MH_1 \cdot A') \cap B' = \varnothing \),
  simply because
  \( C_1 \cdot A' \subseteq A'_{1} \).  Since \( C_1 \) is dense in \( N^{2} \), we have
  \( N^{2} \subseteq MH_1 \) and thus \( (N^{2} \cdot A') \cap B' = \varnothing
  \). Lemma~\ref{lem:propornionate-containment}, applied to \(A'\) and \(B'\), guarantees that
  \(\mu(B' \setminus N^{2} \cdot A') = 0\), which is possible only when \(\mu(B') = 0\),
  contradicting the assumption that \(B\) is non-negligible.

  We are left with the alternative of the item~\eqref{item:suitable-intersect}, and so there has to
  exist some \( g \in N \) such that \( \mu(gA_{1} \cap B) > 0 \). Since \( A_{1} = H_1 \cdot A \),
  there exists \( h \in H_{1} \) such that \( \mu(gh A \cap B) > 0 \).  It remains to observe that
  \(gh \in N^{3}\) and that \( \mu(g'A \cap B) > 0 \) is an open condition on \(g'\).  This follows
  from the continuity in the weak topology of the group homomorphism \(G\to\Aut(X,\mu)\) associated
  with the measure-preserving action of \( G \) on \( (X,\mu) \) (see, for
  instance,~\cite[Lem.~1.2]{MR3748570}).
\end{proof}

\begin{lemma}
  \label{lem:shift-one-set-open}
  For any non-empty open \( V \subseteq N \) and for any non-negligible Borel set
  \( A \subseteq X \), there exists \( h \in H \) such that
  \[ \mu(\{ x \in A : hx \in V\cdot \mathcal{C} \textrm{ and
    } \pi_{\mathcal{W}}(x) = \pi_{\mathcal{W}}(hx) \}) > 0. \]
\end{lemma}
\begin{proof}
  Let \( \zeta : X \to \mathcal{W} \) be the Borel bijection
  \( \zeta(x) = (\pi_{\mathcal{W}}(x), x) \) and consider the push-forward measure
  \( \zeta_{*}\mu \), which for \( Z \subseteq \mathcal{W} \) can be expressed as
  \[ \zeta_{*}\mu(Z) = \int_{\mathcal{C}} \mu_{c}(Z_{c})\, d\nu_{\mathcal{W}}(c). \]
  Let \( (h_{n})_{n \in \mathbb{N}} \) be an enumeration of \( H \) and set
  \[ W_{n} = \{ (c, x) \in \mathcal{W} : \pi_{\mathcal{W}}(x) = \pi_{\mathcal{W}}(h_{n}x) \textrm{
      and } h_{n}x \in V \cdot \mathcal{C} \}. \]
  We claim that \( \bigcup_{n}W_{n} = \mathcal{W} \).  Indeed, for each \( (c,x) \in \mathcal{W} \)
  the set of \( g \in G \) such that \( gx \in V \cdot c \) is non-empty and open, hence there is
  \( h_{n} \in H \) such that \( h_{n}x \in V \cdot c \).

  Finally, \(A\) is non-negligible by assumption, i.e., \( 0 < \mu(A) = \zeta_{*}\mu (\zeta(A)) \),
  so there exists \( W_{n} \) such that \( \zeta_{*}\mu (\zeta(A) \cap W_{n}) > 0 \), which
  translates into the required
  \[ \mu(\{ x \in A : h_{n}x \in V \cdot \mathcal{C} \textrm{ and
    } \pi_{\mathcal{W}}(x) = \pi_{\mathcal{W}}(h_{n}x) \}) > 0. \qedhere \]
\end{proof}

\begin{lemma}
  \label{lem:shift-proportionate-sets-dense}
  For all non-negligible \( \mathcal{W} \)-proportionate Borel sets \( A, B \subseteq X \), there
  exists \( h \in H \) such that
  \[ \mu(\{ x \in A : hx \in B \textrm{ and } \pi_{\mathcal{W}}(x) = \pi_{\mathcal{W}}(hx)\}) > 0. \]
\end{lemma}

\begin{proof}
  The plan is to reduce the setup of this lemma to that of
  Lemma~\ref{lem:shift-proportionate-sets-open}. Let \( V \subseteq N \) be a symmetric neighborhood
  of the identity that is furthermore small enough to guarantee that \( \mathcal{W} \) is
  \( V^{4} \)-lacunary. Apply Lemma~\ref{lem:shift-one-set-open} to find \( h_{1}\in H \) such that
  for
  \[ A' = \{ x \in A : h_{1}x \in V\cdot \mathcal{C} \textrm{ and } \pi_{\mathcal{W}}(x) =
    \pi_{\mathcal{W}}(h_{1}x) \} \]
  one has \( \mu(A') > 0 \). Set \( A_{1} = h_{1}A' \),
  \( B_{1} = \pi_{\mathcal{W}}^{-1}(\{ c \in \mathcal{C} : \mu_{c}(A_{1}) > 0 \}) \cap B \) and note
  that \( A_{1} \) and \( B_{1} \) are non-negligible \( \mathcal{W} \)-proportionate sets.
  Moreover, \( A_{1} \subseteq V \cdot \mathcal{C} \) by construction.

  Repeat the same steps for \( B_{1} \) and find \( h_{2} \in H \) such that for
  \[ B_{1}' = \{ x \in B_{1} : h_{2}x \in V\cdot \mathcal{C} \textrm{ and } \pi_{\mathcal{W}}(x) =
    \pi_{\mathcal{W}}(h_{2}x) \} \]
  we have \( \mu(B_{1}') > 0 \). Set \( B_{2} = h_{2}B_{1}' \) and
  \( A_{2} = A_{1} \cap \pi_{\mathcal{W}}^{-1}(\{ c \in \mathcal{C} : \mu_{c}(B_{2}) > 0\}) \). Once
  again, sets \( A_{2} \) and \( B_{2} \) are non-negligible, \( \mathcal{W} \)-proportionate and
  are both contained in \( V \cdot \mathcal{C} \).

  We now apply Lemma~\ref{lem:shift-proportionate-sets-open} to sets \(A_{2}, B_{2}\) and
  \(\mathcal{W}\), viewed as a \(V\)-lacunary tessellation, yielding an open \( U \subseteq V^{3} \)
  such that \( \mu(gA_{2} \cap B_{2}) > 0 \) for all \( g \in U \). Note that since
  \( U \subseteq V^{3} \) and \( \mathcal{W} \) is, in fact, \( V^{4} \)-lacunary, the equality
  \( \pi_{\mathcal{W}}(x) = \pi_{\mathcal{W}}(gx) \) holds for all \( x \in V \cdot \mathcal{C} \)
  and \( g \in U \). We conclude that \(\mu(h_{2}^{-1}gh_{1}A \cap B) > 0\) for all \( g \in U \)
  and hence any \( h \in h_{2}^{-1}Uh_{1} \cap H \) satisfies the conclusion of the lemma.
\end{proof}

  A measure-preserving partial transformation \( T : A \to B \) is \textbf{\( \mathcal{W} \)-coherent} if
  \( \mu \)-almost surely one has \( \pi_{\mathcal{W}}(x) = \pi_{\mathcal{W}}(Tx) \).

\begin{lemma}
  \label{lem:equimeasurable-match}
  For all \( \mathcal{W} \)-equimeasurable Borel sets \( A, B \subseteq X \), there exists a
  \( \mathcal{W} \)-coherent \( H \)-decomposable measure-preserving bijection \( T : A \to B \).
\end{lemma}

\begin{proof}
  Let \( (h_{n})_{n \in \mathbb{N}} \) be an enumeration of \( H \).
  Consider the set
  \[ A_{0} = \bigl\{x \in A : h_{0}x \in B \textrm{ and } \pi_{\mathcal{W}}(x) =
    \pi_{\mathcal{W}}(h_{0}x) \bigr\}, \]
  and let \( B_{0} = h_{0}A_{0} \). Note that the sets \( A \setminus A_{0} \) and
  \( B \setminus B_{0} \) are \( \mathcal{W} \)-equimeasurable, so we may continue in the same
  fashion and construct sets \( A_{k} \) such that
  \[ A_{k} = \Bigl\{x \in A \setminus \bigsqcup_{i<k}A_{i} : h_{k}x \in B\setminus \bigsqcup_{i<k}
    B_{i} \textrm{ and } \pi_{\mathcal{W}}(x) = \pi_{\mathcal{W}}(h_{k}x)\Bigr\}. \]
  We define \( T : \bigsqcup_{k \in \mathbb{N}} A_{k} \to \bigsqcup_{k \in \mathbb{N}} B_{k} \) by
  the condition \( Tx = h_{k}x \) for \( x \in A_{k} \).

  Sets \(A \setminus \bigsqcup_{k \in
    \mathbb{N}} A_{k}\) and \(B \setminus \bigsqcup_{k \in \mathbb{N}} B_{k}\) are
  \(\mathcal{W}\)-equimeasurable. If either one of them (and thus necessarily both of them) were
  non-negligible, Lemma~\ref{lem:shift-proportionate-sets-dense} would yields an element \(h \in H\)
  that moves a portion of \(A \setminus \bigsqcup_{k \in
    \mathbb{N}} A_{k}\) into \(B \setminus \bigsqcup_{k \in \mathbb{N}} B_{k}\), contradicting the
  construction. We conclude that
  \[ \mu(A \setminus \bigsqcup_{k \in \mathbb{N}} A_{k}) = 0 = \mu(B \setminus \bigsqcup_{k \in
      \mathbb{N}} B_{k}) \] and \(T\) is therefore as required.
\end{proof}

\begin{lemma}
  \label{lem:equimeasurable-approximation}
  Suppose that \( \mathcal{W} \) is a cocompact tessellation over the cross-section \(\mathcal{C}\).
  Let \( A, B \subseteq X \) be \( \mathcal{W} \)-equimeasurable Borel sets. For any
  \( \epsilon > 0 \), and any \( \mathcal{W} \)-coherent partial transformation \( T : A \to B \), there
  exists a \( \mathcal{W} \)-coherent \( H \)-decomposable \( \tilde{T} : A \to B \) such that
  \( \esssup_{x \in A} D(Tx, \tilde{T}x) < \epsilon \).
\end{lemma}

\begin{proof}
  Let \( \mathcal{V} \) be a \( K' \)-cocompact tessellation over some cross-section
  \( \mathcal{C}' \) such that the diameter of each region in \( \mathcal{V} \) is less than
  \( \epsilon \). Suppose \( \mathcal{W} \) is \( K \)-cocompact. By
  Lemma~\ref{lem:lacunary-partition}, we can find a finite partition of
  \( \mathcal{C}' = \bigsqcup_{i \le n} \mathcal{C}'_{i} \) such that each \( \mathcal{C}'_{i} \)
  is \( K'K^{2}K' \)-lacunary, which guarantees that, for each \( i \), every \( \mathcal{W}_{c} \)
  intersects at most one class \( \mathcal{V}_{c'} \), \( c' \in \mathcal{C}'_{i} \). For each
  \( i,j < n \) set
  \( A_{(i,j)} = \{ x \in A : \pi_{\mathcal{V}}(x) \in \mathcal{C}'_{i}, \pi_{\mathcal{V}}(Tx) \in
  \mathcal{C}_{j}'\} \) and \( B_{(i,j)} = TA_{(i,j)} \). We re-enumerate sets \( A_{(i,j)} \) and
  \( B_{(i,j)} \) as a sequence \( A_{k} \), \( B_{k} \), \( k \le n^{2} \) and note that for all
  \( x,y \in A_{k} \) one has
  \[ \pi_{\mathcal{W}}(x) =\pi_{\mathcal{W}}(y) \implies \bigl(\pi_{\mathcal{V}}(x) =
    \pi_{\mathcal{V}}(y) \textrm{ and } \pi_{\mathcal{V}}(Tx) = \pi_{\mathcal{V}}(Ty)\bigr). \]
  Moreover, sets \( A_{k} \) and \( T(A_{k}) \) are \( \mathcal{W} \)-equimeasurable, so
  Lemma~\ref{lem:equimeasurable-match} yields \( \mathcal{W} \)-coherent \( H \)-decomposable
  partial transformations \( T_{k} : A_{k} \to T(A_{k}) \). The transformation
  \( \tilde{T} : A \to B \) can now be defined by the condition \( \tilde{T}x = T_{k}x\) whenever
  \( x \in A_{k} \). It is easy to check that \( \tilde{T} \) is as claimed.
\end{proof}

\begin{proof}[Proof of Theorem~\ref{thm:approximation-by-cocycle-in-dense-set}]
  By Proposition~\ref{prop:union-of-smooth}, we have a sequence \((\mathcal V_k)_{k}\) of cocompact
  tessellations such that \(\mathcal R_G=\bigcup_k \mathcal R_{\mathcal V_k}\).  Let
  \( A_{0} = \{ x \in X : \pi_{\mathcal{V}_{0}}(x) = \pi_{\mathcal{V}_{0}}(Tx) \} \). Use
  Lemma~\ref{lem:equimeasurable-approximation} to find an \( H \)-decomposable
  partial transformation \( T_{0} : A_{0} \to T(A_{0}) \) that satisfies the inequality
  \( \esssup_{x \in A_{0}}D(T_{0}x, Tx) < \epsilon \). Set
  \[ A_{k} = \Bigl\{ x \in X : \pi_{\mathcal{V}_{k}}(x) = \pi_{\mathcal{V}_{k}}(Tx) \textrm{ and } x
    \not \in \bigsqcup_{l < k}A_{l} \Bigr\} \]
  and note that \( A_{k} \), \( k \in \mathbb{N} \), form a partition of \( X \) because
  \(\mathcal R_G=\bigcup_k \mathcal R_{\mathcal V_k}\). Construct partial transformations
  \( T_{k} : A_{k} \to T(A_{k}) \) via repeated applications of
  Lemma~\ref{lem:equimeasurable-approximation} to the tessellations \( \mathcal{V}_{k} \).  The
  element \( S \in \fgr{G \acts X} \) defined for \(x \in A_{k}\) by \( Sx = T_{k}x \) satisfies the
  conclusion of the theorem.
\end{proof}

\section{Orbital transformations}
\label{sec:orbit-transf}

Let \(G \acts X\) be a \emph{free} measure-preserving action of a locally compact Polish group on a
standard probability space. Recall that the identification of \(G\) with its orbits induces the
cocycle\index{Cocycle} map \(\rho : \mathcal{R}_G \to G\), defined by \(\rho(x, gx) = g\). Moreover,
every \(T \in [\mathcal{R}_G]\) has an associated cocycle \(\rho_T : X \to G\) determined by the
condition \(T(x) = \rho_T(x)x\) for all \(x \in X\).

Fix a \emph{right}-invariant Haar measure \(\lambda\) on \(G\). Since any orbit \( [x]_{\eqr_{G}} \)
can be identified with the group \(G\) via the map \( G \ni g \mapsto gx \in [x]_{\eqr_{G}} \), the
measure \(\lambda\) can be pushed forward through this identification to define a collection of
measures \((\lambda_{x})_{x \in X}\) on \(X\). These measures are given by
\( \lambda_{x}(A) = \lambda(\{ g \in G : gx \in A\}) \). The right invariance of \(\lambda\) ensures
that \(\lambda_{x}\) depends only on the orbit \( [x]_{\eqr_{G}} \) and is independent of the choice
of the base point; that is, \(\lambda_{x} = \lambda_{y}\) whenever \( x \eqr_{G} y \).

This section focuses on two main facts: the so-called mass-transport principle, given in
Eq.~\eqref{eq:mass-transport-principle} below, and the non-singularity of the transformations
induced by elements of \(\fgr{G \acts X}\) onto orbits of the action, formulated in
Proposition~\ref{prop:orbital-transformations}.  Both of these topics have been discussed in the
literature in various related contexts, including, for instance,~\cite[Appen.~A]{MR3748570} and the
treatise~\cite{MR1799683}.  However, we are not aware of any specific reference from which
Eq.~\eqref{eq:mass-transport-principle} and Proposition~\ref{prop:orbital-transformations} can be
readily deduced.  The following derivations are therefore included for the reader's convenience,
with the disclaimer that these results are likely to be known to experts.

The freeness of the action allows us to identify the equivalence relation \( \eqr_{G} \) with
\( X \times G \) via \( \Phi : X \times G \to \eqr_{G} \), \( \Phi(x, g) = (x, gx) \). The
push-forward \( \Phi_{*} (\mu \times \lambda) \) of the product measure is denoted by \( M \) and
can equivalently be defined by
\[ M(A) = \int_{X} \lambda_{x}(A_{x})\, d\mu(x), \]
where \( A \subseteq \eqr_{G} \) and \( A_{x} = \{ y \in X : (x,y) \in A\} \).

In general, the flip transformation \( \sigma : \eqr_{G} \to \eqr_{G} \), \( \sigma(x,y) = (y,x) \),
does not preserve the measure \( M \). Set \( \Psi : X \times G \to X \times G \) to be
the involution \( \Psi = \Phi^{-1} \circ \sigma \circ \Phi \),
which simplifies to \( \Psi(x, g) = (gx, g^{-1}) \).
Following the computation as in~\cite[Prop.~A.11]{MR3748570}, one can easily check that
\( \Psi_{*}(\mu \times \lambda) = \mu \times \widehat{\lambda} \), where \( \widehat{\lambda} \) is
the associated \emph{left}-invariant measure, \( \widehat{\lambda}(A) = \lambda(A^{-1}) \). If we
define the measure \( \widehat{M} \) on \( \eqr_{G} \) to be
\[ \widehat{M}(A) = \Phi_{*}(\mu \times \widehat{\lambda}) = \int_{X} \widehat{\lambda}_{x}(A_{x})\,
  d\mu(x), \]
then \( \sigma_{*}M = \widehat{M} \), and also \(\sigma_{*}\widehat{M} = M\), since
\(\sigma^{-1} = \sigma\). In particular, \( \sigma \) is \( M \)-invariant if and only if
\( \lambda = \widehat{\lambda} \), i.e., \( G \) is unimodular.

A function \( f : \eqr_{G} \to \mathbb{R} \) is \( M \)-integrable if and only if
\( X \times G \ni (x,g) \mapsto f(x, gx) \) is \( (\mu \times \lambda) \)-integrable.  Using
Fubini's theorem and noting that \(\Phi\circ\Psi\inv=\Phi\circ\Psi=\sigma\circ\Phi\),
we get the following chain of identities for such a function \(f\):
\begin{align*}
\int_X\int_Gf(x,gx)\,d\lambda(g)d\mu(x)&=\int_{X \times G} f(x, gx) \, d(\mu
\times \lambda)(x,g)\\
&=\int_{X\times G}f\circ \Phi\, d(\Psi\inv_*(\mu\times \widehat\lambda))\\
&=\int_{X\times G}f\circ \Phi \circ\Psi\inv\, d(\mu\times \widehat\lambda)\\
&=\int_{X\times G}f\circ \sigma\circ\Phi\, d(\mu\times \widehat\lambda)\\
\int_X\int_Gf(x,gx)\,d\lambda(g)d\mu(x)
&=\int_X \int_G f(gx,x)\,d\widehat\lambda(g)d\mu(g).
\end{align*}

Let \( \Delta : G \to \mathbb{R}^{>0} \) be the \emph{left Haar modulus}\index{Haar modulus} given
for \( g \in G \) by \( \lambda(gA) = \Delta(g)\lambda(A) \). Recall that
\( \Delta : G \to \mathbb{R}^{>0} \) is a continuous homomorphism
(see~\cite[Prop.~7]{MR0175995}). The measures \( \lambda \) and \( \widehat{\lambda} \) belong to
the same measure class, with the Radon--Nikodym derivative
\( \frac{d\widehat{\lambda}}{d\lambda} (g) = \Delta(g^{-1}) \) for all \( g \in G \)
(see~\cite[p.~79]{MR0175995}).  The identities above translate into the following:
\begin{equation}
  \label{eq:mass-transport-principle}
  \int_{X} \int_{G} f(x, g \cdot x)\, d\lambda(g) d\mu(x) =
  \int_{X} \int_{G} \Delta(g^{-1}) f(g \cdot x, x)\, d\lambda(g) d\mu(x).
\end{equation}
When the group \( G \) is unimodular, this expression attains a very symmetric form and is known as
the \textbf{mass-transport principle}\index{Mass-transport principle}:
\begin{equation}
  \label{eq:mass-transport-principle-unimodular}
  \int_{X} \int_{G} f(x, g \cdot x)\, d\lambda(g) d\mu(x) =
  \int_{X} \int_{G}  f(g \cdot x, x)\, d\lambda(g) d\mu(x).
\end{equation}

Any automorphism \( T \in \fgr{G \acts X} \) induces, for each \( x \in X \), a transformation of
the \( \sigma \)-finite measure space \( (X, \lambda_{x}) \). In general, \( T \) does not preserve
\( \lambda_{x} \); however, it is always non-singular, and the Radon--Nikodym derivative
\( \frac{dT_{*}\lambda_{x}}{d\lambda_{x}} \) can be described explicitly. Note that the full group
\( \fgr{G \acts X} \) admits two natural actions on the equivalence relation \( \eqr_{G} \): the
\emph{left} action \( l \) is given by \( l_{T}(x,y) = (Tx,y) \), and the \emph{right} action
\( r \) is defined as \( r_{T}(x,y) = (x,Ty) \). A straightforward verification
(see~\cite[Lem.~A.9]{MR3748570}) shows that \( l \) is always \( M \)-invariant. Since
\( r_{T} \circ \sigma = \sigma \circ l_{T} \), for all \(T \in \fgr{G \acts X}\), we have
\[ (r_{T})_{*} \widehat{M} = (r_{T} \circ \sigma)_{*} M = (\sigma \circ l_{T})_{*} M = \sigma_{*} M = \widehat{M}.\]

Let \( \Theta = \Phi^{-1} \circ r_{T} \circ \Phi \), i.e., \( \Theta(x,g) = (x, \rho_{Tg}(x)) \).
The equality \( (r_{T})_{*} \widehat{M} = \widehat{M} \) is equivalent to
\( \Theta_{*}(\mu \times \widehat{\lambda}) \) = \( \mu \times \widehat{\lambda} \).  The latter
implies that for each Borel set \( B \subseteq G \) and all measurable sets \( A \subseteq X \), we
have
\begin{displaymath}
  \begin{aligned}
    \int_{A} \widehat{\lambda}(B) \, d\mu
    &=  (\mu \times \widehat{\lambda})(A \times B) = \Theta_{*}(\mu \times \widehat{\lambda})(A \times B) \\
    &= (\mu \times \widehat{\lambda})\bigl(\{ (x, g) \in X \times G : (x, \rho_{Tg}(x)) \in A \times B\}\bigr) \\
    \textrm{Fubini's theorem} &= \int_{A} \widehat{\lambda}(\{ g \in G : \rho_{Tg}(x) \in B \})\, d\mu(x) \\
    &= \int_{A} \widehat{\lambda}(\{ g \in G : gx \in T^{-1}Bx \})\, d\mu(x),
  \end{aligned}
\end{displaymath}
which is possible only if
\( \widehat{\lambda}(\{ g \in G : gx \in T^{-1}Bx \}) = \widehat{\lambda}(B) \) for \( \mu \)-almost
all \( x \). Passing to the measures on the orbits, this translates for each \( B \) into
\( \widehat{\lambda}_{x}(T^{-1}Bx) = \widehat{\lambda}_{x}(Bx) \). If
\( (B_{n})_{n \in \mathbb{N}} \) is a countable algebra of Borel sets in \( G \) that generates the
whole Borel \( \sigma \)-algebra, then for each \( x \in X \), \( (B_{n}x)_{n \in \mathbb{N}} \) is
an algebra of Borel subsets of the orbit \( [x]_{\mathcal{R}_{G}} \), which generates the Borel
\( \sigma \)-algebra on it. We have established that for \( \mu \)-almost all \( x \in X \), the two
measures, \( \widehat{\lambda}_{x} \) and \( T_{*}\widehat{\lambda}_{x} \), coincide on each
\( B_{n}x \), \( n \in \mathbb{N} \), thus \( \mu \)-almost surely
\( \widehat{\lambda}_{x} = T_{*}\widehat{\lambda}_{x} \).

The equality \( \frac{d\widehat{\lambda}}{d\lambda} (g) = \Delta(g^{-1}) \) translates into
\( \frac{d\widehat{\lambda}_{x}}{d\lambda_{x}} (y) = \Delta(\rho(x,y)^{-1}) = \Delta(\rho(y,x)) \),
and the Radon--Nikodym derivative \( \frac{dT_{*}\lambda_{x}}{d\lambda_{x}} \) can now be computed as
follows:
\begin{displaymath}
  \begin{aligned}
    \frac{dT_{*}\lambda_{x}}{d\lambda_{x}}(y)
    &= \frac{dT_{*}\lambda_{x}}{dT_{*}\widehat{\lambda}_{x}}(y) \cdot
      \frac{dT_{*}\widehat{\lambda}_{x}}{d\widehat{\lambda}_{x}}(y) \cdot
      \frac{d\widehat{\lambda}_{x}}{d\lambda_{x}}(y) \\
    \textrm{\( T \) preserves \( \widehat{\lambda}_{x} \)}
    &= \frac{dT_{*}\lambda_{x}}{dT_{*}\widehat{\lambda}_{x}}(y) \cdot
      \frac{d\widehat{\lambda}_{x}}{d\lambda_{x}}(y)
      =\frac{d\lambda_{x}}{d\widehat{\lambda}_{x}}(T^{-1}y) \cdot
      \frac{d\widehat{\lambda}_{x}}{d\lambda_{x}}(y) \\
    &= \Bigl(\frac{d\widehat{\lambda}_{x}}{d\lambda_{x}}(T^{-1}y)\Bigr)^{-1}
      \cdot \frac{d\widehat{\lambda}_{x}}{d\lambda_{x}}(y) \\
    &= \Delta(\rho(x, T^{-1}y)^{-1})^{-1} \Delta(\rho(x, y)^{-1}) \\
    &= \Delta\bigl(\rho(x, T^{-1}y) \cdot \rho(y, x)\bigr) = \Delta(\rho_{T^{-1}}(y)).
  \end{aligned}
\end{displaymath}
We summarize the content of this section into a proposition.

\begin{proposition}
  \label{prop:orbital-transformations}
  Let \( G \) be a locally compact Polish group acting freely \( G \acts X \) on a standard
  probability space \( (X, \mu) \). Let \( \lambda \) be a right Haar measure on \( G \),
  \( \Delta : G \to \mathbb{R}^{>0} \) be the corresponding Haar modulus, and let
  \( (\lambda_{x})_{x \in X} \) be the family of measures obtained by pushing \( \lambda \) onto
  orbits via the action map. Each \( T \in \fgr{G \acts X} \) induces a non-singular transformation
  of \( (X, \lambda_{x}) \) for almost every \( x \in X \), and moreover, one has
  \( \lambda_{x} (T^{-1}A) = \int_{A} \Delta(\rho_{T^{-1}}(y)) \, d\lambda_{x}(y)\) for all Borel
  sets \( A \subseteq X \). If \( G \) is unimodular, then \( T_{*}\lambda_{x} = \lambda_{x} \) for
  \(\mu\)-almost all \( x \in X \).
\end{proposition}

For future reference, we isolate a simple lemma, which is an immediate consequence of Fubini's
theorem.

\begin{lemma}
  \label{lem:zero-set-fubini}
  Let \( G \) be a locally compact Polish group acting freely on a standard probability space
  \( (X, \mu) \). Let \( \lambda \), \( \widehat{\lambda} \), \( (\lambda_{x})_{x \in X} \), and
  \( (\widehat{\lambda}_{x})_{x \in X} \) be as above. For any Borel set \( A \subseteq X \), the
  following are equivalent:
  \begin{enumerate}
    \item\label{item:mA-zero} \( \mu(A) = 0 \);
    \item\label{item:lx-zero} \( \lambda_{x}(A) = 0 \) for \( \mu \)-almost all \( x \in X \);
    \item\label{item:lhx-zero} \( \widehat{\lambda}_{x}(A) = 0 \) for \( \mu \)-almost all \( x \in X \).
  \end{enumerate}
\end{lemma}

\begin{proof}
  \noindent~\eqref{item:mA-zero} \( \iff \) {}~\eqref{item:lx-zero} Using Fubini's Theorem on
  \( (X \times G, \mu \times \lambda) \) to rearrange the order of quantifiers, one has:
    \begin{multline*}
      \mu(A) = 0 \iff \forall g \in G\ \forall^{\mu} x \in X\ gx \not \in A \\
      \iff \forall^{\mu} x \in X\ \forall^{\lambda} g \in G\ gx \not \in A\iff \forall^{\mu}
      x \in X\ \lambda_{x}(A) = 0.
\end{multline*}

\noindent~\eqref{item:lx-zero} \( \iff \) {}~\eqref{item:lhx-zero} is evident, since \( \lambda \)
and \( \widehat{\lambda} \) are equivalent measures, hence so are \( \lambda_{x} \) and
\( \widehat{\lambda}_{x} \) for all \( x \in X \).
\end{proof}

\section{The Hopf decomposition of elements of the full group}
\label{sec:hopf-decomp-elem}

Fix an element \( T \in \fgr{G \acts X} \) of the full group of a free measure-preserving action of
a locally compact Polish group \(G\).  As explained in Section~\ref{sec:orbit-transf}, \(T\) acts
naturally in a non-singular manner on each \(G\)-orbit.  This action thus has a Hopf decomposition
(see Appendix~\ref{sec:hopf-decomposition-appendix}). We will now explain how to interpret this
decomposition globally, thereby generalizing the fact that, when \(G\) is discrete, any element of
the full group decomposes the space into periodic and aperiodic parts.

Let \( \mathcal{C} \) be a cocompact cross-section, and let \( \mathcal{V}_{\mathcal{C}} \) be
the Voronoi tessellation associated with some proper norm on \(G\) (see
Appendix~\ref{sec:tessellations}). Set \( \pi_{\mathcal{C}} : X \to \mathcal{C} \) to be the
projection map given by the condition \( (\pi_{\mathcal{C}}(x), x) \in \mathcal{V}_{\mathcal{C}}\)
for all \( x \in X \). The dissipative and conservative sets of the transformation \(T\) are defined
as follows:
\begin{displaymath}
  \begin{aligned}
    D_T &= \bigl\{ x \in X : \exists n \in \mathbb{N}\ \forall k \in \mathbb{Z} \textrm{ such that
        } |k| \ge n \textrm{ one has } \pi_{\mathcal{C}}(x) \ne \pi_{\mathcal{C}}(T^{k}x) \bigr\}, \\
    C_T &= \bigl\{ x \in X: \forall n \in \mathbb{N}\ \exists k_{1}, k_{2} \in \mathbb{Z} \textrm{
        such that } \\
      & \qquad \qquad k_{1} \le -n, n \le k_{2} \textrm{ and
        } \pi_{\mathcal{C}}(T^{k_{1}}x) = \pi_{\mathcal{C}}(x) = \pi_{\mathcal{C}}(T^{k_{2}}x) \bigr\}.
  \end{aligned}
\end{displaymath}
In plain words, the dissipative set \( D_T \) consists of those points \( x \) whose orbit has a
finite intersection with the Voronoi region of \( x \). The conservative set \( C_T \), on the other
hand, collects all the points whose orbit is bi-recurrent in the region. We argue in
Proposition~\ref{prop:hopf-decomposition-full-group} that the sets \( D_T \) and \( C_T \) induce the Hopf
decomposition for \( T \restriction_{[x]_{\eqr_{T}}} \) for almost every \( x \in X \); in
particular, \( D_T \sqcup C_T \) is a partition of \( X \), which is independent of the choice of the
cross-section \( \mathcal{C} \).

\begin{lemma}
  \label{lem:dc-partition}
  The sets \( D_T \) and \( C_T \) partition the phase space: \( X = D_T \sqcup C_T \).
\end{lemma}
\begin{proof}
  Define sets \( N_{+} \) and \( N_{-} \) according to
  \begin{displaymath}
    \begin{aligned}
      N_{+} &= \{x \in X \setminus (D_T \sqcup C_T) : \forall k \ge 1\ \pi_{\mathcal{C}}(T^{k}x)
              \ne \pi_{\mathcal{C}}(x)\}, \\
      N_{-} &= \{x \in X \setminus (D_T \sqcup C_T) : \forall k \ge 1\ \pi_{\mathcal{C}}(T^{-k}x)
              \ne \pi_{\mathcal{C}}(x)\}, \\
    \end{aligned}
  \end{displaymath}
  and note that
  \( X \setminus (D_T \sqcup C_T) \subseteq \bigcup_{k \in \mathbb{Z}} T^{k}(N_{+} \cup N_{-}) \). To
  show that \( X = D_T \sqcup C_T \) it is enough to verify that \( \mu(N_{+}) = 0 = \mu(N_{-}) \).

  This is done by noting that these sets admit pairwise disjoint copies using piecewise translations
  by powers of~\( T \). In view of the fact that \( T \) is measure-preserving, this implies that
  \( N_{+} \) and \( N_{-} \) are null. To be more precise, set \( N_{-}^{0} = N_{-} \) and define
  inductively \( N^{n}_{-} = \{T^{k(x)}x : x \in N^{n-1}_{-}\} \), where \( k(x) \ge 1 \) is the
  smallest natural number such that \( \pi_{\mathcal{C}}(T^{k(x)}x) = \pi_{\mathcal{C}}(x) \). Note
  that \( k(x) \) is well-defined, for otherwise \( x \) would belong to \( D_{T} \). Sets
  \( N^{n}_{-} \), \( n \in \mathbb{N} \), are pairwise disjoint, and have the same measure since
  \( T \) is measure-preserving. We conclude that \( \mu(N_{-}) = 0 \). The argument for
  \( \mu(N_{+}) = 0 \) is similar.
\end{proof}

\begin{proposition}[Hopf decomposition\index{Hopf decomposition}]
  \label{prop:hopf-decomposition-full-group}
  Let \( G \acts X \) be a free measure-preserving action of a locally compact Polish group on a
  standard probability space \( (X, \mu) \). Let \( \lambda \) be a right Haar measure on \( G \)
  and \( (\lambda_{x})_{x \in X} \) be the push-forward of \( \lambda \) onto the orbits as
  described in Section~\ref{sec:orbit-transf}. For any element \( T \in \fgr{G \acts X} \), the
  measurable \( T \)-invariant partition \( X = D_T \sqcup C_T \) defined above satisfies that for
  \( \mu \)-almost all \( x \in X \) the partition
  \( [x]_{\eqr_{G}} = ([x]_{\eqr_{G}} \cap D_T) \sqcup ([x]_{\eqr_{G}}\cap C_T)\) is the Hopf
  decomposition for \( T\restriction_{[x]_{\eqr_{G}}} \) on \( ([x]_{\eqr_{G}}, \lambda_{x}) \).
  Moreover, there is only one partition \( X = D_T \sqcup C_T \) satisfying this property up to null
  sets.
\end{proposition}

\begin{proof}
  According to Proposition~\ref{prop:orbital-transformations}, we may assume that for all
  \( x \in X \) the map \( T\restriction_{[x]_{\eqr_{G}}} : [x]_{\eqr_{G}} \to [x]_{\eqr_{G}} \) is
  a non-singular transformation with respect to \( \lambda_{x} \) and satisfies
  \( \lambda_{x}(TA) = \int_{A} \Delta(\rho_{T}(y))\, d\lambda_{x}(y) \) for all Borel
  \( A \subseteq X \).

  Let \( [x]_{\mathcal{R}_{G}} = D_{x} \sqcup C_{x} \), \( x \in X \), denote the Hopf
  decomposition for \( T\restriction_{[x]_{\eqr_{G}}} \). For any \( c \in \mathcal{C} \), the set
  \[ \widetilde{W}_{c} = \bigl\{ x \in (\mathcal{V}_{\mathcal{C}})_{c} : T^{k}x \not \in
    (\mathcal{V}_{\mathcal{C}})_{c} \textrm{ for all } k \ge 1 \bigr\} \]
  is a wandering set and therefore \( \widetilde{W}_{c} \subseteq D_{x} \) up to a null set. If
  \( x \in D_T \) satisfies \( x \in (\mathcal{V}_{\mathcal{C}})_{c} \), \( c \in \mathcal{C} \),
  then \( [x]_{\eqr_{G}} \cap (\mathcal{V}_{\mathcal{C}})_{c} \) is finite, and therefore
  \( [x]_{\eqr_{G}} \cap (\mathcal{V}_{\mathcal{C}})_{c} \subseteq \bigcup_{k \in \mathbb{Z}}
  T^{k}\widetilde{W}_{c} \), whence also
  \[ [x]_{\eqr_{G}} \cap D_T \subseteq \bigcup_{c \in \mathcal{C} \cap [x]_{\eqr_{G}}} \bigcup_{k \in
      \mathbb{Z}}T^{k}\widetilde{W}_{c} \subseteq D_{x}. \]

  \begin{claim*}
    We have \( \lambda_{x}([x]_{\eqr_{G}}\cap C_T \cap D_{x}) = 0 \) for each \( x \in X \).
  \end{claim*}
  \begin{cproof}
    Otherwise we can find \( c \in \mathcal{C} \cap [x]_{\eqr_{G}} \) and a wandering set
    \( W \subseteq [x]_{\eqr_{G}}\cap (\mathcal{V}_{\mathcal{C}})_{c} \cap C_{T} \) of positive measure,
    \( \lambda_{x}(W) > 0 \). Construct a sequence of sets \( W_{n} \) by setting \( W_{0} = W \)
    and
    \begin{multline*}
      W_{n} = \bigl\{T^{k_{n}(y)}y : y \in W_{0} \textrm{ and } k_{n}(y) \textrm{ is minimal such
        that } \\
      \pi_{\mathcal{C}}(T^{k_{n}(y)}) = \pi_{\mathcal{C}}(y) \textrm{ and } T^{k_{n}(y)}y
      \not \in \bigcup_{k < n} W_{k}\bigr\},
    \end{multline*}
    where the value of \( k_{n}(y) \) is well-defined for each \( y \in W_{0} \) and
    \( n \in \mathbb{N} \), since all points in \( C_{T} \) return to their Voronoi domain infinitely
    often. Define a transformation \( S_{n} : W_{0} \to W_{n} \) as \( S_{n}(y) = T^{k_{n}(y)}y \),
    and note that for all \( n \in \mathbb{N} \) one has
    \( \rho_{S_{n}}(y) \in \rho((\mathcal{V}_{\mathcal{C}})_{c}, (\mathcal{V}_{\mathcal{C}})_{c})
    \), where, as earlier, \(\rho\) and \(\rho_{S_{n}}\) denote the cocycle maps.  The region
    \( \rho((\mathcal{V}_{\mathcal{C}})_{c}, (\mathcal{V}_{\mathcal{C}})_{c}) \) is precompact,
    since \( \mathcal{C} \) is cocompact, and therefore using continuity of the Haar modulus
    \( \Delta : G \to \mathbb{R}^{>0} \) one can pick \( \epsilon > 0 \) such that
    \( \Delta(\rho_{S_{n}}(y)) > \epsilon \) for all \( y \in W_{0} \) and all
    \( n \in \mathbb{N} \).

    Since \( S_{n} \) is composed of powers of \( T \),
    Proposition~\ref{prop:orbital-transformations} ensures that
    \[ \lambda_{x}(S_{n}W_{0}) = \int_{W_{0}} \Delta(\rho_{S_{n}}(y))\, d\lambda_{x}(y), \]
    whence \( \lambda_{x}(S_{n}W_{0}) \ge \epsilon \lambda_{x}(W_{0}) \) for each
    \( n \in \mathbb{N} \).  We now arrive at a contradiction, as \( W_{n} \),
    \( n \in \mathbb{N} \), form a pairwise disjoint infinite family of subsets of
    \( (\mathcal{V}_{\mathcal{C}})_{c} \) whose measure is uniformly bounded away from zero by
    \( \epsilon \lambda_{x}(W_{0}) \), which is impossible, since
    \( \lambda_{x}((\mathcal{V}_{\mathcal{C}})_{c}) < \infty \) by cocompactness of
    \( \mathcal{C} \). This finishes the proof of the claim.
  \end{cproof}

  We have established by now that \( D_T \cap [x]_{\eqr_{G}} \subseteq D_{x} \) and, up to a null set,
  \( C_T \cap [x]_{\eqr_{G}} \subseteq C_{x} \) by the claim above. Finally,
  \(\mu(X \setminus (D_T \sqcup C_T)) = 0\) implies via Lemma~\ref{lem:zero-set-fubini}
  \( \lambda_{x}((D_T \cap [x]_{\eqr_{G}}) \sqcup (C_T \cap [x]_{\eqr_{G}})) = 0 \) for \( \mu \)-almost
  all \( x \in X \), and therefore
  \[ \lambda_{x}((D_T \cap [x]_{\eqr_{G}}) \triangle D_{x}) = 0 = \lambda_{x}((C_T \cap [x]_{\eqr_{G}})
    \triangle C_{x}) \]
  \( \mu \)-almost surely.  Sets \( D_{T} \) and \( C_{T} \) thus satisfy the conclusion
  of the proposition.

  For the uniqueness part of the proposition, suppose \( D_T, C_T \) and \( D'_T, C'_T \) are two partitions
  of \( X \) such that
  \[ \lambda_{x}(D_T \triangle D_{x}) = 0 = \lambda_{x} (D'_T \triangle D_{x}) \textrm{ and }
    \lambda_{x}(C_T \triangle C_{x}) = 0 = \lambda_{x} (C'_T \triangle C_{x}) \]
  for \( \mu \)-almost all \( x \in X \). One therefore also has
  \( \forall^{\mu}x \in X\ \lambda_{x}(D_T \triangle D'_T) = 0 = \lambda_{x}(C_T \triangle C'_T) \), and
  hence \( \mu(D_T \triangle D'_T) = 0 \) and \( \mu(C_T \triangle C'_T) = 0 \) by Lemma~\ref{lem:zero-set-fubini}.
\end{proof}

We end this section with a natural definition which will be useful for analyzing elements of the full group.

\begin{definition}
  Let \( G \acts X \) be a free measure-preserving action of a locally compact Polish group on a
  standard probability space \( (X, \mu) \), and let \( T\in [G\acts X] \).  Consider the
  \(T\)-invariant partition \(X=D_T\sqcup C_T\) provided by the Hopf decomposition of \(T\) as per
  the previous proposition.  We say that \(T\) is
  \textbf{dissipative}\index{Transformation!dissipative}\index{Dissipative transformation} when \(D_T=X\) and that \(T\) is
  \textbf{conservative}\index{Transformation!conservative}\index{Conservative transformation} when \(C_T=X\).
\end{definition}

When \(G\) is discrete, observe that \(T\) is dissipative if and only if it is aperiodic (all its
orbits are infinite), and that \(T\) is conservative if and only if it is periodic (all its orbits
are finite).

\begin{example}
  Let us give a general example of dissipative elements of the full group.  Let
  \( G \overset{\alpha}{\acts} X \) be a free measure-preserving action of a locally compact Polish
  group on a standard probability space \( (X, \mu) \). If \(g\in G\) generates a discrete infinite
  subgroup, then the element of the full group \(\alpha(g)\) is dissipative.  Indeed, the action of
  \(\alpha(g)\) on each orbit is isomorphic to the \(g\)-action by left translation on \(G\) endowed
  with its right Haar measure, which is dissipative since it admits a Borel fundamental domain and
  has only infinite orbits.  For instance, if \(G = \mathbb{R}\), such a domain is given by the
  interval \([0,g)\) (or \((g, 0]\), if \(g\) is negative).
\end{example}

In Chapter~\ref{chap:example-transf}, we build an interesting example of a conservative element in
the full group of any free measure-preserving flow: its action on each orbit is actually ergodic, and
its cocycle is bounded.

\section{\texorpdfstring{\(\LL^1\)}{L1} full groups and \texorpdfstring{\(\LL^1\)}{L1} orbit
  equivalence}
\label{sec:texorpdfstr-full-gro}

We now restrict ourselves to the setup where the acting group \(G\) is locally compact Polish and
\emph{compactly generated}, endowed with a maximal compatible norm \(\norm\cdot\) (the existence of
such a norm for locally compact Polish group is equivalent to being compactly generated,
see~\cite[Cor.~2.8 and Thm.~2.53]{MR4327092}).  For such a group, as explained in
Section~\ref{sec:l1-full-groups-bound}, it makes sense to talk about the associated \(\LL^1\) full
group by endowing the group with a maximal norm.

The following definition is the natural extension of the notion of \(\LL^1\) orbit equivalence to the locally compact case, stated in terms of full groups.

\begin{definition}\label{def:l1oe}
  Let \(\alpha\) and \(\beta\) be the respective measure-preserving actions of two locally compact
  Polish compactly generated groups \(G\) and \(H\) on a standard probability space \((X,\mu)\).  We
  say that \(\alpha\) and \(\beta\) are \textbf{\(\LL^1\) orbit
    equivalent}\index{L1@\(\LL^{1}\)!orbit equivalence} when there is a measure-preserving
  transformation \(S\in\Aut(X,\mu)\) such that for all \(g\in G\) and all \(h\in H\),
  \[
    S\alpha(g)S\inv\in [H\overset{\beta}{\acts} X]_1\text{ and }
    S\inv\beta(h)S\in[G\overset{\alpha}{\acts} X]_1.
  \]
  In other words, up to conjugating \(\alpha\) by \(S\), we have that the image of \(\alpha\) is
  contained in the \(\LL^1\) full group of \(\beta\), and the image of \(\beta\) is contained in the
  \(\LL^1\) full group of \(\alpha\).
\end{definition}

We now show that \(\LL^1\) full groups do remember actions up to \(\LL^1\) orbit equivalence as
abstract groups. This is done by finding a spatial realization of the isomorphism between the full
groups.  Such techniques originated in the work of H.~Dye~\cite{dyeGroupsMeasurePreserving1959} and
have been greatly generalized by D.~H.~Fremlin~\cite[384D]{MR2459668}.  We recall that a subgroup
\(G\) of \(\Aut(X, \mu)\) is said to have \textbf{many involutions}\index{Transformation group!with
  many involutions} if for any non-trivial measurable \(A \subseteq X\) there exists a non-trivial
involution \(U \in G\) such that \(\supp U \subseteq A\). The group of quasi-measure-preserving
transformations of \((X, \mu)\) is denoted by \(\Aut^{*}(X, \mu)\).

\begin{theorem}[Fremlin]
  \label{thm:fremlin-reconstruction}
  Let \(G, H\) be subgroups of \(\Aut(X,\mu)\) with many involutions. For any isomorphism
  \(\psi : G \to H\) there exists \(S \in \Aut^{*}(X, \mu)\) such that \(\psi(T) = STS^{-1}\) for
  all \(T \in G\).
\end{theorem}

\begin{proposition}
  \label{prop:l1-full-oe-implies-l1-oe}
  If the \(\LL^{1}\) full groups of two ergodic measure-preserving actions of locally compact
  compactly generated Polish groups are isomorphic as abstract groups, then the two actions are
  \(\LL^1 \) orbit equivalent.
\end{proposition}

\begin{proof}
  Denote by \(G\overset{\alpha}{\acts}\) and \(H\overset{\beta}{\acts}\) the two actions on the same
  standard probability space \((X, \mu)\). Since the \(\LL^{1}\) full groups of ergodic actions have
  many involutions (see, for example, Lemma~\ref{lem:many-involutions}), any isomorphism
  \(\psi : \lfgr{G\overset{\alpha}{\acts}X} \to \lfgr{H\overset{\beta}{\acts}X}\) admits a spatial
  realization by some \(S \in \Aut^{*}(X, \mu)\). The Radon--Nikodym derivative of \(S_{*}\mu\) with
  respect to \(\mu\) is easily seen to be preserved by every element of
  \(\lfgr{H\overset{\beta}{\acts}X}\), and hence must be constant by ergodicity. We conclude that
  \(S \in \Aut(X, \mu)\), and therefore by the definition the actions \(\alpha\) and \(\beta\) are
  \(\LL^1\) orbit equivalent.
\end{proof}

\begin{remark}
  Similarly to the finitely generated case~\cite[Sec.~8.1]{MR4398251}, one could define \(\LL^1\)
  full orbit equivalence between actions as equality of \(\LL^1\) full groups up to conjugacy, which
  is a strengthening of \(\LL^1\) orbit equivalence (indeed the latter only requires inclusion of
  each acting group in the \(\LL^1\) full group of the other acting group).  It would be interesting
  to have examples of actions which are \(\LL^1\) orbit equivalent, but not \(\LL^1\) fully orbit
  equivalent.
\end{remark}

We end this section by showing that \(\LL^1\) orbit equivalence is equivalent to a stronger
definition where we ask that, up to conjugating \(\alpha\) by \(S\), we moreover have that, on a
full measure set \(X_0\subseteq X\), the \(\alpha\) and \(\beta\) orbits coincide.  This will be a
direct consequence of the following proposition. The proof of this proposition is the same as that
of~\cite[Prop.~3.8]{MR3464151} which was not stated in the level of generality we need. Since it is
short, we reproduce it here. We emphasize that when the acting groups are not discrete,
the full measure set \(X_0\) may very well fail to be \(\alpha\)-invariant or \(\beta\)-invariant.
In particular, when we say that the orbits coincide on \(X_0\) we simply mean that
they induce the same partition on \(X_0\).

\begin{proposition}
  \label{prop:oeincl}
  Let \(G\) and \(H\) be two locally compact Polish groups acting in a Borel measure-preserving
  manner on a standard probability space \((X,\mu)\), denote by \(\alpha\) the \(G\)-action and
  suppose that \(\alpha(G)\leq [H \acts X]\). Then there is a full measure Borel subset
  \(X_0\subseteq X\) such that
  \[\mathcal R_G\cap \left(X_0\times X_0\right)\subseteq \mathcal R_H.\]
\end{proposition}
\begin{proof}
  Let \(\lambda\) be the Haar measure on \(G\). Since \(\alpha(G)\leq[H \acts X]\), for all
  \(g\in G\) and almost all \(x\in X\), we have \(g x\in Hx\). By Fubini's theorem, this implies
  that the Borel set
  \[X_0 =\{x\in X: \text{ for }\lambda\text{-almost all }g\in G,\text{ we have } gx\in Hx\}\]
  has full measure. Now let \(x\in X_0\), and let \(g_1\in G\) be such that \(g_1x\in X_0\). We want
  to show that \(g_1x\in Hx\).

  Since \(x\) and \(g_1x\) are in \(X_0\), the sets
  \[A =\{g\in G:gx\in Hx\} \quad \textrm{and} \quad B =\{g\in G:gx\in Hg_1x\}\]
  have full measure and so \(A \cap B\) has full measure. Take \(g\in A\cap B\), and note that
  \(gx\in Hx\cap Hg_1x\), so the two orbits \(Hx\) and \(Hg_1x\) intersect, hence \(g_1x\in Hx\).
\end{proof}

\begin{corollary}\label{cor:def-l1-oe}
  Let \(G\) and \(H\) be compactly generated locally compact Polish groups, and let
  \(\norm{\cdot}_{G}\) and \(\norm{\cdot}_{H}\) be maximal norms on \(G\) and \(H\), respectively.
  Two measure-preserving actions of \(G\) and \(H\) on a standard probability space \((X,\mu)\) are
  \(\LL^1\) orbit equivalent if and only if they can be conjugated so as to share the same orbits on
  a full measure Borel subset \(X_0\subseteq X\), i.e.,
  \(\eqr_{G} \cap (X_{0} \times X_{0}) = \eqr_{H} \cap (X_{0} \times X_{0})\), and there exist Borel
  maps \(\gamma_{G} : G \times X_{0}\to H \) and \(\gamma_H : H \times X_{0}\to G \) such that:
  \begin{enumerate}
  \item\label{item:orbit-map} for all \(x\in X_0\), \( g\cdot x = \gamma_{G}(g,x)\cdot x\) and
    \(h \cdot x = \gamma_{H}(h,x)\cdot x\) whenever \( gx \in X_{0}\) and \(hx \in X_{0} \);
  \item\label{item:integrability} \( \int_{X_{0}} \norm{\gamma_{G}(g,x)}_{H}\, d\mu(x) < +\infty \)
    and \( \int_{X_{0}}\norm{\gamma_{H}(h,x)}_{G}\, d\mu(x)<+\infty\) for all \(g \in G\) and all
    \(h \in H\).
  \end{enumerate}
\end{corollary}

\begin{proof}
  We may assume that the actions \(G \acts X\) and \(H \acts X\) are Borel.  By the definition of
  \(\LL^1\) full groups, the conditions stated in the corollary are sufficient to establish
  \(\LL^1\) orbit equivalence.  Conjugating the two actions, we may also assume that they share the
  same full group. Since the \(\LL^1\) full groups contain the acting groups, we can apply
  Proposition~\ref{prop:oeincl} twice to obtain a full measure Borel subset \(X_0\subseteq X\) on
  which the orbits of the two actions coincide.

  Let \(\mathcal{F}(G)\) and \(\mathcal{F}(H)\) denote the Effros Borel spaces associated with \(G\)
  and \(H\), respectively.  The orbit equivalence relations \(\eqr_{G}\) and \(\eqr_{H}\) are Borel,
  and consequently, the maps
  \begin{displaymath}
    \begin{aligned}
      G \times X_{0} \ni (g, x) &\mapsto \{h \in H: gx = hx\} \in \mathcal{F}(H), \\
      H \times X_{0} \ni (h, x) &\mapsto \{g \in G: hx = gx\} \in \mathcal{F}(G)
    \end{aligned}
  \end{displaymath}
  are also Borel~\cite[Thm.~7.1.2]{MR1425877}.  Note that \(\{h \in H: gx = hx\} \ne \varnothing\)
  whenever \(gx \in X_{0}\), and similarly, \(\{g \in G: hx = gx\} \ne \varnothing\) provided that
  \(hx \in X_{0}\).  By applying the Kuratowski--Ryll-Nardzewski
  selectors~\cite[12.13]{kechris_classical_1995}, we can find Borel maps \(\gamma_{G}\) and
  \(\gamma_{H}\) that satisfy item~\eqref{item:orbit-map} and the
  inequalities
  \[\snorm{\gamma_{G}(g,x)}_{H} < D_{H}(x,gx) + 1, \quad
    \snorm{\gamma_{H}(h,x)}_{G} < D_{G}(x,hx) + 1, \]
  where \(D_{H}\) and \(D_{G}\) are the metrics induced on the orbits by the respective actions.
  The integrability condition~\eqref{item:integrability} now follows from the assumption that the
  actions have been conjugated to satisfy \(G \le \lfgr{H \acts X}\) and \(H \le \lfgr{G \acts X}\).
\end{proof}

We will demonstrate in the final chapter that there exist free ergodic \(\R\)-flows that are not
\(\LL^1\) orbit equivalent. This result will be established by connecting \(\LL^1\) orbit
equivalence to flip Kakutani equivalence. In the discrete amenable setting, a key result due to
T.~Austin shows that entropy is preserved under \(\LL^1\) orbit
equivalence~\cite{austinBehaviourEntropyBounded2016}.

\begin{question}
  Let \(G\) be an amenable non-discrete non-compact compactly generated locally compact Polish
  group. Are there free measure-preserving ergodic actions of \(G\) which are not \(\LL^1\) orbit
  equivalent?
\end{question}



\chapter[Derived \texorpdfstring{\( \LL^{1} \)}{L1} full groups]{Derived
  \texorpdfstring{\( \LL^{1} \)}{L1} full groups for locally compact amenable groups}
\label{chap:derived-l1-full-group}

Given a measure-preserving action of a Polish normed group \((G,\norm\cdot)\) on \( (X,\mu) \), the
\textbf{derived \( \LL^1 \) full group}\index{L1@\(\LL^{1}\)!full group!derived}\index{Derived L1@Derived \(\LL^{1}\) full group}
\( \derived(\lfgr{G\acts X}) \) of the action is, by definition, the (topological) derived subgroup
of the \(\LL^1\) full group \(\lfgr{G\acts X}\).  Recall that this means that
\( \derived(\lfgr{G\acts X}) \) is the closure in \( \lfgr{G\acts X} \) of the subgroup generated by
commutators, i.e.,~elements of the form \( TUT\inv U\inv \), where \( T,U\in \lfgr{G\acts X} \).
Provided the \(G\)-action is aperiodic, the latter can be described as a subgroup of
\(\lfgr{G\acts X}\) in three different ways, using the fact that \( \lfgr{G\acts X} \) is
\emph{induction friendly}, as explained in Section~\ref{sec:derived-equals-symmetric} (see
Corollary~\ref{cor:all-subgroups-are-equal}):
\begin{itemize}
\item \( \derived(\lfgr{G\acts X}) \) is the closure of the group generated by involutions;
\item \( \derived(\lfgr{G\acts X}) \) is the closure of the group generated by \(3\)-cycles;
\item \( \derived(\lfgr{G\acts X}) \) is the closure of the group generated by periodic elements.
\end{itemize}
In particular, all periodic elements of \( \lfgr{G\acts X} \) actually belong to
\( \derived(\lfgr{G\acts X}) \) (see Lemma~\ref{lem:periodic-elements-are-in-S(G)} for the proof of
this specific statement).

Compared to the previous chapter, we impose one further restriction on the acting group, and
consider actions of a locally compact \emph{amenable} Polish normed group
\( (G, \snorm{ \cdot } ) \).
Appendix~G of~\cite{bekkaKazhdanProperty2008}\index{Polish group!amenable} contains an excellent
review of amenability for both general topological groups and locally compact groups.
As before, we fix a measure-preserving
action \( G \acts X \) on a standard probability space \( (X,\mu) \), and let
\( D : \eqr_{G} \to \mathbb{R}^{\ge 0} \) denote the family of metrics induced onto the orbits by
the norm. To ensure our results encompass both the non-compactly generated case and the situation in
which the \(\LL^1\) full group coincides with the entire full group of the action (as
in~\cite{MR3748570}), we do not impose the condition that the norm be either proper or maximal. In
particular, the norm may be bounded, which in turn implies that the metric \(D\) on the orbits is
also bounded.

In Section~\ref{sec:dense-chain-subgr}, we construct a dense increasing chain of subgroups in
\(\derived(\lfgr{G\acts X})\). This dense chain is utilized in the subsequent sections. In
Section~\ref{sec:whirly-amenability}, we show that the amenability of the group \(G\) is reflected
in the \emph{whirly amenability} of \(\derived(\lfgr{G\acts X})\). Meanwhile, in
Section~\ref{sec:topol-gener}, we prove, by a Baire category argument, that
\(\derived(\lfgr{G\acts X})\) contains a dense \(2\)-generated subgroup.

\section{Dense chain of subgroups}
\label{sec:dense-chain-subgr}

An equivalence relation \( \eqr \subseteq \eqr_{G} \) is said to be \textbf{uniformly bounded} if
there is \( M > 0 \) and \( X' \subseteq X \) such that \( \mu(X \setminus X') = 0 \) and
\( \sup_{(x_{1}, x_{2}) \in \mathcal{R}'} D(x_{1}, x_{2}) \le M \), where
\( \eqr' = \eqr \cap X' \times X' \).

\begin{lemma}
  \label{lem:exhaustive-bounded-smooth-atomless}
  Let \( (G, \snorm \cdot) \) be a locally compact amenable Polish normed group acting on a standard
  probability space \( (X, \mu) \). There exists a sequence of cross-sections \( \mathcal{C}_{n} \),
  \(n \in \mathbb{N}\), and tessellations \( \mathcal{W}_{n} \) over \( \mathcal{C}_{n} \) such that
  for all \( n \in \mathbb{N} \)
  \begin{enumerate}
  \item \( \mathcal{R}_{\mathcal{W}_{n}} \subseteq \mathcal{R}_{\mathcal{W}_{n+1}} \) and
    \( \bigcup_{k \in \mathbb{N}} \mathcal{R}_{\mathcal{W}_{k}} = \mathcal{R}_{G} \) (up to a null
    set);
  \item \( \mathcal{R}_{\mathcal{W}_{n}} \) is uniformly bounded.
  \end{enumerate}
\end{lemma}

\begin{proof}
  Let \( \mathcal{C} \) be a cocompact cross-section, \( \mathcal{V}_{\mathcal{C}} \) be the Voronoi
  tessellation over~\( \mathcal{C} \), \( \pi_{\mathcal{V}_{\mathcal{C}}} : X \to \mathcal{C} \) be
  the associated reduction, and \( \nu = (\pi_{\mathcal{V}_{\mathcal{C}}})_{*}\mu \) be the
  push-forward measure on \( \mathcal{C} \). Recall that
  \( \mathcal{R}_{\mathcal{V}_{\mathcal{C}}} \) is uniformly bounded since \( \mathcal{C} \) is
  cocompact. Let \( E \) be the equivalence relation obtained by restricting \( \mathcal{R}_{G} \)
  onto \( \mathcal{C} \). By a theorem of A.~Connes, J.~Feldman, and B.~Weiss~\cite{MR662736},
  \( E \) is hyperfinite on an invariant set of \( \nu \)-full measure. Throwing away a
  \( G \)-invariant null set, we may write \( E = \bigcup_{n} E_{n} \), where
  \( (E_{n})_{n \in \mathbb{N}} \) is an increasing sequence of Borel equivalence relations with
  finite classes. For \( m, n \in \mathbb{N} \), define \( A_{n, m} \) to be the set of points in
  the cross-section whose \( E_{n} \)-class is bounded in diameter by \( m \):
  \[ A_{n, m} = \bigl\{ c \in \mathcal{C} : D(c_{1},c_{2}) \le m \textrm{ for all } c_{1}, c_{2} \in
    \mathcal{C} \textrm{ such that } c_{1} E_{n} c \textrm{ and } c_{2} E_{n} c \bigr\}. \]
  Note that the sets \( A_{n,m} \) are \( E_{n} \)-invariant, nested, and
  \( \bigcup_{m} A_{n,m} = \mathcal{C} \) for every \( n \in \mathbb{N} \). Pick \( m_{n} \) so
  large as to ensure \( \nu(\mathcal{C} \setminus A_{n,m_{n}}) < 2^{-n} \) and let
  \( B_{n} = \bigcap_{k \ge n} A_{k,m_{k}} \). The sets \( B_{n} \) are \( E_{n} \)-invariant,
  increasing, and \( \lim_{n}\nu(B_{n}) = \nu(\mathcal{C}) \). Define equivalence relations
  \( F_{n} \) on \( \mathcal{C} \) by setting \( c_{1} F_{n} c_{2} \) whenever \( c_{1} = c_{2} \)
  or \( c_{1}, c_{2} \in B_{n} \) and \( c_{1} E_{n} c_{2} \). Note that
  \( D(c_{1}, c_{2}) \le m_{n} \) whenever \( c_{1} F_{n} c_{2} \). Let
  \( \mathcal{C}_{n} \subseteq \mathcal{C} \) be a Borel transversal for \( F_{n} \) and define
  \( \mathcal{W}_{n} = \{ (c, x) \in \mathcal{C}_{n} \times X: c F_{n}
  \pi_{\mathcal{V}_{\mathcal{C}}}(x) \} \).  It is straightforward to check that each
  \( \mathcal{W}_{n} \) is a tessellation over \( \mathcal{C}_{n} \), and the equivalence relations
  \( \mathcal{R}_{\mathcal{W}_{n}} \) satisfy the conclusions of the lemma.
\end{proof}

The equivalence relations \( \mathcal{R}_{\mathcal{W}_{n}} \) produced by
Lemma~\ref{lem:exhaustive-bounded-smooth-atomless} give rise to a nested chain of groups
\( \fgr{\mathcal{R}_{\mathcal{W}_{0}}} \leq \fgr{\mathcal{R}_{\mathcal{W}_{1}}} \leq \cdots \). Some
basic facts about such groups can be found in Appendix~\ref{sec:tessellations}. The following lemma
establishes that such a chain is dense in the \emph{derived} \( \LL^{1} \) full group.

\begin{lemma}
  \label{lem:exhaustive-chain-subgroups}
  Let \( (G, \snorm \cdot) \) be a locally compact amenable Polish normed group acting on a standard
  probability space \( (X, \mu) \) and let \( (\mathcal{R}_{n})_{n \in \mathbb{N}} \) be a sequence
  of equivalence relations as in Lemma~\ref{lem:exhaustive-bounded-smooth-atomless}. If the action
  is aperiodic, then the union \( \bigcup_{n} [\mathcal{R}_{n}] \) is contained in the derived
  \( \LL^{1} \) full group \( \derived(\lfgr{G \acts X}) \) and is dense therein.
\end{lemma}

\begin{proof}
  By definition, \( \fgr{\mathcal{R}_{n}} \) is a subgroup of \( \fgr{\mathcal{R}_{G}} \). Since
  equivalence relations \( \mathcal{R}_{n} \) are uniformly bounded, we actually have
  \( \fgr{\mathcal{R}_n} \leq \lfgr{ G \acts X} \), and the topology induced by the \( \LL^{1} \)
  metric on \( \fgr{\mathcal{R}_{n}} \) coincides with the topology induced from
  \( \fgr{\mathcal{R}_{G}} \). Moreover, in view of
  Proposition~\ref{prop:full-group-smooth-generated-periodic}, \( \fgr{\mathcal{R}_{n}} \) is
  topologically generated by periodic transformations, so we actually have
  \( \fgr{\mathcal{R}_{n}} \leq \derived(\lfgr{G \acts X}) \) as a consequence of
  Lemma~\ref{lem:periodic-elements-are-in-S(G)} and Corollary~\ref{cor:all-subgroups-are-equal}.

  It remains to verify that the union \( \bigcup_{n} \fgr{\mathcal{R}_{n}} \) is dense in
  \( \derived(\lfgr{G \acts X}) \). To this end, recall that by
  Corollary~\ref{cor:all-subgroups-are-equal}, the derived \( \LL^{1} \) full group
  \( \derived(\lfgr{G \acts X}) \) is topologically generated by involutions. So let
  \( U \in \derived(\lfgr{G \acts X}) \) be an involution and set
  \( X_{n} = \{x \in X : (x,U(x))\in \mathcal{R}_{n} \} \), \( n \in \mathbb{N} \). Note that
  \( X_{n} \) is \( U \)-invariant since \( U \) is an involution. Moreover, \( \mu(X_{n}) \to 1 \)
  as \( \bigcup_{n} \mathcal{R}_{n} = \mathcal{R}_{G} \), and thus the induced transformations
  \( U_{X_{n}} \in \fgr{\eqr_{n}} \) converge to \( U \) in the topology of \( \lfgr{G \acts X} \).
  We conclude that \( \bigcup_{n}\fgr{\mathcal{R}_{n}} \) is dense in the derived \( \LL^{1} \) full
  group.
\end{proof}

\begin{corollary}
  \label{cor:dense-chain-subgroups}
  Let \( (G, \snorm \cdot) \) be a locally compact amenable Polish normed group acting on a standard
  probability space \( (X, \mu) \). Suppose that almost every orbit of the action is uncountable.
  There exists a chain \( H_{0} \le H_{1} \le \cdots \le \derived(\lfgr{G \acts X}) \) of closed
  subgroups such that \( \bigcup_{n} H_{n} \) is dense in \( \derived(\lfgr{G \acts X}) \), and each
  \( H_{n} \) is isomorphic to \( \LL^{0}(Y_{n}, \nu_{n}, \Aut([0,1], \lambda)) \) for some standard
  Lebesgue space \( (Y_{n}, \nu_{n}) \). If, moreover, each orbit of the action has measure zero,
  then one can assume that all \( (Y_{n}, \nu_{n}) \) are atomless and each \( H_{n} \) is
  isomorphic to \( \LL^{0}([0,1], \lambda, \Aut([0,1], \lambda)) \).
\end{corollary}

\begin{proof}
  Apply Lemmas~\ref{lem:exhaustive-bounded-smooth-atomless} and~\ref{lem:exhaustive-chain-subgroups}
  to get a dense chain of subgroups
  \( \fgr{\mathcal{R}_{0}} \le \fgr{\mathcal{R}_{1}} \le \cdots \le \derived(\lfgr{G \acts X}) \)
  and use Corollary~\ref{cor:maharam-identification-with-l0} to deduce that each
  \( \fgr{\mathcal{R}_{n}} \) has the desired form.
\end{proof}

\begin{corollary}\label{cor:amenable-periodic-dense}
  Let \( (G, \snorm \cdot) \) be a locally compact amenable Polish normed group acting on a standard
  probability space \( (X, \mu) \).  If the action is aperiodic, then the set of periodic elements
  is dense in the derived \( \LL^1 \) full group \( \derived(\lfgr{G\acts X}) \).
\end{corollary}
\begin{proof}
  Consider a chain of subgroups \(\fgr{\eqr_{n}}\) given by
  Lemma~\ref{lem:exhaustive-chain-subgroups}. Periodic elements are dense in these groups for their
  natural topology (see Proposition~\ref{prop:full-group-smooth-generated-periodic} and the
  discussion preceding it).  These topologies are compatible with the standard Borel structure of
  \(\Aut(X,\mu)\) induced by the weak topology and therefore must refine the \(\LL^1\) topology by
  the standard automatic continuity arguments~\cite[Sec.~1.6]{MR1425877}.  Hence, periodic elements
  are dense in all of \( \derived(\lfgr{G\acts X}) \), as claimed.
\end{proof}

Corollary~\ref{cor:amenable-periodic-dense}, together with
Proposition~\ref{prop:induction-friendly-maximal-norm}, shows that the \( \LL^1 \) norm is maximal
on derived \( \LL^1 \) full groups of aperiodic measure-preserving actions of locally compact
amenable Polish normed groups (see Section~\ref{sec:l1-full-groups-bound} for a brief reminder on
the maximality of norms).  In particular, such groups are boundedly generated
by~\cite[Thm.~2.53]{MR4327092}.

\begin{theorem}\label{thm:lc-amenable-derived-maximal}
  Let \( (G, \snorm \cdot) \) be a locally compact amenable Polish normed group acting on a standard
  probability space \( (X, \mu) \). If the action is aperiodic, then the \(\LL^1\) norm is maximal
  on the derived \(\LL^1\) full group \( \derived(\lfgr{G\acts X}) \).
\end{theorem}

We do not know if the amenability hypothesis can be removed, even when \( G \) is discrete and the
action is free.

\section{Whirly amenability}
\label{sec:whirly-amenability}

Lemma~\ref{lem:exhaustive-chain-subgroups} is a powerful tool to deduce various dynamical properties
of derived \( \LL^{1} \) full groups. Recall that a Polish group \( G \) is said to be
\textbf{whirly amenable}\index{Polish group!whirly amenable} if it is amenable and, for any
continuous action of \( G \) on a compact space, any invariant measure is supported on the set of
fixed points of the action. In particular, each such action has to have some fixed points, so whirly
amenable groups are extremely amenable\index{Polish group!extremely amenable}, meaning
that all their continuous actions on compact spaces have fixed points.

\begin{proposition}
  \label{prop:whirly-amenability-full-groups-smooth-relation}
  Let \( \mathcal{R} \) be a smooth measurable equivalence relation on a standard Lebesgue space
  \( (X, \mu) \). If \( \mu \) is atomless, then the full group \( \fgr{\mathcal{R}} \) is whirly
  amenable.
\end{proposition}

\begin{proof}
  In view of Proposition~\ref{prop:full-group-smooth-representation}, the full group
  \( \fgr{\mathcal{R}} \) is isomorphic to
  \[ \LL^{0}([0,1], \lambda, \Aut([0,1], \lambda))^{\epsilon_{0}} \times \Aut([0,1],
    \lambda)^{\kappa_{0}} \times \prod_{n \ge 1} \LL^{0}([0,1], \lambda,
    \mathfrak{S}_{n})^{\epsilon_{0}}, \]
  where \( \mathfrak{S}_{n} \) is the group of permutations of an \( n \)-element set, and
  \( \epsilon_{n} \in \{0,1\} \), \( \kappa_{0} \le \aleph_{0} \). Since a product of whirly
  amenable groups is whirly amenable, it suffices to show that the groups appearing in the
  decomposition above, namely \( \LL^{0}([0,1], \lambda, \Aut([0,1], \lambda)) \),
  \( \Aut([0,1], \lambda) \), and \( \LL^{0}([0,1], \lambda, \mathfrak{S}_{n}) \), \( n \ge 1 \),
  are whirly amenable.

  The group \( \Aut([0,1], \lambda) \) is whirly amenable by~\cite{MR1891002} (it is, in fact, a
  so-called Levy group). Finally, we apply a theorem of V.~Pestov and
  F.~M.~Schneider~\cite{MR3711882}, which asserts that a group \( \LL^{0}([0,1], \lambda, G) \) is
  whirly amenable if and only if \( G \) is amenable. This readily implies the whirly amenability of
  \( \LL^{0}([0, 1], \lambda, \mathfrak{S}_{n}) \) and
  \( \LL^{0}([0,1], \lambda, \Aut([0,1], \lambda)) \).
\end{proof}

\begin{remark}
  The assumption of \( \mu \) being atomless cannot be omitted in the proposition above. Indeed,
  \( \fgr{\mathcal{R}} \) will factor onto \( \mathfrak{S}_{n} \) for some \( n \ge 2 \), as long as
  an \( \mathcal{R} \)-class contains at least \( 2 \) atoms of \( \mu \) of the same measure.
  However, if all \( \mu \)-atoms within each \( \mathcal{R} \)-class have distinct measures, then
  the restriction of \( \fgr{\eqr} \) onto the atomic part of \( X \) is trivial, which suffices to
  conclude the whirly amenability of the group \( \fgr{\mathcal{R}} \).
\end{remark}

\begin{theorem}
  \label{thm:derived-whirly-amenable}
  Let \( G \acts X \) be a measure-preserving action of an amenable locally compact Polish normed
  group on a standard probability space \( (X, \mu) \). If the action is aperiodic, then the derived
  \( \LL^{1} \) full group \( \derived(\lfgr{G \acts X}) \) is whirly amenable. In particular,
  \( \lfgr{G\acts X} \) is amenable.
\end{theorem}

\begin{proof}
  Lemma~\ref{lem:exhaustive-chain-subgroups} shows that \( \derived(\lfgr{G \acts X}) \) has an
  increasing dense chain of subgroups \( H_{n} \) of the form \( \fgr{\mathcal{R}_{n}} \), where
  \( \mathcal{R}_{n} \) are smooth measurable equivalence relations on \( X \).
  Proposition~\ref{prop:whirly-amenability-full-groups-smooth-relation} applies and shows that the
  groups \( H_{n} \) are whirly amenable. The latter is sufficient to conclude the whirly
  amenability of \( \derived(\lfgr{G \acts X}) \), as any invariant measure for the action of the
  derived group is also invariant for the induced \( H_{n} \)-actions.  Hence, it has to be
  supported on the intersection of the fixed points of all \( H_{n} \), which coincides with the set
  of fixed points for the action of \( \derived(\lfgr{G \acts X}) \).

  The fact that \( \lfgr{G\acts X} \) is amenable now follows from the fact that every abelian group
  is amenable and that every amenable extension of an amenable group must itself be amenable (for
  instance, see~\cite[Prop.~G.2.2]{bekkaKazhdanProperty2008}).
\end{proof}

\begin{remark}
  Note that, in general, \( \lfgr{G\acts X} \) is not extremely amenable. For flows, it factors onto
  \( \R \) via the index map (see Chapter~\ref{chap:index-map}).  Since \( \R \) admits continuous
  actions on compact spaces without fixed points, \( \lfgr{\mathbb{R} \acts X} \) is not extremely
  amenable (and in particular, it is not whirly amenable) for any free measure-preserving flow.
\end{remark}

\begin{corollary}\label{cor:amenability-equivalences}
  Let \( G \acts X \) be a free measure-preserving action of a unimodular locally compact Polish
  group on a standard probability space \( (X,\mu) \). The following are equivalent:
  \begin{enumerate}
  \item\label{item:G-amenable} \( G \) is amenable.
  \item\label{item:l1-G-amenable} \( \lfgr{ G \acts X} \) is amenable.
  \item\label{item:d-l1-amenable} The derived \( \LL^1 \) full group
    \( \derived(\lfgr{G \acts X}) \) is amenable.
  \item\label{item:d-l1-extremely-amenable} The derived \( \LL^1 \) full group
    \( \derived(\lfgr{G \acts X}) \) is extremely amenable.
  \item\label{item:d-l1-whirly-amenable} The derived \( \LL^1 \) full group
    \( \derived(\lfgr{G \acts X}) \) is whirly amenable.
  \end{enumerate}
\end{corollary}
\begin{proof}
  We established the
  implication~\eqref{item:G-amenable}\( \implies \)\eqref{item:d-l1-whirly-amenable} in
  Theorem~\ref{thm:derived-whirly-amenable}. The chain of
  implications~\eqref{item:d-l1-whirly-amenable}\( \implies
  \)\eqref{item:d-l1-extremely-amenable}\( \implies \)\eqref{item:d-l1-amenable} is straightforward,
  and~\eqref{item:d-l1-amenable}\( \implies \)\eqref{item:l1-G-amenable} follows from the stability
  of amenability under group extensions, which was already discussed in
  Theorem~\ref{thm:derived-whirly-amenable}.

  For the last implication~\eqref{item:l1-G-amenable}\( \implies \)~\eqref{item:G-amenable}, first
  recall that the orbit full group of the action is generated by involutions.  It follows that the
  orbit full group is topologically generated by involutions whose cocycles are integrable
  (actually, one can even ask that the cocycles are bounded). In particular, the \( \LL^1 \) full
  group \( \lfgr{G \acts X} \) is dense in the orbit full group, and so,
  assuming~\eqref{item:l1-G-amenable}, we conclude that the orbit full group \( \fgr{ G \acts X } \)
  is amenable. The amenability of \( G \) then follows from~\cite[Thm.~5.1]{MR3748570}.
\end{proof}

\begin{remark}
  We have to require unimodularity to apply~\cite[Thm.~5.1]{MR3748570}. It seems likely that the
  unimodularity hypothesis can be dropped in this result, but we do not pursue this question
  further.
\end{remark}

\section{Topological generators}
\label{sec:topol-gener}

We now concern ourselves with the question of determining the topological rank of derived
\(\LL^{1}\) full groups. Our approach will be based on the dense chain of subgroups established in
Corollary~\ref{cor:dense-chain-subgroups}, and the first step is to study the topological rank of
the group \( \LL^{0}([0,1], \Aut([0,1])) \).

Let \( (Y, \nu) \) and \( (Z, \lambda) \) be standard Lebesgue spaces. Consider the product space
\( Y \times Z \) equipped with the product measure \( \nu \times \lambda \) and let
\( \mathcal{R} \) be the product of the discrete equivalence relation on \( Y \) and the
anti-discrete on \( Z \); in other words, \( (y_{1}, z_{1}) \mathcal{R} (y_{2}, z_{2}) \) if and
only if \( y_{1} = y_{2} \). As discussed in Appendix~\ref{sec:disintegration-measure}, the
following two groups are one and the same:
\begin{enumerate}
  \item the full group \( \fgr{\mathcal{R}} \);
  \item the topological group \( \LL^{0}(Y, \nu, \Aut(Z, \lambda)) \).
\end{enumerate}
In particular, we may and do endow \(\fgr{\mathcal{R}}\) with the Polish group topology
induced by its natural identification with \( \LL^{0}(Y, \nu, \Aut(Z, \lambda)) \).

Suppose additionally that \( (Z, \lambda) \) is atomless. Pick a hyperfinite ergodic
measure-preserving equivalence relation \( E \) on \( (Z,\lambda) \).  We claim that
\( \mathrm{APER}(Z) \cap \fgr{E} \) is dense in \( \Aut(Z, \lambda) \), where \(\mathrm{APER}(Z)\)
stands for the collection of aperiodic automorphisms of \(Z\).  Indeed, first note that
by~\cite[Prop.~3.1]{MR2583950}, the full group \(\fgr{E}\) is weakly dense in \(\Aut(Z,\lambda)\).
Let us then pick any aperiodic \(T\in \fgr{E}\).  It follows from~\cite[Thm.~2.4]{MR2583950} that
the \(\Aut(Z,\lambda)\)-conjugacy class of \( T \) is weakly dense in \(\Aut(X,\lambda)\).  By the
continuity of the conjugacy action and weak density, the \(\fgr{E}\)-conjugacy class of \( T \) is
weakly dense as well, which proves our claim since this conjugacy class is clearly contained in
\( \mathrm{APER}(Z) \cap \fgr{E} \).

Now set \( \mathcal{R}_{0} = \mathrm{id}_{Y} \times E \) to be the equivalence relation on
\( Y \times Z \) given by the condition \( (y_{1}, z_{1}) \mathcal{R}_{0} (y_{2}, z_{2}) \) whenever
\( y_{1} = y_{2} \) and \( z_{1} E z_{2} \). A standard application of the Jankov--von Neumann
uniformization theorem yields the following lemma.

\begin{lemma}
  \label{lem:aperiodic-R0-dense-R}
  \( \mathrm{APER}(Y \times Z) \cap \fgr{\mathcal{R}_{0}} \) is dense in
  \( \fgr{\mathcal{R}} \ismph \LL^{0}(Y, \nu, \Aut([0,1], \lambda)) \).
\end{lemma}

Our first goal is to establish that the topological rank of \( \fgr{\mathcal{R}} \) is \( 2 \). We
do so by first verifying this under the assumption that \( (Y, \nu) \) is atomless and then deducing
the general case.

We say that a Polish group \( G \) is
\textbf{generically \( k \)-generated}\index{Polish group!generically \(k\)-generated}, where
\( k \in \mathbb{N} \), if the set of \( k \)-tuples \( (g_{1}, \ldots, g_{k}) \in G^{k} \) that
generate a dense subgroup of \( G \) is dense in \( G^{k} \). Note that the set of such tuples is
always a \( G_{\delta} \) set, so if \( G \) is generically \( k \)-generated, then a comeager set
of \( k \)-tuples generates a dense subgroup of \( G \).

\begin{proposition}\label{prop:generic-2-generated-R}
  Suppose that \( (Y, \nu) \) is atomless.  The Polish group \( \fgr{\mathcal{R}} \) is generically
  \( 2 \)-generated.
\end{proposition}
\begin{proof}
  By~\cite[Thm~5.1]{MR3568978}, the set of pairs
  \[ (S,T) \in (\mathrm{APER}(Y \times Z) \cap \fgr{\mathcal{R}_{0}}) \times
    \fgr{\mathcal{R}_{0}} \]
  such that \( \overline{\langle S,T \rangle} = \fgr{\mathcal{R}_{0}} \) is dense \( G_{\delta} \)
  for the uniform topology. In view of Lemma~\ref{lem:aperiodic-R0-dense-R}, this implies that
  \( \fgr{\mathcal{R}} \) is generically \( 2 \)-generated.
\end{proof}

\begin{lemma}
  \label{lem:rank-product-ge-rank-factors}
  For all Polish groups \( G \) and \( H \), one has
  \[ \tprank(G \times H) \ge \max\{\tprank(G), \tprank(H)\}. \]
  If \( G \times H \) is generically \( k \)-generated, then so are \( G \) and \( H \) as well.
\end{lemma}
\begin{proof}
  The inequality on ranks is immediate from the trivial observation that if
  \( \langle (g_{1}, h_{1}), \ldots, (g_{k}, h_{k}) \rangle \) is dense in \( G \times H \), then
  \( \langle g_{1}, \ldots, g_{k} \rangle \) is dense in \( G \) and
  \( \langle h_{1}, \ldots, h_{k} \rangle \) is dense in \( H \).

  Suppose \( G \times H \) is generically \( k \)-generated.  Pick an open set
  \( U \subseteq G^{k} \) and note that \( U \times H^{k} \) corresponds to an open subset of
  \( (G \times H)^{k} \) via the isomorphism \( (G \times H)^{k} \ismph G^{k} \times H^{k} \).
  Since \( G \times H \) is generically \( k \)-generated, there is a tuple
  \( (g_{i}, h_{i})_{i=1}^{k} \in (G \times H)^{k} \) that generates a dense subgroup and
  \( (g_{i}, h_{i})_{i=1}^{k} \in U \times H^{k} \). We conclude that \( (g_{i})_{i=1}^{k} \in U \)
  generates a dense subgroup of \( G \), and the lemma follows.
\end{proof}

\begin{lemma}
  \label{lem:L0-times-Gn-has-same-rank}
  For any Polish group \( G \)
  \[ \tprank(\LL^{0}([0,1], \lambda, G)) = \tprank\bigl(\LL^{0}([0,1], \lambda, G) \times
    G^{\mathbb{N}}\bigr). \]
  If \( \LL^{0}([0,1], \lambda, G) \) is generically \( k \)-generated for some
  \( k \in \mathbb{N} \), then so is the group
  \( \LL^{0}([0,1], \lambda, G) \times G^{\mathbb{N}} \).
\end{lemma}

\begin{proof}
  In view of Lemma~\ref{lem:rank-product-ge-rank-factors},
  \( \tprank(\LL^{0}([0,1], \lambda, G)) \le \tprank\bigl(\LL^{0}([0,1], \lambda, G) \times
  G^{\mathbb{N}}\bigr) \), and since the group \( G \) is separable, we only need to consider the
  case when the rank \( \tprank(\LL^{0}([0,1], \lambda, G)) \) is finite.

  It is notationally convenient to shrink the interval and work with the group
  \[ \LL^{0}([0,1/2], \lambda, G) \times G^{\mathbb{N}} \]
  instead, as it can naturally be viewed as a closed subgroup of \( \LL^{0}([0,1], \lambda, G) \) via
  the identification \( f \times (g_{i})_{i \in \mathbb{N}} \mapsto \zeta \), where
  \begin{displaymath}
    \zeta(t) =
    \begin{cases}
      f(t) & \textrm{if \( 0 \le t < 1/2 \)}, \\
      g_{i} & \textrm{if \( 1-2^{-i-1} \le t < 1-2^{-i-2} \) for \( i \in \mathbb{N} \)}.
    \end{cases}
  \end{displaymath}
  Pick families \( (\xi_{l})_{l \in \mathbb{N}} \) dense in \( \LL^{0}([0,1/2], \lambda, G) \), and
  \( (h_{m})_{m \in \mathbb{N}} \) dense in \( G \).

  Let us call a function \( \alpha : \mathbb{N} \to \mathbb{N} \) a multi-index if
  \( \alpha(i) = 0 \) for all but finitely many \( i \in \mathbb{N} \). We use
  \( \mathbb{N}^{<\mathbb{N}} \) to denote the set of all multi-indices. Given
  \( \alpha \in \mathbb{N}^{< \mathbb{N}} \), we define
  \( h_{\alpha} = (h_{\alpha(i)})_{i \in \mathbb{N}} \in G^{\mathbb{N}} \). Note that
  \( \{h_{\alpha} : \alpha \in \mathbb{N}^{<\mathbb{N}}\} \) is dense in \( G^{\mathbb{N}} \), and
  thus \( \{\xi_{l} \times h_{\alpha} : l \in \mathbb{N}, \alpha \in \mathbb{N}^{<\mathbb{N}}\} \)
  is a dense family in \( \LL^{0}([0,1/2], \lambda, G) \times G^{\mathbb{N}} \).

  Pick a tuple \( f_{1}, \ldots, f_{k} \in \LL^{0}([0,1], \lambda, G) \) that generates a dense
  subgroup. For each pair \( (l,\alpha) \in \mathbb{N} \times \mathbb{N}^{<\mathbb{N}} \), there
  exists a sequence of reduced words \( (w_{n}^{l, \alpha})_{n \in \mathbb{N}} \) in the free group
  on \( k \) generators such that \( w^{l,\alpha}_{n}(f_{1}, \ldots, f_{k}) \) converges to
  \( \xi_{l} \times h_{\alpha} \) in measure. By passing to a subsequence, we may assume that
  \( w^{l,\alpha}_{n}(f_{1}, \ldots, f_{k}) \to \xi_{l} \times h_{\alpha} \) pointwise almost
  surely. In other words, the set
  \[ P_{l,\alpha} = \bigl\{ t \in [0,1] : w^{l,\alpha}_{n}(f_{1}, \ldots, f_{k})(t) \to (\xi_{l}
    \times h_{\alpha})(t) \bigr\} \]
  has Lebesgue measure \( 1 \) for each
  \( (l, \alpha) \in \mathbb{N} \times \mathbb{N}^{<\mathbb{N}} \), and hence so does the set
  \[ P = \bigcap_{l \in \mathbb{N}} \bigcap_{\alpha \in \mathbb{N}^{<\mathbb{N}}} P_{l,\alpha}. \]
  Pick some \( t_{j} \in P \cap [1-2^{-j-1}, 1-2^{-j-2}) \), \( j \in \mathbb{N} \), and set
  \begin{displaymath}
    \tilde{f}_{i}(t) =
    \begin{cases}
      f_{i}(t) & \textrm{for \( 0 \le t < 1/2 \)}, \\
      f_{i}(t_{j}) & \textrm{for \( 1-2^{-j-1} \le t < 1-2^{-j-2} \) for \( j \in \mathbb{N} \)}.
    \end{cases}
  \end{displaymath}
  Elements \( \tilde{f}_{i} \) naturally belong to
  \( \LL^{0}([0,1/2], \lambda, G) \times G^{\mathbb{N}} \), and we claim that they generate a dense
  subgroup therein, witnessing
  \( \tprank(\LL^{0}([0,1/2], \lambda, G) \times G^{\mathbb{N}}) \le k \). To this end, recall that
  \( w_{n}^{l,\alpha}(f_{1}, \ldots, f_{k}) \to \xi_{l} \times h_{\alpha} \) pointwise almost
  surely. In particular,
  \[ w_{n}^{l,\alpha}(f_{1}, \ldots, f_{k})\restriction_{[0,1/2]} \to \xi_{l} \times
    h_{\alpha}\restriction_{[0,1/2]} \]
  in measure and, for each \( j \in \mathbb{N} \),
  \[ w_{n}^{l,\alpha}(f_{1}, \ldots, f_{k})(t_{j}) \to (\xi_{l} \times h_{\alpha})(t_{j}) =
    h_{\alpha(j)} \] is guaranteed by choosing \( t_{j} \in P \). We conclude that
  \[ w_{n}^{l,\alpha}(\tilde{f}_{1}, \ldots, \tilde{f}_{k}) \to \xi_{l} \times h_{\alpha} \] in
  \( \LL^{0}([0,1/2], \lambda, G) \times G^{\mathbb{N}} \), and therefore
  \[ \tprank(\LL^{0}([0,1/2], \lambda, G) \times G^{\mathbb{N}}) \le k. \]

  Finally, suppose that \( \LL^{0}([0,1], \lambda, G) \) is generically \( k \)-generated. Choose
  open sets \( U_{i} \subseteq \LL^{0}([0,1/2], \lambda, G) \times G^{\mathbb{N}}
  \),\( 1 \le i \le k \).  Shrinking them if necessary, we may assume that all \(U_{i}\) have the
  form \(U_{i} = A_{0}^{i} \times A_{1}^{i} \times \cdots \times A_{n}^{i} \times G^{\mathbb{N}}\),
  where \(A_{0}^{i}\) is open in \(\LL^{0}([0,1/2], \lambda, G)\), and \(A_{j}^{i}\), \(j \ge 1\),
  are open in \(G\).

  Pick \( V_{i} \subseteq \LL^{0}([0,1], \lambda, G) \), \( 1 \le i \le k \), to consist of those
  functions \(f\) satisfying \(f|_{[0, 1/2]} \in A_{0}\) and \(f(t) \in A_{j}\) for all
  \(t \in [1-2^{-j-1}, 1-2^{-j-2})\), \(1 \le j \le n\). Note that
  \( V_{i} \cap \LL^{0}([0,1/2], \lambda, G) \times G^{\mathbb{N}} = U_{i} \).

  Since \( \LL^{0}([0,1], \lambda, G) \) is assumed to be generically \( k \)-generated, there is a
  tuple \( (f_{1}, \ldots, f_{k}) \) generating a dense subgroup in \( \LL^{0}([0,1], \lambda, G) \)
  such that \( f_{i} \in V_{i} \) for each \( i \). Running the above construction, we get a tuple
  \[ (\tilde{f}_{1}, \ldots, \tilde{f}_{k}) \in \LL^{0}([0,1/2], \lambda, G) \times G^{\mathbb{N}}
  \]
  such that \( \tilde{f}_{i} \in U_{i} \), \( 1 \le i \le k \), whence
  \( \LL^{0}([0,1/2], \lambda, G) \times G^{\mathbb{N}} \) is generically \( k \)-generated.
\end{proof}

Lemma~\ref{lem:L0-times-Gn-has-same-rank} remains valid if we take the product with a finite power
of \( G \), which follows from Lemma~\ref{lem:rank-product-ge-rank-factors}.

\begin{corollary}
  \label{cor:L0-times-Gm-has-same-rank}
  For any Polish group \( G \) and any \( m \in \mathbb{N} \), one has
  \[ \tprank(\LL^{0}([0,1], \lambda, G)) = \tprank(\LL^{0}([0,1], \lambda, G)) \times G^{m}. \]
  If \( \tprank(\LL^{0}([0,1], \lambda, G)) \) is generically \( k \)-generated for some
  \( k \in \mathbb{N} \), then so is the group \( \LL^{0}([0,1], \lambda, G) \times G^{m} \).
\end{corollary}

We may now strengthen Proposition~\ref{prop:generic-2-generated-R} by dropping the assumption on
\( (Y,\nu) \) being atomless.

\begin{proposition}
  \label{prop:L0-std-Lebesgue-generically-2-generated}
  Let \( (Y,\nu) \) be a standard Lebesgue space and \( (Z,\lambda) \) be a standard probability
  space.  The Polish group \( \LL^{0}(Y,\nu, \Aut(Z,\lambda)) \) is generically \( 2 \)-generated.
\end{proposition}

\begin{proof}
  Let \( Y_{a} \) be the set of atoms of \( Y \), put \( Y_{0} = Y \setminus Y_{a} \) and
  \( \nu_{0} = \nu\restriction_{Y_{0}} \). The group \( \LL^{0}(Y,\nu, \Aut(Z,\lambda)) \) is
  naturally isomorphic to
  \[ \LL^{0}(Y_{0}, \nu_{0}, \Aut(Z,\lambda)) \times \Aut(Z,\lambda)^{|Y_{a}|}. \]
  An application of Proposition~\ref{prop:generic-2-generated-R} together with
  Lemma~\ref{lem:L0-times-Gn-has-same-rank} or Corollary~\ref{cor:L0-times-Gm-has-same-rank}
  (depending on whether \( Y_{a} \) is infinite or not) finishes the proof.
\end{proof}

\begin{proposition}
  \label{prop:increasing-union-k-generated} Let \( G \) be a Polish group and let
  \( H_{0} \le H_{1} \le \cdots \le G \) be a dense chain of Polish subgroups,
  \( \overline{\bigcup_{n}H_{n}} = G \). If each \( H_{n} \) is generically \( k \)-generated, then
  \( G \) is generically \( k \)-generated.
\end{proposition}

\begin{proof}
  We need to show that for any open \( U \subseteq G^{k} \) and any open \( V \subseteq G \) there
  is a tuple \( (g_{1}, \ldots, g_{k}) \in U \) such that
  \( \langle g_{1}, \ldots, g_{k} \rangle \cap V \ne \varnothing \). Since groups \( H_{n} \) are
  nested and \( \bigcup_{n}H_{n} \) is dense in \( G \), there is \( n \) so large that
  \( U \cap H_{n}^{k} \ne \varnothing \) and \( V \cap H_{n} \ne \varnothing \). It remains to use
  the fact that \( H_{n} \) is generically \( k \)-generated to find the required tuple.
\end{proof}

\begin{theorem}
  \label{thm:generically-2-generated-derived}
  Let \( G \acts X \) be a measure-preserving action of a locally compact amenable Polish normed
  group on a standard probability space \( (X, \mu) \). If almost every orbit of the action is
  uncountable, then the derived \( \LL^{1} \) full group \( \derived(\lfgr{G \acts X}) \) is
  generically \( 2 \)-generated and has topological rank \(2\).
\end{theorem}
\begin{proof}
  In view of Corollary~\ref{cor:dense-chain-subgroups}, there is a chain of subgroups
  \[ H_{0} \le H_{1} \le \cdots \le \derived(\lfgr{G \acts X}), \quad \overline{\bigcup_{n}H_{n}} =
    \derived(\lfgr{G \acts X}),\]
  where each \( H_{n} \) is isomorphic to \( \LL^{0}(Y_{n},\nu_{n}, \Aut([0,1], \lambda)) \) for
  some standard Lebesgue space \( (Y_{n}, \nu_{n}) \). By
  Proposition~\ref{prop:L0-std-Lebesgue-generically-2-generated}, every \( H_{n} \) is generically
  \( 2 \)-generated, and we may apply Proposition~\ref{prop:increasing-union-k-generated} to
  conclude that \( \derived(\lfgr{G \acts X}) \) is generically \( 2 \)-generated.  In particular,
  its topological rank is at most \( 2 \).  To see that its topological rank is actually equal to
  \( 2 \), simply note that \( \derived(\lfgr{G \acts X}) \) is not abelian (e.g., by the proof of
  Proposition~\ref{prop:symmetric-equals-alternating}).
\end{proof}

The assumption for orbits to be uncountable is essential, and
Theorem~\ref{thm:generically-2-generated-derived} is in striking contrast to the dynamical
interpretation of the topological rank of derived \( \LL^{1} \) full groups for actions of discrete
groups. As shown in~\cite[Thm.~4.3]{MR4398251}, given an aperiodic measure-preserving action of a
finitely generated group \( \Gamma \acts X \), the topological rank of
\( \derived(\lfgr{\Gamma \acts X}) \) is finite if and only if the action has finite Rokhlin
entropy.



\chapter{The index map for \texorpdfstring{\(\LL^1\)}{L1} full groups of flows}
\label{chap:index-map}

We now turn our attention to \textbf{flows}\index{Flow}, i.e., measure-preserving actions of \(\R\).
Since the group of reals is locally compact, amenable, unimodular, and, of course, Polish, all of
the results in the previous chapters apply to \(\mathbb{R}\)-flows.  A much more in-depth
understanding of \(\LL^{1}\) full groups of flows is possible and is based on the existence of the
so-called \emph{index map}, which we define and investigate in this chapter. This map is a
continuous homomorphism from the \( \LL^{1} \) full group of the flow to the additive group of
reals, which can be thought of as measuring the average shift distance. When the flow is ergodic,
such averages are the same across orbits. By taking the ergodic decomposition of the flow
\(\mathcal{F}\), we can adopt a slightly more general vantage point and view the index map
\( \ind \) as a homomorphism into the \( \LL^{1} \) space of functions on the space of invariant
measures \( (\mathcal{E}, p) \), \( \ind : \lfgr{\mathcal{F}} \to \LL^{1}(\mathcal{E}, p, \R) \).

Understanding the kernel of the index map is a task of fundamental importance. We will subsequently
identify \( \ker \ind \) with the topological derived subgroup of \( \lfgr{\mathcal{F}} \)
(Theorem~\ref{thm:index-kernel-is-derived-subgroup}). This will allow us to describe the
abelianizations of \( \LL^{1} \) full groups of flows and estimate the number of their topological
generators.

It has already been mentioned that any element \( T \) of a full group of a flow induces Lebesgue
measure-preserving transformations on orbits (Section~\ref{sec:orbit-transf}). When~\( T \)
furthermore belongs to the \( \LL^{1} \) full group, these transformations are special---they leave
``half-lines'' invariant up to a set of finite measure. Such transformations form the so-called
\textbf{commensurating group}\index{Transformation group!commensurating}. Let us therefore begin
with a more formal treatment of this group, which has already appeared in the literature before, for
instance in~\cite{robertsonNegativedefinitekernels1998}.

\section{Self-commensurating automorphisms of a subset}
\label{sec:comm-transf}

Consider an infinite measure space \( (Z, \lambda) \). We say that two measurable sets
\( A, B \subseteq Z \) are \textbf{commensurate}\index{Commensurate sets} if the measure of their
symmetric difference is finite, \( \lambda(A \triangle B) < \infty \). The relation of being
commensurate is an equivalence relation, and all sets of finite measure fall into a single
class. Note also that if \( A \) and \( B \) are both commensurate to some \( C \), then so is the
intersection \( A \cap B \); in other words, all equivalence classes of commensurability are closed
under finite intersections.

Let \( \comm(B) \) denote the set of all measurable \( A \subseteq Z \) that are commensurate
to \( B \). Fix some \( Y \subseteq Z \) and consider the semigroup of measure-preserving
transformations between elements of \( \comm(Y) \). More precisely, let \( \isom{}(Y, \lambda) \) be
the set of measure-preserving partial bijections \( T : A \to B \) between sets \( A, B \in \comm(Y) \),
which we call the \textbf{self-commensurating semigroup} of \((Y,\lambda)\).

Recall that we denote the domain of \( T \) as \( \dom T \) and its range as \( \rng T \). For
partial transformations
\( S: A \to B \) and \( T: A' \to B' \), the composition \( T \circ S \) has a domain given by
\( A \cap S^{-1}(A') \). As always, we identify two maps if they differ only on a null set. Since
the classes of commensurability are closed under finite intersections, the set
\( \isom{}(Y, \lambda) \) forms a semigroup with respect to composition.

This semigroup carries a natural equivalence relation: \( T \sim S \) whenever the transformations
disagree on a set of finite measure, \( \lambda(\{ x : Tx \ne Sx \}) < \infty \). This equivalence
is, moreover, a congruence, i.e., if \( T_{1} \sim S_{1} \) and \( T_{2} \sim S_{2} \), then
\( T_{1} \circ T_{2} \sim S_{1} \circ S_{2} \). One may therefore push the semigroup structure from
\( \isom{}(Y, \lambda) \) onto the set of equivalence classes, which we denote by
\( \Autm{}(Y, \lambda) \). An important observation is that \( \Autm{}(Y, \lambda) \) is a group.
Indeed, the identity corresponds to the map \( x \mapsto x \) on \( Y \), and for a representative
\( T \in \isom{}(Y, \lambda) \), its inverse inside \( \Autm{}(Y, \lambda) \) is, naturally, given
by \( T^{-1} : \rng T \to \dom T \). We call \( \Autm{}(Y, \lambda) \) the
\textbf{self-commensurating automorphism group} of \( Y \).

The self-commensurating semigroup admits an important homomorphism into the reals,
\(\ind : \isom{}(Y, \lambda) \to \mathbb{R}\), called the \textbf{index map}\index{Index map!for self-commensurating automorphisms} and
defined by
\[ \ind(T) = \lambda(\dom T \setminus \rng T) - \lambda(\rng T \setminus \dom T). \]

\begin{lemma}
  \label{lem:properties-of-the-index-map} For all \( T \in \isom{}(Y, \lambda) \), the index map
  satisfies the following:
  \begin{enumerate}
    \item\label{lem:ambient-set} if \( A \in \comm(Y) \) is such that
      \( \dom T \subseteq A \) and \( \rng T \subseteq A \), then
      \[ \ind(T) = \lambda(A \setminus \rng T) - \lambda(A \setminus \dom T) ;\]
    \item\label{lem:restriction-preserves-index} if \( T' \in \isom{}(Y, \lambda) \) is a
      restriction of \( T' \), that is \( T' = T\restriction_{\dom T'} \), then
      \( \ind(T') = \ind(T) \).
  \end{enumerate}
\end{lemma}

\begin{proof}
  {}\eqref{lem:ambient-set}~If \( A \subseteq Z \) is commensurate to \( Y \) and
  \( \dom T \subseteq A \), \( \rng T \subseteq A \), then
  \begin{displaymath}
    \begin{aligned}
      \ind(T) &= \lambda(\dom T \setminus \rng T) - \lambda(\rng T \setminus \dom T) \\
              &= \lambda(A \setminus \rng T) - \lambda (A \setminus (\dom T \cup \rng T)) \\
      & \qquad - (\lambda(A \setminus \dom T) - \lambda (A \setminus (\dom T \cup \rng T))) \\
      &= \lambda(A \setminus \rng T) - \lambda(A \setminus \dom T).
    \end{aligned}
  \end{displaymath}

  {}\eqref{lem:restriction-preserves-index}~If \( T' \in \isom{}(Y, \lambda) \) is a restriction of
  \( T \), then
  \[ T(\dom T \setminus \dom T') = \rng T \setminus \rng T'. \]
  Thus, for any \( A \in \comm(Y) \) containing both \( \dom T \) and \( \rng T \),
  item~\eqref{lem:ambient-set} implies
  \begin{displaymath}
    \begin{aligned}
      \ind(T) &= \lambda(A \setminus \dom T) - \lambda(A \setminus \rng T) \\
              &= \lambda(A \setminus \dom T') - \lambda(\dom T \setminus \dom T') \\
      & \qquad
      - (\lambda(B \setminus \rng T') - \lambda(\rng T \setminus \rng T')) \\
      &= \lambda(A \setminus \dom T') - \lambda(A \setminus \rng T') = \ind(T'),
    \end{aligned}
  \end{displaymath}
  where the equality \(\lambda(\dom T \setminus \dom T') = \lambda(\rng T \setminus \rng T') \) is
  based on \(T\) being measure-preserving.
\end{proof}

\begin{proposition}
  \label{prop:index-map-homomorphism}
  The index map \( \ind : \isom{}(Y, \lambda) \to \mathbb{R} \) is a homomorphism. Moreover, if
  \( T, S \in \isom{}(Y, \lambda) \) are equivalent, \( T \sim S \), then \( \ind(T) = \ind(S) \).
\end{proposition}

\begin{proof}
  In view of Lemma~\ref{lem:properties-of-the-index-map}\eqref{lem:restriction-preserves-index}, to
  check that \( \ind(T_{1} \circ T_{2}) = \ind(T_{1}) + \ind(T_{2}) \), we may pass to restrictions
  of these transformations and assume that \( \rng T_{2} = \dom T_{1} \). Pick a set
  \( A \in \comm(Y) \) large enough to contain the domains and ranges of \( T_{1} \) and
  \( T_{2} \). By Lemma~\ref{lem:properties-of-the-index-map}\eqref{lem:ambient-set},
  \begin{displaymath}
    \begin{aligned}
      \ind(T_{1} \circ T_{2}) &= \lambda(A \setminus \rng T_{1}) - \lambda( A \setminus \dom T_{2}) \\
      &= \lambda(A \setminus \rng T_{1}) - \lambda(A \setminus \dom T_{1})
      + \lambda(A \setminus \rng T_{2}) - \lambda( A \setminus \dom T_{2}) \\
      &= \ind(T_{1}) + \ind(T_{2}).
    \end{aligned}
  \end{displaymath}

  For the moreover part, suppose that \( T, S \in \isom{Y}(Y, \lambda) \) are equivalent. Let
  \( U \) be the restriction of \( T \) and \( S \) onto the set \( \{x : Tx = Sx\} \). Using
  Lemma~\ref{lem:properties-of-the-index-map}\eqref{lem:restriction-preserves-index} once again, we
  get \( \ind(T) = \ind(U) = \ind(S) \). Hence, the index map is invariant under the equivalence
  relation \(\sim\).
\end{proof}

The proposition above implies that the index map respects the relation \( \sim \), and hence gives
rise to a map from \( \Autm{}(Y, \lambda) \) to the reals.

\begin{corollary}
  \label{cor:index-map-commensurate-group}
  The index map factors to a group homomorphism \[ \ind : \Autm{}(Y, \lambda) \to \mathbb{R}. \]
\end{corollary}

\section{The commensurating automorphism group}
\label{sec:comm-autom}

Let us again consider an infinite measure space \( (Z, \lambda) \) and \( Y \subseteq Z \) a
measurable subset. We now define the \textbf{commensurating automorphism group of \( Y \) in
  \( Z \)} as the group of all measure-preserving transformations \( T \in \Aut(Z, \lambda) \) such
that \( \lambda(Y \triangle T(Y)) < \infty \).  We denote this group by \( \Aut_Y(Z,\lambda) \).

Every \( T \in \Aut_Y(Z,\lambda) \) naturally gives rise to an element of \( \Autm{}(Y, \lambda) \)
by considering its restriction \( T \restriction_{Y} \).  The following lemma shows that in this
case we may use any other set \( A \) commensurate to \( Y \) instead without changing the
corresponding element of the commensurating group.

\begin{lemma}
  \label{lem:commensurate-condition-is-invariant}
  Let \( T \in \Aut(Z, \lambda) \) be a measure-preserving automorphism. If
  \( T\restriction_{A} \in \isom{}(Y, \lambda) \) for some \( A \in \comm(Y) \), then
  \( T\restriction_{B} \in \isom{}(Y, \lambda) \) and \( T\restriction_{B} \sim T\restriction_{A} \)
  for all \( B \in \comm(Y) \).
\end{lemma}

\begin{proof}
  Since commensuration is an equivalence relation and \(A\) is commensurate to \(Y\), the assumption
  \( T\restriction_{A} \in \isom{}(Y, \lambda) \) is equivalent to
  \( \lambda(A \triangle T(A)) < \infty \). Moreover, given \( B \in \comm(Y) \), we only need to
  show that \( \lambda(B \triangle T(B)) \) is finite in order to conclude that
  \( T\restriction_{B} \in \isom{}(Y,\lambda) \). So we compute
  \begin{displaymath}
    \begin{aligned}
      \lambda(B \triangle T(B))
      =& \lambda(B \setminus T(B)) + \lambda(T(B) \setminus B) \\
      \le& \lambda(A \setminus T(A)) + \lambda(B \setminus A) + \lambda(T(A \setminus B)) \\
       & \qquad + \lambda(T(A) \setminus A) + \lambda(A \setminus B) + \lambda(T(B \setminus A)) \\
      =&\lambda(A \triangle T(A)) + 2 \lambda(A \triangle B ) < \infty.
    \end{aligned}
  \end{displaymath}
  Thus, the measure \( \lambda(B \triangle T(B)) \) is finite. Hence
  \( T\restriction_{B} \in \isom{}(Y, \lambda) \) for all \( B \in \comm(Y) \). Finally,
  \( T\restriction_{A} \sim T\restriction_{B} \), since these transformations agree on
  \( A \cap B \).
\end{proof}

To summarize, if \( T\restriction_{A} \in \isom{}(Y, \lambda) \) for some \( A \in \comm(Y) \), then
all restrictions \( T\restriction_{B} \), \( B \in \comm(Y) \), are pairwise equivalent. Hence, they
correspond to the same element \( T\restriction_{Y} \in \Autm{}(Y, \lambda) \). According to
Proposition~\ref{prop:index-map-homomorphism}, the index \( \ind(T\restriction_{Y}) \) of this
element can be computed as
\( \ind(T\restriction_{Y}) = \lambda(B \setminus T(B)) - \lambda(B \setminus T^{-1}(B)) \) for any
\( B \in \comm(Y) \).

\section[Index map on L1 full groups]{Index map on \texorpdfstring{\( \LL^{1} \)}{L1} full groups
of \texorpdfstring{\( \mathbb{R} \)-flows}{R-flows}}
\label{chap:index-map-l1-full-groups}

Let \( \mathcal{F} = \mathbb{R} \acts X \) be a free measure-preserving Borel flow, let
\( \lfgr{\mathcal{F}} \) be the associated \( \LL^{1} \) full group, where we endow \( \mathbb{R} \)
with the standard Euclidean norm, and let \( T \in \lfgr{\mathcal{F}} \). The action of
\(r \in \mathbb{R}\) upon \(x \in X\) is denoted additively by \(x + r\). Recall that the cocycle of
\( T \) is denoted by \( \rho_{T} : X \to \mathbb{R} \) and is defined by the equality
\(T(x)=x+\rho_T(x)\) for all \(x\in X\). We are going to argue that, on every orbit, \( T \) induces
a measure-preserving transformation that belongs to the commensurate group of
\( \mathbb{R}^{\ge 0} \), when the orbit is identified with the real line.

Consider the function \( f : \mathcal{R}_{\mathcal{F}} \to \{-1,0, 1\} \) defined by
\begin{displaymath}
  f(x,y) =
  \begin{cases}
    1 & \textrm{if } x < y < T(x),\\
    -1 & \textrm{if } T(x) < y < x, \\
    0 & \textrm{otherwise}.
  \end{cases}
\end{displaymath}
One can think of \( f \) as a ``charge function'' that spreads charge \( +1 \) over each interval
\( (x, T(x)) \) and \( -1 \) over \( (T(x), x) \). Note that we have both
\[ \int_{\mathbb{R}}f(x, x + r)\, d\lambda(r) = \rho_{T}(x) \quad \textrm{and} \quad
  \int_{\mathbb{R}}\abs{f(x, x + r)}\, d\lambda(r) = \abs{\rho_{T}(x)}. \]
Since \( T \) belongs to the \( \LL^{1} \) full group, its cocycle is integrable, which means that
\( f \) is \( M \)-integrable (see Section~\ref{sec:orbit-transf}). By the mass-transport principle,
the following integrals are equal and finite:
\[ \int_{X} \int_{\mathbb{R}} \abs{f(x, x + r)}\, d\lambda(r) d\mu(x) = \int_{X} \int_{\mathbb{R}}
  \abs{f(x+r, x)}\, d\lambda(r) d\mu(x). \]
In particular, the integral \( \int_{\mathbb{R}}\abs{f(x +r,x)}\, d\lambda(r) \) is finite for
almost all \(x\).

Let \( T_{x} \in \Aut(\mathbb{R}, \lambda) \) denote the transformation induced by \( T \) onto the
orbit of \( x \) obtained by identifying the origin of the real line with \( x \), which is
measure-preserving by Proposition~\ref{prop:orbital-transformations}.  One can reinterpret the
integral \( \int_{\mathbb{R}}\abs{f(x +r,x)}\, d\lambda(r) \) as follows:
\begin{displaymath}
  \begin{aligned}
    \int_{\mathbb{R}}|f(x +r,x)| \, d\lambda(r) &
    = \lambda\bigl(\mathbb{R}^{\ge 0} \setminus T_{x}(\mathbb{R}^{\ge 0})\bigr)
                                                  + \lambda\bigl(T_{x}(\mathbb{R}^{\ge 0})\setminus \mathbb{R}^{\ge 0}\bigr) \\
    &=\lambda (\mathbb{R}^{\ge 0}\bigtriangleup T_x(\mathbb{R}^{\ge 0})).
  \end{aligned}
\end{displaymath}
In particular, \( T_x\restriction_{\mathbb{R}^{\ge 0}} \) belongs to the commensurating group of
\( \mathbb{R}^{\ge 0} \).
Observe that we also have
\[
\int_{\mathbb{R}}f(x +r,x) \, d\lambda(r)
= \lambda\bigl(\mathbb{R}^{\ge 0} \setminus T_{x}(\mathbb{R}^{\ge 0})\bigr)
- \lambda\bigl(T_{x}(\mathbb{R}^{\ge 0})\setminus \mathbb{R}^{\ge 0}\bigr),
\]
which is equal to the index of
\( T_{x}\restriction_{\mathbb{R}^{\ge 0}} \). By Section~\ref{sec:comm-autom},
\( \ind(T_{x}\restriction_{\mathbb{R}^{\ge 0}}) = \ind(T_{y}\restriction_{\mathbb{R}^{\ge 0}}) \)
whenever \( x \mathcal{R}_{\mathcal{F}} y \). For any \( T \in \lfgr{\mathcal{F}} \), we therefore
have an \(\mathcal{F}\)-invariant measurable map \( h_{T} : X \to \mathbb{R} \) given by
\( h_{T}(x) = \int_{\mathbb{R}}f(x +r,x) \, d\lambda(r) \).  Note that for any
\( \mathcal{F} \)-invariant set \( Y \subseteq X \), the mass-transport principle yields
\begin{equation}
  \label{eq:integral-cocycle-h}
  \int_{Y}\rho_{T}(x)\, d\mu(x) = \int_{Y}h_{T}(x)\, d\mu(x).
\end{equation}

Let \( (\mathcal{E}, p) \), \( X \ni x \mapsto \nu_{x} \in \mathcal{E} \), be the ergodic
decomposition of \( (X, \mu, \mathcal{F}) \) (see Appendix~\ref{sec:ergod-decomp}). Since the map
\( h_{T} \) is \( \mathcal{F} \)-invariant, it produces a map
\( \tilde{h}_{T} : \mathcal{E} \to \mathbb{R} \) via \( \tilde{h}(\nu) = h(x) \) for any \( x \)
such that \( \nu = \nu_{x} \) or, equivalently, via
\[ \tilde{h}_{T}(\nu) = \int_{X}\int_{\mathbb{R}}f(x+r, x)\, d\lambda(r)d\nu(x). \]
Note also that
\begin{equation}\label{eq:index-is-a-contraction}
  \int_{\mathcal{E}} \abs{\tilde{h}_{T}(\nu)}\, dp(\nu)
  \leq \int_{X}\int_{\mathbb{R}}\abs{f(x+r, x)}\, d\lambda(r)d\mu(x)
  =\int_{X}\abs{\rho_T(x)}\,d\mu(x),
\end{equation}
thus \( \tilde{h}_{T} \in \LL^{1}(\mathcal{E},p, \mathbb{R}) \). We can now define the index map of
a (possibly non-ergodic) flow as a function
\( \ind : \lfgr{\mathcal{F}} \to \LL^{1}(\mathcal{E}, p,\mathbb{R}) \).

\begin{definition}
  \label{def:index-map}
  Let \( \mathcal{F} = \mathbb{R} \acts X \) be a free measure-preserving flow on a standard
  probability space \( (X, \mu) \); let also \( (\mathcal{E}, p) \) be the space of
  \( \mathcal{F} \)-invariant ergodic probability measures, where \( p \) is the probability measure
  yielding the disintegration of \( \mu \). The \textbf{index map}\index{Index map!on the \( \LL^1 \) full group} is the function
  \( \ind : \lfgr{\mathcal{F}} \to \LL^{1}(\mathcal{E},p, \mathbb{R}) \) given by
  \[ \ind(T)(\nu) = \tilde{h}_{T}(\nu) = \int_{X}\int_{\mathbb{R}}f(x+r, x)\, d\lambda(r)d\nu(x). \]
\end{definition}

\begin{proposition}
  \label{prop:index-map-is-surjective}
  For any free measure-preserving flow \( \mathcal{F} = \mathbb{R} \acts X \), the index map
  \( \ind : \lfgr{\mathcal{F}} \to \LL^{1}(\mathcal{E},p, \mathbb{R}) \) is a continuous and
  surjective homomorphism. Furthermore, its kernel consists of all \(T\in \lfgr{\mathcal{F}}\)
  satisfying, for almost all \(y\in X\),
  \begin{equation}\label{eq:kernel-of-index}
  \lambda_y(\{x \in \supp T : x < y \le Tx\}) =
  \lambda_y(\{x \in \supp T : Tx < y \le x\}).
  \end{equation}
\end{proposition}

\begin{proof}
  The index map is a homomorphism, since, as we have discussed earlier, \( h_{T}(x) \) is equal to
  the index of \( T_{x}\restriction_{\mathbb{R}^{\ge 0}} \). Continuity follows from the fact that
  \( \ind \) is \(1\)-Lipschitz as a direct consequence of Eq.~\eqref{eq:index-is-a-contraction}.
  To see surjectivity, pick any \( \tilde{h} \in \LL^{1}(\mathcal{E},p, \mathbb{R}) \) and view it
  as a map \( h : X \to \mathbb{R} \) via the identification \( h(x) = \tilde{h}(\nu_{x}) \). Define
  the automorphism \( T \in \Aut(X, \mu) \) by \( T(x) = x + h(x) \). It is straightforward to check
  that \( T \in \lfgr{\mathcal{F}} \) and \( \ind(T) = h \).

  Finally, according to the definition of the index map, the kernel is the set of all
  \(T\in\lfgr{\mathcal{F}}\) for which, almost everywhere in \(X\), the condition \(h_{T}(y) = 0 \)
  holds.  By the definition of \(h_T\) and the charge function \(f\), this translates to the
  relationship
  \[\lambda_y(\{x \in \supp T : x < y < Tx\}) = \lambda_y(\{x \in \supp T : Tx < y < x\}).\]
  Since \(\lambda_y\) is atomless, the above equality is equivalent to the desired condition.
\end{proof}

The quotient group \( \lfgr{\mathcal{F}}/ \ker{\ind} \) naturally inherits the quotient norm given
by
\[ \norm{T \ker{\ind}}_{1} = \inf_{S \in \ker{\ind}} \norm{T S}_{1}. \]
By Proposition~\ref{prop:index-map-is-surjective}, the index map induces an isomorphism between
\( \lfgr{\mathcal{F}}/ \ker{\ind} \) and \( \LL^{1}(\mathcal{E},p, \mathbb{R}) \). We argue that this
isomorphism is, in fact, an isometry.

\begin{proposition}\label{prop:index-map}
  The index map \( \ind \) induces an isometric isomorphism from
  \( \lfgr{\mathcal{F}}/ \ker{\ind} \) onto \( \LL^{1}(\mathcal{E},p, \mathbb{R}) \), where the
  former is endowed with the quotient norm and the latter bears the usual \( \LL^{1} \) norm.
\end{proposition}

\begin{proof}
  Since \( \int_{X}|h_{T}(x)| \, d\mu(x) = \int_{\mathcal{E}} |\tilde{h}_{T}(\nu)| \, dp(\nu) \), it
  suffices to show that for all \( T \in \lfgr{\mathcal{F}} \)
  \[ \inf_{S \in \ker\ind} \norm{T S}_{1} = \int_{X}|h_{T}|\, d\mu. \]
  Let \(T\in \lfgr{\mathcal{F}}\).  We first show the inequality
  \( \displaystyle\inf_{S \in \ker\ind} \norm{T S}_{1} \ge \int_{X}|h_{T}|\, d\mu \).

  Pick any \(S \in \ker \ind\).  For any \(\mathcal{F}\)-invariant measurable \(Y \subseteq X\),
  \(\int_{Y} \rho_{S} \, d\mu = 0\) and
  \[\int_{Y}\rho_{TS}\, d\mu = \int_{Y} \rho_{T}(S(x))\, d\mu(x) + \int_{Y} \rho_{S}(x)\, d\mu(x) =
    \int_{Y} \rho_{T}\, d\mu = \int_{Y} h_{T} \, d\mu,\]
  where we rely on Eq.~\eqref{eq:integral-cocycle-h} and \(S\) being measure-preserving.  Consider
  the \(\mathcal{F}\)-invariant sets
  \[Y^{<0} = \{x \in X : h_{T}(x) < 0\} \quad
    \textrm{and} \quad Y^{\ge0} =\{x \in X : h_{T}(x) \ge 0\}.\]
  The norm \(\norm{TS}_{1}\) can be estimated from below as follows.
  \begin{displaymath}
    \begin{aligned}
      \norm{TS}_{1}
      &= \int_{X}|\rho_{TS}|\, d\mu  = \int_{Y^{< 0}}|\rho_{TS}|\, d\mu + \int_{Y^{\ge 0}}|\rho_{TS}|\, d\mu \\
      &\geq \abs{\int_{Y^{< 0}}\rho_{TS}\, d\mu} + \abs{\int_{Y^{\ge 0}}\rho_{TS}\, d\mu} \\
      &= \abs{\int_{Y^{< 0}}h_{T}\, d\mu} + \abs{\int_{Y^{\ge 0}}h_{T}\, d\mu} \\
      &= -\int_{Y^{< 0}}h_{T}\, d\mu + \int_{Y^{\ge 0}}h_{T}\, d\mu = \int_{X}|h_{T}|\, d\mu. \\
    \end{aligned}
  \end{displaymath}
  We conclude that
  \[ \inf_{S \in \ker{\ind}} \norm{TS}_{1} \ge \int_{X}|h_{T}|\, d\mu. \]

  For the other direction, consider a transformation \( T' \) defined by \( T'(x) = x + h_{T}(x) \);
  note that \( T' \in \lfgr{\mathcal{F}} \), \( \rho_{T'}(x) = h_{T'}(x) = h_{T}(x) \) for all
  \( x \in X \), and \( T^{-1}T' \in \ker \ind \). Therefore
  \[ \inf_{S \in \ker{\ind}} \norm{TS}_{1} \le \norm{TT^{-1}T'}_{1} = \norm{T'}_{1} =
    \int_{X}|h_{T'}|\, d\mu = \int_{X}|h_{T}|\, d\mu, \]
  and the desired equality of norms follows.
\end{proof}

Using similar reasoning, we obtain the following characterization of the \(\LL^1\) full group and
the index map, where for all \(T\in[\mathcal R_{\mathcal F}]\), we let \(r_{T}\) be the
measure-preserving transformation of \((\mathcal R_{\mathcal F},M)\) given by
\(r_{T}(x,y)=(x,T(y))\) (see Section~\ref{sec:orbit-transf}).

\begin{proposition}\label{prop:commensuration-of-L1}
  Let \(\mathcal F = \R \acts X\) be a free measure-preserving \(\R\)-flow.  Consider the set
  \(\mathcal R^{\geq 0}=\{(x,y)\in\mathcal R_{\mathcal F} : x\geq y\}\).  Then for every
  \(T\in [\R \acts X]\), we have
  \[
    \norm{T}_1=M\left(\mathcal R^{\geq 0}\bigtriangleup r_T(\mathcal R^{\geq 0})\right).
  \]
  In particular, the \(\LL^1\) full group of \(\mathcal F\) can be seen as the commensurating group
  of \(\mathcal R^{\geq 0}\) inside the full group of \(\mathcal R\).  Moreover, in the ergodic
  case, the index of \(T\) as defined above is equal to its index as a commensurating transformation
  of the set \(\mathcal R^{\geq 0}\) in the sense of Section~\ref{sec:comm-transf}.
\end{proposition}
\begin{proof}
  Through the identification \((x,t)\mapsto(x,x+t)\), the measure-preserving transformation \(r_T\)
  is acting on \(X\times \R\) as \(\mathrm{id}_{X}\times T_x\), and the set \(\mathcal R^{\geq 0}\)
  becomes \(X\times \R^{\geq 0}\).  We then have
  \begin{align*}
    M( \mathcal R^{\geq 0}\bigtriangleup r_T(\mathcal R^{\geq 0}))
    &=\int_X \lambda(\R^{\geq 0}\bigtriangleup (T_x(\R^{\geq 0})))\, d\mu(x)\\
    &=\int_X\abs{\rho_T}\, d\mu
  \end{align*}
  by the mass-transport principle, which yields the conclusion, since by the definition of the norm
  \(\norm{T}_1=\int_X\abs{\rho_T}\, d\mu\).

  The moreover part follows from a similar computation.
\end{proof}

\begin{remark}
  The full group of \(\mathcal R\) embeds via \(T\mapsto r_T\) into the group of measure-preserving
  transformations of \((\mathcal R,M)\).  One could use this and the fact that the commensurating
  automorphism group of \(\mathcal R^{\geq 0}\) is a Polish group in order to give another proof
  that \(\LL^1\) full groups of measure-preserving \(\mathbb{R}\)-flows are themselves Polish.
\end{remark}



\chapter[Bounded elements of the full group]{Orbitwise ergodic bounded elements of full groups}
\label{chap:example-transf}

The purpose of this chapter is to contrast some of the differences in the dynamics of the elements
of full groups of \( \mathbb{Z} \)-actions and those arising from \( \mathbb{R} \)-flows.  Let
\( S \in \fgr{\mathbb{Z} \acts X} \) be an element of the full group of a measure-preserving
aperiodic transformation, and let \( \rho_{S^{k}} : X \to \mathbb{Z} \) be the cocycle associated
with \( S^{k} \) for \( k \in \mathbb{Z} \). Since \( \mathbb{Z} \) is a discrete group, the
conservative part in the Hopf decomposition for \( S \) (see
Appendix~\ref{sec:hopf-decomposition-appendix}) reduces to the set of periodic orbits. In
particular, an aperiodic \( S \in \fgr{\mathbb{Z} \acts X} \) has to be dissipative, hence
\( |\rho_{S^{k}}(x)| \to \infty \) as \( k \to \infty \). When \( S \) belongs to the \( \LL^{1} \)
full group of the action, a theorem of R.~M.~Belinskaja~\cite[Thm.~3.2]{MR0245756} strengthens this
conclusion and asserts that for almost all \( x \) in the dissipative component of \( S \), either
\( \rho_{S^{k}}(x) \to +\infty \) or \( \rho_{S^{k}}(x) \to -\infty \).

Given an arbitrary free measure-preserving flow \( \mathbb{R} \acts X \), we construct an example of
an aperiodic \( S \in \lfgr{\mathbb{R} \acts X} \) for which the signs in
\( \{\rho_{S^{k}}(x) : k \in \mathbb{N} \} \) keep alternating indefinitely for almost all
\( x \in X \). In fact, we present a transformation that acts ergodically on each orbit of the flow
(in particular, it is conservative and globally ergodic as soon as the flow is ergodic).  Moreover,
we ensure it has a uniformly bounded cocycle. Our argument uses a variant of the well-known cutting
and stacking construction adapted for infinite measure spaces. Additional technical difficulties
arise from the necessity to work across all orbits of the flow simultaneously.  The transformation
will arise as a limit of special partial transformations we call \emph{castles}, which we now define.

The \textbf{pseudo full group}\index{Pseudo full group} of the flow is the set of injective Borel
maps \( \varphi: \dom\varphi \to \rng\varphi \) between Borel sets \( \dom \varphi \subseteq X \),
\( \rng \varphi \subseteq X \), for which there exists a countable Borel partition
\( (A_n)_{n \in \mathbb{N}} \) of the domain \( \dom\varphi \) and a countable family of reals
\( (t_n)_{n \in \mathbb{N}} \) such that \( \varphi(x)=x+t_n \) for every \( x\in A_n \). Such maps
are measure-preserving isomorphisms between \( (\dom \varphi, \mu \restriction_{\dom \varphi}) \)
and \( (\rng \varphi, \mu \restriction_{\rng \varphi}) \), in other words they are
partial transformations. The
\textbf{support}\index{Transformation!support} of \( \varphi \) is
the set
\[
  \supp\varphi = \{ x\in \dom\varphi : \varphi(x)\neq x\}\cup\{x\in\rng\varphi : \varphi^{-1}(x)\neq
  x \}.
\]
Given \( \varphi \) in the pseudo full group and a Borel set \( A\subseteq X \), we let
\[ \varphi(A)=\{\varphi(x) : x\in A\cap \dom\varphi\}. \]
In particular, \( \varphi(A)=\varnothing \) if \( A \) is disjoint from \( \dom\varphi \). A
\textbf{castle} is an element \( \varphi \) of the pseudo full group of the flow such that for
\( B=\dom \varphi\setminus\rng\varphi \) the sequence \( (\varphi^k(B))_{k\in\N} \) consists of
pairwise disjoint subsets which cover its support. Since \( \varphi \) is measure-preserving, for
almost every \( x\in B \) there is \( k\in\N \) such that \( \varphi^k(x)\not\in\dom \varphi \). It
follows that \( \varphi^{-1} \) is also a castle. The set \( B \) is called the \textbf{basis} of the
castle, and the basis of its inverse \( C \) is called its \textbf{ceiling}, which is equal to
\( \rng\varphi\setminus\dom\varphi \). Observe that if two castles have disjoint supports, then
their union is also a castle. We denote by \( \vec\varphi:B\to C \) the element of the pseudo full
group which takes every element of the basis of \( \varphi \) to the corresponding element of the
ceiling.

\begin{remark}
  Equivalently, one could define a castle as an element \( \varphi \) of the pseudo full group which
  induces a \emph{graphing} consisting of finite segments only
  (see~\cite[Sec.~17]{kechrisTopicsOrbitEquivalence2004} for the definition of a graphing). It
  induces a partial order \( \le_\varphi \) defined by \( x\leq _\varphi y \) if and only if there
  is \( k\in \N \) such that \( y=\varphi^k(x) \). The basis of the castle is the set of minimal
  elements, while the ceiling is the set of maximal ones. Finally, \( \vec\varphi \) is the map
  which takes a minimal element to the unique maximal element above it.
\end{remark}

\begin{theorem}
  \label{thm:example-conservative-ergodic-unbounded}
  Let \( \mathbb{R} \acts X \) be a free measure-preserving flow. There exists
  \( S\in \fgr{\mathbb{R} \acts X} \) that acts ergodically on every orbit of the flow and whose
  cocycle is bounded by \( 4 \). Moreover, the signs in
  \( \{\rho_{S^{k}}(x) : k \in \mathbb{N} \} \) keep changing indefinitely for almost all
  \( x \in X \).
\end{theorem}

\begin{proof}
  Fix a free measure-preserving flow \( \mathbb{R} \acts X \), and let \( \mathcal{C} \subset X \)
  be a cross-section.

  We recall some notation from Section~\ref{sec:prelim-flows}.  Since \( \mathcal{C} \) is lacunary,
  for any \( c \in \mathcal{C} \), the function
  \( \gap_{\mathcal{C}}(c)=\min\{ r > 0 : c+r \in \mathcal{C}\} \) is well-defined. This gives the
  first return map \( \sigma_{\mathcal{C}} : \mathcal{C} \to \mathcal{C} \) via
  \( \sigma_{\mathcal{C}}(c) = c + \gap_{\mathcal{C}}(c) \), which is Borel.  There is also a
  natural bijective correspondence between \( X \) and the set
  \( \{(c, t) \in \mathcal{C} \times \mathbb{R}^{\ge 0} : c \in \mathcal{C}, 0 \le t <
  \gap_{\mathcal{C}}(c)\} \).  Let \( \lambda_{c}^{\mathcal{C}} \) be the ``Lebesgue measure'' on
  \( c + [0, \gap_{\mathcal{C}}(c)) \) given by
  \[ \lambda_{c}^{\mathcal{C}}(A) = \lambda(\{r \in \mathbb{R} : 0 \le r < \gap_{\mathcal{C}}(c),
    c+r \in A\}). \]
  The measure \( \mu \) on \( X \) can then be disintegrated as
  \( \mu(A) = \int_{\mathcal{C}} \lambda_{c}^{\mathcal{C}}(A)\, d\nu(c) \) for some finite (but not
  necessarily probability) measure \( \nu \) on \( \mathcal{C} \), as explained at the end of
  Section~\ref{sec:prelim-flows}.

  Let \( (\mathcal{C}_{n})_{n \in \mathbb{N}} \) be a vanishing sequence of markers---a sequence of
  nested cross-sections \( \mathcal{C}_{1} \supset \mathcal{C}_{2} \supset \mathcal{C}_{3} \cdots \)
  with an empty intersection: \( \bigcap_{n \in \mathbb{N}} \mathcal{C}_{n} = \varnothing \). We may
  arrange \( \mathcal{C}_{1} \) to be such that
  \( \gap_{\mathcal{C}_{1}}(c) \in \mathopen (2, 3 \mathclose ) \) for all
  \( c \in \mathcal{C}_{1} \). Put
  \[ \mathcal{C}_{0} = \{ c + k : c \in \mathcal{C}_{1}, k\in\{0,1, 2\} \} \]
  and \( Y = \mathcal{C}_{1} + [0,2) \). Note that \( \mu(X \setminus Y) \leq \frac{1}{3} \). Our
  first goal is to define an element \( \varphi \) of the pseudo full group with domain and range
  equal to \( Y \) such that for almost every \( x \in Y \), the action of \( \varphi \) on the
  intersection of the orbit of \( x \) with \( Y \) is ergodic and has a cocycle bounded by \( 3
  \). It will then be easy to modify \( \varphi \) to an element of the full group whose action on
  each orbit of the flow is ergodic at the cost of increasing the cocycle bound to \( 4 \). \medskip

  Our first partial transformation \( \varphi \) will arise as the limit of a sequence of castles
  \( (\varphi_n)_{n \in \mathbb{N}} \), with each \( \varphi_n \) belonging to the pseudo full group
  of \( \mathcal{R}_{\mathcal{C}_n} \). We also use another family of castles
  \( (\psi_n)_{n \in \mathbb{N}} \) which allows us to extend \( \varphi_n \) by ``going back'' from
  its ceiling to its basis while keeping the cocycle bound (this is our main adjustment compared to
  the usual cutting and stacking procedure). Both sequences of castles will have their cocycles
  bounded by \( 3 \). Here are the basic constraints that these sequences have to satisfy:

  \begin{enumerate}
  \item for all \( n\geq 1 \), \( Y=\supp \varphi_n\sqcup \supp\psi_n \);
  \item for all \( n\geq 1 \), \( \varphi_{n+1} \) extends \( \varphi_n \);
  \item \( \mu(\supp\psi_n) \) tends to \( 0 \) as \( n \) tends to \( +\infty \).
  \end{enumerate}

  The bases and ceilings of \( (\varphi_{n})_{n \in \mathbb{N}} \) and
  \( (\psi_{n})_{n \in \mathbb{N}} \) will satisfy additional constraints that will enable us to
  make the induction work and ensure ergodicity on each orbit of the flow.  In order to specify
  these constraints properly, we introduce the following notation.

  Each orbit of the flow comes with the linear order \( < \) inherited from \( \mathbb{R} \) via
  \( x<y \) if and only if \( y=x+t \) for some \( t>0 \). Set \( \kappa_{\mathcal{C}_n}(x) \) to be
  the minimum of the intersection of \( \mathcal{C}_n \) with the cone \( \{y \in X : y \geq x\} \).

  Let \( \mathcal{D}_1 = \mathcal{C}_1+2 \subseteq \mathcal{C}_0 \) and \( \mathcal{D}_n \) be the set
  of those \( x\in \mathcal{D}_1 \) which are maximal in \( \kappa_{\mathcal{C}_{n}}^{-1}(c) \) among
  points of \( \mathcal{D}_{1} \) for some \( c \in \mathcal{C}_{n} \); in other words,
  \[ \mathcal{D}_{n} = \{ x \in \mathcal{D}_{1} : (x, \kappa_{\mathcal{C}_{n}}(x)) \cap
    \mathcal{C}_{0} = \varnothing\}. \]
  Note that by construction, the distance between \( x \) and \( \kappa_{\mathcal{C}_{n}}(x) \) is
  less than \( 1 \) for each \( x \in \mathcal{D}_{n} \). Let \( \iota_n \) be the map
  \( \mathcal{C}_n\to \mathcal{D}_n \) which assigns to \( c \in \mathcal{C}_n \) the \( < \)-least
  element of \( \mathcal{D}_n \) that is greater than \( c \).

  \begin{figure}[hbt]
    \centering
    \begin{tikzpicture}
      \filldraw (0.5,0) circle (2pt) node[anchor=south,yshift=0.5mm](C11) {\( \mathcal{C}_{1} \)}
      node [anchor=north,yshift=-1mm](C21) {\( \mathcal{C}_{2} \)};
      \filldraw (1.3,0) circle (0.8pt);
      \filldraw (2.1,0) circle (0.8pt) node[anchor=south,yshift=0.5mm](D11) {\( \mathcal{D}_{1} \)};
      \filldraw (2.6,0) circle (1.5pt)  node[anchor=south,yshift=0.5mm](C12) {\( \mathcal{C}_{1} \)};
      \filldraw (3.1,0) circle (0.8pt);
      \filldraw (3.7,0) circle (0.8pt) node[anchor=south,yshift=0.5mm](D12) {\( \mathcal{D}_{1} \)};
      \filldraw (4.4,0) circle (1.5pt) node[anchor=south,yshift=0.5mm](C13) {\( \mathcal{C}_{1} \)};
      \filldraw (5.0,0) circle (0.8pt);
      \filldraw (5.7,0) circle (0.8pt) node[anchor=south,yshift=0.5mm](D13) {\( \mathcal{D}_{1} \)}
      node[anchor=north,yshift=-1mm](D21) {\( \mathcal{D}_{2} \)};
      \filldraw (6.3,0) circle (2pt) node[anchor=south,yshift=0.5mm](C14) {\( \mathcal{C}_{1} \)}
      node [anchor=north,yshift=-1mm](C22) {\( \mathcal{C}_{2} \)};
      \filldraw (6.9,0) circle (0.8pt);
      \filldraw (7.5,0) circle (0.8pt)  node[anchor=south,yshift=0.5mm](D14) {\( \mathcal{D}_{1} \)};
      \filldraw (8.2,0) circle (1.5pt)  node[anchor=south,yshift=0.5mm](C15) {\( \mathcal{C}_{1} \)};
      \filldraw (9,0) circle (0.8pt);
      \filldraw (9.6,0) circle (0.8pt)  node[anchor=north,yshift=-1mm](D22) {\( \mathcal{D}_{2} \)}
      node[anchor=south,yshift=0.5mm](D15) {\( \mathcal{D}_{1} \)};
      \filldraw (10.4,0) circle (2pt) node[anchor=south,yshift=0.5mm](C11) {\( \mathcal{C}_{1} \)}
      node [anchor=north,yshift=-1mm](C21) {\( \mathcal{C}_{2} \)};
      \filldraw (11.0,0) circle (0.8pt);
    \end{tikzpicture}
    \caption{An example of cross-sections \( \mathcal{C}_{0} \) (all points), \( \mathcal{C}_{1} \)
      (dots of size \raisebox{0.2mm}{\protect\tikz\protect\filldraw (0,0) circle (1.5pt);} and
      above), \( \mathcal{C}_{2} \) (marked as {\protect\tikz \protect\filldraw (0,0) circle
        (2.5pt);}) and \( \mathcal{D}_{1} \), \( \mathcal{D}_{2} \).}
    \label{fig:Cn-and-Dn}
  \end{figure}

  The bases and ceilings of \( \varphi_n \) and \( \psi_n \) are as follows:
  \begin{itemize}
  \item the basis of \( \varphi_n \) is \( A_n = \mathcal{C}_n+\left[0,\frac{1}{2^n}\right) \);
  \item the ceiling of \( \varphi_n \) is
    \( B_n = \mathcal{D}_n+\left[-\frac{1}{2}-\frac{1}{2^n},-\frac{1}{2}\right) \);
  \item the basis of \( \psi_n \) is
    \( E_n = \mathcal{D}_n+\left [-\frac{1}{2},-\frac{1}{2}+\frac{1}{2^n}\right) \);
  \item the ceiling of \( \psi_n \) is
    \( F_n = \mathcal{C}_n+\left [\frac{1}{2}, \frac{1}{2}+\frac{1}{2^n}\right) \).
  \end{itemize}
  Furthermore, we impose two \textbf{translation conditions}, which help us to preserve the above
  concrete definitions of the bases and ceilings at the inductive step when we construct
  \( \varphi_{n+1} \) and \( \psi_{n+1} \):
  \begin{itemize}
  \item \( \vec\varphi_n(c+t)=\iota_n(c)+t-\frac{1}{2}-\frac{1}{2^n} \) for all
    \( c\in\mathcal{C}_n \) and all \( t\in \left[0,\frac{1}{2^n}\right) \).
  \item \( \vec \psi_n(d+t)=\iota_n^{-1}(d)+t+1 \) for all \( d\in\mathcal{D}_n \) and all
    \( t\in \left[-\frac{1}{2},-\frac{1}{2}+\frac{1}{2^n}\right) \).
  \end{itemize}

  The first step of the construction consists of the castle \( \varphi_1 : x \mapsto x +1 \), which
  has basis \( A_1=\mathcal{C}_1+[0,\frac{1}{2}) \) and ceiling \( B_1=\mathcal{D}_1 +[-1,-\frac{1}{2}) \),
  and the castle \( \psi_1 : x \mapsto x -1 \) with basis \( E_{1} = \mathcal{D}_1+[-\frac{1}{2},0) \)
  and ceiling \( F_1=\mathcal{C}_1+[\frac{1}{2},1) \).

  We now concentrate on the induction step: suppose \( \varphi_n \) and \( \psi_n \) have been built
  for some \( n\geq 1 \); let us construct \( \varphi_{n+1} \) and \( \psi_{n+1} \).

  \medskip

  The strategy is to split the bases of \( \varphi_{n} \) and \( \psi_{n} \) into two equal
  intervals and ``interleave'' the ``two halves'' of \( \varphi_{n} \) with ``one half'' of
  \( \psi_{n} \) followed by ``gluing'' adjacent ceilings and bases within the same
  \( \mathcal{C}_{n+1} \) segment (see Figure~\ref{fig:inductive-step}). To this end, we introduce
  two intermediate castles \( \tilde \varphi_n \) and \( \tilde \psi_n \) that will ensure that
  \( \varphi_{n+1} \) ``wiggles'' more than \( \varphi_n \), yielding ergodicity of the final
  transformation.

  Define two new half-measure subsets of the bases \( A_n \) and \( E_n \) respectively:
  \begin{itemize}
  \item \( A_n^1 = \mathcal{C}_n+\left[0,\frac{1}{2^{n+1}}\right) \);
  \item
    \( E_n^0 = \mathcal{D}_n+ \left[-\frac{1}{2}+ \frac{1}{2^{n+1}}, -\frac{1}{2}+\frac{1}{2^n}\right) \);
  \end{itemize}
  and let
  \[ B_n^0 = \vec\varphi_n(A_n^1) =
    \mathcal{D}_n+\left[-\frac{1}{2}-\frac{1}{2^{n}},-\frac{1}{2}-\frac{1}{2^{n+1}}\right) ,\] and
  \[ F_n^0 = \vec\psi_n(E_n^0)= \mathcal{C}_n +
    \left[\frac{1}{2}+\frac{1}{2^{n+1}},\frac{1}{2}+\frac{1}{2^n}\right) ,\]
  where the two equalities are consequences of the translation conditions. Let \( G_n \) be the
  \( \psi_n \)-saturation of \( E_n^0 \), and note that the restriction of \( \psi_n \) to \( G_n \)
  is a castle with support \( G_n \), whose basis is \( E^0_n \) and whose ceiling is \( F_n^0 \).
  Finally, let
  \[ A_n^0 = A_n\setminus A_n^1=\mathcal{C}_n+\left[\frac{1}{2^{n+1}},\frac{1}{2^n}\right). \]

  We define the partial transformation
  \( \xi_n:B_n^0\sqcup F_n^0\to E_n^0\sqcup A_n^0 \) to be used for ``gluing together''
  \( \varphi_n \) and the restriction of \( \psi_n \) to \( G_n \):
  \begin{itemize}
        \item \( \xi_n(b)=b+\frac{3}{2^{n+1}}\in E_n^0 \) for all \( b\in B_n^0 \) and
        \item \( \xi_n(f)=f-\frac{1}{2}\in A_n^0 \) for all \( f\in F_n^0 \).
  \end{itemize}
  Set \( \tilde \varphi_n= \varphi_n \sqcup \xi_n \sqcup \psi_{n\restriction G_n} \), whereas
  \( \tilde \psi_n \) is simply the restriction of \( \psi_n \) onto the complement of \( G_n
  \). Observe that \( \tilde{\varphi}_n \) has basis \( A_n^1 \) and ceiling
  \[ B_n^1= B_n\setminus B_n^0=\mathcal{D}_n + \left[-\frac{1}{2}-\frac{1}{2^{n+1}},-\frac{1}{2}\right) ,\]
  while \( \tilde \psi_n \) has basis
  \[ E^1_n= E_n\setminus E_n^0=\mathcal{D}_n+\left[-\frac{1}{2},-\frac{1}{2}+\frac{1}{2^{n+1}}\right) \]
  and ceiling
  \[ F_n^1= F_n\setminus
    F_n^0=\mathcal{C}_n+\left[\frac{1}{2},\frac{1}{2}+\frac{1}{2^{n+1}}\right). \]
  We continue to have \( Y=\supp \tilde \varphi_n\sqcup \supp \tilde \psi_n \), but the support of
  \( \tilde{\psi}_{n} \) is half the support of \( \psi_{n} \), meaning that
  \( \mu(\supp \tilde \psi_n)=\frac{1}{2} \mu(\supp \psi_n) \).

  \begin{figure}[hbt]
 \centering
 \begin{tikzpicture}
   \draw (0,0) rectangle (0.55, 1.2);
   \draw (0.9,0) rectangle (1.45, 1.2);
   \draw (1.8,0) rectangle (2.35, 1.2);
   \draw (2.7,0) rectangle (3.25, 1.2);
   \draw (3.6,0) rectangle (4.15, 1.2);
   \draw (4.5,0) rectangle (5.05, 1.2);
   \draw (0.275, 1.3) node[anchor=south] {\( A_{n} \)};
   \draw (4.775, 1.3) node[anchor=south] {\( B_{n} \)};
   \draw (4.775, 0.9) node {\( B^{0}_{n} \)};
   \draw (4.775, 0.3) node {\( B^{1}_{n} \)};
   \draw (2.525, 1.3) node[anchor=south] {\( \longrightarrow \)};
   \draw (2.525, 1.5) node[anchor=south] {\( \varphi_{n} \)};
   \draw (0, -2) rectangle (0.55, -0.8);
   \draw (0.9,-2) rectangle (1.45, -0.8);
   \draw (1.8,-2) rectangle (2.35, -0.8);
   \draw (2.7,-2) rectangle (3.25, -0.8);
   \draw (0.275, -2.1) node[anchor=north] {\( F_{n} \)};
   \draw (1.625, -2.2) node[anchor=north] {\( \longleftarrow \)};
   \draw (1.625, -2.4) node[anchor=north] {\( \psi_{n} \)};
   \draw (2.975, -2.1) node[anchor=north] {\( F_{n} \)};
   \draw (7.1,0) rectangle (7.65, 1.2);
   \draw (8.0,0) rectangle (8.55, 1.2);
   \draw (8.9,0) rectangle (9.45, 1.2);
   \draw (9.8,0) rectangle (10.35, 1.2);
   \draw (10.7,0) rectangle (11.25, 1.2);
   \draw (7.375, 1.3) node[anchor=south] {\( A_{n} \)};
   \draw (10.975, 1.3) node[anchor=south] {\( B_{n} \)};
   \draw (9.175, 1.3) node[anchor=south] {\( \longrightarrow \)};
   \draw (9.175, 1.5) node[anchor=south] {\( \varphi_{n} \)};
   \draw (7.1, -2) rectangle (7.65, -0.8);
   \draw (8.0,-2) rectangle (8.55, -0.8);
   \draw (8.9,-2) rectangle (9.45, -0.8);
   \draw (9.8,-2) rectangle (10.35, -0.8);
   \draw (7.35, -2.1) node[anchor=north] {\( F_{n} \)};
   \draw (8.725, -2.2) node[anchor=north] {\( \longleftarrow \)};
   \draw (8.725, -2.4) node[anchor=north] {\( \psi_{n} \)};
   \draw (10.075, -2.1) node[anchor=north] {\( E_{n} \)};
   \draw[dashed] (-0.4, 0.6) -- (5.4, 0.6);
   \draw[dashed] (-0.4, -1.4) -- (3.7, -1.4);
   \draw[dashed] (6.7, 0.6) -- (11.7, 0.6);
   \draw[dashed] (6.7, -1.4) -- (10.7, -1.4);
   \draw[->, rounded corners=3mm] (5.05,0.9) -- (5.5, 0.9) -- (5.5, -1.1) -- (3.25,-1.1);
   \draw[->, rounded corners=3mm] (0,-1.1) -- (-0.4, -1.1) -- (-0.4, 0.3) -- (0,0.3);
   \draw[->, rounded corners=3mm] (11.25,0.9) -- (11.7, 0.9) -- (11.7, -1.1) -- (10.35,-1.1);
   \draw[->, rounded corners=3mm] (7.1,-1.1) -- (6.6, -1.1) -- (6.6, 0.3) -- (7.1,0.3);
   \draw[->] (7.1, -1.7) -- (3.25, -1.7);
   \draw[->, rounded corners=3mm] (5.05, 0.3) -- (6.05, 0.3) -- (6.05, 0.9) -- (7.1, 0.9);
   \draw (0.275,0.9) node {\( A^{1}_{n} \)};
   \draw (0.275,0.3) node {\( A^{0}_{n} \)};
   \draw (0.275,-1.1) node {\( F_{n}^{0} \)};
   \draw (0.275,-1.7) node {\( F_{n}^{1} \)};
   \draw (2.975,-1.1) node {\( E_{n}^{0} \)};
   \draw (2.975,-1.7) node {\( E_{n}^{1} \)};
   \draw (-0.1,-0.4) node {\( \xi_{n} \)};
   \draw (5.25,-0.4) node {\( \xi_{n} \)};
   \draw (7.375,0.9) node {\( A^{1}_{n} \)};
   \draw (7.375,0.3) node {\( A^{0}_{n} \)};
   \draw (10.975, 0.9) node {\( B^{0}_{n} \)};
   \draw (10.975, 0.3) node {\( B^{1}_{n} \)};
   \draw (7.375,-1.1) node {\( F_{n}^{0} \)};
   \draw (7.375,-1.7) node {\( F_{n}^{1} \)};
   \draw (10.075,-1.1) node {\( E_{n}^{0} \)};
   \draw (10.075,-1.7) node {\( E_{n}^{1} \)};
   \draw (6.9,-0.4) node {\( \xi_{n} \)};
   \draw (11.4,-0.4) node {\( \xi_{n} \)};
   \draw (6.05,1) node[anchor=south] {\( \xi'_{n} \)};
   \draw (5.025,-1.7) node[anchor=north] {\( \xi''_{n} \)};
 \end{tikzpicture}
 \caption{Inductive step.}
 \label{fig:inductive-step}
\end{figure}

  \medskip

  The ceiling of \( \tilde \varphi_n \) is equal to
  \( B_n^1=\mathcal{D}_n+\left[-\frac{1}{2}-\frac{1}{2^{n+1}},-\frac{1}{2}\right) \), whereas we
  need the ceiling of \( \varphi_{n+1} \) to be equal to
  \( B_{n+1}=\mathcal{D}_{n+1}+\left[-\frac{1}{2}-\frac{1}{2^{n+1}},-\frac{1}{2}\right) \). We
  obtain the required \( \varphi_{n+1} \) and \( \psi_{n+1} \) out of \( \tilde \varphi_n \) and
  \( \tilde \psi_n \), respectively, by ``passing through each element of
  \( \mathcal{C}_n\setminus\mathcal{C}_{n+1} \) ''.

  Note that \( \mathcal{D}_{n+1} \) is equal to the set of \( d\in \mathcal{D}_n \) such that
  \( \kappa_{\mathcal{C}_{n}}(d)\in \mathcal{C}_{n+1} \). Each \( x\in B_{n}^1\setminus B_{n+1} \)
  can be written uniquely as \( x=d+t \) where \( d\in\mathcal{D}_n\setminus \mathcal{D}_{n+1} \) and
  \( t\in \left[-\frac{1}{2}-\frac{1}{2^{n+1}},-\frac{1}{2}\right) \). Set
  \[
        \xi'_n(x)=\kappa_{\mathcal{C}_{n}}(d)+t+\frac{1}{2}+\frac{1}{2^{n+1}},
  \]
  and note that \( \xi'_n(x) \) belongs to
  \( (\mathcal{C}_n\setminus\mathcal{C}_{n+1})+\left[0,\frac{1}{2^{n+1}}\right)=A_{n}^1\setminus
  A_{n+1} \), hence \( \xi'_n \) is a measure-preserving bijection from \( B_n^1\setminus B_{n+1} \)
  onto \( A_n^1\setminus A_{n+1} \).

  The transformation \( \varphi_{n+1} \) is set to be \( \tilde \varphi_n\sqcup \xi'_n \), and we
  claim that it is a castle with basis \( A_{n+1} \) and ceiling \( B_{n+1} \). This amounts to
  showing that for all \( x\in A_{n+1} \), there is \( k\in\N \) such that \( \varphi_{n+1}^k(x) \)
  is not defined. Pick \( x\in A_{n+1} \) and write it as \( c_0+t \) for some
  \( c_0\in\mathcal{C}_{n+1} \) and \( t\in [0,\frac{1}{2^{n+1}}) \). Let \( c_1 \) be the successor
  of \( c_0 \) in \( \mathcal{C}_n \), which we suppose is not an element of
  \( \mathcal{C}_{n+1} \). By the construction of \( \tilde \varphi_n \) and \( \xi'_n \), there is
  \( k\in\N \) such that \( \xi'_n(\tilde \varphi_n^k(x))\in c'+[0,\frac{1}{2^{n+1}}) \), which
  means that \( \varphi_{n+1}^{k+1}(x)\in c'+[0,\frac{1}{2^{n+1}}) \). Iterating this argument, we
  eventually find \( k_{0}, p\in\N \) such that
  \( \varphi_{n+1}^{k_{0}}(x)\in c_p+[0,\frac{1}{2^{n+1}}) \) for some \( c_p\in\mathcal{C}_n \)
  such that the successor \( c_{p+1} \) of \( c_p \) in \( \mathcal{C}_n \) belongs to
  \( \mathcal{C}_{n+1} \). By the definition of \( \tilde \varphi_n \), we must have some
  \( l\in\N \) such that
  \( \varphi_{n+1}^{k_{0}+l}(x) = \tilde \varphi_n^l(\varphi_{n+1}^{k_{0}}(x))\in B_{n+1} \),
  whereas \( \varphi_{n+1}^{k_{0}+l+1}(x) \) is not defined. Thus, \( \varphi_{n+1} \) is indeed a
  castle.

  The extension \( \psi_{n+1} \) of \( \tilde{\psi}_{n} \) is defined similarly by connecting
  adjacent segments of \( F^{1}_{n} \) and \( E_{n}^{1} \) by a translation. More specifically, each
  \( x\in F^1_n\setminus F_{n+1} \) can be written uniquely as \( x=c+t \) for some
  \( c\in \mathcal{C}_n\setminus \mathcal{C}_{n+1} \) and
  \( t \in [\frac{1}{2},\frac{1}{2}+\frac{1}{2^{n+1}}) \). The restriction of
  \( \kappa_{\mathcal{C}_{n}} \) to \( \mathcal{D}_n \) is a bijection
  \( \mathcal{D}_n\to \mathcal{C}_n \). We denote its inverse by \( p_n \) and let
  \( \xi''_n(x)=p_n(c)+t-1 \). The map \( \psi_{n+1} =\tilde\psi_n\sqcup \xi''_n \) can be checked
  to be a castle with basis \( E_{n+1} \) and ceiling \( F_{n+1} \) as desired. It also follows that
  the translation conditions continue to be satisfied by both \( \varphi_{n+1} \) and
  \( \psi_{n+1} \).

  The transformations \( \varphi_{n} \) extend each other, so \( \varphi=\bigcup_n \varphi_n \) is
  an element of the pseudo full group supported on \( Y = \supp\varphi_{n} \sqcup \supp\psi_{n} \).
  Note also that \[ \mu(\supp\psi_{n+1}) = \mu(\supp\psi_{n})/2, \]
  and therefore \( \dom \varphi = Y = \rng\varphi \). We claim that \( \varphi \), seen as a
  measure-preserving transformation of \( Y \), induces an ergodic measure-preserving transformation
  on \( (y+\mathbb R)\cap Y \) for almost all \( y\in Y \), where \( y+\mathbb R \) is endowed with
  the Lebesgue measure. This follows from the fact that \( \varphi \) induces a
  \textbf{rank-one}\index{Transformation!rank-one}
  transformation of the infinite measure space \( (y+\mathbb R)\cap Y \): for all Borel
  \( A \subseteq (y+\mathbb R)\cap Y \) of finite Lebesgue measure and all \( \epsilon>0 \), there
  are \( B\subseteq(y+\mathbb R)\cap Y \), \( k\in\N \), and a subset
  \( F\subseteq \{0,\dots, k\} \) such that \( B,\varphi(B),\dots,\varphi^k(B) \) are pairwise
  disjoint and
  \[
        \lambda( A\bigtriangleup (\bigsqcup_{f\in F} \varphi^f(B)))<\epsilon.
  \]
  Indeed, at each step \( n \) for every \( c \in\mathcal{C}_n \), the iterates of
  \( c+[0,\frac{1}{2^n}) \) by the restriction of \( \varphi_n \) to the interval
  \( [c, \iota_n(c)) \) are disjoint ``intervals of size \( 2^{-n} \)'', i.e., sets of the form
  \( t+[0,\frac{1}{2^n}) \), and these iterates cover a proportion \( 1-\frac{1}{2^n} \) of
  \( [c, \iota_n(c)) \) (the rest of this interval being \( [c, \iota_n(c)) \cap \supp \psi_{n} \)).

  It remains to extend \( \varphi \) supported on \( Y \) to a measure-preserving transformation
  \( S \) with \( \supp S = X \). Let \( Z = X \setminus Y \) be the leftover set,
  \[ Z = \{ c + t : c \in \mathcal{C}_{1} : 2 \le t < \gap_{\mathcal{C}_{1}}(c)\} ,\]
  and put
  \[ Z'= \{ c + t : c \in \mathcal{C}_{1},\ 2- \gap_{\mathcal{C}_{1}}(c) \le t < 2 \}. \]
  Figure~\ref{fig:construction-of-S} illustrates an interval between \( c \in \mathcal{C}_{1} \) and
  \( c' = \sigma_{\mathcal{C}_{1}}(c) \). Within this gap, \( Z \) corresponds to
  \( [c+2, c+2+\gap_{\mathcal{C}_{1}}(c)) \), and \( Z' \) is an interval of the exact same length
  adjacent to it on the left. Note that \( Z' \subseteq Y \) by construction. Let
  \( \eta : Z' \to Z \) be the natural translation map,
  \( \eta(x) = x + \gap_{\mathcal{C}_{1}}(c) \) for all \( x \in Z' \) satisfying
  \( x \in c + [0, \gap_{\mathcal{C}_{1}}(c)) \). Observe that \( \eta \) is a measure-preserving
  bijection, and its cocycle is bounded by \( 1 \).

  \begin{figure}[htb]
    \centering
    \begin{tikzpicture}
        \draw (-1.0,0) -- (0,0);
        \draw (1.0,0) -- (5.5,0);
        \draw (0.5, 0) node {\( \cdots \)};
        \filldraw (-0.5,0) circle (1pt);
        \draw (-0.5,0) node[anchor=north] {\( c \)};
        \draw (3.55,0) circle (1pt);
        \filldraw (5, 0) circle (1pt);
        \draw (5,0) node[anchor=north] {\( c' \)};
        \draw [->,domain=80:15] plot ({2.65 + 0.9*cos(\x)}, {0.7*sin(\x) - 0});
        \draw [<-,domain=-165:-15] plot ({2.65 + 0.9*cos(\x)}, {0.6*sin(\x) - 0});
        \draw[dashed] [->,domain=165:15] plot ({4.00 + 0.4*cos(\x)}, {0.4*sin(\x) - 0});
        \draw[dashed] [<-,domain=-155:-5] plot ({3.10 + 1.3*cos(\x)}, {0.8*sin(\x) - 0});
        \draw [decorate,decoration={brace,amplitude=4pt},yshift=7mm] (4.15,0) -- (4.95, 0)
        node [black, midway, anchor=south, yshift=1.5mm] {\( Z \) };
        \draw [decorate,decoration={brace,amplitude=4pt},yshift=7mm] (3.25,0) -- (4.05, 0)
        node [black, midway, anchor=south, yshift=1.5mm] {\( Z' \) };
        \draw (4.1,-0.1) -- (4.1, 0.1);
        \draw (2.6, -0.35) node {\( \varphi \)};
        \draw (4.3, -0.65) node {\( S \)};
    \end{tikzpicture}
    \caption{Construction of the transformation \( S \).}
    \label{fig:construction-of-S}
  \end{figure}

  We now rewire the orbits of \( \varphi \) and define \( S : X \to X \) as follows (see
  Figure~\ref{fig:construction-of-S}):
  \begin{displaymath}
    S(x) =
    \begin{cases}
        \varphi(x) & \textrm{if \( x \not \in Z \cup Z' \)}; \\
        \eta(x) & \textrm{if \( x \in Z' \)}; \\
        \varphi(\eta^{-1}(x)) & \textrm{if \( x \in Z \)}.
    \end{cases}
  \end{displaymath}

  It is straightforward to verify that \( S \) is a free measure-preserving transformation, and the
  distance \( D(x, Sx) \le 4 \) for all \( x \in X \) because \( |\rho_{\varphi}(x)| \le 3 \) and
  \( |\rho_{\eta}(x)| \le 1 \) for all \( x \) in their domains. Note that for every \( y \in Y \),
  the intersection of the \( S \)-orbit with \( Y \) coincides with its \( \varphi \)-orbit. Since
  \( \varphi \) is ergodic on each orbit of the flow intersected with \( Y \), and considering that
  \( X = Y \sqcup Z \) and \( S^{-1}(Z) \subseteq Y \), it follows that \( S \) is ergodic on every
  orbit of the flow. Therefore, \( S \) satisfies the conclusion of the theorem.
\end{proof}

\begin{remark}
  The bound \( 4 \) in the formulation of Theorem~\ref{thm:example-conservative-ergodic-unbounded}
  is of no significance, as by rescaling the flow, it can be replaced with any \( \epsilon>0 \).
\end{remark}



\chapter[Conservative transformations]{Conservative and intermitted transformations}
\label{chap:intermitted-transformations}

Interesting dynamics of conservative transformations is present only in the non-discrete case, as it
reduces to periodicity for countable group actions. Chapter~\ref{chap:example-transf} provides an
illustrative construction of a conservative automorphism and shows that they exist in \( \LL^{1} \)
full groups of all free flows. The present chapter is devoted to the study of such elements. The
central role is played by the concept of an intermitted transformation, which is related to the
notion of induced transformation. Using this tool, we show that all conservative elements of
\( \lfgr{\mathbb{R} \acts X} \) can be approximated by periodic automorphisms, and hence belong to
the derived \( \LL^1 \) full group of \( \mathbb{R} \acts X \); see
Corollary~\ref{cor:conservative-belong-to-derived-full-group}.

Throughout the chapter, we fix a free measure-preserving flow \( \mathbb{R} \acts X \) on a standard
Lebesgue space \( (X, \mu) \). Given a cross-section \( \mathcal{C} \subset X \), recall that we
defined an equivalence relation \(\eqr_{\mathcal{C}}\) by declaring \( x \eqr_{\mathcal{C}} y \)
whenever there is \( c \in \mathcal{C} \) such that both \( x \) and \( y \) belong to the gap
between \( c \) and \( \sigma_{\mathcal{C}}(c) \). More formally, \( x \eqr_{\mathcal{C}} y \) if
there is \( c \in \mathcal{C} \) such that \( \rho(c, x) \ge 0 \), \( \rho(c, y) \ge 0 \) and
\( \rho(x, \sigma_{\mathcal{C}}(c)) > 0 \), \( \rho(y, \sigma_{\mathcal{C}}(c)) > 0 \).  Such an
equivalence relation is smooth.

Now let \( T \in \fgr{\mathbb{R} \acts X}\) be a conservative transformation. Under the action of
\( T \), almost every point returns to its \( \eqr_{\mathcal{C}} \)-class infinitely often, which
suggests the idea of the first return map.

\begin{definition}
  \label{def:intermitted-transformation}
  The \textbf{intermitted transformation}\index{Transformation!intermitted}
  \( T_{\eqr_{\mathcal{C}}} : X \to X \) is defined by
\[ T_{\eqr_{\mathcal{C}}}x = T^{n(x)}x,\quad \textrm{where
  } n(x) = \min\{n \ge 1 : x \eqr_{\mathcal{C}}T^{n(x)}x\}.\]
\end{definition}

The map \( T_{\eqr_{\mathcal{C}}} \) is well-defined, since \( T \) is conservative, and it preserves
the measure \( \mu \), since \( T_{\eqr_{\mathcal{C}}} \) belongs to the full group of \( T \).

\begin{remark}
  \label{rem:intermitted-transformation-for-equivalence-relations}
  The concept of an intermitted transformation \( T_{E} \) makes sense for any equivalence relation
  \( E \) for which the intersection of any orbit of \( T \) with any \( E \)-class is either empty
  or infinite. In particular, intermitted transformations can be considered for any conservative
  \( T \in \fgr{ G \acts X } \) in a full group of a locally compact group action. For instance,
  with a cocompact cross-section \( \mathcal{C} \) we can associate an equivalence relation of lying
  in the same cell of the Voronoi tessellation (see Appendix~\ref{sec:tessellations}). Such an
  equivalence relation does have the aforementioned transversal property, and hence the intermitted
  transformation is well-defined.

  Note also the following connection with the more familiar construction of the induced
  transformation. Let \( T \in \Aut(X, \mu) \), and let \( A \subseteq X \) be a set of positive
  measure.  Recall that the induced transformation \(T_{A} \in \Aut(X,\mu)\) is supported on the set
  \(A\) and is defined for \(x \in A\) by \(T_{A}x = T^{n(x)}x\), where
  \(n(x) = \min\{n \ge 1 : T^{n}x \in A\}\).  Define \( \mathcal{A} \) to be the equivalence
  relation with two classes: \( A \) and \( X \setminus A \). The induced transformations
  \( T_{A} \) and \( T_{X \setminus A} \) commute and satisfy
  \( T_{A} \circ T_{X \setminus A} = T_{\mathcal{A}} \).
\end{remark}

The next lemma forms the core of this chapter. It shows that the operation of taking an intermitted
transformation does not increase the norm. As we discuss later in
Remark~\ref{rem:counterexample-R2}, the analog of this statement is false even for
\( \mathbb{R}^{2} \)-flows, which perhaps justifies the technical nature of the argument.

\begin{lemma}
  \label{lem:TEC-estimate}
  Let \( T \in \lfgr{\mathbb{R} \acts X} \) be a conservative automorphism, and let
  \( \mathcal{C} \) be a cross-section. Let also \( Y \) be the set of points where \( T \) and
  \( T_{\eqr_{\mathcal{C}}} \) differ: \( Y = \{x \in X : Tx \ne T_{\eqr_{\mathcal{C}}}x\} \). One
  has \( \int_{Y}|\rho_{T_{\eqr_{\mathcal{C}}}}|\, d\mu \le \int_{Y} |\rho_{T}|\, d\mu \).
\end{lemma}

\begin{proof}
  By the definition of \( Y \), for any \( x \in Y \), the arc from \( x \) to \( Tx \) crosses at
  least one point of \( \mathcal{C} \). We may therefore represent \( |\rho_{T}(x)| \) as the sum of
  the distance from \( x \) to the first point of \( \mathcal{C} \) along the arc plus the rest of
  the arc. More formally, for \( x \in X \), let \( \pi_{\mathcal{C}}(x) \) be the unique
  \( c \in \mathcal{C} \) such that \( x \in c + [0, \gap_{\mathcal{C}}(c)) \). Define
  \( \alpha : Y \to \mathbb{R}^{\ge 0} \) by
  \begin{displaymath}
    \alpha(x) =
    \begin{cases}
      |\rho(x, \sigma_{\mathcal{C}}(\pi_{\mathcal{C}}(x)))|, & \textrm{if \( \rho(x, Tx) > 0 \)}, \\
      |\rho(x, \pi_{\mathcal{C}}(x))| & \textrm{if \( \rho(x, Tx) < 0\)}.
    \end{cases}
  \end{displaymath}
  Note that \( \alpha(x) \le |\rho_{T}(x)| \), and set \( \beta(x) = |\rho_{T}(x)| - \alpha(x) \), so that
  \[ \int_{Y}|\rho_{T}|\, d\mu = \int_{Y}\alpha\, d\mu + \int_{Y} \beta\, d\mu. \]
  For instance, in the context of Figure~\ref{fig:conservative-dynamics},
  \( \alpha(x_{4}) = \rho(x_{4}, c_{2}) \) and \( \beta(x_{4}) = \rho(c_{2}, x_{5}) \). Let us
  partition \( Y = Y' \sqcup Y'' \), where
  \[ Y' = \bigl\{x \in Y : \rho(x, Tx) \textrm{ and } \rho(x, T_{\eqr_{\mathcal{C}}}x) \textrm{ have
      the same sign or \( T_{\eqr_{\mathcal{C}}}x = x \)}\,\bigr\}, \]
  and \( Y'' = Y \setminus Y' \) consists of those \( x \in Y \) for which the signs of
  \( \rho(x, Tx) \) and \( \rho(x, T_{\eqr_{\mathcal{C}}}x) \) are different. For example, referring to
  the same figure, \( x_{0} \in Y'' \), while \( x_{2} \in Y' \).

  To prove the lemma, it is enough to show two inequalities:
  \begin{equation}
    \label{eq:alpha-estimate}
   \int_{Y'}|\rho_{T_{\eqr_{\mathcal{C}}}}(x)|\, d\mu(x) \le \int_{Y}\alpha(x)\, d\mu(x),
  \end{equation}
  \begin{equation}
    \label{eq:beta-estimate}
    \int_{Y''}|\rho_{T_{\eqr_{\mathcal{C}}}}(x)|\, d\mu(x) \le \int_{Y}\beta(x)\, d\mu(x).
  \end{equation}

  Eq.~\eqref{eq:alpha-estimate} is straightforward, since the equality of signs of \( \rho(x, Tx) \)
  and \( \rho(x, T_{\eqr_{\mathcal{C}}}x) \) implies that \( T_{\eqr_{\mathcal{C}}}x \) is closer
  than \( x \) to the point \( c \in \mathcal{C} \), which is crossed by the arc from \( x \) to
  \( Tx \). For example, the point \( x_{2} \) in Figure~\ref{fig:conservative-dynamics} satisfies
  \[ |\rho_{T_{\eqr_{\mathcal{C}}}}(x_{2})| = \rho(x_{2},x_{4}) \le \rho(x_{2},c_{2}) =
    \alpha(x_{2}).\]
  Thus \( |\rho_{T_{\eqr_{\mathcal{C}}}}(x)| \le \alpha(x) \) for all \( x \in Y' \), and so
  \[ \int_{Y'} |\rho_{T_{\eqr_{\mathcal{C}}}}|\, d\mu \le \int_{Y'} \alpha\, d\mu \le \int_{Y}
    \alpha\, d\mu, \]
  which establishes~\eqref{eq:alpha-estimate}. The other inequality will require a bit more work.

  For \( x \in Y'' \), let \( N(x) \ge 1 \) be the smallest integer such that the sign of
  \( \rho(x, T^{N(x)+1}x) \) is opposite to that of \( \rho_{T}(x) \). In less formal terms,
  \( N(x) \) is the smallest integer such that the arc from \( T^{N(x)}x \) to \( T^{N(x)+1}x \)
  jumps over \( x \). In particular, points \( T^{k}x \), \( 1 \le k \le N(x) \), are all on the
  same side relative to \( x \), while \( T^{N(x)+1}x \) is on the other side of it. We consider the
  map \( \eta : Y'' \to X \) given by \( \eta(x) = T^{N(x)}x \). The properties of this map will be
  crucial for establishing the inequality~\eqref{eq:beta-estimate}, so let us provide some
  explanations first.

  \begin{figure}[htb]
    \centering
    \begin{tikzpicture}
      \filldraw (-0.4,0.1) -- (-0.4, -0.1) node[anchor=north] {\( c_{0} \)};
      \filldraw (1.8,0.1) -- (1.8, -0.1) node[anchor=north] {\( c_{1} \)};
      \filldraw (4.5,0.1) -- (4.5, -0.1) node[anchor=north] {\( c_{2} \)};
      \filldraw (7.1,0.1) -- (7.1, -0.1) node[anchor=north] {\( c_{3} \)};
      \filldraw (8.7,0.1) -- (8.7, -0.1) node[anchor=north] {\( c_{4} \)};

      \filldraw (1,0) circle (1pt) node[anchor=north] {\( x_{0} \)};
      \filldraw (2.3,0) circle (1pt) node[anchor=north] {\( x_{1} \)};
      \filldraw (3.3,0) circle (1pt) node[anchor=north] {\( x_{2} \)};
      \filldraw (5.2,0) circle (1pt) node[anchor=north] {\( x_{3} \)};
      \filldraw (4.0,0) circle (1pt) node[anchor=north] {\( x_{4} \)};
      \filldraw (7.8,0) circle (1pt) node[anchor=north] {\( x_{5} \)};
      \filldraw (6.2,0) circle (1pt) node[anchor=north] {\( x_{6} \)};
      \filldraw (9.6,0) circle (1pt) node[anchor=north] {\( x_{7} \)};
      \filldraw (-1.5,0) circle (1pt) node[anchor=north] {\( x_{8} \)};
      \filldraw (0.3,0) circle (1pt) node[anchor=north] {\( x_{9} \)};
      \draw[->,domain=165:15] plot ({1.65 + 0.65*cos(\x)}, {0.4*sin(\x) - 0});
      \draw[->,domain=165:15] plot ({2.8 + 0.5*cos(\x)}, {0.4*sin(\x) - 0});
      \draw[->,domain=165:10] plot ({4.25 + 0.95*cos(\x)}, {0.6*sin(\x) - 0});
      \draw[->,domain=-15:-165] plot[yshift=-3mm] ({4.6 + 0.55*cos(\x)}, {0.4*sin(\x) - 0});
      \draw[->,domain=165:10] plot ({5.88 + 1.9*cos(\x)}, {0.7*sin(\x) - 0});
      \draw[->,domain=-15:-170] plot[yshift=-3mm] ({7.0 + 0.75*cos(\x)}, {0.5*sin(\x) - 0});
      \draw[->,domain=165:10] plot ({7.9 + 1.7*cos(\x)}, {0.6*sin(\x) - 0});
      \draw[->,domain=-8:-175] plot[yshift=-3mm] ({4.05 + 5.45*cos(\x)}, {1.1*sin(\x) - 0});
      \draw[->,domain=165:15] plot ({-0.6 + 0.9*cos(\x)}, {0.4*sin(\x) - 0});
    \end{tikzpicture}
    \caption{Dynamics of a conservative orbit.}
    \label{fig:conservative-dynamics}
  \end{figure}

  Consider once again Figure~\ref{fig:conservative-dynamics}, which shows a partial orbit of a point
  \( x_{0} \) for \( x_{i} = T^{i}x_{0} \) up to \( i \le 9 \) and several points
  \( c_{i} \in \mathcal{C} \). First, as we have already noted before, \( x_{0} \in Y \), since
  \( \neg x_{0} \eqr_{\mathcal{C}} x_{1} \); moreover, \( x_{0} \in Y'' \), since
  \( x_{9} = T_{\eqr_{\mathcal{C}}}x_{0} \) is to the left of \( x_{0} \), while \( x_{1} \) is to
  the right of it, so \( \rho(x_{0}, x_{1}) \) and \( \rho(x_{0}, x_{9}) \) have opposite signs.
  Also, \( N(x_{0}) = 7 \), because \( x_{8} \) is the first point in the orbit to the left of
  \( x_{0} \), thus \( \eta(x_{0}) = x_{7} \). In general, we have
  \( T^{N(x)+1}x \ne T_{\eqr_{\mathcal{C}}}x \). However, the equality
  \( T^{N(x)+1}x = T_{\eqr_{\mathcal{C}}}x \) holds when \( x \in Y'' \) and the points
  \( T^{N(x)+1}x \) and \( x \) are \( \eqr_{\mathcal{C}} \)-equivalent.

  The next point in the orbit \( x_{1} \) is not in \( Y \), whereas \( x_{2} \in Y \) but
  \( x_{2} \not \in Y'' \), because \( T_{\eqr_{\mathcal{C}}}x_{2} = x_{4} \) and both
  \( \rho(x_{2}, x_{3}) \) and \( \rho(x_{2}, x_{4}) \) are positive. The point \( x_{3} \) belongs
  to \( Y'' \) and has \( N(x_{3}) = 1 \) with \( \eta(x_{3}) = x_{4} \). The points
  \( x_{4}, x_{5}, x_{6} \) are in \( Y \), but whether any of them are elements of \( Y'' \) is not clear
  from Figure~\ref{fig:conservative-dynamics}, as the orbit segment is too short to clarify the
  values of \( T_{\eqr_{\mathcal{C}}}x_{i} \), \( i = 4,5,6 \). However, if \( x_{4}, x_{5}, x_{6} \)
  happen to lie in \( Y'' \), then \( N(x_{5}) = 1 \) with \( \eta(x_{5}) = x_{6} \), and
  \( N(x_{4}) = 3 \), \( N(x_{6}) = 1 \), \( \eta(x_{4}) = \eta(x_{6}) = x_{7} = \eta(x_{0}) \). In
  particular, the function \( x \mapsto \eta(x) \) is not necessarily one-to-one, but we are going
  to argue that it is always finite-to-one.

 \begin{claim}\textit{If \( x, y \in Y'' \) are distinct points such that
    \( \eta(x) = \eta(y) \), then \( \neg x \eqr_{\mathcal{C}} y \).}
 \end{claim}
 \begin{cproof}
        Suppose \( x, y \in Y'' \)
        satisfy \( \eta(x) = \eta(y) \).
        The definition of \( \eta \) implies that \( x \) and \( y \)
  must belong to the same orbit of \( T \), and we may assume without loss of generality that
  \( y = T^{k_{0}}x \) for some \( k_{0} \ge 1 \). If the orbit of \( x \) and \( y \) is aperiodic,
  it implies that \( N(x) > k_{0} \) and \( N(y) + k_{0} = N(x) \), \( N(y) \ge 1 \). However,
  even if the orbit is periodic, either \( N(y) + k_{0} = N(x) \) for the smallest positive integer
  \( k_{0} \) such that \( y = T^{k_{0}}x \) or \( N(x) + k_{0}' = N(y) \) for the smallest positive
  integer \( k_{0}' \) such that \( x = T^{k_{0}'}y \). Interchanging the roles of \( x \) and
  \( y \) if necessary, we may therefore assume that \( N(y) + k_{0} = N(x) \) holds for some
  \( k_{0} \ge 1 \), \( T^{k_{0}}x = y \), regardless of the type of orbit we consider.

  Suppose \( x \) and \( y \) are \( \eqr_{\mathcal{C}} \)-equivalent. Let \( k \ge 1 \) be the
  smallest natural number for which \( x \) and \( T^{k}x \) are \( \eqr_{\mathcal{C}} \)-equivalent.
  By the assumption \( x \eqr_{\mathcal{C}} y \) and the choice of \( k_{0} \) we have
  \( k \le k_{0} < N(x) \). By the definition of \( N(x) \), all points \( T^{i}x \),
  \( 1 \le i \le N(x) \), are on the same side of \( x \). In particular, this applies to \( Tx \)
  and \( T^{k}x \), which shows that \( \rho(x, Tx) \) and \( \rho(x, T_{\eqr_{\mathcal{C}}}x) \) have
  the same sign, thus \( x \not \in Y'' \).
  \end{cproof}

  The above claim implies that the function \( x \mapsto \eta(x) \) is finite-to-one, for the arc
  from \( \eta(x) \) to \( T\eta(x) \) intersects only finitely many
  \( \eqr_{\mathcal{C}} \)-equivalence classes, and the preimage of \( \eta(x) \) picks at most one
  point from each such class. Note also that \( \eta(x) \in Y \) for all \( x \in Y'' \), but
  \( \eta(x) \) may not be an element of \( Y'' \). Among the \( \eqr_{\mathcal{C}} \)-equivalence
  classes that the arc from \( \eta(x) \) to \( T\eta(x) \) crosses, two are special: the intervals
  that contain \( T\eta(x) \) and \( \eta(x) \), respectively. Our goal will be to bound the sum of
  \( |\rho_{T_{\eqr_{\mathcal{C}}}}(x)| \) over the points \( x \) with the same \( \eta(x) \) value
  by \( \beta(\eta(x)) \) (see Claim 3 below). For a typical point \( x \), we can bound
  \( |\rho_{T_{\eqr_{\mathcal{C}}}}(x)| \) simply by the length of the interval of its
  \( \eqr_{\mathcal{C}} \)-class. For example, Figure~\ref{fig:conservative-dynamics} does not
  specify \( T_{\eqr_{\mathcal{C}}}x_{4} \), but we can be sure that
  \( |\rho_{T_{\eqr_{\mathcal{C}}}}(x_{4})| \le \rho(c_{1},c_{2}) \). In view of Claim 1, such an
  estimate comes close to showing that the sum of \( |\rho_{T_{\eqr_{\mathcal{C}}}}(x)| \) over
  \( x \) with the same image \( \eta(x) \) is bounded by \( |\rho(\eta(x), T\eta(x))| \). It merely
  comes close due to the two special \( \eqr_{\mathcal{C}} \)-classes mentioned above, where our
  estimate needs to be improved. The next claim shows that one of these special cases is of no
  concern as \( x \) is never \( \eqr_{\mathcal{C}} \)-equivalent to \( \eta(x) \).

  \begin{claim} \textit{For all \( x \in Y'' \), we have \( \neg x \eqr_{\mathcal{C}} \eta(x) \).}
  \end{claim}
  \begin{cproof}
  Suppose, towards a contradiction, that
  \( x \eqr_{\mathcal{C}} \eta(x) \), and let \( k \ge 1 \) be the smallest integer for which
  \( x \eqr_{\mathcal{C}} T^{k}(x) \); in particular, \( T_{\eqr_{\mathcal{C}}}x = T^{k}x \). Note that
  \( k \le N(x) \) by the assumption, and by the definition of \( N(x) \), \( \rho(T^{k}x, x) \) has
  the same sign as \( \rho_{T}(x) \), whence \( x \not \in Y'' \).
  \end{cproof}

  Pick some \( y \in Y \) with non-empty preimage \( \eta^{-1}(y) \), and let
  \( z_{1}, \ldots, z_{n} \in Y'' \) be all the elements in \( \eta^{-1}(y) \). For instance, in the
  situation depicted in Figure~\ref{fig:conservative-dynamics}, we may have \( n = 3 \) and
  \( z_{1} = x_{0} \), \( z_{2} = x_{4} \), \( z_{3} = x_{6} \), and \( y = x_{7} \). The following
  claim unlocks the path toward the inequality~\eqref{eq:beta-estimate}.

  \begin{claim} \textit{In the above notation,
    \( \sum_{i=1}^{n}|\rho_{T_{\eqr_{\mathcal{C}}}}(z_{i})| \le \beta(y) \).}
  \end{claim}
  \begin{cproof}
    Recall that the arc from \( y \) to \( Ty \) crosses at least one point in \( \mathcal{C} \). If
    \( c \in \mathcal{C} \) is the closest to \(y\) among such points, then \( \beta(y) \) is
    defined to be \( |\rho(c, Ty)| \). For instance, in the notation of
    Figure~\ref{fig:conservative-dynamics}, \( \beta(x_{7}) = |\rho(c_{4},x_{8})| \).  Each point
    \( z_{i} \) is located under the arc from \( y \) to \( Ty \), and by Claim 2, no point
    \( z_{i} \) belongs to the interval from \( c \) to \( y \). In the language of our concrete
    example, no point \( z_{i} \) can be between \( c_{4} \) and \( x_{7} \). As discussed before,
    \( |\rho_{T_{\eqr_{\mathcal{C}}}}(x)| \) is always bounded by the length of the gap to which
    \( x \) belongs. This is sufficient to prove the claim if no \( z_{i} \) is equivalent to
    \( Ty \), as in this case the whole \( \eqr_{\mathcal{C}} \)-equivalence class of every
    \( z_{i} \) is fully contained under the interval between \( c \) and \( Ty \), and distinct
    \( z_{i} \) represent distinct \( \eqr_{\mathcal{C}} \)-classes by Claim 1. This is the
    situation depicted in Figure~\ref{fig:conservative-dynamics}, and our argument boils down to the
    inequalities
  \[
    \begin{aligned}
    |\rho_{T_{\eqr_{\mathcal{C}}}}(x_{0})| + |\rho_{T_{\eqr_{\mathcal{C}}}}(x_{4})| +
    |\rho_{T_{\eqr_{\mathcal{C}}}}(x_{5})| &\le |\rho(c_{0}, c_{1})| + |\rho(c_{1}, c_{2})| +
                                             |\rho(c_{2}, c_{3})| \\
      &\le |\rho(c_{0}, c_{4})| \le \beta(x_{7}).
    \end{aligned}
  \]

  Suppose there is some \( z_{i} \) such that \( z_{i} \eqr_{\mathcal{C}} Ty \). By Claim 1, such
  \( z_{i} \) must be unique, and we assume without loss of generality that
  \( z_{1} \eqr_{\mathcal{C}}Ty \). For example, this situation would occur if in
  Figure~\ref{fig:conservative-dynamics} \( Tx_{7} \) were equal to \( x_{9} \). Let \( c' \) be the
  first element of \( \mathcal{C} \) over which goes the arc from \( z_{1} \) to \( Tz_{1} \) (it
  would be the point \( c_{1} \) in Figure~\ref{fig:conservative-dynamics}). It is enough to show
  that \( |\rho_{T_{\eqr_{\mathcal{C}}}}(z_{1})| \le |\rho(T_{\eqr_{\mathcal{C}}}z_{1}, c')| \), as
  we can use the previous estimate for all other \( |\rho_{T_{\eqr_{\mathcal{C}}}}(z_{i})| \),
  \( i \ge 2 \).  Note that \( T_{\eqr_{\mathcal{C}}}z_{1} = Ty \), and \( z_{1} \in Y'' \) by
  assumption, which implies that the signs of \( \rho(z_{1}, T_{\eqr_{\mathcal{C}}}z_{1}) \) and
  \( \rho(z_{1}, c') \) are different. The latter is equivalent to saying that \( z_{1} \) is
  between \( T_{\eqr_{\mathcal{C}}}z_{1} \) and \( c' \), i.e.,
  \( |\rho(T_{\eqr_{\mathcal{C}}}z_{1}, c')| = |\rho_{T_{\eqr_{\mathcal{C}}}}(z_{1})| + |\rho(z_{1},
  c')| \), and the claim follows.
  \end{cproof}

  We are now ready to finish the proof of this lemma. We have already shown that \( \eta \) is
  finite-to-one, so let \( Y_{n}'' \subseteq Y'' \), \( n \ge 1 \), be such that
  \( x \mapsto \eta(x) \) is \( n \)-to-one on \( Y_{n}'' \). Let \( R_{n} = \eta(Y_{n}'') \), and
  recall that \( R_{n} \subseteq Y \). The sets \( R_{n} \) are pairwise disjoint. Let
  \( \phi_{k,n} : R_{n} \to Y_{n}'' \), \( 1 \le k \le n\), be Borel bijections that pick the
  \( k \)th point in the preimage: \( Y_{n}'' = \bigsqcup_{i=1}^{n}\phi_{k,n}(R_{n}) \). Note that
  the maps \( \phi_{k,n} : R_{n} \to \phi_{k,n}(R_{n})\) are measure-preserving, since they belong
  to the pseudo full group of \( T \), and
  \( \sum_{k=1}^{n}|\rho_{T_{\eqr_{\mathcal{C}}}}(\phi_{k,n}(x))| \le \beta(x) \) for all
  \( x \in R_{n} \) by Claim 3. One now has

  \begin{displaymath}
    \begin{aligned}
      \int_{Y_{n}''}|\rho_{T_{\eqr_{\mathcal{C}}}}(x)|\, d\mu(x)
      &= \sum_{k=1}^{n}\int_{\phi_{k,n}(R_{n})}|\rho_{T_{\eqr_{\mathcal{C}}}}(x)|\, d\mu(x) \\
      \because \textrm{\( \phi_{n,k} \) are measure-preserving}
      &= \int_{R_{n}}\sum_{k=1}^{n}|\rho_{T_{\eqr_{\mathcal{C}}}}(\phi_{k,n}^{-1}(x))|\, d\mu(x) \\
      \because \textrm{Claim 3} &\le \int_{R_{n}} \beta(x) \, d\mu(x). \\
    \end{aligned}
  \end{displaymath}

  Summing these inequalities over \( n \), we get
  \begin{displaymath}
    \begin{aligned}
      \int_{Y''}|\rho_{T_{\eqr_{\mathcal{C}}}}(x)|\, d\mu(x)
      &= \sum_{n=1}^{\infty}\int_{Y_{n}''}
        |\rho_{T_{\eqr_{\mathcal{C}}}}(x)|\, d\mu(x) \\
      &\le \sum_{n=1}^{\infty} \int_{R_{n}}\beta(x)\, d\mu(x)
        \le \int_{Y}\beta(x)\, d\mu(x),
    \end{aligned}
  \end{displaymath}
  where the last inequality is based on the fact that the sets \( R_{n} \) are pairwise disjoint.
  This finishes the proof of the inequality~\eqref{eq:beta-estimate} as well as the lemma.
\end{proof}

Several important facts follow easily from Lemma~\ref{lem:TEC-estimate}.  For one, it implies that
for any cross-section \( \mathcal{C} \), the intermitted transformation \( T_{\eqr_{\mathcal{C}}} \)
belongs to \( \lfgr{\mathbb{R} \acts X} \).  In fact, we have the following inequality on the norms.

\begin{corollary}
  \label{cor:norm-intermitted-estimate}
  For any intermitted transformation \( T_{\eqr_{\mathcal{C}}} \), one has
  \( \snorm{T_{\eqr_{\mathcal{C}}}}_{1} \le \norm{T}_{1} \).
\end{corollary}

\begin{proof}
  By the definition of the set \( Y \) in Lemma~\ref{lem:TEC-estimate},
  \( \rho_{T_{\eqr_{\mathcal{C}}}}(x) = \rho_{T}(x) \) for all \( x \not \in Y \), hence
  \begin{displaymath}
    \begin{aligned}
      \int_{X}|\rho_{T_{\eqr_{\mathcal{C}}}}|\,d\mu
      &= \int_{X \setminus Y} |\rho_{T_{\eqr_{\mathcal{C}}}}|\, d\mu
        + \int_{Y} |\rho_{T_{\eqr_{\mathcal{C}}}}|\, d\mu \\
      \because \textrm{Lemma~\ref{lem:TEC-estimate}}
      &\le \int_{X \setminus Y}|\rho_{T}|\, d\mu + \int_{Y} |\rho_{T}|\, d\mu = \int_{X} |\rho_{T}|\, d\mu,
    \end{aligned}
  \end{displaymath}
  which shows \( \snorm{T_{\eqr_{\mathcal{C}}}}_{1} \le \snorm{T}_{1} \).
\end{proof}

\begin{remark}
  \label{rem:counterexample-R2}
  As we discussed in Remark~\ref{rem:intermitted-transformation-for-equivalence-relations}, the
  concept of an intermitted transformation applies more broadly than the case of one-dimensional
  flows. We mention, however, that the analog of Lemma~\ref{lem:TEC-estimate} and
  Corollary~\ref{cor:norm-intermitted-estimate} does not hold even for free measure-preserving
  \( \mathbb{R}^{2} \)-flows. Consider an annulus depicted in Figure~\ref{fig:T-circle} and let
  \( T \) be the rotation by an angle \( \alpha \) around the center of this annulus. Let the
  equivalence relation \( E \) consist of two classes, each composing half of the ring. For a point
  \( x \) such that \( \neg x E Tx \), \( T_{E}x \) will be close to the other side of the class. It
  is easy to arrange the parameters (the angle \( \alpha \) and the radii of the annulus) so that
  \( \snorm{\rho_{T_{E}}(x)} > \snorm{\rho_{T}(x)} \) for all \( x \) such that \( Tx \ne T_{E}x \).

  \def\AngleValue{18.7}
  \begin{figure}[htb]
    \centering
    \begin{subfigure}{4.9cm}
      \begin{tikzpicture}
        \draw (0,0) circle (1cm);
        \draw (0,0) circle (1.5cm);
        \draw (0, 1.0) -- (0, 1.5);
        \draw (0, -1.0) -- (0, -1.5);
        \foreach \x in {0,...,15} {
          \fill ({40+\x*\AngleValue}:1.25) circle (1pt);
        }
        \draw (1.25,0) node {\( \vdots \)};
        \draw (0,0) -- (40:1.25);
        \draw (0,0) -- ({40+\AngleValue}:1.25);
        \draw (49.35:0.8) node {\( \alpha \)};
        \draw[->,domain=160:210] plot ({2*cos(\x)}, {2*sin(\x) - 0});
        \draw (-2.25, 0) node {\( T \)};
        \draw ({40+2*\AngleValue}:1.86) node {\( x \)};
        \draw[->] ({40+2*\AngleValue}:1.7) -- ({40+2*\AngleValue}:1.3);
        \draw ({40+13*\AngleValue}:1.25) node[yshift=-7mm, xshift=2mm] {\( T_{E}x \)};
        \draw[->] ({40+13*\AngleValue}:1.7) -- ({40+13*\AngleValue}:1.3);
      \end{tikzpicture}
      \caption{}
      \label{fig:T-circle}
    \end{subfigure}
    \begin{subfigure}{4cm}
      \begin{tikzpicture}
        \draw (0,0) circle (0.7cm);
        \draw (0,0) circle (1.1cm);
        \draw (-2,-1.5) rectangle (2,1.5);
        \draw[dashed] (0,-1.5) -- (0,1.5);
        \draw[->,domain=160:210] plot ({1.5*cos(\x)}, {1.5*sin(\x) - 0});
      \end{tikzpicture}
      \caption{}
      \label{fig:T-circle-tile}
    \end{subfigure}
    \caption{Construction of a conservative transformation \( T \) with
      \( \snorm{T_{E}}_{1} > \snorm{T}_{1} \).}
    \label{fig:counterexample-intermitted-estimate-R2}
  \end{figure}

  Every free measure-preserving flow \( \mathbb{R}^{2} \acts X \) admits a tiling of its orbits by
  rectangles. The transformation \( T \in \lfgr{\mathbb{R}^{2} \acts X} \) can be defined similarly
  to Figure~\ref{fig:T-circle} on each rectangle of the tiling by splitting each tile into two
  equivalence classes as in Figure~\ref{fig:T-circle-tile}. The resulting transformation \( T \)
  will have bounded orbits and satisfy \( \snorm{T_{E}}_{1} > \snorm{T}_{1} \) relative to the
  equivalence relation \( E \) whose classes are the half tiles.
\end{remark}

When the gaps in a cross-section \( \mathcal{C} \) are large, \( x \) and \( Tx \) will often be
\( \eqr_{\mathcal{C}} \)-equivalent, and it is therefore natural to expect that
\( T_{\eqr_{\mathcal{C}}} \) will be close to \( T \). This intuition is indeed valid, and the
following approximation result is the most important consequence of Lemma~\ref{lem:TEC-estimate}.

\begin{lemma}
  \label{lem:approximation-by-intermitted}
  Let \( T \in \lfgr{\mathbb{R} \acts X} \) be a conservative transformation. For any
  \( \epsilon > 0 \), there exists \( M \) such that for any cross-section \( \mathcal{C} \) with
  \( \gap_{\mathcal{C}}(c) \ge M \) for all \( c \in \mathcal{C} \), one has
  \( \snorm{T \circ T^{-1}_{\eqr_{\mathcal{C}}}}_{1} < \epsilon \).
\end{lemma}

\begin{proof}
  Let \( A_{K} = \{ x \in X: |\rho_{T}(x)| \ge K \} \), \( K \in \mathbb{R}^{\ge 0} \), be the set
  of points whose cocycle is at least \( K \) in absolute value. Since
  \( T \in \lfgr{\mathbb{R} \acts X} \), we may pick \( K \ge 1 \) so large that
  \( \int_{A_{K}}|\rho_{T}|\, d\mu < \epsilon/4 \). Pick any real \( M \) such that
  \( 2K^{2}/M < \epsilon/4 \). We claim that it satisfies the conclusion of the lemma. To verify
  this, we pick a cross-section \( \mathcal{C} \) with all gaps having a size of at least \( M
  \). Set as before \( Y = \{ x \in X : Tx \ne T_{\eqr_{\mathcal{C}}}x \} \). Since
  \[ \bigl\lVert T \circ T^{-1}_{\eqr_{\mathcal{C}}}\bigr\rVert_{1} = \int_{Y} D(Tx,
    T_{\eqr_{\mathcal{C}}}x)\, d\mu(x),\]
  our task is to estimate this integral. This can be done in a rather crude way. We can simply use
  the triangle inequality
  \( D(Tx,T_{\eqr_{\mathcal{C}}}x) \le |\rho_{T}(x)| + |\rho_{T_{\eqr_{\mathcal{C}}}}(x)| \), and
  deduce
  \begin{displaymath}
    \int_{Y}D(Tx, T_{\eqr_{\mathcal{C}}}x)\, d\mu(x) \le \int_{Y} |\rho_{T}|\, d\mu
    + \int_{Y}|\rho_{T_{\eqr_{\mathcal{C}}}}|\, d\mu \le 2 \int_{Y} |\rho_{T}|\, d\mu,
  \end{displaymath}
  where the last inequality is based on Lemma~\ref{lem:TEC-estimate}.

  It remains to show that \( \int_{Y} |\rho_{T}|\, d\mu < \epsilon/2 \).  Let
  \( \widetilde{X} = \{c + [K, \gap_{\mathcal{C}}(c)-K] : c \in \mathcal{C} \} \) be the region that
  leaves out intervals of length \( K \) on both sides of each point in \( \mathcal{C} \). Note that
  for any \( x \in \widetilde{X}\setminus A_{K} \) one has \( x\eqr_{\mathcal{C}}Tx \) and thus
  \( T_{\eqr_{\mathcal{C}}}x = Tx \) for such points. Therefore,
  \( Y \subseteq A_{K} \sqcup B_{K} \), where \( B_{K} = X \setminus (\widetilde{X} \cup A_{K}) \),
  and thus
  \[ \int\limits_{Y}|\rho_{T}|\, d\mu \le \int\limits_{A_{K}} |\rho_{T}|\, d\mu +
    \int\limits_{B_{K}}|\rho_{T}|\, d\mu < \epsilon/4 + K \cdot 2K/M < \epsilon/2. \qedhere
  \]
\end{proof}

\begin{lemma}
  \label{lem:approximation-by-periodic}
  Let \( T \in \lfgr{\mathbb{R} \acts X} \) be a conservative transformation.  For any
  \( \epsilon > 0 \) there exists a periodic transformation \( P \in \fgr{T} \) such that
  \( \snorm{T\circ P^{-1}}_{1} < \epsilon \).
\end{lemma}

\begin{proof}
  By Lemma~\ref{lem:approximation-by-intermitted}, we can find a cocompact cross-section
  \( \mathcal{C} \) such that \( \snorm{T\circ T^{-1}_{\eqr_{\mathcal{C}}}} < \epsilon/2 \). Let
  \( \tilde{M} \) be an upper bound for gaps in \( \mathcal{C} \). Recall that the cocycle
  \( |\rho_{T_{\eqr_{\mathcal{C}}}}(x)| \) is uniformly bounded by \( \tilde{M} \), and, in fact,
  the same is true for any element in the full group of \( T_{\eqr_{\mathcal{C}}} \). In particular,
  we may use Rokhlin's lemma to find a periodic \( P \in \fgr{T_{\eqr_{\mathcal{C}}}} \) such that
  \( \snorm{T_{\eqr_{\mathcal{C}}} \circ P^{-1}} < \epsilon/2\tilde{M} \), and conclude that
  \( \snorm{T_{\eqr_{\mathcal{C}}} \circ P^{-1}}_{1} < \epsilon/2 \). We therefore have
  \[ \snorm{T \circ P^{-1}}_{1} \le \snorm{T\circ T^{-1}_{\eqr_{\mathcal{C}}}}_{1} +
    \snorm{T_{\eqr_{\mathcal{C}}} \circ P^{-1}}_{1} < \epsilon. \qedhere\]
\end{proof}

\begin{corollary}
  \label{cor:conservative-belong-to-derived-full-group}
  If \( T \in \lfgr{\mathbb{R} \acts X} \) is conservative, then \( T \) belongs to the derived
  \( \LL^1 \) full group \( \derived(\lfgr{\mathbb{R} \acts X}) \).  In particular, its index
  satisfies \( \ind(T) = 0 \).
\end{corollary}

\begin{proof}
  By Lemma~\ref{lem:approximation-by-periodic}, every conservative transformation
  \(T \in \lfgr{\mathbb{R} \acts X}\) lies in the closed subgroup generated by the periodic
  elements. This subgroup is equal to the derived \(\LL^{1}\) full group
  \(\derived(\lfgr{\mathbb{R} \acts X})\) by Corollary~\ref{cor:all-subgroups-are-equal}. Since the
  range of \(\ind\) is abelian, its kernel contains all commutators, and thus
  \(\derived(\lfgr{\mathbb{R} \acts X}) \subseteq \ker \ind\).
\end{proof}



\chapter[Dissipative transformations]{Dissipative and monotone transformations}
\label{chap:peri-appr-monot}

The previous chapter studied conservative transformations, whereas this one concentrates on
dissipative ones.  Our goal will be to show that any dissipative
\( T \in \lfgr{\mathbb{R} \acts X} \) of index \( \ind(T) = 0 \) belongs to the derived subgroup
\( \derived(\lfgr{\mathbb{R} \acts X}) \).  Recall that conversely, \emph{every} element of the
(topological) derived group \( \derived(\lfgr{\mathbb{R} \acts X}) \) has index zero since the index
map is a continuous group homomorphism taking values in an abelian topological group.  We begin by
describing some general aspects of the dynamics of dissipative automorphisms.

Recall that according to Proposition~\ref{prop:hopf-decomposition-full-group}, any transformation
\( T \in \fgr{\mathbb{R} \acts X} \) induces a \( T \)-invariant partition of the phase space
\( X = C_T \sqcup D_T \) such that \( T|_{C_T} \) is conservative and \( T|_{D_T} \) is dissipative.
Formally speaking, a transformation is said to be dissipative if the partition trivializes to
\( D_T = X \).  For the purpose of this chapter, it is, however, convenient to widen this notion
just a bit by allowing \( T \) to have fixed points.

\begin{definition}
  \label{def:dissipative-transformation} A transformation \( T \in \fgr{\mathbb{R} \acts X} \) is
  said to be \textbf{dissipatively supported}\index{Transformation!dissipatively supported} if
  \( D_T = \supp T \), where \( D_T \) is the dissipative element of the Hopf decomposition for
  \( T \).
\end{definition}

\section{Orbit limits and monotone transformations}
\label{sec:orbit-limits}

We begin by showing that the dynamics of dissipatively supported transformations in \( \LL^{1} \)
full groups of \( \mathbb{R} \)-flows is similar to those in \( \LL^{1} \) full groups of
\( \mathbb{Z} \)-actions.  We do so by establishing an analog of R.~M.~Belinskaja's
result~\cite[Thm.~3.2]{MR0245756}. Recall that a sequence of reals is said to have an almost
constant sign if all but finitely many elements of the sequence have the same sign.

\begin{proposition}
  Let \(S\) be a measure-preserving transformation of the real line that commensurates the set
  \(\R^{-}\).  Suppose that \(S\) is dissipatively supported. Then for almost all \(x\in\R\), the
  sequence of reals \((S^k(x)-x)_{k\in\N}\) has an almost constant sign.
\end{proposition}
\begin{proof}
  Let \(Q\) be the set of reals \(x\) such that the sequence \((S^k(x)-x)_{k\in\N}\) does not have
  an almost constant sign.  Assume, to the contrary, that \(Q\) has positive measure.  Since \(S\)
  is dissipative, we can find a Borel wandering set \(A\subseteq\R\) for \(S\) that intersects \(Q\)
  non-trivially.  All the translates of \(Q' = Q\cap A\) are disjoint, and for all \(x\in Q'\), the
  sequence \((S^k(x)-x)_{k\in\N}\) does not have an almost constant sign.

  Since \(S\) is dissipatively supported, for almost all \(x\in Q'\), the sequence of absolute
  values \((|S^k(x)|)_{k\in\N}\) tends to \(+\infty\) (see
  Proposition~\ref{prop:dissipative-property}).  In particular, there are infinitely many points
  \(y\) in the \(S\)-orbit of \(x\) such that \(y<0\) but \(S(y)>0\).  Because the map
  \(Q'\times \Z\to \R\), which sends \((x,k)\) to \(S^k(x)\), is measure-preserving, it follows that
  the set of \(y<0\) such that \(S(y)>0\) has infinite measure.  This contradicts the fact that
  \(S\) commensurates the set \(\R^-\).
\end{proof}

\begin{corollary}\label{cor:finite-arcs}
  Let \( T \in \lfgr{\mathbb{R} \acts X} \) be dissipatively supported. For almost all
  \(x\in \supp T\), the sequence \( (\rho(x, T^{k}(x)))_{k \in \mathbb{N}} \) has an almost constant
  sign.
\end{corollary}
\begin{proof}
  Let \(T\in \lfgr{\mathbb{R} \acts X}\).  For all \(x\in X\), denote by \(T_x\) the
  measure-preserving transformation of \(\mathbb{R}\) induced by \(T\) on the \(\mathbb{R}\)-orbit
  of \(x\).  By the proof of Proposition~\ref{prop:commensuration-of-L1}, the integral
  \[
    \int_X \lambda(\R^{\geq 0}\bigtriangleup (T_x(\R^{\geq 0})))\, d\mu(x)
  \]
  is finite. In particular, for almost every \(x\in X\), the transformation \(T_x\) commensurates
  the set \(\R^{\geq 0}\). The conclusion now follows directly from the previous proposition.
\end{proof}

For any dissipatively supported transformation in an \( \LL^{1} \) full group of a free locally
compact Polish group action and for almost every \( x \in X \), \( \rho(x, T^{n}x) \to \infty \) as
\( n \to \infty \), in the sense that \( \rho(x, T^{n}x) \) eventually escapes any compact subset of
the acting group. In the context of flows, Corollary~\ref{cor:finite-arcs} strengthens this
statement and implies that \( \rho(x, T^{n}x) \) must converge to either \( +\infty \) or
\( -\infty \).

\begin{corollary}
  \label{cor:evasive-set-limit-infinity}
  If \( T \in \lfgr{\mathbb{R} \acts X} \) is dissipatively supported, then for almost every point
  \( x \in \supp T \), either \( \lim\limits_{n \to \infty} \rho(x, T^{n}x) = +\infty \) or
  \( \lim\limits_{n \to \infty}\rho(x, T^{n}x) = -\infty \). \qed
\end{corollary}

In view of this corollary, there is a canonical \( T \)-invariant decomposition of \( \supp T \)
into ``positive'' and ``negative'' orbits.

\begin{definition}
  \label{def:positive-negative-evasive}
  Let \( T \in \lfgr{\mathbb{R} \acts X} \) be a dissipatively supported automorphism.  Its support
  is partitioned into \( \vec{X} \sqcup \cev{X} \), where
  \begin{displaymath}
    \begin{aligned}
      \vec{X} &= \bigl\{ x \in \supp T : \lim\limits_{n\to\infty}\rho(x,T^{n}x) = +\infty\bigr\}, \\
      \cev{X} &= \bigl\{ x \in \supp T : \lim\limits_{n\to\infty}\rho(x,T^{n}x) = -\infty\bigr\}.\\
    \end{aligned}
  \end{displaymath}
  The set \( \vec{X} \) is said to be \textbf{positive evasive}\index{Set!evasive}, and
  \( \cev{X} \) is \textbf{negative evasive}.
\end{definition}

According to Corollary~\ref{cor:finite-arcs}, for almost every \( x \in \supp T \),
\emph{eventually} either all \( T^{n}x \) are to the right of \( x \) or all are to the left of
it. There are points \( x \) for which the adverb ``eventually'' can, in fact, be dropped.

\begin{corollary}
  \label{cor:monotone-induced-transformation}
  Let \( T \in \lfgr{\mathbb{R} \acts X} \) be a dissipatively supported transformation, and let
  \begin{displaymath}
    \begin{aligned}
      \vec{A} &= \{ x \in \vec{X} : \rho(x, T^{n}x) > 0 \textrm{ for all } n \ge 1 \},\\
      \cev{A} &= \{ x \in \cev{X} : \rho(x, T^{n}x) < 0 \textrm{ for all } n \ge 1 \}.\\
    \end{aligned}
  \end{displaymath}
  The set \( A = \vec{A} \sqcup \cev{A} \) is a complete section for \( T|_{\supp T} \).
\end{corollary}

\begin{proof}
  We need to show that almost every orbit of \( T \) intersects \( A \). Let \( x \in \supp T \) and
  suppose for definiteness that \( x \in \vec{X} \). Since
  \( \lim_{n \to \infty}\rho(x, T^{n}x) = +\infty \), we can define
  \( n_{0} = \max\{n \in \mathbb{N} : \rho(x, T^{n}x) \le 0\} \), and then
  \( T^{n_{0}}x \in \vec{A} \).
\end{proof}

\begin{definition}
  \label{def:monotone-transformation}
  A dissipatively supported transformation \( T \in \lfgr{\mathbb{R} \acts X} \) is
  \textbf{monotone}\index{Transformation!monotone} if \( \rho(x, Tx) > 0 \) for almost all
  \( x \in \vec{X} \), and \( \rho(x, Tx)< 0 \) for almost all \( x \in \cev{X} \).
\end{definition}

\begin{corollary}
  \label{cor:dissipative-periodic-monotone}
  Let \( T \in \lfgr{\mathbb{R} \acts X} \) be a dissipatively supported transformation. There is a
  complete section \( A \subseteq \supp T \) and a periodic transformation
  \( P \in \lfgr{\mathbb{R} \acts X} \cap \fgr{T} \) such that \( T = P \circ T_{A} \) and
  \( T_{A} \) is monotone.
\end{corollary}

\begin{proof}
  Take \( A \) to be as in Corollary~\ref{cor:monotone-induced-transformation} and note that
  \( P = T \circ T_{A}^{-1} \) is periodic and satisfies the conclusions of the corollary.
\end{proof}

As we discussed at the beginning of the chapter, our goal is to show that the kernel of the index
map coincides with the derived subgroup of \( \lfgr{\mathbb{R} \acts X} \). Note that if
\( T = P \circ T_{A} \) is as above, then \( \ind(T) = \ind(T_{A}) \), and coupled with the results
of Chapter~\ref{chap:intermitted-transformations}, it will suffice to show that all monotone
transformations of index zero belong to \( \derived(\lfgr{\mathbb{R} \acts X}) \). This will be the
focus of the rest of this chapter and will take some effort to achieve, but the main strategy is to
show that such transformations can be approximated by periodic transformations, which is the content of
Theorem~\ref{thm:periodic-construction} below.

\section{Arrival and departure sets}
\label{sec:arriv-depart-sets}

Throughout the rest of this chapter, we fix a cross-section \( \mathcal{C} \subset X \) and a
monotone transformation \( T \in \lfgr{\mathbb{R} \acts X} \).

Let us recall a few definitions and facts from Section~\ref{sec:prelim-flows}. The lacunarity of
\(\mathcal{C}\) provides the gap function \(\gap_{\mathcal{C}}:\mathcal{C}\to (0,+\infty)\), and the
induced map \(\sigma_{\mathcal{C}}:\mathcal{C}\to\mathcal{C}\) taking \(c\) to
\(c+\gap_{\mathcal{C}}(c)\), whose orbits are the intersections of the flow's orbits with
\(\mathcal{C}\).  We let \( \lambda_{c}^{\mathcal{C}} \) be the Lebesgue measure on
\( c + [0, \gap_{\mathcal{C}}(c)) \) given by
\[ \lambda_{c}^{\mathcal{C}}(A) = \lambda(\{t \in [0,\gap_{\mathcal{C}}(c)) : c+t \in A\}).\]
The measure \( \mu \) can be disintegrated as
\( \mu(A) = \int_{\mathcal{C}} \lambda_{c}^{\mathcal{C}}(A)\, d\nu(c) \) for some finite (but not
necessarily probability) measure \( \nu \) on \( \mathcal{C} \).  Finally, we use the convenient
notation
\( A(c)=A\cap \bigl( c+[0,\gap_{\mathcal{C}}(c)\bigr)=A\cap [c]_{\mathcal{R}_{\mathcal{C}}} \).
Note that \(\lambda_{c}(A(c)) = \lambda_{c}^{C}(A)\) for all \(c \in \mathcal{C}\), where
\(\lambda_c\) denotes the Lebesgue measure on the whole orbit of \(c\).

We now introduce some essential additional terminology concerning our fixed monotone transformation
\( T \in \lfgr{\mathbb{R} \acts X} \).  The \textbf{arrival set}\index{Set!arrival}
\( A_{\mathcal{C}} \) is the set of the first visitors to \( \eqr_{\mathcal{C}} \) classes:
\( A_{\mathcal{C}} = \{ x \in \supp T : \neg x \eqr_{\mathcal{C}} T^{-1}x\} \). Analogously, the
\textbf{departure set}\index{Set!departure} \( D_{\mathcal{C}} \) is defined to be
\( D_{\mathcal{C}} = \{ x \in \supp T : \neg x \eqr_{\mathcal{C}} Tx\} \). We also let
\( \vec{A}_{\mathcal{C}} \) denote \( A_{\mathcal{C}} \cap \vec{X} \) and
\( \cev{A}_{\mathcal{C}} = A_{\mathcal{C}} \cap \cev{X} \); likewise for \( \vec{D}_{\mathcal{C}} \)
and \( \cev{D}_{\mathcal{C}} \). Note that \( T(D_{\mathcal{C}}) = A_{\mathcal{C}} \), and thus
\( T^{-1}(A_{\mathcal{C}}) = D_{\mathcal{C}} \). There is, however, another useful map from
\( A_{\mathcal{C}} \) onto \( D_{\mathcal{C}} \).

\begin{figure}[htb]
  \centering
  \begin{tikzpicture}
    \draw (0,0) -- (10,0);
    \filldraw (1.5,0) circle (1.5pt);
    \filldraw (8.5,0) circle (1.5pt);
    \draw (1.5, 0.3) node {\( c \)};
    \draw (8.5, 0.3) node {\( c' \)};
    \foreach \x in {0.7, 2.5, 3.9, 4.7, 6.1, 7.7, 9.5} {
      \filldraw (\x,0) circle (1pt);
    }
    \draw [->,domain=90:15,yshift=1mm] plot ({0.1 + 0.6*cos(\x)}, {0.6*sin(\x) - 0});
    \draw [->,domain=165:15] plot ({1.6 + 0.9*cos(\x)}, {0.9*sin(\x) - 0});
    \draw [->,domain=165:15,yshift=0.5mm] plot ({3.2 + 0.7*cos(\x)}, {0.7*sin(\x) - 0});
    \draw [->,domain=165:15,yshift=1.2mm] plot ({4.3 + 0.4*cos(\x)}, {0.4*sin(\x) - 0});
    \draw [->,domain=165:15,yshift=0.5mm] plot ({5.4 + 0.7*cos(\x)}, {0.7*sin(\x) - 0});
    \draw [->,domain=165:15,yshift=0.3mm] plot ({6.9 + 0.8*cos(\x)}, {0.8*sin(\x) - 0});
    \draw [->,domain=165:15] plot ({8.6 + 0.9*cos(\x)}, {0.9*sin(\x) - 0});
    \draw [->,domain=165:100,yshift=1.2mm] plot ({10.0 + 0.5*cos(\x)}, {0.5*sin(\x) - 0});
    \draw (2.5, -0.3) node {\( x \in \vec{A}_{\mathcal{C}} \)};
    \draw (7.7, -0.3) node {\( T^{4}x \in \vec{D}_{\mathcal{C}} \)};
    \draw [->,domain=-165:-15, yshift=-4mm] plot ({5.1 + 2.6*cos(\x)}, {0.6*sin(\x)});
    \draw (5.1, -0.7) node {\( t_{\mathcal{C}}(x) = 4 \)};
  \end{tikzpicture}
  \caption{Arrival and departure sets.}
  \label{fig:arrival-departure}
\end{figure}

We define the \textbf{transfer value} \( t_{\mathcal{C}} : A_{\mathcal{C}} \to \mathbb{N} \) by the
condition \[ t_{\mathcal{C}}(x) = \min\{ n \ge 0: T^{n}x \in D_{\mathcal{C}} \} \]
and the \textbf{transfer function} \( \tau_{\mathcal{C}} : A_{\mathcal{C}} \to D_{\mathcal{C}} \) is
defined to be \( \tau_{\mathcal{C}}(x) = T^{t_{\mathcal{C}}(x)}x \). Note that
\( \tau_{\mathcal{C}} \) is measure-preserving. The transfer value introduces a partition of the
arrival set \( A_{\mathcal{C}} = \bigsqcup_{n \in \mathbb{N}} A_{\mathcal{C}}^{n} \), where
\( A_{\mathcal{C}}^{n} = t_{\mathcal{C}}^{-1}(n) \).  By applying the transfer function, we also
obtain a partition for the departure set:
\( D_{\mathcal{C}} = \bigsqcup_{n \in \mathbb{N}}D_{\mathcal{C}}^{n} \), where
\( D_{\mathcal{C}}^{n} = \tau_{\mathcal{C}}(A_{\mathcal{C}}^{n}) \).

In plain words, \( t_{\mathcal{C}}(x) + 1 \) is the number of points in
\( [x]_{\ceqr_{T}} \cap [x]_{\eqr_{\mathcal{C}}} \). Therefore if
\( \lambda_{c}^{\mathcal{C}}(A_{\mathcal{C}}^{n}) \ge \lambda_{c}^{\mathcal{C}}(A_{\mathcal{C}}^{m})
\) for some \( n \ge m \) then also
\( \lambda_{c}^{\mathcal{C}}([A_{\mathcal{C}}^{n}]_{\ceqr_{T}}) \ge
\lambda_{c}^{\mathcal{C}}([A_{\mathcal{C}}^{m}]_{\ceqr_{T}}) \) since
\[
  \lambda_{c}^{\mathcal{C}}([A_{\mathcal{C}}^{n}]_{\ceqr_{T}}) = (n+1)\lambda_c^{\mathcal
    C}(A^n_{\mathcal C}) \ge (m+1)\lambda_c^{\mathcal C}(A^m_{\mathcal{C}}) =
  \lambda_{c}^{\mathcal{C}}([A_{\mathcal{C}}^{m}]_{\ceqr_{T}}).
\]
In Sections~\ref{sec:coher-modif} and~\ref{sec:peri-appr}, we modify the transformation \( T \) on
the arrival and departure sets, and we want to do this in a way that affects as many orbits as
possible, as measured by \( \lambda_{c}^{\mathcal{C}} \). This amounts to using sets
\( A_{\mathcal{C}}^{n} \) (and \( D_{\mathcal{C}}^{n} \)) with as high values of \( n \) as
possible. The next lemma will be helpful in conducting such a selection in a measurable way across
all of \( c \in \mathcal{C} \).

\begin{lemma}
  \label{lem:copious-forward-arrival-sets}
  Let \( A \subseteq X \) be a measurable set with a measurable partition
  \( A = \bigsqcup_{n} A_{n} \) and let \( \xi : \mathcal{C} \to \mathbb{R}^{\ge 0}\) be a
  measurable function such that \( \xi(c) \le \lambda_{c}^{\mathcal{C}}(A) \) for all
  \( c \in \mathcal{C} \).  There are measurable \( \nu : \mathcal{C} \to \mathbb{N} \) and
  \( r : \mathcal{C} \to \mathbb{R}^{\ge 0} \) such that for any \( c \in \mathcal{C} \) for which
  \( \xi(c) > 0 \), one has
  \begin{equation}
    \label{eq:nu-r-condition}
    \lambda_{c}^{\mathcal{C}} \bigl(\bigl(\mkern-6mu \bigsqcup_{n > \nu(c)}\mkern-6muA_{n}\bigr)
    \cup \bigl( A_{\nu(c)} \cap (c + [0, r(c)])\bigr)\bigr) = \xi(c).
  \end{equation}
\end{lemma}

\begin{proof}
  For \( c \in \mathcal{C}\) such that \( \xi(c) > 0 \), set
  \[ \nu(c) = \min\bigl\{n \in \mathbb{N} : \lambda_{c}^{\mathcal{C}}\bigl(\bigsqcup_{k >
      n}A_{k}\bigr) < \xi(c) \bigr\}. \]
  Note that one necessarily has
  \( \lambda_{c}^{\mathcal{C}}(A_{\nu(c)}) \ge \xi(c) - \lambda_{c}^{\mathcal{C}}\bigl(\bigsqcup_{n
    > \nu(c)}A_{n}\bigr)>0 \).  Set
  \[ r(c) = \min\bigl\{ a \ge 0 : \lambda_{c}^{\mathcal{C}}\bigl(A_{\nu(c)} \cap (c+[0,a])\bigr) =
    \xi(c) - \lambda_{c}^{\mathcal{C}}\bigl(\mkern-6mu\bigsqcup_{n > \nu(c)}\mkern-8muA_{n}\bigr)
    \bigr\}. \] These functions \( \nu \) and \( r \) satisfy the conclusions of the lemma.
\end{proof}

\begin{remark}
  Note that Eq.~\eqref{eq:nu-r-condition} does not specify the functions \(\nu\) and \(r\) uniquely.
  For instance, although \(\lambda_{c}^{\mathcal{C}}(A_{\nu(c)})>0\), there might be \(\delta > 0\)
  such that, for some \(c\in \mathcal{C}\),
  \[ \lambda_{c}^{\mathcal{C}}\bigl(A_{\nu(c)} \cap (c+[r(c),r(c)+\delta])\bigr)=0.\]
  In this case, replacing \(r(c)\) with \(r(c)+\delta\) does not change the validity of
  Eq.~\eqref{eq:nu-r-condition}.
\end{remark}

\begin{definition}
  \label{def:copious-arrival-set}
  Consider the partition of the positive arrival set
  \( \vec{A}_{\mathcal{C}} = \bigsqcup_{n}\vec{A}_{\mathcal{C}}^{n} \) and let
  \( \xi : \mathcal{C} \to \mathbb{R}^{\ge 0} \), \( r : \mathcal{C} \to \mathbb{R}^{\ge 0} \), and
  \( \nu : \mathcal{C} \to \mathbb{N} \) be as in Lemma~\ref{lem:copious-forward-arrival-sets}. The
  set \( \vec{A}^{\copious}_{\mathcal{C}} \) defined by the condition
  \[ \vec{A}^{\copious}_{\mathcal{C}}(c) = \mkern-6mu\bigsqcup_{n >
      \vec{\nu}(c)}\mkern-6mu\vec{A}_{\mathcal{C}}^{n}(c) \cup \bigl( A_{\mathcal{C}}^{\vec{\nu}(c)}
    \cap (c + [0, \vec{r}(c)])\bigr) \quad \textrm{for all } c \in \mathcal{C} \]
  is said to be the \textbf{positive \( \xi \)-copious arrival set}\index{Set!copious}.  The
  \textbf{positive \( \xi \)-copious departure set} is given by
  \( \vec{D}_{\mathcal{C}}^{\copious} = \tau_{\mathcal{C}}(\vec{A}_{\mathcal{C}}^{\copious}) \). The
  definitions of the \textbf{negative \( \xi \)-copious arrival} and \textbf{departure sets} use the
  partition \( \cev{A}_{\mathcal{C}} = \bigsqcup_{n}\cev{A}_{\mathcal{C}}^{n} \) of the negative
  arrival set and are analogous.
\end{definition}

Copious sets maximize the measure \( \lambda_{c}^{\mathcal{C}} \) of their saturation under the
action of \( T \). In other words, among all subsets \( A' \subseteq \vec{A}_{\mathcal{C}} \) for
which \( \lambda_{c}^{\mathcal{C}}(A') = \xi(c) \), the measure
\( \lambda_{c}^{\mathcal{C}}([A']_{\ceqr_{T}}) \) is maximal when
\( A'(c) = \vec{A}_{\mathcal{C}}^{\copious}(c) \). In particular, if
\( \lambda_{c}^{\mathcal{C}}(\vec{A}_{\mathcal{C}}^{\copious}) \) is close to
\( \lambda_{c}^{\mathcal{C}}(\vec{A}_{\mathcal{C}}) \), then we expect
\( \lambda_{c}^{\mathcal{C}}([\vec{A}_{\mathcal{C}}^{\copious}]_{\ceqr_{T}}) \) to be close to
\( \lambda_{c}^{\mathcal{C}}([\vec{A}_{\mathcal{C}}]_{\ceqr_{T}}) \). The following lemma quantifies
this intuition.

\begin{lemma}
  \label{lem:saturation-of-copious-sets}
  Let \( \xi : \mathcal{C} \to \mathbb{R}^{\ge 0} \) be such that
  \( \xi(c) \le \lambda_{c}^{\mathcal{C}}(\vec{A}_{\mathcal{C}}) \) for all \( c \in \mathcal{C} \),
  and let \( \vec{A}_{\mathcal{C}}^{\copious} \) be the \( \xi \)-copious arrival set constructed in
  Lemma~\ref{lem:copious-forward-arrival-sets}. If there exists \( 1/2 > \delta > 0 \) such that
  \( \xi(c) \ge (1 - \delta) \lambda_{c}^{\mathcal{C}}(\vec{A}_{\mathcal{C}}) \) for all
  \( c \in \mathcal{C} \), then
  \[ \lambda_{c}^{\mathcal{C}}([\vec{A}_{\mathcal{C}}(c) \setminus
    \vec{A}^{\copious}_{\mathcal{C}}(c)]_{\ceqr_{T}}) \le
    \frac{\delta}{1-\delta}\lambda_{c}^{\mathcal{C}}(\vec{X}) \quad \textrm{for all } c \in
    \mathcal{C}, \]
  and therefore also
  \( \mu([\vec{A}_{\mathcal{C}} \setminus \vec{A}^{\copious}_{\mathcal{C}}]_{\ceqr_{T}}) \le
  \frac{\delta}{1-\delta}\mu(\vec{X}) \).

  An analogous statement is valid for the negative arrival set \( \cev{A}_{\mathcal{C}} \).
\end{lemma}

\begin{proof}
  Let \( \nu \) be as in Lemma~\ref{lem:copious-forward-arrival-sets} and note that
  \[ \bigsqcup_{k > \nu(c)}\vec{A}_{\mathcal{C}}^{k}(c) \subseteq
    \vec{A}_{\mathcal{C}}^{\copious}(c) \subseteq \bigsqcup_{k \ge
      \nu(c)}\vec{A}_{\mathcal{C}}^{k}(c) \]
  whenever \( c \in \mathcal{C} \) satisfies \( \xi(c) > 0\). Recall that for
  \( x \in \vec{A}_{\mathcal{C}}^{n} \) we have \( x \eqr_{\mathcal{C}} T^{k} \) for all
  \( 0 \le k \le n \) and the sets \( T^{k}(\vec{A}_{\mathcal{C}}^{n}) \) are pairwise disjoint. In
  particular,
  \begin{equation}
    \label{eq:xi-nu-bound}
    \begin{split}
      \lambda_{c}^{\mathcal{C}}(\vec{X})
      &\ge \lambda_{c}^{\mathcal{C}}
        \bigl(\bigl[\mkern-6mu\bigsqcup_{k \ge \nu(c)} \mkern-6mu
        \vec{A}_{\mathcal{C}}^{k}(c)\bigr]_{\ceqr_{T}}\bigr)
        \ge (\nu(c) + 1) \lambda_{c}^{\mathcal{C}}\bigl(\mkern-6mu\bigsqcup_{k \ge \nu(c)}\mkern-6mu
        \vec{A}_{\mathcal{C}}^{k}(c)\bigr) \\
      &\ge (\nu(c) +1) \lambda_{c}^{\mathcal{C}}(\vec{A}_{\mathcal{C}}^{\copious}) = (\nu(c) + 1)\xi(c).
    \end{split}
  \end{equation}
  Note also that \( \xi(c) \ge (1 - \delta)\lambda_{c}^{\mathcal{C}}(\vec{A}_{\mathcal{C}}) \)
  implies
  \begin{equation}
    \label{eq:complement-copiuous-estimate}
    \lambda_{c}^{\mathcal{C}}(\vec{A}_{\mathcal{C}} \setminus \vec{A}^{\copious}_{\mathcal{C}}) \le \xi(c) \delta/(1-\delta).
  \end{equation}
  For any \( c \in \mathcal{C} \), we have
  \begin{displaymath}
    \begin{aligned}
      \lambda_{c}^{\mathcal{C}}([\vec{A}_{\mathcal{C}}(c) \setminus \vec{A}^{\copious}_{\mathcal{C}}(c)]_{\ceqr_{T}})
      &\le \lambda_{c}^{\mathcal{C}}(\{T^{k}x : x\in \vec{A}_{\mathcal{C}}(c)\setminus
        \vec{A}_{\mathcal{C}}^{\copious}(c), 0 \le k \le \vec{\nu}(c) \}) \\
      & \le ( \vec{\nu}(c) + 1)\lambda_{c}^{\mathcal{C}}(\vec{A}_{\mathcal{C}}\setminus
        \vec{A}_{\mathcal{C}}^{\copious})\\
      \because~\eqref{eq:complement-copiuous-estimate}&\le (\vec{\nu}(c)+1) \vec{\xi}(c) \delta/(1-\delta) \\
      \because~\eqref{eq:xi-nu-bound} &\le \lambda_{c}^{\mathcal{C}}(\vec{X})\delta/(1-\delta).
    \end{aligned}
  \end{displaymath}
  The inequality for the measure \( \mu \) follows by disintegrating \( \mu \) into
  \( \int_{\mathcal{C}} \lambda_{c}^{\mathcal{C}}(\, \cdot \,) \, d\nu(c) \), as discussed at the
  outset of this section.

  The argument for the negative arrival set is completely analogous.
\end{proof}

\section{Coherent modifications}
\label{sec:coher-modif}

We remind the reader that our goal is to show that any dissipatively supported transformation
\( T \in \lfgr{\mathbb{R} \acts X} \) of index \( \ind(T) = 0 \) can be approximated by periodic
transformations. One approach to ``loop'' the orbits of \( T \) is by mapping
\( \vec{D}_{\mathcal{C}}(c) \) to \( \cev{A}_{\mathcal{C}}(c) \) and \( \cev{D}_{\mathcal{C}}(c) \)
to \( \vec{A}_{\mathcal{C}}(c) \) (cf.~Figure~\ref{fig:psi-and-psip}). For such a modification to
work, the measures \( \lambda_{c}^{\mathcal{C}}(\vec{D}_{\mathcal{C}}(c)) \) and
\( \lambda_{c}^{\mathcal{C}}(\cev{A}_{\mathcal{C}}(c)) \) have to be equal. Recall that
\( \ind(T) = 0 \) implies that for almost every \( c \in \mathcal{C} \), the measure of points
\( x \) such that \( x \le c < Tx \) equals the measure of those \( y \) for which
\( Ty < c \le y \). If one could guarantee that
\( T(\vec{D}_{\mathcal{C}}(c)) = \vec{A}_{\mathcal{C}}(\sigma_{\mathcal{C}}(c)) \), then the
aforementioned modification would indeed work. In the case of \( \mathbb{Z} \)-actions, the
discreteness of the acting group allows one to find a cross-section \( \mathcal{C} \) for which this
condition does hold. For flows, however, we must deal with the possibility that
\( T(\vec{D}_{\mathcal{C}}(c)) \) can be ``scattered'' (see
Figure~\ref{fig:general-arrival-departure}) along the orbit and be unbounded, which is the key
reason for the increased complexity compared to the argument for \( \mathbb{Z} \)-actions.

Since we cannot hope to ``loop'' all the orbits of \( T \), we will do the next best thing, and apply
the modification of Figure~\ref{fig:psi-and-psip} on ``most'' orbits as measured by
\( \lambda_{c}^{\mathcal{C}} \). Copious sets discussed in Section~\ref{sec:arriv-depart-sets} have
large saturations under \( T \), but, generally speaking, fail to satisfy
\( T(\vec{D}_{\mathcal{C}}^{\copious}(c)) =
\vec{A}_{\mathcal{C}}^{\copious}(\sigma_{\mathcal{C}}(c)) \) for the same reason as do the sets
\( \vec{D}_{\mathcal{C}}(c) \). Our strategy is to leverage the ``\(\epsilon\) of room'' provided by
the difference \( \vec{D}_{\mathcal{C}}(c) \setminus \vec{D}_{\mathcal{C}}^{\copious}(c) \) to
transform \( T \) into \( T' \) that will retain the same arrival and departure sets as \( T \),
while additionally satisfying the condition
\( T'(\vec{D}_{\mathcal{C}}^{\copious}(c)) =
\vec{A}_{\mathcal{C}}^{\copious}(\sigma_{\mathcal{C}}(c)) \). In this section, we describe two
abstract modifications of dissipatively supported transformations, and the approximation strategy
outlined above will later be implemented in Section~\ref{sec:peri-appr}.

Since we are about to consider arrival and departure sets of different transformations, we use the
notation \( \vec{A}_{\mathcal{C}}[U] \) to denote the positive arrival set constructed for a
transformation \( U \); likewise for negative arrival and departure sets, etc.

\begin{lemma}
  \label{lem:departure-set-transformation}
  Let \( \phi \) and \( \phi' \) be measure-preserving transformations on \( X \) subject to the
  following conditions:
  \begin{enumerate}
  \item \( \supp(\phi) \subseteq D_{\mathcal{C}} \), \( \supp(\phi') \subseteq A_{\mathcal{C}} \);
  \item \( \phi(\vec{D}_{\mathcal{C}}) = \vec{D}_{\mathcal{C}}\),
    \( \phi(\cev{D}_{\mathcal{C}}) = \cev{D}_{\mathcal{C}} \), and
    \( \phi'(\vec{A}_{\mathcal{C}}) = \vec{A}_{\mathcal{C}}\),
    \( \phi'(\cev{A}_{\mathcal{C}}) = \cev{A}_{\mathcal{C}} \);
  \item \( x \,\eqr_{\mathcal{C}}\, \phi(x) \) and \( x \,\eqr_{\mathcal{C}}\, \phi'(x) \) for all
    \( x \in \supp T \).
  \end{enumerate}
  The transformation \( Ux = \phi' T \phi(x) \) is monotone, \( Ux = Tx \) for all
  \( x \not \in D_{\mathcal{C}} \), and the sets \( D_{\mathcal{C}} \), \( A_{\mathcal{C}} \) remain
  the same:
  \begin{align*}
    \vec{X}[U] &= \vec{X}  &  \cev{X}[U] &= \cev{X}, \\
    \vec{D}_{\mathcal{C}}[U] &= \vec{D}_{\mathcal{C}} & \cev{D}_{\mathcal{C}}[U]
                                         &=
                                           \cev{D}_{\mathcal{C}},  \\
    \vec{A}_{\mathcal{C}}[U] &= \vec{A}_{\mathcal{C}} & \cev{A}_{\mathcal{C}}[U]
                                         &=
                                           \cev{A}_{\mathcal{C}}. \\
  \end{align*}
  Moreover, the integral of lengths of ``departing arcs'' remains unchanged:
  \begin{displaymath}
    \int_{D_{\mathcal{C}}}\mkern-10mu|\rho_{U}|\, d\mu =
    \int_{D_{\mathcal{C}}}\mkern-10mu|\rho_{T}|\, d\mu, \\
  \end{displaymath}
  and the following estimate on \( \int_{X}D(Tx, Ux)\, d\mu(x) \) is available:
  \begin{displaymath}
    \int_{X} D(Tx, Ux)\, d\mu(x) \le 2 \int_{D_{\mathcal{C}}}\mkern-10mu |\rho_{T}(x)| \, d\mu(x).
  \end{displaymath}
\end{lemma}

\begin{figure}[htb]
  \centering
  \begin{tikzpicture}
    \filldraw (5,0) circle (1pt);
    \filldraw (5,-2) circle (1pt);
    \draw (2.3, 0) node {\( \vec{X} \)};
    \draw (2.3, -2) node {\( \cev{X} \)};
    \draw (5.0, -1) node {\( \sigma_{\mathcal{C}}(c) \)};
    \draw[white,pattern=north east lines] (3.5,-0.1) rectangle (4.5, 0.1);
    \draw[white,pattern=north west lines] (5.8,-0.1) rectangle (6.8, 0.1);
    \draw [->,domain=-50:230] plot ({4 + 0.5*cos(\x)}, {0.5*sin(\x) + 0.6});
    \draw [->,domain=-50:230] plot ({6.3 + 0.5*cos(\x)}, {0.5*sin(\x) + 0.6});
    \draw (4, 0.6) node {\( \phi \)};
    \draw (6.3, 0.6) node {\( \phi' \)};
    \draw[white,pattern=north west lines] (3.2, -2.1) rectangle (4.2, -1.9);
    \draw[white,pattern=north east lines] (5.5, -2.1) rectangle (6.5, -1.9);
    \draw[->] (5.0, -0.8) -- (5.0, -0.2);
    \draw[->] (5.0, -1.2) -- (5.0, -1.8);
    \draw [->,domain=130:410] plot ({3.7 + 0.5*cos(\x)}, {0.5*sin(\x) - 2.6});
    \draw [->,domain=130:410] plot ({6.0 + 0.5*cos(\x)}, {0.5*sin(\x) - 2.6});
    \draw (3.7, -2.6) node {\( \phi' \)};
    \draw (6.0, -2.6) node {\( \phi \)};
    \draw (4, -0.4) node {\( \vec{D}_{\mathcal{C}}(c) \)};
    \draw (3.7, -1.6) node {\( \cev{A}_{\mathcal{C}}(c) \)};
    \draw (6.3, -0.4) node {\( \vec{A}_{\mathcal{C}}(\sigma_{\mathcal{C}}(c)) \)};
    \draw (6.0, -1.6) node {\( \cev{D}_{\mathcal{C}}(\sigma_{\mathcal{C}}(c)) \)};
  \end{tikzpicture}
  \caption{The transformation \( U = \phi'T\phi \) defined in
    Lemma~\ref{lem:departure-set-transformation}.}
  \label{fig:construction-phi-prime-T-phi}
\end{figure}

\begin{proof}
  Figure~\ref{fig:construction-phi-prime-T-phi} illustrates the definition of the transformation
  \( U \). The equality of the arrival and departure sets is straightforward to verify. Note that
  \( \phi(\vec{D}_{\mathcal{C}}(c)) = \vec{D}_{\mathcal{C}}(c) \) for all \( c \in \mathcal{C} \),
  and therefore \( \int_{\vec{D}_{\mathcal{C}}} \rho_{\phi}\,d\mu = 0 \). In fact, the following
  four integrals vanish:
  \begin{equation}
    \label{eq:vanishing-cocycles}
    \int_{\vec{D}_{\mathcal{C}}} \mkern-10mu \rho_{\phi}\,d\mu
    = \int_{\cev{D}_{\mathcal{C}}}\mkern-10mu \rho_{\phi}\,d\mu
    = \int_{\vec{A}_{\mathcal{C}}}\mkern-10mu \rho_{\phi'}\,d\mu
    = \int_{\cev{A}_{\mathcal{C}}}\mkern-10mu \rho_{\phi'}\,d\mu = 0.
  \end{equation}
  Observe that \( \rho_{U} \) is positive on \( \vec{D}_{\mathcal{C}} \) and negative on
  \( \cev{D}_{\mathcal{C}} \); thus
  \begin{align*}
    \int_{D_{\mathcal{C}}}\mkern-10mu|\rho_{U}|\, d\mu =
    & \int_{\vec{D}_{\mathcal{C}}}\mkern-10mu \rho_{\phi'T\phi} \, d\mu -
      \int_{\cev{D}_{\mathcal{C}}}\mkern-10mu \rho_{\phi'T\phi}\, d\mu \\
    \because \textrm{cocycle
    identity}=
    & \int_{\vec{D}_{\mathcal{C}}}\mkern-10mu \rho_{\phi}\, d\mu + \int_{\vec{D}_{\mathcal{C}}}\mkern-10mu
      \rho_{T}(\phi(x))\, d\mu(x) + \int_{\vec{D}_{\mathcal{C}}}\mkern-10mu \rho_{\phi'}(T\phi(x))\, d\mu(x) \\
    &-\int_{\cev{D}_{\mathcal{C}}}\mkern-10mu \rho_{\phi}\, d\mu - \int_{\cev{D}_{\mathcal{C}}}\mkern-10mu
      \rho_{T}(\phi(x))\, d\mu(x) - \int_{\cev{D}_{\mathcal{C}}}\mkern-10mu \rho_{\phi'}(T\phi(x))\, d\mu(x) \\
    =& \int_{\vec{D}_{\mathcal{C}}}\mkern-10mu \rho_{\phi}\, d\mu + \int_{\vec{D}_{\mathcal{C}}}\mkern-10mu
       \rho_{T}\, d\mu + \int_{\vec{A}_{\mathcal{C}}}\mkern-10mu \rho_{\phi'}\, d\mu \\
    &-\int_{\cev{D}_{\mathcal{C}}}\mkern-10mu \rho_{\phi}\, d\mu - \int_{\cev{D}_{\mathcal{C}}}\mkern-10mu
      \rho_{T}\, d\mu - \int_{\cev{A}_{\mathcal{C}}}\mkern-10mu \rho_{\phi'}\, d\mu \\
    \because~\textrm{Eq.~\eqref{eq:vanishing-cocycles}}=
    & \int_{\vec{D}_{\mathcal{C}}}\mkern-10mu \rho_{T} \, d\mu -
      \int_{\cev{D}_{\mathcal{C}}}\mkern-10mu \rho_{T} \, d\mu
      = \int_{D_{\mathcal{C}}}\mkern-10mu |\rho_{T}|\, d\mu.
  \end{align*}
  Finally, note that for any \( x \in D_{\mathcal{C}} \), the arc from \( x \) to \( Tx \)
  intersects the arc from \( T^{-1}\phi'T\phi(x) \) to \( \phi'T\phi(x) \) (both arcs go over the
  same point of \( \mathcal{C} \)), and therefore
  \begin{displaymath}
    \begin{aligned}
      D(Tx, Ux) \le |\rho_{T}(x)| + |\rho_{T}(T^{-1}\phi'T\phi(x))|.
    \end{aligned}
  \end{displaymath}
  Integration over \( D_{\mathcal{C}} \) yields
  \[ \int_{X} D(Tx, Ux)\, d\mu(x) = \int_{D_{\mathcal{C}}} D(Tx, Ux)\, d\mu(x) \le 2
    \int_{D_{\mathcal{C}}}|\rho_{T}(x)| \, d\mu(x). \qedhere\]
\end{proof}

\begin{lemma}
  \label{lem:gluing-arriving-and-departure-sets}
  Let \( T \in \lfgr{\mathbb{R} \acts X} \) be a monotone transformation.  Let
  \( F \subseteq D_{\mathcal{C}} \) be such that
  \( \lambda_{c}^{\mathcal{C}}(\vec{F}) = \lambda_{c}^{\mathcal{C}}(\cev{F}) \) for all
  \( c \in \mathcal{C} \), and the function
  \( \mathcal{C} \ni c \mapsto \lambda_{c}^{\mathcal{C}}(F) \) is
  \( \sigma_{\mathcal{C}} \)-invariant (i.e.,
  \(\lambda_{c}^{\mathcal{C}}(F) = \lambda^{\mathcal{C}}_{c'}(F)\) whenever \( c \) and \( c' \)
  belong to the same orbit of the flow).  Let \( Z \subseteq A_{\mathcal{C}} \) be the arrival
  subset that corresponds to \( F \), i.e., \( Z = T(F) \).  Let \( \psi : \vec{F} \to \cev{Z} \)
  and \( \psi' : \cev{F} \to \vec{Z} \) be any measure-preserving transformations such that
  \( \psi(x) \eqr_{\mathcal{C}} x \) and \( \psi'(x) \eqr_{\mathcal{C}} x \) for all \( x \) in the
  corresponding domains.  Define \( V : X \to X \) by the following formula:
  \begin{displaymath}
    Vx =
    \begin{cases}
      \psi(x) & \textrm{if } x \in \vec{F}, \\
      \psi'(x) & \textrm{if } x \in \cev{F}, \\
      Tx & \textrm{otherwise.}
    \end{cases}
  \end{displaymath}
  The transformation \( V \) is a measure-preserving automorphism from the full group
  \( \fgr{\mathbb{R} \acts X} \), and \( Vx = Tx \) for all \( x \not \in F \).  The integral of
  distances \( D(Tx, Vx) \) can be estimated as follows:
  \begin{displaymath}
    \int_{X}D(Tx, Vx)\, d\mu(x) \le 2 \int_{D_{\mathcal{C}}} \mkern-10mu |\rho_{T}(x)|\, d\mu(x).
  \end{displaymath}
\end{lemma}

\noindent The following figure illustrates the notions of
Lemma~\ref{lem:gluing-arriving-and-departure-sets}.

\begin{figure}[htb]
  \centering
  \begin{tikzpicture}
    \filldraw (5,0) circle (1pt);
    \filldraw (5,-1) circle (1pt);
    \draw (2.3, 0) node {\( \vec{X} \)};
    \draw (2.3, -1.0) node {\( \cev{X} \)};
    \draw (5.0, -0.5) node {\( c \)};
    \draw[white,pattern=north east lines] (3.5,-0.1) rectangle (4.5, 0.1);
    \draw[white,pattern=north west lines] (5.8,-0.1) rectangle (6.8, 0.1);
    \draw [->,domain=140:40] plot ({5.15 + 1.2*cos(\x)}, {1.2*sin(\x) - 0.5});
    \draw (5.15, 1.0) node {\( T \)};
    \draw[white,pattern=north west lines] (3.2, -1.1) rectangle (4.2, -0.9);
    \draw[white,pattern=north east lines] (5.4, -1.1) rectangle (6.4, -0.9);
    \draw [->,domain=-40:-140] plot ({4.8 + 1.2*cos(\x)}, {1.2*sin(\x) - 0.5});
    \draw (4.8, -2.0) node {\( T \)};
    \draw[->,dashed] (4.0, -0.2) -- (3.75, -0.85);
    \draw[->,dashed] (5.95, -0.85) -- (6.25, -0.15);
    \draw (3.5,-0.5) node {\( \psi \)};
    \draw (6.5,-0.5) node {\( \psi' \)};
    \draw (3.7, 0.4) node {\( \vec{F} \)};
    \draw (3.4, -1.4) node {\( \cev{Z} \)};
    \draw (6.7, 0.4) node {\( \vec{Z} \)};
    \draw (6.2, -1.4) node {\( \cev{F} \)};
  \end{tikzpicture}
  \caption{The transformation \( V \) defined in
    Lemma~\ref{lem:gluing-arriving-and-departure-sets}.}
  \label{fig:construction-outline}
\end{figure}

\begin{proof}
  It is straightforward to verify that \( V \) is a measure-preserving transformation.  For the
  integral inequality, note that for any \( x \in \vec{F} \) one has
  \[ D(Tx,Vx) \le |\rho_{T}(x)| + |\rho_{T}(T^{-1}x)| ,\] and therefore
  \begin{displaymath}
    \begin{aligned}
      \int_{\vec{F}}D(Tx,Vx)\, d\mu(x) \le \int_{\vec{F}}|\rho_{T}|\, d\mu +
      \int_{\cev{F}}|\rho_{T}|\, d\mu = \int_{F} |\rho_{T}| \, d\mu \le \int_{D_{\mathcal{C}}}\mkern-10mu
      |\rho_{T}|\, d\mu.
    \end{aligned}
  \end{displaymath}
  A similar inequality holds for \( \int_{\cev{F}}D(Tx,Vx)\, d\mu \), and the lemma follows.
\end{proof}

\section{Periodic approximations}
\label{sec:peri-appr}

We now have all the ingredients necessary to prove that monotone transformations can be approximated
by periodic automorphisms. Our arguments follow the approach outlined at the beginning of
Section~\ref{sec:coher-modif}.

In the following lemma, we assume that the Lebesgue measure of those \( x \in \vec{X} \) that jump
over any given \( c \in \mathcal{C} \) is bounded from above by some \( \beta \), and that most of
such jumps --- of measure at least \( \gamma \) --- are between adjacent
\( \eqr_{\mathcal{C}} \)-classes. We are going to construct a periodic approximation \( P \) of the
transformation \( T \) with an explicit bound on \( \int_{X}D(Tx,Px)\, d\mu(x) \), which can be made
small for a sufficiently sparse cross-section \( \mathcal{C} \). When the flow is ergodic, this
lemma alone suffices to conclude that \( T \in \derived(\lfgr{\mathbb{R} \acts X}) \).
Theorem~\ref{thm:periodic-construction} builds upon
Lemma~\ref{lem:periodic-construction-one-cross-section} and treats the general case.

\begin{lemma}
  \label{lem:periodic-construction-one-cross-section}
  Let \( T \in \lfgr{\mathbb{R} \acts X} \) be a monotone transformation, let \( K > 0 \) be a
  positive real, and let \( J = \{ x \in \supp T : |\rho_{T}(x)| \ge K \} \).  Let \( \mathcal{C} \)
  be a cross-section such that \( \gap_{\mathcal{C}}(c) > K \) for all \( c \in \mathcal{C} \).  Let
  \( 0 < \gamma < \beta \) be reals such that for all \( c \in \mathcal{C} \):
  \begin{align*}
    \lambda_{c}^{\mathcal{C}}&(\{x \in \vec{X} : x < \sigma_{\mathcal{C}}(c) \le Tx,
                               Tx\,\eqr_{\mathcal{C}}\,\sigma_{\mathcal{C}}(c)\}) > \gamma,\\
    \lambda_{c}^{\mathcal{C}}&(\{x \in \cev{X} : Tx < c \le x,
                               Tx\,\eqr_{\mathcal{C}}\,\sigma^{-1}_{\mathcal{C}}(c)\}) > \gamma,\\
    \lambda_{c}^{\mathcal{C}}&(\{x \in \vec{X} : x < \sigma_{\mathcal{C}}(c) \le Tx\}) < \beta,\\
    \lambda_{c}^{\mathcal{C}}&(\{x \in \cev{X} : Tx < c \le x\}) < \beta.
  \end{align*}
  There exists a periodic transformation \( P \in \lfgr{\mathbb{R} \acts X} \) such that
  \( \supp P \subseteq \supp T \) and
  \[ \int_{X} D(Tx, Px)\, d\mu(x) \le 5 \int_{D_{\mathcal{C}}} \mkern-10mu |\rho_{T}|\, d\mu +
    \int_{J}|\rho_{T}|\, d\mu + \frac{K(\beta-\gamma)}{\gamma}\mu(\supp T). \]
\end{lemma}

\begin{proof}
  Let \( D_{\mathcal{C}} \) and \( A_{\mathcal{C}} \) be the departure and arrival sets of \( T \).
  Figure~\ref{fig:general-arrival-departure} depicts the arrival set \( \vec{A}_{\mathcal{C}}(c) \) and
  the departure set \( \vec{D}_{\mathcal{C}}(c) \) for an element \( c \) of the cross-section
  \( \mathcal{C} \).  Note that the preimages \( T^{-1}(\vec{A}_{\mathcal{C}}(c)) \) may come from
  different (possibly infinitely many) \( \eqr_{\mathcal{C}} \)-equivalence classes; likewise,
  the images \( T(\vec{D}_{\mathcal{C}}(c)) \) of the departure set may visit several
  \( \eqr_{\mathcal{C}} \)-equivalence classes.

  \begin{figure}[htb]
    \centering
    \begin{tikzpicture}[scale=0.85]
      \draw (-2.0,0) node {\( \vec{X} \)};
      \draw[->] (5.3,0) -- (6.7,0) node[pos=0.5,anchor=south] {\( \tau_{\mathcal{C}} \)};
      \draw (-0.55, 0) node {\( \ldots \)};
      \draw (0.9, 0) node {\( \ldots \)};
      \draw (11.1, 0) node {\( \ldots \)};
      \filldraw (-1.1,0) circle (1.5pt);
      \filldraw (0.4,0) circle (1.5pt);
      \filldraw (3,0) circle (1.5pt);
      \filldraw (9,0) circle (1.5pt);
      \filldraw (11.6,0) circle (1.5pt);
      \draw[white, pattern=north east lines] (-1.5, -0.1) rectangle (-1.2, 0.1);
      \draw[white, pattern=north east lines] (0.0, -0.1) rectangle (0.3, 0.1);
      \draw[white, pattern=north east lines] (1.9, -0.1) rectangle (2.9, 0.1);
      \draw[white, pattern=north west lines] (3.1, -0.1) rectangle (4.4, 0.1);
      \draw [->,domain=165:15, yshift=+1mm] plot ({2.0 + 1.8*cos(\x)}, {0.5*sin(\x)});
      \draw [->,domain=165:15, yshift=+1mm] plot ({1.0 + 2.4*cos(\x)}, {0.5*sin(\x)});
      \draw [->,yshift=-2mm] (2.55,0) to [out=-30,in=210] (4.0,0);
      \draw[white, pattern=north east lines] (7.6, -0.1) rectangle (8.9, 0.1);
      \draw[white, pattern=north west lines] (9.1, -0.1) rectangle (10.2, 0.1);
      \draw[white, pattern=north west lines] (11.7, -0.1) rectangle (12.2, 0.1);
      \draw [->,yshift=-2mm] (8.5,0) to [out=-40,in=220] (9.5,0);
      \draw [->,domain=165:15, yshift=+1mm] plot ({10.0 + 2.0*cos(\x)}, {0.5*sin(\x)});
      \draw [->,yshift=+2mm] (2.1,0) to [out=30,in=150] (10.0,0);
      \draw [decorate,decoration={brace,mirror,amplitude=4pt},yshift=-5mm] (3.15, 0) -- (4.35, 0)
      node [black, midway, anchor=north, yshift=-1.5mm] {\(\vec{A}_{\mathcal{C}}(c) \) };
      \draw [decorate,decoration={brace,mirror,amplitude=4pt},yshift=-5mm] (7.65, 0) -- (8.85, 0)
      node [black, midway, anchor=north, yshift=-1.5mm] {\(\vec{D}_{\mathcal{C}}(c) \) };
      \draw[dotted] (3,1.0) -- (3,0);
      \draw (3,1.25) node {\( c \)} ;
      \draw[dotted] (9,1.0) -- (9,0);
      \draw (9,1.25) node {\( \sigma_{\mathcal{C}}(c) \)} ;
    \end{tikzpicture}
    \caption{The arrival set \( \vec{A}_{\mathcal{C}}(c) \) and the departure
      set \( \vec{D}_{\mathcal{C}}(c) \) for some \( c \in \mathcal{C} \).}
    \label{fig:general-arrival-departure}
  \end{figure}

  Set \( \xi(c) = \gamma \) to be the constant function.  In view of the assumptions on
  \( \gamma \), we may apply Lemma~\ref{lem:copious-forward-arrival-sets} to get positive and
  negative \( \xi \)-copious arrival sets
  \( \vec{A}^{\copious}_{\mathcal{C}} \subseteq A_{\mathcal{C}} \) and
  \( \cev{A}^{\copious}_{\mathcal{C}} \subseteq A_{\mathcal{C}} \), as well as the corresponding
  departure sets
  \( \vec{D}^{\copious}_{\mathcal{C}} = \tau_{\mathcal{C}}(\vec{A}^{\copious}_{\mathcal{C}}) \) and
  \( \cev{D}^{\copious}_{\mathcal{C}} = \tau_{\mathcal{C}}(\cev{A}^{\copious}_{\mathcal{C}}) \). For
  the sets
  \( A^{\copious}_{\mathcal{C}} = \vec{A}^{\copious}_{\mathcal{C}} \sqcup
  \cev{A}^{\copious}_{\mathcal{C}} \) and
  \( D^{\copious}_{\mathcal{C}} = \vec{D}^{\copious}_{\mathcal{C}} \sqcup
  \cev{D}^{\copious}_{\mathcal{C}} \), we have
  \( \lambda^{\mathcal{C}}_{c}(A^{\copious}_{\mathcal{C}}(c)) = 2\gamma =
  \lambda^{\mathcal{C}}_{c}(D^{\copious}_{\mathcal{C}}(c)) \) for all \( c \in \mathcal{C} \). Let
  \begin{displaymath}
    \begin{aligned}
      A^{\circ}_{\mathcal{C}}
      &= \bigl\{ x \in \vec{A}_{\mathcal{C}} : T^{-1}x \,\eqr_{\mathcal{C}}\,
        \sigma_{\mathcal{C}}^{-1}(\pi_{\mathcal{C}}(x))\bigr\} \cup \bigl\{ x \in
        \cev{A}_{\mathcal{C}} : T^{-1}x \,\eqr_{\mathcal{C}}\,
        \sigma_{\mathcal{C}}(\pi_{\mathcal{C}}(x))\bigr\}, \\
      D^{\circ}_{\mathcal{C}}
      &= \bigl\{ x \in \vec{D}_{\mathcal{C}} : Tx
        \,\eqr_{\mathcal{C}}\,\sigma_{\mathcal{C}}(\pi_{\mathcal{C}}(x)) \bigr\}
        \cup \bigl\{ x \in \cev{D}_{\mathcal{C}} : Tx
        \,\eqr_{\mathcal{C}}\,\sigma_{\mathcal{C}}^{-1}(\pi_{\mathcal{C}}(x)) \bigr\},
    \end{aligned}
  \end{displaymath}
  be the set of arcs that jump from/to \emph{the next} \( \eqr_{\mathcal{C}} \)-equivalence class.
  By the assumptions of the lemma, we have
  \( \lambda^{\mathcal{C}}_{c}(\vec{D}^{\circ}_{\mathcal{C}}(c)) \ge \gamma \) and
  \(\lambda^{\mathcal{C}}_{c}(\vec{A}^{\circ}_{\mathcal{C}}(c) ) \ge \gamma \) for all \( c \in \mathcal{C} \).
  Let \( \phi \) be any measure-preserving transformation such that:
  \begin{itemize}
  \item \( \phi \) is supported on \( D_{\mathcal{C}} \);
  \item \( \phi(\vec{D}_{\mathcal{C}}) = \vec{D}_{\mathcal{C}} \) and
    \( \phi(\cev{D}_{\mathcal{C}}) = \cev{D}_{\mathcal{C}} \);
  \item \( \phi(x) \,\eqr_{\mathcal{C}}\, x \) for all \( x \in X \);
  \end{itemize}
  and moreover
  \begin{equation}
    \label{eq:phi-copious-set}
    \phi(D^{\copious}_{\mathcal{C}}) \subseteq D^{\circ}_{\mathcal{C}}.
  \end{equation}
  Select a transformation \( \phi' \) such that
  \begin{itemize}
  \item \( \phi' \) is supported on \( A_{\mathcal{C}} \);
  \item \( \phi'(\vec{A}_{\mathcal{C}}) = \vec{A}_{\mathcal{C}} \) and
    \( \phi'(\cev{A}_{\mathcal{C}}) = \cev{A}_{\mathcal{C}} \);
  \item \( \phi'(x)\, \eqr_{\mathcal{C}}\, x \) for all \( x \in X \);
  \end{itemize}
  and moreover
  \begin{equation}
    \label{eq:phip-copious-set}
    \phi'(T\circ\phi(D^{\copious}_{\mathcal{C}})) = A^{\copious}_{\mathcal{C}}.
  \end{equation}
  Figure~\ref{fig:phi-and-phip} illustrates these maps. Note that while in general
  \( \tau_{\mathcal{C}}\bigl(\vec{A}^{\circ}(c)\bigr) \ne \vec{D}^{\circ}(c) \), one has
  \( \tau_{\mathcal{C}}\bigl(\vec{A}^{\copious}(c)\bigr) = \vec{D}^{\copious}(c) \) for all
  \( c \in \mathcal{C} \) by the definition of the \( \xi \)-copious departure set.

  \begin{figure}[htb]
    \centering
    \begin{tikzpicture}
      \draw (-0.2,0) node {\( \vec{X} \)};
      \filldraw (2,0) circle (1.5pt);
      \filldraw (8,0) circle (1.5pt);
      \draw (2,0) node[yshift=2.5mm] {\( c \)};
      \draw[white,pattern=north east lines] (0.5,-0.1) rectangle (1.9,0);
      \draw[white,pattern=north west lines] (2.1,-0.1) rectangle (3.5,0);
      \draw[white,pattern=north west lines] (2.8,0.0) rectangle (3.8,0.1);
      \draw [->,yshift=-2mm] (1.2,0) to [out=-30,in=210] (2.8,0);
      \draw[white,pattern=north east lines] (6.5,-0.1) rectangle (7.9,0);
      \draw[white,pattern=north west lines] (8.1,-0.1) rectangle (9.5,0);
      \draw[white,pattern=north east lines] (6.2,0.0) rectangle (7.2,0.1);
      \draw [->,yshift=-2mm] (7.2,0) to [out=-30,in=210] (8.8,0);
      \draw [decorate,decoration={brace,mirror,amplitude=4pt},yshift=-5mm] (2.15, 0) -- (3.45, 0)
      node [black, midway, anchor=north, yshift=-1.5mm] {\(\vec{A}^{\circ}_{\mathcal{C}}(c) \)
      };
      \draw [decorate,decoration={brace,mirror,amplitude=4pt},yshift=-5mm] (6.55, 0) -- (7.85, 0)
      node [black, midway, anchor=north, yshift=-1.5mm] {\(\vec{D}^{\circ}_{\mathcal{C}}(c) \)
      };
      \draw [decorate,decoration={brace,amplitude=4pt},yshift=+2.5mm] (2.85, 0) -- (3.75, 0)
      node [black, midway, anchor=south, yshift=+1.5mm] {\(\vec{A}^{\copious}_{\mathcal{C}}(c) \)
      };
      \draw [decorate,decoration={brace,amplitude=4pt},yshift=+2.5mm] (6.25, 0) -- (7.15, 0)
      node [black, midway, anchor=south, yshift=+1.5mm] {\(\vec{D}^{\copious}_{\mathcal{C}}(c) \)
      };

      \draw[->] (4.3, 0.7) -- (5.7,0.7) node[pos=0.5, anchor=south] {\( \tau_{\mathcal{C}}
        \)};
      \draw[<-, dashed] (3.8,-0.05) .. controls (4.3,-0.3) and (3.5,-0.5) .. (3,-0.2);
      \draw (4.25, -0.3) node {\( \phi' \)};
      \draw[->, dashed] (6.2,-0.05) .. controls (5.7,-0.3) and (6.5,-0.5) .. (7,-0.2);
      \draw (4.25, -0.3) node {\( \phi' \)};
      \draw (5.75, -0.3) node {\( \phi \)};
    \end{tikzpicture}
    \caption{Automorphism \( \phi \) maps \( D^{\copious}_{\mathcal{C}}(c) \) into
      \( D^{\circ}_{\mathcal{C}}(c) \) and \( (\phi')^{-1} \) sends
      \( A^{\copious}_{\mathcal{C}}(c) \) into \( A^{\circ}_{\mathcal{C}}(c) \).}
    \label{fig:phi-and-phip}
  \end{figure}

  Let \( U \) be the transformation obtained by applying
  Lemma~\ref{lem:departure-set-transformation} to \( T \), \( \phi \), and \( \phi' \). The
  automorphism \( U \) satisfies
  \( U(\vec{D}^{\copious}_{\mathcal{C}}(c)) =
  \vec{A}^{\copious}_{\mathcal{C}}(\sigma_{\mathcal{C}}(c)) \) and
  \( U(\cev{D}^{\copious}_{\mathcal{C}}(c)) =
  \cev{A}^{\copious}_{\mathcal{C}}(\sigma^{-1}_{\mathcal{C}}(c)) \) for all \( c \in \mathcal{C}
  \). Choose a measure-preserving transformation
  \( \psi : \vec{D}^{\copious}_{\mathcal{C}} \to \cev{A}^{\copious}_{\mathcal{C}} \) such that
  \( x\, \eqr_{\mathcal{C}}\, \psi(x) \) for all \( x \) in the domain of \( \psi \). Set
  \( \psi' = \tau_{\mathcal{C}}^{-1} \circ \psi^{-1} \circ \tau_{\mathcal{C}}^{-1} :
  \cev{D}^{\copious}_{\mathcal{C}} \to \vec{A}^{\copious}_{\mathcal{C}} \).  Let \( V \) be the
  transformation produced by Lemma~\ref{lem:gluing-arriving-and-departure-sets} applied to \( U \),
  \( \psi \), and \( \psi' \) (see Figure~\ref{fig:psi-and-psip}).
  \begin{figure}[htb]
    \centering
    \begin{tikzpicture}
      \draw (-0.5,0.05) node {\( \vec{X} \)};
      \draw (-0.5,-0.95) node {\( \cev{X} \)};
      \filldraw (2,0) circle (1.5pt);
      \filldraw (9,0) circle (1.5pt);
      \draw (2,-0.5) node {\( c \)};
      \draw[white,pattern=north east lines] (0.5,-0.1) rectangle (1.9,0.1);
      \draw[white,pattern=north west lines] (2.1,-0.1) rectangle (3.5,0.1);
      \draw [->,yshift=2mm] (1.2,0) to [out=30,in=150] (2.8,0);
      \draw[white,pattern=north east lines] (7.5,-0.1) rectangle (8.9,0.1);
      \draw[white,pattern=north west lines] (9.1,-0.1) rectangle (10.5,0.1);
      \draw [->,yshift=2mm] (8.2,0) to [out=30,in=150] (9.8,0);
      \draw [decorate,decoration={brace,amplitude=4pt},yshift=5mm] (2.15, 0) -- (3.45, 0)
      node [black, midway, anchor=south, yshift=+1.5mm] {\(\vec{A}^{\copious}_{\mathcal{C}}(c) \)
      };
      \draw [decorate,decoration={brace,amplitude=4pt},yshift=5mm] (7.55, 0) -- (8.85, 0)
      node [black, midway, anchor=south, yshift=+1.5mm] {\(\vec{D}^{\copious}_{\mathcal{C}}(c) \)
      };
      \draw [->,yshift=+2mm] (3.5,0) to [out=20,in=160] (7.5,0);
      \draw (5.5, 0.9) node {\( \tau_{\mathcal{C}} \)};
      \draw[white,pattern=north west lines] (0.5,-1.1) rectangle (1.9,-0.9);
      \draw[white,pattern=north east lines] (2.1,-1.1) rectangle (3.5,-0.9);
      \draw [->,yshift=-2mm] (2.8,-1) to [out=-150,in=-30] (1.2,-1);
      \draw[white,pattern=north west lines] (7.5,-1.1) rectangle (8.9,-0.9);
      \draw[white,pattern=north east lines] (9.1,-1.1) rectangle (10.5,-0.9);
      \draw [->,yshift=-2mm] (9.8,-1) to [out=-150,in=-30] (8.2,-1);
      \filldraw (2,-1) circle (1.5pt);
      \filldraw (9,-1) circle (1.5pt);
      \draw [decorate,decoration={brace,mirror,amplitude=4pt},yshift=-5mm] (2.15, -1.0) -- (3.45, -1.0)
      node [black,midway, anchor=north, yshift=-1.5mm] {\(\cev{D}^{\copious}_{\mathcal{C}}(c) \)
      };
      \draw [decorate,decoration={brace,mirror,amplitude=4pt},yshift=-5mm] (7.55, -1.0) -- (8.85, -1.0)
      node [black,midway, anchor=north, yshift=-1.5mm] {\(\cev{A}^{\copious}_{\mathcal{C}}(c) \)
      };
      \draw [->,yshift=-2mm] (7.5,-1) to [out=-160,in=-20] (3.5,-1);
      \draw (5.5, -1.9) node {\( \tau_{\mathcal{C}} \)};
      \draw[->, dashed] (8.2, -0.2) -- (8.2, -0.8) node[pos=0.5, anchor=east] {\( \psi \)};
      \draw[->, dashed] (2.8, -0.8) -- (2.8, -0.2) node[pos=0.5, anchor=west] {\( \psi' =
        \tau_{\mathcal{C}}^{-1} \circ \psi^{-1} \circ \tau_{\mathcal{C}}^{-1} \)};
    \end{tikzpicture}
    \caption{Construction of the automorphism \( V \) from \( U \), \( \psi \), and \( \psi' \).}
    \label{fig:psi-and-psip}
  \end{figure}
  Finally, set \( P : X \to X \) to be
  \begin{displaymath}
    Px =
    \begin{cases}
      Vx & \textrm{if } x \in [D^{\copious}_{\mathcal{C}}]_{\ceqr_{V}},\\
      x  & \textrm{otherwise}. \\
    \end{cases}
  \end{displaymath}
  We claim that \( P \) satisfies the conclusions of the lemma. It is periodic, since the
  transformation \( \psi'\circ \tau_{\mathcal{C}} \circ \psi \circ \tau_{\mathcal{C}} \) is the
  identity map, and \( \supp P \subseteq \supp T \) by construction. It remains to estimate
  \( \int_{X} D(Tx, Px)\, d\mu(x) \).
  \begin{displaymath}
    \begin{aligned}
      \int_{X}D(Tx, Px)\, d\mu(x)
      &\le \int_{X}D(Tx,Ux)\, d\mu(x) +
        \int_{X}D(Ux, Vx)\, d\mu(x) \\
      & \qquad + \int_{X}D(Vx, Px)\, d\mu(x) \\
      &\le [\textrm{Estimates of Lemma~\ref{lem:departure-set-transformation} and
        Lemma~\ref{lem:gluing-arriving-and-departure-sets}}] \\
      &\le 4\int_{D_{\mathcal{C}}} \mkern-10mu |\rho_{T}|\, d\mu + \int_{X}D(Vx, Px)\, d\mu(x).
    \end{aligned}
  \end{displaymath}
  We concentrate on estimating \( \int_{X}D(Vx, Px)\, d\mu(x) \). Recall that \( Tx = Ux = Vx \) for
  all \( x \not \in D_{\mathcal{C}} \); hence, \( \rho_{T}(x) = \rho_{V}(x) \) for
  \( x \not \in D_{\mathcal{C}} \). Set
  \( \Psi = (\vec{X} \cup \cev{X}) \setminus [D^{\copious}_{\mathcal{C}}]_{\ceqr_{V}} \) and note
  that \( Vx = Ux \) for \( x \in \Psi \).  Therefore, using the conclusion of
  Lemma~\ref{lem:departure-set-transformation}, we have
  \begin{equation}
    \label{eq:cocycle-bound-T-and-V}
    \int_{D_{\mathcal{C}} \cap \Psi} |\rho_{V}|\, d\mu
    = \int_{D_{\mathcal{C}} \cap \Psi} |\rho_{U}|\, d\mu \le
    \int_{D_{\mathcal{C}}}|\rho_{U}|\, d\mu = \int_{D_{\mathcal{C}}}|\rho_{T}|\, d\mu.
  \end{equation}
  The integral \( \int_{X}D(Vx, Px) \, d\mu(x) \) can now be estimated as follows.
  \begin{displaymath}
    \begin{aligned}
      \int_{X}D(Vx, Px) \, d\mu(x)
      &= \int_{\Psi} |\rho_{V}|\, d\mu \\
      &\le \int_{\Psi \setminus D_{\mathcal{C}}} |\rho_{V}|\, d\mu + \int_{D_{\mathcal{C}} \cap \Psi}
        |\rho_{V}|\, d\mu \\
      \because \textrm{\( Tx = Vx \) for \( x \not \in D_{\mathcal{C}} \) and
      Eq.~\eqref{eq:cocycle-bound-T-and-V}}
      &\le \int_{\Psi \setminus D_{\mathcal{C}}} |\rho_{T}|\, d\mu + \int_{D_{\mathcal{C}}}
        |\rho_{T}|\, d\mu. \\
    \end{aligned}
  \end{displaymath}
  Finally, we consider the integral \( \int_{\Psi \setminus D_{\mathcal{C}}} |\rho_{T}|\, d\mu \)
  and partition its domain \( \Psi \setminus D_{\mathcal{C}} \) as
  \( (J \cap (\Psi \setminus D_{\mathcal{C}})) \sqcup ((\Psi \setminus D_{\mathcal{C}}) \setminus
  J)\), which yields
  \begin{displaymath}
    \begin{aligned}
      \int_{\Psi \setminus D_{\mathcal{C}}} |\rho_{T}|\, d\mu
      &\le \int_{J}|\rho_{T}|\, d\mu + K\mu(\Psi) \\
      &\le \int_{J}|\rho_{T}|\, d\mu + \frac{K(\beta - \gamma)}{\gamma}\mu(\supp T),
    \end{aligned}
  \end{displaymath}
  where the last inequality follows from Lemma~\ref{lem:saturation-of-copious-sets} with
  \( \delta = 1 - \gamma/\beta\).  Combining all the inequalities together, we get
  \begin{displaymath}
    \int_{X} D(Tx, Px)\, d\mu(x) \le 5 \int_{D_{\mathcal{C}}} |\rho_{T}|\, d\mu +
    \int_{J}|\rho_{T}|\, d\mu +
    \frac{K(\beta - \gamma)}{\gamma}\mu(\supp T). \qedhere
  \end{displaymath}
\end{proof}

Lemma~\ref{lem:periodic-construction-one-cross-section} allows us to approximate, with a periodic
transformation, a monotone \( T \) for which the Lebesgue measure of points jumping over any given
\( c \in X \) is roughly constant across orbits.  To deal with the general case, we simply need to
split the phase space \( X \) into countably many segments that are invariant under the flow and
apply Lemma~\ref{lem:periodic-construction-one-cross-section} on each of them separately. Small care
needs to be taken to ensure that the values \( (\beta-\gamma)/\gamma \), which appear in the
formulation of Lemma~\ref{lem:periodic-construction-one-cross-section}, remain uniformly small
across the partition of \( X \). Details are presented in the following theorem.  Let us recall that
\(\lambda_c\) denotes the Lebesgue measure on the entire orbit of \(c\), as discussed in
Section~\ref{sec:orbit-transf} (the measure \(\lambda_c^{\mathcal C}\), which we have used
throughout this chapter, corresponds to the Lebesgue measure restricted to the interval
\(c+[0,\gap_{\mathcal{C}}(c))\,\)).

\begin{theorem}
  \label{thm:periodic-construction}
  Let \( T \in \lfgr{\mathbb{R} \acts X} \) be a monotone transformation that belongs to the kernel
  of the index map. For any \( \epsilon > 0 \), there exists a periodic transformation
  \( P \in \lfgr{\mathbb{R} \acts X} \) such that \( \supp P \subseteq \supp T \) and
  \( \int_{X} D(Tx, Px)\, d\mu(x) < \epsilon \).
\end{theorem}

\begin{proof}
  Without loss of generality, we assume \(\epsilon\le 1\).  Let \( K_{\epsilon} \ge 1 \) be such
  that for the set
  \[ J_{\epsilon} = \{x \in \supp T : |\rho_{T}(x)| \ge K_{\epsilon}\} \]
  one has \( \int_{J_{\epsilon}}|\rho_{T}|\, d\mu < \epsilon/18 \). Pick a cross-section
  \( \mathcal{C} \) with gaps so large that
  \[ 2K_{\epsilon}^{2}/\gap_{\mathcal{C}}(c) < \epsilon/15 \]
  for all \( c \in \mathcal{C} \), which ensures
  \begin{equation}
    \label{eq:gaps-in-C-are-large}
    K_{\epsilon} \cdot \mu(D_{\mathcal{C}}\setminus J_{\epsilon}) \le \epsilon/15.
  \end{equation}
  Note also that Eq.~\eqref{eq:gaps-in-C-are-large} holds for any cross-section
  \( \mathcal{C}' \subseteq \mathcal{C} \), since \( D_{\mathcal{C}'} \subseteq D_{\mathcal{C}} \)
  and \( \gap_{\mathcal{C}'}(c) \ge \gap_{\mathcal{C}}(c) \) for all \( c \in \mathcal{C}' \).

  For any positive real \( \alpha > 0 \), the positive real
  \( \delta(\alpha) = \epsilon\alpha/(5\cdot 3K_{\epsilon})\) satisfies
  \( \delta(\alpha) < \alpha \) and
  \( 2\delta(\alpha)/(\alpha - \delta(\alpha)) < \epsilon/3K_{\epsilon} \).  We may therefore pick
  countably many positive reals \( \alpha_{n} > 0 \), \( \delta_{n} >0 \), \( n \ge 1 \), such that
  \( \mathbb{R}^{>0} = \bigcup_{n}(\alpha_{n} - \delta_{n}/2, \alpha_{n} + \delta_{n}/2) \) and
  \begin{equation}
    \label{eq:alpha-delta-condition}
    \Bigl(\frac{2\delta_{n}}{\alpha_{n} - \delta_{n}}\Bigr) <
    \frac{\epsilon}{3K_{\epsilon}} \quad \forall n \ge 1.
  \end{equation}
  Define intervals \( I_{n} = (\alpha_{n} - \delta_{n}/2, \alpha_{n} + \delta_{n}/2) \),
  \( n \ge 1 \).

  Let \( \zeta : \mathcal{C} \to \mathbb{R}^{\ge 0} \) be the map that measures the set of forward
  arcs over its argument:
  \[ \zeta(c) = \lambda_c\bigl(\{x \in \vec{X} : x < c \le Tx \}\bigr). \]
  Our assumption that \(T\) lies in the kernel of the index map implies that \(\zeta\) also measures
  the set of backward arcs over its argument.  Specifically, by Eq.~\eqref{eq:kernel-of-index} from
  Proposition~\ref{prop:index-map-is-surjective}, after discarding an invariant null set, we have
  for every \(c\in\mathcal C\),
  \[
    \lambda_c(\{x \in \supp T : x < c \le Tx\}) = \lambda_c(\{ x \in \supp T : Tx < c \le x\}).
  \]

  Set \( \mathcal{C}_{1} = \zeta^{-1}(I_{1}) \) and construct inductively
  \( \mathcal{C}_{n} = \zeta^{-1}(I_{n}) \setminus \bigl[\bigcup_{k < n}\mathcal{C}_{k}\bigr]_{\eqr}
  \), where \(\eqr\) denotes the orbit equivalence relation induced by the flow
  \(\mathbb{R}\acts X\).  The sets \( \mathcal{C}_{n} \) are pairwise disjoint, and moreover,
  \( ( c_{1}, c_{2} )\notin \eqr \) for all \( c_{1} \in \mathcal{C}_{n_{1}} \),
  \( c_{2} \in \mathcal{C}_{n_{2}} \), \( n_{1} \ne n_{2} \).  Let
  \( \chi_{n} : \mathcal{C}_{n} \to \mathbb{N} \), \( n \ge 1 \), be the function defined by
  \begin{displaymath}
    \begin{aligned}
      \chi_{n}(c) &= \min\Bigl\{
      m \in \mathbb{N}: \\
      &\lambda_c\bigl(\bigl\{x \in \vec{X} : x < c
        \le Tx, D(x,c) \le m, D(Tx,c) \le m\bigr\}\bigr) > \zeta(c) - \delta_{n}/2\\
      \text{and }
      &\lambda_c\bigl(\bigl\{x \in \cev{X} : Tx < c
        \le x, D(x,c) \le m, D(Tx,c) \le m\bigr\}\bigr) > \zeta(c) - \delta_{n}/2\Bigr\}.
    \end{aligned}
  \end{displaymath}
  Set \( \mathcal{C}'_{n,1} = \chi^{-1}_{n}(1) \) and define inductively
  \( \mathcal{C}'_{n,m} = \chi^{-1}(m) \setminus \bigl[\bigcup_{k <
    m}\mathcal{C}'_{n,k}\bigr]_{\eqr}\).  Let \( X_{n,m} \) denote the saturated set
  \( [\mathcal{C}'_{n,m}]_{\eqr} \).  Finally, for all \( m,n \ge 1 \), let
  \( \mathcal{C}_{n,m} \subseteq \mathcal{C}'_{n,m} \) be a sub-section satisfying
  \( \gap_{\mathcal{C}_{n,m}}(c) > m \) for all \( c \in \mathcal{C}_{n,m} \).  The sets
  \( \mathcal{C}_{n,m} \) and \( X_{n,m} \) satisfy the following conditions:
  \begin{enumerate}
  \item \( \mathcal{C}_{n,m} \) is a cross-section for the restriction of the flow onto
    \( X_{n,m} \);
  \item the sets \( X_{n,m} \), \( m, n \ge 1 \), are pairwise disjoint.
  \item\label{item:size-of-departure-set} \( \zeta(c) \in I_{n} \) and, for all
    \( c \in \mathcal{C}_{n,m} \), we have
    \begin{align*}
      \lambda_{c}^{\mathcal{C}_{n,m}}\bigl(\{x \in \vec{X} : x < \sigma_{\mathcal{C}_{n,m}}(c) \le Tx\}\bigr)
      &> \alpha_{n} - \delta_{n} \text{ and } \\
      \lambda_{c}^{\mathcal{C}_{n,m}}\bigl(\{x \in \cev{X} : Tx < c \le x\}\bigr)
      &> \alpha_{n} - \delta_{n}.
    \end{align*}
  \end{enumerate}

  Let \( T_{n,m} \) denote the restriction of \( T \) onto \( X_{n,m} \). Apply
  Lemma~\ref{lem:periodic-construction-one-cross-section} to the transformation \( T_{n,m} \),
  cross-section \( \mathcal{C}_{n,m} \) with gaps of size at least \( K_{\epsilon} \), and
  \( \beta = \alpha_{n} + \delta_{n} \), \( \gamma = \alpha_{n} - \delta_{n} \). Let \( P_{n,m} \)
  be the resulting periodic transformation on \( X_{n,m} \). Set \( P = \bigsqcup_{n,m}P_{n,m}
  \). We claim that \( P \) satisfies the conclusions of the theorem. Set
  \( \mathcal{C}' = \bigsqcup_{n,m} \mathcal{C}_{n,m} \) and note that
  \( \mathcal{C}' \subseteq \mathcal{C} \), whence \( D_{\mathcal{C}'} \subseteq D_{\mathcal{C}} \).
  We can now split the integral \( \int_{X} D(Tx, Px)\, d\mu(x)\) as:
  \[
    \int_{X}D(Tx, Px)\, d\mu(x) = \sum_{n,m} \int_{X_{n,m}}D(T_{n,m}x, P_{n,m}x)\, d\mu(x). \]
  Applying Lemma~\ref{lem:periodic-construction-one-cross-section}, we obtain the following
  chain of inequalities:
  \begin{align*}
    \int_{X}D(Tx, Px)\, d\mu(x)
    &\le  5 \sum_{n,m}\int_{D_{\mathcal{C}_{n,m}}} |\rho_{T}|\, d\mu +
      \sum_{n,m}\int_{J_{\epsilon} \cap X_{n,m}}|\rho_{T}|\, d\mu\\
    &\quad+\sum_{n,m}K_{\epsilon}\Bigl(\frac{2\delta_{n}}{\alpha_{n}-\delta_{n}}\Bigr)\mu(X_{n,m}) \\
    \because \textrm{Eq.~\eqref{eq:alpha-delta-condition}}
    & \le 5 \int_{D_{\mathcal{C}}} |\rho_{T}|\, d\mu + \int_{J_{\epsilon}}|\rho_{T}|\, d\mu +
      (\epsilon/3)\mu(X) \\
    & \le 5 \int_{D_{\mathcal{C}}\setminus J_{\epsilon}} |\rho_{T}|\, d\mu +
      6 \int_{J_{\epsilon}}|\rho_{T}|\, d\mu +  \epsilon/3 \\
    \because \textrm{choice of \( K_{\epsilon} \)}
    & < 5K_{\epsilon}\mu(D_{\mathcal{C}}\setminus J_{\epsilon}) + \epsilon/3 + \epsilon/3 \\
    \because \textrm{Eq.~\eqref{eq:gaps-in-C-are-large}} &\le \epsilon,
  \end{align*}
  and the theorem follows.
\end{proof}

\begin{corollary}
  \label{cor:dissipative-index-zero-derived-subgroup}
  Let \( \mathbb{R} \acts X \) be a measure-preserving flow and
  \( T \in \lfgr{\mathbb{R} \acts X} \) be a dissipatively supported transformation.  If
  \( \ind(T) = 0 \), then \( T \in \derived(\lfgr{\mathbb{R} \acts X}) \).
\end{corollary}

\begin{proof}
  By Corollary~\ref{cor:dissipative-periodic-monotone}, there is a monotone transformation \( U \)
  and a periodic transformation \( P \) such that \( T = U \circ P \). Since
  \( P \in \derived(\lfgr{\mathbb{R} \acts X}) \) by Corollary~\ref{cor:all-subgroups-are-equal}, it
  remains to show that \( U \) belongs to the derived subgroup. The latter follows from
  Theorem~\ref{thm:periodic-construction}, since \( \ind(U) = \ind(T) - \ind(P) = 0 \).
\end{proof}



\chapter{Conclusions}
\label{chap:conclusions}

Our objective in this last chapter is to draw several conclusions regarding the structure of the
\( \LL^{1} \) full groups of measure-preserving flows. The analysis conducted in
Chapters~\ref{chap:intermitted-transformations} and~\ref{chap:peri-appr-monot} leads to the most
technically challenging result of our work, which is the following theorem.

\begin{theorem}
  \label{thm:index-kernel-is-derived-subgroup}
  Let \( \mathcal{F}: \mathbb{R} \acts X \) be a free measure-preserving flow on a standard
  probability space. The kernel of the index map coincides with the derived subgroup
  \( \derived(\lfgr{\mathcal{F}}) \).
\end{theorem}

\begin{proof}
  The inclusion \( \derived(\lfgr{\mathcal{F}}) \subseteq \ker{\ind} \) is automatic since the image
  of \( \ind \) is abelian. For the other direction, pick a transformation \( T \in \ker{\ind} \)
  and consider its Hopf decomposition \( X = C \sqcup D \) provided by
  Proposition~\ref{prop:hopf-decomposition-full-group}. We have \( T = T_{C} \circ T_{D} \), where
  \( T_{C} \in \lfgr{\mathcal{F}} \) is conservative and \( T_{D} \in \lfgr{\mathcal{F}} \) is
  dissipatively supported. According to
  Corollary~\ref{cor:conservative-belong-to-derived-full-group}, \( \ind(T_{C}) = 0 \) and
  \( T_{C} \in \derived(\lfgr{\mathcal{F}}) \), whence
  \( \ind(T_{D}) = \ind(T) - \ind(T_{C}) = 0 \). Therefore, the dissipative part \( T_{D} \)
  satisfies the assumptions of Corollary~\ref{cor:dissipative-index-zero-derived-subgroup}, which
  yields \( T_{D} \in \derived(\lfgr{\mathcal{F}}) \), and hence
  \( T \in \derived(\lfgr{\mathcal{F}}) \) as desired.
\end{proof}

\section{Topological ranks of \texorpdfstring{\(\LL^{1}\)}{L1} full groups}
\label{sec:topological-ranks}

Empowered with the result above and Theorem~\ref{thm:generically-2-generated-derived}, we can
estimate the topological ranks of \( \LL^{1} \) full groups of flows.  We recall the following
well-known inequalities.

\begin{proposition}
  \label{prop:topo-rank-surjective-homomorphism}
  Let \( \phi : G \to H \) be a surjective continuous homomorphism of Polish groups. The topological
  rank \( \tprank(G) \) satisfies
  \[ \tprank(H) \le \tprank(G) \le \tprank(H) + \tprank(\ker{\phi}).  \]
\end{proposition}

\begin{proposition}
  \label{prop:rank-inequality-l1-full-group}
  Let \( \mathcal{F} : \mathbb{R} \acts X \) be a free measure-preserving flow on a standard
  probability space \( (X, \mu) \). The topological rank \( \tprank(\lfgr{\mathcal{F}}) \) is finite
  if and only if the flow has finitely many ergodic components. Moreover, if \( \mathcal{F} \) has
  exactly \( n \) ergodic components, then
  \[ n+1 \le \tprank(\lfgr{\mathcal{F}}) \le n+3. \]
\end{proposition}

\begin{proof}
  Let \( \mathcal{E} \) be the space of probability invariant ergodic measures of the flow, and let
  \( p \) be the probability measure on \( \mathcal{E} \) such that
  \( \mu = \int_{\mathcal{E}}\nu \, dp(\nu) \) (see Appendix~\ref{sec:ergod-decomp}).
  Proposition~\ref{prop:index-map-is-surjective} shows that the index map
  \( \ind : \lfgr{\mathcal{F}} \to \LL^{1}(\mathcal{E}, p) \) is continuous and surjective. An
  application of Proposition~\ref{prop:topo-rank-surjective-homomorphism} yields
  \begin{equation}
    \label{eq:tprank-inequality}
    \tprank(\LL^{1}(\mathcal{E}, p)) \le \tprank(\lfgr{\mathcal{F}}) \le
    \tprank(\LL^{1}(\mathcal{E}, p)) + \tprank(\ker \ind) = \tprank(\LL^{1}(\mathcal{E}, p)) + 2,
  \end{equation}
  where the last equality is based on Theorem~\ref{thm:index-kernel-is-derived-subgroup} and
  Theorem~\ref{thm:generically-2-generated-derived}. Since \( \LL^{1}(\mathcal{E}, p) \) is a Banach
  space, its topological rank is finite if and only if its dimension is finite, which is equivalent
  to \( (\mathcal{E}, p) \) being purely atomic with finitely many atoms. We have shown that
  \( \tprank(\lfgr{\mathcal{F}}) \) is finite if and only if the flow has only finitely many ergodic
  components. The moreover part of the proposition follows from the
  inequality~\eqref{eq:tprank-inequality} and the observation that
  \( \tprank(\LL^{1}(\mathcal{E}, p)) = \dim(\LL^{1}(\mathcal{E}, p)) + 1 \).
\end{proof}

As already mentioned in the introduction, we conjecture that the topological rank completely
remembers the number of ergodic components.

\begin{conjecture}
  Let \( \mathcal F \) be a free measure-preserving flow. If it has exactly \( n \) ergodic components,
  then \( \tprank(\lfgr{\mathcal F}) = n+1 \).
\end{conjecture}

Provided the conjecture holds, we have a priori no way of distinguishing \(\LL^1\) full groups of
ergodic flows as topological groups.  For \(\Z\)-actions, it is a consequence of Belinskaja's
theorem that there are many \(\LL^1\) full groups.  The following two sections explore analogues of
her result for flows, demonstrating the existence of numerous \(\LL^1\) full groups associated with
free ergodic flows. While we currently lack a concrete method to distinguish these groups, their
geometric properties---discussed in the final section---may provide valuable insights in this
direction.

\section{Katznelson's conjugation theorem}
\label{sec:katzn-conj-theor}

R.~M.~Belinskaja~\cite{MR0245756} showed that if measure-preserving transformations
\(T, U \in \Aut(X,\mu)\) generate the same orbit equivalence relation, i.e.,
\(\eqr_{T} = \eqr_{U}\), and \(U \in \lfgr{T}\), then \(T\) and \(U\) are conjugated.  Y.~Katznelson
found a different argument and isolated a sufficient condition for the conjugacy of
measure-preserving transformations (see~\cite[Theorem~A.1]{2201.06662}).  In the following, for
\(T \in \Aut(X, \mu)\), \(x \in X\), and \(A \subseteq \mathbb{Z}\), we let \(T^{A}x\) denote the set
\(\{T^{k}x : k \in A\}\).

\begin{theorem}[Katznelson]
  \label{thm:kaznelson-conjugations-automorphisms}
  Suppose \(T, U \in \Aut(X,\mu)\) are measure-preserving transformations that generate the same
  orbit equivalence relation, \(\eqr_{T} = \eqr_{U}\). If the symmetric difference
  \(T^{\mathbb{N}}x \bigtriangleup U^{\mathbb{N}}x\) is finite for almost all \(x\), then \(T\) and
  \(U\) are conjugated by an element from the full group \(\fgr{T} = \fgr{U}\).
\end{theorem}

The analog of this result for free measure-preserving flows will be proved shortly in
Theorem~\ref{thm:katznelson-conjugation-flows}.  But first, we discuss an important application of
Theorem~\ref{thm:kaznelson-conjugations-automorphisms}. Consider a free measure-preserving flow
\(\mathcal{F} : \mathbb{R} \acts X\).  Given a dissipatively supported transformation
\(T \in \fgr{\mathcal{F}} \) (in the sense of Definition~\ref{def:dissipative-transformation}),
Proposition~\ref{prop:dissipative-property} implies that almost every non-trivial \(T\)-orbit
\([x]_{\eqr_{T}}\) is a discrete subset of \([x]_{\eqr}\) unbounded both from below and from
above. The order induced on \([x]_{\eqr_{T}}\) by the flow may disagree with the \(T\)-order of
points.  One may therefore define the \textbf{\(\mathcal{F}\)-reordering} of \(T\) to be the first
return transformation \(\reor{T}\) induced by the ordering of the flow on the orbits of \(T\):
\[\reor{T}x = x + \min\{ r > 0 : x + r \in [x]_{\eqr_{T}}\} \qquad \textrm{for } x \in \supp T.\]
Note that \(T\) and \(\reor{T}\) generate the same orbit equivalence relation,
\(\eqr_{T} = \eqr_{\reor{T}}\).

If \(T\) belongs to the \(\LL^{1}\) full group of the flow, either
\(T^{\mathbb{N}}x \bigtriangleup \tilde{T}^{\mathbb{N}}x\) or
\(T^{\mathbb{N}}x \bigtriangleup \tilde{T}^{-\mathbb{N}}x\) is finite, depending on whether
\(\lim_{n}\rho(x,T^{n}x) = +\infty\) or \(\lim_{n}\rho(x,T^{n}x) = -\infty\)
(cf.~Corollary~\ref{cor:evasive-set-limit-infinity}).  Which symmetric difference is finite may
depend on the point \(x \in X\), and Theorem~\ref{thm:kaznelson-conjugations-automorphisms} can be
used to show that \(T\) and its reordering \(\tilde{T}\) are flip-conjugate.

\begin{definition}
  \label{def:flip-conjugacy}
  Let \((X_{1}, \mu_{1})\) and \((X_{2}, \mu_{2})\) be standard probability spaces, and let
  \(T_{i} \in \Aut(X_{i}, \mu_{i})\), \(i = 1, 2\). Measure-preserving transformations \(T_{1} \)
  and \(T_{2}\) are \textbf{flip-conjugate}\index{Transformation!flip-conjugate} if there exist an
  isomorphism of measure spaces \(S : X_{1} \to X_{2} \) and measurable partitions
  \(X_{1} = X_{1}^{-}\sqcup X_{1}^{+} \), \( X_{2} = X_{2}^{-}\sqcup X_{2}^{+}\) such that
  \begin{enumerate}
  \item \(S(X_{1}^{-}) = X_{2}^{-}\) and \(S(X_{1}^{+}) = X_{2}^{+}\);
  \item \(X_{1}^{-}, X_{1}^{+}\) are \(T_{1}\)-invariant, and \(X_{2}^{-}, X_{2}^{+}\) are
    \(T_{2}\)-invariant;
  \item \(ST_{1}\restriction_{X_{1}^{+}}S^{-1} = T_{2}\restriction_{X_{2}^{+}}\) and
    \(ST_{1}\restriction_{X_{1}^{-}}S^{-1} = T^{-1}_{2}\restriction_{X_{2}^{-}}\).
  \end{enumerate}
\end{definition}

Note that when one of the \(T_i\)'s is ergodic, our definition of flip-conjugacy coincides with the
standard one, which requires \(X_i^-\) or \(X_i^+\) to have full measure.

\begin{proposition}
  \label{prop:conjugated-reordering}
  Any dissipatively supported \(T \in \lfgr{\mathcal{F}}\) and its \(\mathcal{F}\)-reordering
  \(\reor{T}\) are flip-conjugated by an element from the full group \(\fgr{T} = \fgr{\tilde{T}}\).
\end{proposition}

\begin{proof}
  Consider the decomposition \(\supp T = \cev{X} \sqcup \vec{X}\) into the positive and negative
  orbits as in Definition~\ref{def:positive-negative-evasive}.  In particular,
  \(T^{\mathbb{N}}x \triangle \reor{T}^{\mathbb{N}}x \) and
  \(T^{\mathbb{N}}x \triangle \reor{T}^{-\mathbb{N}}x \) are finite for \(x \in \vec{X}\) and
  \(x \in \cev{X}\), respectively.  Theorem~\ref{thm:kaznelson-conjugations-automorphisms} implies
  that there exist automorphisms \(S_{1} \in \fgr{T \restriction_{{\vec{X}}}}\) and
  \(S_{2} \in \fgr{T \restriction_{{\cev{X}}}}\) such that
  \(S_{1}T\restriction_{\vec{X}}S_{1}^{-1} = \reor{T}\restriction_{\vec{X}}\) and
  \(S_{2}T\restriction_{\cev{X}}S_{2}^{-1} = \reor{T}^{-1}\restriction_{\cev{X}}\).  The
  transformation \(S\) given by
  \begin{displaymath}
    Sx =
    \begin{cases}
      S_{1}x & \textrm{if \(x \in \vec{X}\)},\\
      S_{2}x & \textrm{if \(x \in \cev{X}\)},\\
      x & \textrm{otherwise}\\
    \end{cases}
  \end{displaymath}
  belongs to the full group \(\fgr{T}\) and witnesses the flip-conjugacy of \(T\) and \(\reor{T}\).
\end{proof}

The transformation conjugating \(T\) and \(U\) in
Theorem~\ref{thm:kaznelson-conjugations-automorphisms} can be written fairly explicitly. This is
done in terms of the function \(\delta\) defined as follows. Suppose \( (\Omega, \lambda) \) is a
(possibly infinite) measure space, and let \( A, B \subseteq \Omega \) be measurable sets such that
\( \lambda(A \bigtriangleup B) < +\infty \). We set
\( \delta(A, B) = \lambda(A \setminus B) - \lambda(B \setminus A) \).  This function satisfies a few
properties which the reader can easily verify.

\begin{proposition}
  \label{prop:nabla-properties}
  Suppose \( (\Omega, \lambda) \) is a measure space. For all \( A, B, C, a \subseteq \Omega \)
  such that
  \( \lambda(A \bigtriangleup B), \lambda(B \bigtriangleup C), \lambda(A \bigtriangleup C),
  \lambda(a) < +\infty \), the following holds:
  \begin{enumerate}
  \item \( \delta(A, C) = \delta(A, B) + \delta(B, C) \);
  \item \( \delta(A, A) = 0 \) and \( \delta(A, B) = - \delta(B,A) \);
  \item \( \delta(A \bigtriangleup a, B) = \delta(A,B) + (\lambda(a) - 2 \lambda(a \cap A)) \).
  \end{enumerate}
\end{proposition}

Any orbit of a measure-preserving transformation can be endowed with a counting measure. Given \(T\)
and \(U\) as in the statement of Theorem~\ref{thm:kaznelson-conjugations-automorphisms}, set
\(\tau(x) = \delta(U^{\mathbb{N}}x, T^{\mathbb{N}}x)\) and define \(Sx = U^{\tau(x)}x\).  One can
verify that \(S \in \fgr{U} = \fgr{T}\) and \(STS^{-1} = U\) (further details can be found
in~\cite[Theorem~A.1]{2201.06662}).

Let now \( \mathcal{F}_{1} \) and \( \mathcal{F}_{2} \) be measure-preserving flows on a standard
probability space \( (X, \mu) \); we denote the actions of \( r \in \mathbb{R} \) upon \( x \in X \)
by \( x +_{1} r \) and \( x +_{2} r \), respectively. Suppose that their full groups coincide,
\( \fgr{\mathcal{F}_{1}} = \fgr{\mathcal{F}_{2}} \), and so the flows share the same orbits,
\(\eqr_{\mathcal{F}_{1}} = \eqr_{\mathcal{F}_{2}}\).  For \(x \in X\), let
\(s_{i}(x) = x +_{i} [0,\infty)\), \(i = 1,2\), denote the ``right half-orbit'' of \(x\).  A natural
analog of the condition \(|T^{\mathbb{N}}x \bigtriangleup U^{\mathbb{N}}x| < \infty\) from
Theorem~\ref{thm:kaznelson-conjugations-automorphisms} would be to require finiteness of the
Lebesgue measure of \(s_{1}(x) \bigtriangleup s_{2}(x)\) for all \(x \in X\). This condition alone,
however, is not sufficient for the conjugacy of \(\mathcal{F}_{1}\) and \(\mathcal{F}_{2}\).

Each flow induces a copy of the Lebesgue measure onto orbits via
\[\lambda_{i,x}(A) = \lambda(\{r \in \mathbb{R} : x+_{i} r \in A\}).\]
Since we assume \(\fgr{\mathcal{F}_{1}} = \fgr{\mathcal{F}_{2}}\), and so
\( \mathcal{F}_{2} \subseteq \fgr{\mathcal{F}_{1}} \), \( \lambda_{1,x} \) is a translation
invariant measure relative to the action of \( \mathcal{F}_{2} \), and therefore must differ from
\( \lambda_{2,x} \) by a constant: there is an orbit invariant measurable function
\( c : X \to \mathbb{R}^{>0} \) such that \( \lambda_{2,x} = c(x)\lambda_{1,x} \).  Any element in
\(\fgr{\mathcal{F}_{1}} = \fgr{\mathcal{F}_{2}}\) preserves \(\lambda_{i,x}\), \(i = 1,2\), and
therefore cannot conjugate \(\mathcal{F}_{1}\) into \(\mathcal{F}_{2}\) unless \(c(x)\) is
constantly equal to \(1\).

When the flows are ergodic, \(c(x) = c\) is a constant, and one may renormalize the flows without
changing the full groups.  Let \( \mathcal{F}_{2}' \) be the rescaling of \( \mathcal{F}_{2} \)
given by \( x +_{2}' r = x +_{2} cr \). It is straightforward to check that
\( \lambda_{2,x}'(A) = c^{-1}\lambda_{2,x}(A) = \lambda_{1,x}(A) \), and the flows \( \mathcal{F}_{1} \)
and \( \mathcal{F}_{2}' \) induce the same measure onto orbits.

After this renormalization, the finiteness of the measure \(s_{1}(x) \bigtriangleup s_{2}(x)\) for
all \(x \in X\) is indeed sufficient to establish the conjugacy of the flows.

\begin{theorem}
  \label{thm:katznelson-conjugation-flows}
  Let \( \mathcal{F}_{i} \), \( i = 1,2 \), be free measure-preserving flows that share the same
  orbits, \(\eqr_{\mathcal{F}_{1}} = \eqr_{\mathcal{F}_{2}}\), and induce the same measures
  \((\lambda_{x})_{x \in X}\) onto orbits.  If
  \( \lambda_{x}(s_{1}(x)\bigtriangleup s_{2}(x)) < +\infty \), \( x \in X \), then the flows are
  conjugate by a measure-preserving transformation \(S \in \fgr{\mathcal{F}_{1}}\).
\end{theorem}

\begin{proof}
  Let \( n : X \times \mathbb{R} \to \mathbb{R} \) be the
  \(\mathcal{F}_{1}, \mathcal{F}_{2}\)-cocycle defined by \( x +_{2} r = x +_{1} n(x,r) \). Since
  \(\mathcal{F}_{1}\) and \(\mathcal{F}_{2}\) induce the same measure on the orbits,
  \( n(x, \cdot) : \mathbb{R} \to \mathbb{R} \) is a Lebesgue measure-preserving automorphism:
  \begin{displaymath}
    \begin{aligned}
      \lambda(n(x, A)) &= \lambda_{1,x}(\{x +_{1} n(x,r) : r \in A\})  \\
                       &= \lambda_{2,x}(\{x +_{2} r : r \in A\}) = \lambda(A).
    \end{aligned}
  \end{displaymath}
  For \( x \in X \) and \( r \in \mathbb{R} \cup \{+ \infty\} \) let
  \begin{displaymath}
    s_{i,r}(x) =
    \begin{cases}
      x +_{i} [0,r) & \textrm{if } r \ge 0, \\
      x +_{i} [r,0) & \textrm{if } r < 0.
    \end{cases}
  \end{displaymath}
  In particular, \( s_{i}(x) = s_{i, +\infty}(x) \).  Note that
  \begin{equation}
    \label{eq:shift-half-orbit}
    \begin{split}
      s_{1}(x +_{2} r) &=  s_{1}(x) \bigtriangleup s_{1, n(x,r)}(x), \\
      s_{2}(x +_{2} r) &=  s_{2}(x) \bigtriangleup s_{2, r}(x). \\
    \end{split}
  \end{equation}
  Also, considering the cases \(r < 0\) and \(r \ge 0\) separately, one can easily verify that for
  all \(r \in \mathbb{R}\) and \(i = 1,2\)
  \[ \lambda_{i,x}(s_{i, r}(x)) - 2 \lambda_{i,x}(s_{2}(x) \cap s_{2, r}(x)) = - r. \]
  and, in particular,
  \begin{equation}
    \label{eq:cocycle-correction}
    \begin{split}
      \lambda_{1,x}(s_{1, n(x,r)}(x)) - 2 \lambda_{1,x}(s_{1}(x) \cap s_{1, n(x,r)}(x)) &= - n(x,r), \\
      \lambda_{2,x}(s_{2, r}(x)) - 2 \lambda_{2,x}(s_{2}(x) \cap s_{2, r}(x)) &= - r.
    \end{split}
  \end{equation}
  Put \( \tau(x) = \delta(s_{1}(x), s_{2}(x)) \), then
  \begin{equation}
    \label{eq:tau-shift}
    \begin{split}
      \tau(x+_{2}r)
      &= \delta(s_{1}(x+_{2}r), s_{2}(x+_{2}r)) \\
      \because \textrm{Eq.}~\eqref{eq:shift-half-orbit}
      &= \delta(s_{1}(x) \bigtriangleup s_{1,n(x,r)}(x), s_{2}(x+_{2}r)) \\
      \because \textrm{Prop.}~\ref{prop:nabla-properties}
      &= \delta(s_{1}(x), s_{2}(x+_{2}r)) + \\
      & \qquad \lambda_{1,x}(s_{1, n(x,r)}(x)) - 2
        \lambda_{1,x}(s_{1}(x) \cap s_{1, n(x,r)}(x)) \\
      \because \textrm{Eq.}~\eqref{eq:cocycle-correction}
      &= \delta(s_{1}(x), s_{2}(x+_{2}r)) - n(x,r) \\
      \because \textrm{Prop.}~\ref{prop:nabla-properties}
      &= -\delta(s_{2}(x+_{2}r),s_{1}(x)) - n(x,r) = \delta(s_{1}(x), s_{2}(x+_{2}r)) - \\
      &\qquad (\lambda_{2,x}(s_{2, r}(x)) - 2 \lambda_{2,x}(s_{2}(x) \cap s_{2, r}(x))) - n(x,r)\\
      \because \textrm{Eq.}~\eqref{eq:cocycle-correction}
      &= \delta(s_{1}(x), s_{2}(x)) - n(x,r) + r.\\
    \end{split}
  \end{equation}
  The required transformation \( S : X \to X \) is given by \( Sx = x +_{1} \tau(x) \).
  \begin{displaymath}
    \begin{aligned}
      S(x +_{2} r)
      &= (x +_{2} r) +_{1} \tau(x +_{2} r) = (x +_{1} n(x,r)) +_{1} \tau(x +_{2} r) \\
      \because \textrm{Eq.~\eqref{eq:tau-shift}}
      &= x +_{1} (n(x,r) + \tau(x) - n(x,r) + r) = Sx +_{1} r.
    \end{aligned}
  \end{displaymath}
  Thus \(S\) conjugates \(\mathcal{F}_{1}\) and \(\mathcal{F}_{2}\). It therefore remains to check
  that \( S \) is a measure-preserving bijection. First, note that \( Sx \) satisfies
  \( \delta(s_{1}(Sx), s_{2}(x)) = 0 \). Indeed,
  \( s_{1}(Sx) = s_{1}(x) \bigtriangleup s_{1,\tau(x)}(x) \) (by the analog of
  Eq.~\eqref{eq:shift-half-orbit}), and therefore
  \begin{equation}
    \label{eq:delta-zero}
    \delta(s_{1}(Sx), s_{2}(x)) = \tau(x) - \tau(x) = 0
  \end{equation}
  by Proposition~\ref{prop:nabla-properties}.

  To show injectivity, suppose that \( Sx = Sy \). In view of Eq.~\eqref{eq:delta-zero} and
  Proposition~\ref{prop:nabla-properties},
  \[ \delta(s_{2}(x), s_{2}(y)) = \delta(s_{2}(x), s_{1}(Sx)) + \delta(s_{1}(Sy), s_{2}(y)) = 0. \]
  However, if \( y = x +_{2} r \), then \( s_{2}(y) = s_{2}(x) \bigtriangleup s_{2,r}(x) \) and so
  \( \delta(s_{2}(x), s_{2}(y)) = r \). One concludes that \( r = 0 \) and \( x = y \). We have
  already established that \( S(x +_{2} r) = Sx +_{1} r \), which shows that the range of \( S \) is
  orbit invariant, yielding surjectivity.

  Finally, to show that \( S \) is measure-preserving, it suffices to check that \( S \) preserves
  the Lebesgue measure \( \lambda_{1,x} = \lambda_{2,x} \) on all the orbits.  To this end, let
  \( n' : X \times \mathbb{R} \to \mathbb{R} \) be the \( \mathcal{F}_{1} \)-cocycle (i.e.,
  \( x +_{1} r = x+_{2} n'(x,r) \)). For all \( r' \in \mathbb{R} \), one has
  \begin{displaymath}
    \begin{aligned}
      \lambda_{1,x}(Ss_{1, r'}(x)) &=\lambda_{1,x}(\{ y +_{1} \tau(y) : y \in s_{1,r'}(x)\}) \\
                                   &= \lambda_{1,x}(\{(x +_{1} r) +_{1} \tau(x +_{1} r) : 0 \le r < r'\}) \\
                                   &= \lambda(\{ r + (\tau(x) - r + n'(x,r)) : 0 \le r < r' \}) \\
                                   &= \lambda(\{n'(x,r) : 0 \le r < r'\}) = \lambda(n'(x, [0,r'))) = r'.
    \end{aligned}
  \end{displaymath}
  Hence, \( S \in \Aut(X,\mu) \) is the required conjugation between \( \mathcal{F}_{1} \) and
  \( \mathcal{F}_{2} \).
\end{proof}

In the \(\Z\) case, the above result is the key to Belinskaja's flip-conjugacy result for \(\LL^1\)
orbit equivalence.  Unfortunately, we do not know if it can be useful for proving an analogous
result for flows.  In the next section, we nevertheless obtain a weaker result that shows there are
many \(\LL^1\) full groups.  We leave the following question open.

\begin{question}\label{qu:does-commensurating-always-hold}
  Given two ergodic flows with equal \(\LL^1\) full groups, does there exist a rescaling under
  which they satisfy the hypothesis of the above theorem?
\end{question}

\section{\texorpdfstring{\(\LL^1\)}{L1} orbit equivalence implies flip Kakutani equivalence}
\label{sec:kakut-equiv-isom}

A measure-preserving action of a compactly generated locally compact Polish group can always be
twisted by a continuous automorphism of the group without affecting the \( \LL^1 \) full group.

In the case of \( \Z \)-actions, this takes a particularly simple form since the only non-trivial
automorphism of \( \Z \) is given by \( n\mapsto -n \).  It follows from the results of
R.~M.~Belinskaja~\cite{MR0245756} that this is, up to conjugacy, the only way to get an \(\LL^1\)
orbit equivalence for ergodic \( \Z \)-actions~\cite[Theorem~4.2]{MR3810253}: if \(T_1,T_2\) are two
ergodic measure-preserving transformations that are \(\LL^1\) orbit equivalent, then they are
flip-conjugate: \(T_1\) is conjugate to either \(T_2\) or \(T_2^{-1}\).

As mentioned before, we do not know whether a variant of such rigidity holds when we replace \(\Z\)
by \(\R\) (see Question~\ref{quest:belinskaya-for-flows} below), but, as shown in
Theorem~\ref{thm:flip-kakutani-equivalence}, \(\LL^1\) orbit equivalent free measure-preserving
flows must at least be flip Kakutani equivalent. In particular, there are uncountably many \(\LL^1\)
full groups of free ergodic flows up to abstract group isomorphism.

Let us first define the notion of (flip) Kakutani equivalence of flows.  For the main results about
this concept, the reader may
consult~\cite{katokTimeChangeMonotone1975,katokMonotoneEquivalenceErgodic1977}, where it is called
\textbf{monotone equivalence} of flows.  Given a measure-preserving automorphism
\(T \in \Aut(Z, \nu)\) and a positive integrable function \(f \in \LL^{1}(Z, \nu)\), one can define
the so-called \textbf{suspension flow}\index{Suspension flow} or \textbf{flow under a function} on the space
\[
  X = \{(z, t) : z \in Z,\ 0 \le t < f(z)\}.
\] For \(r \ge 0\), the action \((z, t) + r\) is given by
\[ (z, t) + r = \bigl(T^{k}z, t + r - \sum\limits_{i=0}^{k-1}f(T^{i}z)\bigr), \]
where \( k \ge 0 \) is defined uniquely by the condition
\( \sum_{i=0}^{k-1}f(T^{i}z) \le t + r < \sum_{i=0}^{k}f(T^{i}z) \); similarly, for \( r \le 0 \)
the action is
\[
  (z, t) + r = \bigl(T^{-k}z, t + r + \sum\limits_{i=1}^{k}f(T^{-i}z)\bigr),
\]
where \( k \ge 0 \) satisfies \( 0 \le t + r + \sum\limits_{i=1}^{k}f(T^{-i}z) < f(T^{-k}z) \).
Such a flow preserves the restriction onto \(X\) of the product measure \(\nu \times \lambda\).  The
space \((X, \mu)\), where
\[
  \mu = \dfrac{\nu \times \lambda}{\int_{Z}f\, d\nu}\restriction_{X},
\] is a standard probability space.  The automorphism \(T\) in the suspension flow construction is
called the \textbf{base automorphism}.

\begin{definition}\label{def:flip-kak-eq}
  Two flows are \textbf{flip Kakutani equivalent}\index{Flip Kakutani equivalence} if they are
  isomorphic to suspension flows over flip-conjugate base automorphisms.
\end{definition}

It is important to note that the construction of suspension flows can be reversed through the use of
cross-sections\footnote{In full generality, the definition of a cross-section should actually be
  relaxed, replacing lacunarity with discreteness in each orbit, and only requiring the gap function
  of the cross-section to be integrable.}.  Given a free flow on \((X,\mu)\) and a cocompact
\(U\)-lacunary cross-section \(\mathcal{C} \subseteq X\) for a precompact neighborhood of the
identity \(U\subseteq G\), there is a unique finite measure \(\nu\) on \(\mathcal C\) such that the
map \(U\times \mathcal C\to\mathcal C+U\subseteq X\) taking \((t,c)\) to \(c+t\) is
measure-preserving (see~\cite[Prop.~4.3]{kyedBettiNumbersLocally2015} for the general construction).
The first-return map \(\sigma_\mathcal C: \mathcal C\to\mathcal C\) is measure-preserving, and our
initial flow can be seen as the flow built under the gap function \(\gap_{\mathcal C}\) with the
base transformation \(\sigma_{\mathcal{C}}\).

We require the following key result, established by
D.~Rudolph~\cite{rudolphTwovaluedStepCoding1976}. In light of the preceding discussion, it can be
restated as follows: every free measure-preserving flow is conjugate to a suspension flow with a
two-valued function.

\begin{theorem}[Rudolph]
  \label{thm:rudolph-2-valued-step-coding}
  Let \(t_0\in \R^{>0} \setminus \Q\) be a positive irrational number.  Any free measure-preserving
  flow on a standard probability space admits a cross-section whose gap function takes only the
  values \(1\) and \(t_0\) almost surely.
\end{theorem}

\begin{remark}
  The second-named author has obtained a generalization of this result in the purely Borel context;
  see~\cite{slutskyRegularCrossSections2019}.
\end{remark}

\begin{theorem}
  \label{thm:flip-Kakutani-equivalence-same-orbits}
  Let \(\mathcal{F}, \mathcal{F}'\) be free measure-preserving flows on \((X,\mu)\) that share the
  same orbits, namely \(\eqr_{\mathcal{F}} = \eqr_{\mathcal{F}'}\).  If
  \(\mathcal{F}' \leq \lfgr{\mathcal{F}}\), then \(\mathcal{F}\) and \(\mathcal{F}'\) are flip
  Kakutani equivalent.
\end{theorem}

\begin{proof}
  We denote the flow \(\mathcal{F}\) using our usual notation, \((x,t) \mapsto x + t\).  As
  explained right after Definition~\ref{def:flip-kak-eq}, it suffices to find cross-sections for
  \(\mathcal{F}\) and \(\mathcal{F}'\) such that the corresponding first return automorphisms are
  flip-conjugate.

  Fix some irrational \(t_{0} > 1\), and let \(\mathcal{C} \subseteq X\) be a Borel cross-section
  for \(\mathcal{F}\) such that, outside a Borel \(\mathcal{F}\)-invariant null set, we have
  \(\gap_{\mathcal{C}}(c) \in\{1,t_0\} \) for all \(c \in \mathcal{C}\), as provided by
  Theorem~\ref{thm:rudolph-2-valued-step-coding}. Define the automorphism \(T : X \to X \) by
  \begin{displaymath}
    Tx =
    \begin{cases}
      \sigma_{\mathcal{C}}(c) + \alpha
      & \textrm{if \(x = c + \alpha\) for some \(c \in \mathcal{C}\), \(\alpha \in [0,1]\)},\\
      x & \textrm{otherwise}.
    \end{cases}
  \end{displaymath}
  The transformation \(T\) is obtained by gluing together the identity map, \(x\mapsto x+1\), and
  \(x\mapsto x+t_0\).  Since all these belong to \(\lfgr{\mathcal{F}}\), which is finitely full, we
  have \(T\in \lfgr{\mathcal{F}}\) as well.  Note that \(T\) is dissipatively supported and is
  therefore flip-conjugate to its \(\mathcal{F}'\)-reordering \(\reor{T}\) by
  Proposition~\ref{prop:conjugated-reordering}.  In other words, there is a \(T\)-invariant Borel
  set \(Z \subseteq X\) of full measure, \(\mu(Z) = 1 \), and a \(T\)-invariant Borel partition
  \(Z = Z^{+} \sqcup Z^{-}\) such that \(T\restriction_{Z^{+}}\) is conjugate to
  \(\reor{T}\restriction_{Z^{+}}\) and \(T\restriction_{Z^{-}}\) is conjugate to
  \(\reor{T}^{-1}\restriction_{Z^{-}}\).

  Let \(\nu\) be the measure on \(\mathcal{C}\) given for a Borel \(A \subseteq \mathcal{C}\) by
  \( \nu(A) = \mu(A + [0,1))\).  The measure \(\mu \restriction_{\mathcal{C} +[0,1)}\) is naturally
  isomorphic to \((\nu \times \lambda)\restriction_{\mathcal{C} + [0,1)} \), where \(\lambda\) is
  the Lebesgue measure on \([0,1]\), and we therefore have
  \[\forall^{\nu \times \lambda} (c, \lambda) \in \mathcal{C} \times [0,1)\quad c + \alpha \in Z.\]
  By Fubini's theorem, this is equivalent to
  \[\forall^{\lambda}\alpha \in [0,1)\ \forall^{\nu} c \in \mathcal{C}\quad (c + \alpha \in Z).\]
  Therefore, there exists some \(\alpha_{0} \in [0,1)\) such that
  \(\nu (\{ c \in \mathcal{C} : c + \alpha_{0} \in Z \}) = 1 \).  Note that
  \(T\restriction_{\mathcal{C} + \alpha_{0}}\) is the first return map on
  \(\mathcal{C} + \alpha_{0}\) in the order of the flow \(\mathcal{F}\), whereas
  \(\reor{T} \restriction_{\mathcal{C} + \alpha_{0}}\) is the first return map in the order induced
  on the orbits by \(\mathcal{F}'\).  Since \(T\restriction_{\mathcal{C} + \alpha_{0}}\) and
  \(\reor{T} \restriction_{\mathcal{C} + \alpha_{0}}\) are flip-conjugate, the flows are flip
  Kakutani equivalent.
\end{proof}

Theorem~\ref{thm:flip-Kakutani-equivalence-same-orbits} has the following straightforward
consequences.

\begin{corollary}
  \label{thm:flip-kakutani-equivalence}
  If two free ergodic measure-preserving flows are \(\LL^{1}\) orbit equivalent, then they are also
  flip Kakutani equivalent.
\end{corollary}
\begin{proof}
  This now follows from the definition of \(\LL^1\) orbit equivalence, see Definition~\ref{def:l1oe}
  and the paragraph thereafter.
\end{proof}

\begin{corollary}
  \label{cor:flip-kakutani-equivalence-same-l1}
  If two free ergodic measure-preserving flows have abstractly isomorphic \(\LL^{1}\) full groups,
  then they are also flip Kakutani equivalent.
\end{corollary}

\begin{proof}
  We have seen in Proposition~\ref{prop:l1-full-oe-implies-l1-oe} that the isomorphism of \(\LL^1\)
  full groups of ergodic flows implies \(\LL^1\) orbit equivalence, so the result follows from the
  previous corollary.
\end{proof}

Kakutani equivalence is a highly non-trivial equivalence relation
(see~\cite{ornsteinEquivalenceMeasurePreserving1982}
or~\cite{MR4850604,kundeAnticlassificationResultsWeakly2023}). It seems likely, however, that
\(\LL^{1}\) full groups of flows contain even more information about the action.  The only
continuous automorphisms of \( \R \) are multiplications by nonzero scalars, and we ask whether the
isomorphism of \(\LL^{1}\) full groups necessarily recovers the action up to such an automorphism.

\begin{question}
  \label{quest:belinskaya-for-flows}
  Let \( \mathcal{F}_{1} \) and \( \mathcal{F}_{2} \) be free ergodic measure-preserving flows with
  isomorphic \( \LL^{1} \) full groups.  Is it true that there exists \( \alpha\in\R^* \) such that
  \( \mathcal{F}_{1} \) and \( \mathcal{F}_{2}\circ m_\alpha \) are isomorphic, where \( m_\alpha \)
  denotes the multiplication by \( \alpha \)?
\end{question}

Note that a positive answer to Question~\ref{qu:does-commensurating-always-hold} would imply a
positive answer to the above question.

\section{Maximality of the \texorpdfstring{\(\LL^1\)}{L1} norm and
  geometry}\label{sec:geom-l1-full-groups}

In this last section, we show that the \(\LL^1\) norm is maximal on \(\LL^1\) full groups of flows.
In particular, it defines their quasi-isometry type.  Exploring this quasi-isometry type further
might lead to topological group invariants capable of distinguishing some ergodic flows.

\begin{theorem}\label{thm:l1-norm-maximal}
  Let \(\mathcal F\) be a free measure-preserving flow.  The \(\LL^1\) norm on \(\lfgr{\mathcal{F}}\) is
  maximal.
\end{theorem}
\begin{proof}
  We have already shown that the \(\LL^1\) norm on the derived \(\LL^1\) full group is maximal (see
  Theorem~\ref{thm:lc-amenable-derived-maximal}).  Denote by \((\mathcal E,p)\) the space of
  \(\mathcal F\)-invariant ergodic probability measures, where \(p\) is the probability measure
  arising from the disintegration of \(\mu\), which we write as \(x\mapsto \nu_x\) (see
  Section~\ref{sec:ergod-decomp}).  The derived \(\LL^1\) full group is equal to the kernel of the
  surjective index map \(\ind : \lfgr{\mathcal F} \to \LL^1(\mathcal E, p, \R)\) and the quotient
  norm on \(\lfgr{\mathcal{F}}/\ker I\) is equal to the \(\LL^1\) norm on
  \(\LL^1(\mathcal E,p, \R)\) by Proposition~\ref{prop:index-map}.  The latter norm is maximal, as
  is the case for any Banach space norm.

  Given a function \(f \in \LL^1(\mathcal E,p, \R) \), let \(U_{f} \in \lfgr{\mathcal{F}}\) be given by
  \(U_{f}(x) = x + f(\nu_{x})\).  The cocycle \(\rho_{U_{f}}(x) = f(\nu_{x})\) is constant on each
  ergodic component and \(\snorm{U_{f}}_{1} = \snorm{f}_{1}\).  Furthermore, \(\ind(U_{f}) = f\). We
  show that \(\snorm{\,\cdot\,}\) is both large-scale geodesic and coarsely proper (see
  Appendix~\ref{sec:maximal-norms} and Proposition~\ref{prop:chara-maximal-norm}, in particular).

  Any \(T \in \lfgr{\mathcal{F}}\) can be written as \(T = (TU_{\ind(T)}^{-1}) U_{\ind(T)}\), where
  the transformation \(TU_{\ind(T)}^{-1} \in \ker \ind = \derived(\lfgr{\mathcal{F}})\), and
  \(\snorm{U_{\ind(T)}}_{1} \le \snorm{T}_{1}\).  In particular, we have
  \(\snorm{TU_{\ind(T)}^{-1}}_{1} \le 2 \snorm{T}_{1}\).

  Since the \(\LL^{1}\) norm is maximal on \(\derived(\lfgr{\mathcal{F}})\), it is large-scale
  geodesic.  In fact, Proposition~\ref{prop:induction-friendly-maximal-norm} establishes that it is
  large-scale geodesic with constant \(K = 2\).  We may therefore express \(TU_{\ind(T)}^{-1}\) as a
  product \(V_{1} \cdots V_{n}\) of elements \(V_{i} \in \derived(\lfgr{\mathcal{F}})\), where each
  \(V_{i}\) has norm at most \(K\) and
  \[
    \sum_{i=1}^n\snorm{V_i}_1\leq K\snorm{TU_{\ind(T)}^{-1}}_{1} \leq 2K\snorm{T}_1.
  \]
  The transformation \(U_{\ind(T)}\) can, for any \(m \ge 1\), also be expressed as a product
  \[U_{\ind(T)} = U_{\ind(T)/m} \cdots U_{\ind(T)/m} = U_{\ind(T)/m}^{m}.\]
  By taking \(m\) sufficiently large, we can ensure that
  \(\snorm{U_{\ind(T)/m}}_{1} = \snorm{\ind(T)/m}_{1} \le K\).  Therefore,
  \(T = (V_{1} \cdots V_{n})(U_{\ind(T)/m} \cdots U_{\ind(T)/m})\), and
  \[\sum_{i=1}^{n}\snorm{V_{i}}_{1} + \sum_{j=1}^{m}\snorm{U_{\ind(T)/m}}_{1} \le 2K \snorm{T}_{1} +
    \snorm{U_{\ind(T)}}_{1} \le 3K\snorm{T}_{1}.\]
  We conclude that the norm \(\snorm{\,\cdot\,}\) on \(\lfgr{\mathcal{F}}\) is large-scale geodesic
  with \(K' = 3K = 6\).

  It remains to prove coarse properness. Let \(\epsilon>0\) and \(R>0\) be positive reals.  By
  Theorem~\ref{thm:lc-amenable-derived-maximal}, there exists \(n\in\N\) so large that every element in
  the derived \(\LL^1\) full group of norm at most \(2R\) is a product of \(n\) elements of norm at
  most \(\epsilon\). Let \(N\) be any integer greater than \(R/\epsilon\). We argue that every
  element of \(\lfgr{\mathcal{F}}\) of norm at most \(R\) is a product of \(2n+N\) elements of norm
  at most \(\epsilon\).

  Indeed, if \(T=(TU_{\ind(T)}\inv) U_{\ind(T)}\) has norm at most \(R\), then
  \[\snorm{TU_{\ind(T)}\inv}_{1} \leq 2\norm T_1\leq 2R,\]
  and \(TU_{\ind(T)}\inv\) can therefore by written as a product of \(n\) elements of
  \(\derived(\lfgr{F})\) each of norm \(\le \epsilon\).  Moreover, \(U_{\ind(T)} = U_{\ind(T)/N}^{N}\)
  and \(\snorm{U_{\ind(T)/N}}_{1} \le \epsilon\) by the choice of \(N\).  The conclusion follows.
\end{proof}

\begin{remark}
  While the proposition above states that \(\LL^1\) full groups of flows are quite large, one can
  use Proposition~\ref{prop:commensuration-of-L1} to show that they satisfy the \emph{Haagerup
    property}\index{Haagerup property}.  In other words, such groups admit a coarsely proper affine
  action on a Hilbert space (namely, the affine Hilbert space
  \(\chi_{\mathcal R^{\geq 0}}+\LL^2(\mathcal R,M)\)).
\end{remark}

Corollary~\ref{cor:flip-kakutani-equivalence-same-l1} along
with~\cite[Sec.~12]{ornsteinEquivalenceMeasurePreserving1982} implies that there are uncountably
many \(\LL^1\) full groups of ergodic free flows up to topological group isomorphism.  It would be
interesting if their geometry allowed us to distinguish these groups.  However, we do not even know
the answer to the following question.

\begin{question}
  Are there two free ergodic measure-preserving flows with non-quasi-isometric \(\LL^1\) full
  groups?
\end{question}


\appendix

\chapter{Normed groups}
\label{chap:normed-groups}

We chose to present our work in the framework of groups equipped with compatible norms rather than
metrics.  These two frameworks are equivalent, but the former has some stylistic advantages, in our
opinion. In Appendix~\ref{chap:normed-groups}, we remind the reader the concept of a norm on a group
(Section~\ref{sec:norms-on-groups}) and state C.~Rosendal's results on maximal norms
(Section~\ref{sec:maximal-norms}).

\section{Norms on groups}
\label{sec:norms-on-groups}

\begin{definition}
  \label{def:group-norm}
  A \textbf{norm}\index{Norm} on a group \( G \) is a map
  \( \snorm{\cdot}: G \to \mathbb{R}^{\ge 0} \) such that for all \( g, h \in G \)
  \begin{enumerate}
  \item \( \snorm g = 0 \) if and only if \( g = e \);
  \item \( \snorm g = \snorm{g^{-1}} \);
  \item \( \snorm{gh} \le \snorm{g} + \snorm{h} \).
  \end{enumerate}
  If \(G\) is moreover a topological group, a norm \(\norm{\cdot}\) on \(G\) is called
  \textbf{compatible}\index{Norm!compatible} if the balls \(\{g \in G : ||g|| < r\}\), \( r > 0\),
  form a basis of neighborhoods of the identity.  When \( G \) is a Polish group equipped with a
  compatible norm \(\norm\cdot\), we refer to the pair \( (G,\norm\cdot) \) as a \textbf{Polish
    normed group}\index{Polish group!normed}.
\end{definition}

There is a correspondence between (compatible) left-invariant metrics on a group and (compatible)
norms on it.  Indeed, given a left-invariant metric \( d \) on \( G \), the function
\( \snorm{g} = d(e,g) \) is a norm. Conversely, from a norm \( \snorm{\cdot} \) one can recover the
left-invariant metric \( d \) via \( d(g,h) = \snorm{g^{-1}h} \).  Analogously, there is a
correspondence between norms and right-invariant metrics given by \( d(g,h) = \snorm{hg^{-1}} \).

The language of group norms thus contains the same information as the formalism of left-invariant
(or right-invariant) metrics, but it has the stylistic advantage of removing the need of making a
choice between the invariant side, when such a choice is immaterial.

\begin{remark}
  \label{rem:norms-from-non-invariant-metrics}
  Note, however, that there are metrics that are neither left- nor right-invariant, which
  nonetheless induce a group norm via the same formula \( \snorm{g} = d(g,e) \).  Consider for
  example a Polish group \( G \) with a compatible left-invariant metric \( d' \) on it.  If \( G \)
  is not a CLI group, the metric \( d' \) is not complete, but the metric
  \[ d(f,g) = \frac{d'(f,g) + d'(f^{-1},g^{-1})}{2} \]
  is complete. Since \( d(g,e) = d'(g,e) \), we see that \( d \) induces the same norm
  \( \snorm{\cdot} \) as does the left-invariant metric \( d' \).
\end{remark}

There is a canonical way to push a norm onto a factor group.

\begin{proposition}[{see~\cite[Thm.~2.2.10]{MR2459668}}]
  \label{prop:quotient-norm}
  Let \( (G, \snorm{\cdot}) \) be a Polish normed group, and let \( H \trianglelefteq G \) be a
  closed normal subgroup of \( G \). The function
  \[ \snorm{gH}^{G/H} = \inf\{\snorm{gh} : h\in H\} \]
  is a norm on \( G/H \) which is compatible with the quotient topology. In particular,
  \( (G/H, \snorm{\cdot}^{G/H}) \) is a Polish normed group.
\end{proposition}

\begin{definition}
  \label{def:proper-norm}
  A compatible norm \( \snorm{\cdot} \) on a locally compact Polish group \( G \) is
  \textbf{proper}\index{Norm!proper} if all balls \( \{ g \in G : \snorm{g} \le r \} \) are compact.
\end{definition}
R.~A.~Struble~\cite{MR348037} showed that all locally compact Polish groups admit a compatible
proper norm.

\section{Maximal norms}
\label{sec:maximal-norms}

As we noted in Lemma~\ref{lem:quasi-metric-structure-same-L1}, quasi-isometric norms yield the same
\(\LL^{1}\) full groups. C.~Rosendal identified the class of Polish groups that admit maximal norms,
which are unique up to quasi-isometry.  In this section, we state some results from C.~Rosendal's
treatise~\cite{MR4327092}, which are relevant to our work.  For the reader's convenience, we
formulate the following definitions and propositions in the language of group norms as opposed to
left-invariant metrics or écarts, as in the original reference.

\begin{definition}[{\cite[Def.~2.68]{MR4327092}}]
  \label{def:maximal-norm}
  A compatible norm \( \norm\cdot \) on a Polish group~\( G \) is said to be \textbf{maximal} if for any
  compatible norm \( \norm{\cdot}' \) there is a constant \( C>0 \) such that
  \( \snorm{g}' \leq C \snorm{g} + C\) for all \(g \in G\).
\end{definition}

\begin{definition}[{\cite[Prop.~2.15]{MR4327092}}]\label{def:boundedly-generated}
  Let \( G \) be a Polish group. A subset \( A\subseteq G \) is \textbf{coarsely
    bounded}\index{Set!coarsely bounded} if for
  every continuous isometric action of \( G \) on a metric space \( (M,d_M) \), the set
  \( A\cdot m \) is bounded for each \( m\in M \), i.e., there is \( K>0 \) such that
  \[ d_{M}(a_1\cdot m,a_2\cdot m)\leq K \quad \textrm{for all } a_1,a_2\in A. \]
  A Polish group \( G \) is \textbf{boundedly generated}\index{Polish group!boundedly generated} if
  it is generated by a coarsely bounded set.
\end{definition}

\begin{theorem}[{\cite[Thm.~2.73]{MR4327092}}]
  \label{thm:rosendal-boundedly-generated-maximal}
  A Polish group admits a maximal compatible norm if and only if it is boundedly generated.
\end{theorem}

The following characterization is available to establish maximality of a given norm.

\begin{definition}[{\cite[Def.~2.62]{MR4327092}}]
  \label{def:large-scale-geodesic}
  A norm \(\norm{\cdot}\) on a group \(G\) is called \textbf{large-scale
    geodesic}\index{Norm!large-scale geodesic} if there is \(K>0\) such that for any \(g\in G\),
  there are \(g_1,\ldots,g_n\in G\) of norm \(\snorm{g_{i}} \le K\), \(1 \le i \le n\), such that
  \(g=g_1\cdots g_n\) and
  \[
    \sum_{i=1}^{n}\norm{g_i}\leq K \norm g.
  \]
\end{definition}

\begin{definition}[{\cite[Lem.~2.39(2) and Prop.~2.7(5)]{MR4327092}}]
  \label{def:coarsely-proper}
  A norm \(\snorm\cdot\) on a group \(G\) is called
  \textbf{coarsely proper}\index{Norm!coarsely proper} if for every \(\epsilon>0\) and every \(R>0\), there are a finite subset
  \(F\subseteq G\) and \(n\in\N\) such that every element \(g\in G\) of norm at most \(R\) can be
  written as a product
  \[
    g=f_1g_1\cdots f_ng_n,
  \]
  where \(f_1,\dots,f_n\in F\) and each \(g_i\) has norm at most \(\epsilon\).
\end{definition}

\begin{proposition}[{\cite[Prop.~2.72]{MR4327092}}]
  \label{prop:chara-maximal-norm}
  A compatible norm \(\norm\cdot\) on a Polish group \(G\) is maximal if and only if it is both
  large-scale geodesic and coarsely proper.
\end{proposition}


\chapter{Spaces and groups of measurable maps}\label{sec:spaces-and-groups-of-measurable-maps}

\section{\texorpdfstring{\(\LL^0\)}{L0} spaces and convergence in measure}

In this section, we introduce the topology of convergence in measure for spaces of measurable
functions and explore its connections to an \(\LL^1\)-type metric as well as its relationship with
the group \(\Aut(X,\mu)\).

\begin{definition}
  Let \((X,\mu)\) be a standard probability space, and let \(Y\) be a Polish space. The space
  \(\LL^0(X,\mu,Y)\)\index{L0 space@\(\LL^{0}\) space}, often denoted by \(\LL^0(X,Y)\) for brevity,
  consists of equivalence classes of measurable maps \(f : X \to Y\), where functions are identified
  up to null sets. This space is equipped with the \textbf{topology of convergence in
    measure}\index{Convergence in measure}, which is generated by the sets
  \(\tilde{U}_{\epsilon,A}\) defined as follows: for every measurable subset \(A \subseteq X\),
  every open subset \(U \subseteq Y\), and every \(\epsilon > 0\),
  \[
    \tilde U_{\epsilon,A} = \{f \in \LL^0(X,Y) : \mu(f^{-1}(U) \cap A) > \epsilon\}.
  \]
\end{definition}

The topology of convergence in measure is Polish. The justification of this fact is postponed to
Proposition~\ref{prop:polish-topo-cv-measure}. First, we take a brief detour to justify why this
topology is appropriately named the topology of convergence in measure.

\begin{lemma}\label{lem:L0-nbhd-basis}
  Let \( d_{Y} \) be a compatible metric on \( Y \). For \( f_{0} \in \LL^{0}(X,Y) \), a
  neighborhood basis of \( f_{0} \) is given by the sets
  \[
    \bigl\{ f \in \LL^{0}(X,Y) : \mu\bigl(\{ x \in X : d_{Y}(f(x), f_{0}(x)) \ge \epsilon \}\bigr) <
    \epsilon \bigr\}, \quad \epsilon > 0.
  \]
\end{lemma}

\begin{proof}
  Fix \(\epsilon > 0\) and a dense sequence \( (y_{n})_{n} \) in \( Y \). For each \( n \), define
  \[
    A_n = \bigl\{x \in X : d_Y(f_0(x), y_n) < \tfrac{\epsilon}{2}\bigr\}.
  \]
  By the density of \((y_{n})_{n}\), we have \( X = \bigcup_n A_n \). Consequently, there exists
  \( N \ge 1 \) such that \( \mu\bigl(\bigcup_{n < N} A_n\bigr) > 1 - \tfrac{\epsilon}{2}
  \). Without loss of generality, by re-enumerating the sequence \((y_{n})_{n}\) if necessary, we
  may assume that \(\mu(A_{n}) > 0\) for all \(n < N\). Let \( U^n \) denote the open ball of radius
  \( \tfrac{\epsilon}{2} \) centered at \( y_n \), and set
  \[
    \epsilon_{n} = \min\bigl\{\tfrac{\mu(A_{n})}{2}, \tfrac{\epsilon}{2N}\bigr\}.
  \]
  For any \(f \in \LL^{0}(X, Y)\), the set
  \( B = \{x \in X : d_{Y}(f_{0}(x), f(x)) \ge \epsilon\} \) is contained in
  \[
    \Bigl(X \setminus \bigcup_{n < N} A_{n}\Bigr) \cup \bigcup_{n < N} \{x \in A_{n} : d_{Y}(f(x),
    y_{n}) \ge \tfrac{\epsilon}{2}\},
  \]
  and therefore
  \begin{displaymath}
    \begin{aligned}
      \mu(B)
      &\le \mu\Bigl(X \setminus \bigcup_{n < N} A_{n}\Bigr) +
        \sum_{n < N} \mu\bigl(\{x \in A_{n} : d_{Y}(f(x), y_{n}) \ge \tfrac{\epsilon}{2}\}\bigr) \\
      \because \textrm{choice of \(N\)}
      &< \frac{\epsilon}{2} +
        \sum_{n < N} \mu\bigl(\{x \in A_{n} : d_{Y}(f(x), y_{n}) \ge \tfrac{\epsilon}{2}\}\bigr) \\
      &= \frac{\epsilon}{2} + \sum_{n < N} \mu\bigl(A_{n} \setminus f^{-1}(U^{n})\bigr) \\
      &= \frac{\epsilon}{2} + \sum_{n < N} \bigl(\mu(A_{n}) - \mu(f^{-1}(U^{n}) \cap A_{n})\bigr).
    \end{aligned}
  \end{displaymath}
  If \( f \in \tilde{U}^n_{\mu(A_n) - \epsilon_{n}, A_n} \), then
  \[
    \mu(A_{n}) - \mu(f^{-1}(U^{n}) \cap A_{n}) < \epsilon_{n} \le \tfrac{\epsilon}{2N},
  \]
  and therefore the open set \( \bigcap_{n < N} \tilde{U}^n_{\mu(A_n) - \epsilon_{n}, A_n} \)
  satisfies the desired inclusion
  \[
    f_{0} \in \bigcap_{n < N} \tilde{U}^n_{\mu(A_n) - \epsilon_{n}, A_n} \subseteq \bigl\{ f \in
    \LL^0(X, Y) : \mu\bigl(\{x \in X : d_Y(f(x), f_0(x)) \ge \epsilon\}\bigr) < \epsilon \bigr\}.
  \]

  Conversely, given any \(f_0 \in \tilde U_{\epsilon,A}\), we need to show the existence of
  \(\delta > 0\) such that
  \[
    \left\{ f \in \LL^0(X,Y) : \mu\left(\{x \in X : d_Y(f(x), f_0(x)) \ge \delta\}\right) < \delta
    \right\} \subseteq U_{\epsilon, A}.
  \]
  Since \(U\) is open, we have
  \[
    f_0^{-1}(U) \cap A = \bigcup_{n} \left\{x \in A : d_Y(f_0(x), Y \setminus U) >
      \tfrac{1}{n}\right\}.
  \]
  Thus, there exists \(n \in \N\) sufficiently large such that
  \[
    \mu\bigl(\{x \in A : d_Y(f_0(x), Y \setminus U) > \tfrac{1}{n}\}\bigr) > \epsilon.
  \]
  Let
  \begin{equation}
    \label{eq:condition-on-delta}
    \delta = \min\bigl\{\tfrac{1}{n}, \mu\bigl(\{x \in A : d_Y(f_0(x), Y \setminus U) >
    \tfrac{1}{n}\}\bigr) - \epsilon\bigr\}.
  \end{equation}
  Now, consider \(f \in \LL^0(X,Y)\) satisfying
  \begin{equation}
    \label{eq:condition-on-f}
    \mu\left(\left\{x \in X : d_Y(f_0(x), f(x)) \ge \delta\right\}\right) < \delta.
  \end{equation}
  Then,
  \[
    f^{-1}(U) \cap A \supseteq \left\{x \in A : d_Y(f_0(x), f(x)) < \delta \text{ and } d_Y(f_0(x),
      Y \setminus U) > \tfrac{1}{n}\right\},
  \]
  and the latter set can be rewritten as
  \[
    A \setminus \bigl(\{x \in A : d_Y(f_0(x), f(x)) \ge \delta\} \cup \{x \in A : d_Y(f_0(x), Y
    \setminus U) \le \tfrac{1}{n}\}\bigr).
  \]
  Consequently,
  \begin{multline*}
    \mu(f^{-1}(U) \cap A) \ge \mu(A) - \mu\bigl(\{x \in A : d_Y(f_0(x), f(x)) \ge
    \delta\}\bigr) \\ - \mu\bigl(\{x \in A : d_Y(f_0(x), Y \setminus U) \le
    \tfrac{1}{n}\}\bigr).
  \end{multline*}
  Since \(f\) is assumed to satisfy Eq.~\eqref{eq:condition-on-f} and \(\delta\) is chosen according
  to Eq.~\eqref{eq:condition-on-delta}, we get
  \begin{multline*}
    \mu(f^{-1}(U) \cap A) > \mu(A) - \mu\bigl(\{x \in A : d_Y(f_0(x), Y \setminus U) >
    \tfrac{1}{n}\}\bigr) + \epsilon \\
    - \mu\bigl(\{x \in A : d_Y(f_0(x), Y \setminus
    U) \le \tfrac{1}{n}\}\bigr).
  \end{multline*}
  The two negative terms sum to \(-\mu(A)\), and thus we have \(\mu(f^{-1}(U) \cap A) > \epsilon\).
  We conclude that
  \[
    \bigl\{ f \in \LL^0(X,Y) : \mu(\{x \in X : d_Y(f_0(x), f(x)) \ge \delta\}) < \delta \bigr\}
    \subseteq \tilde U_{\epsilon,A}.\qedhere
  \]
\end{proof}

Now suppose that \( d_Y \) is a compatible \emph{bounded} metric on \(Y\); for instance,
\[d_{Y}(y_{1}, y_{2}) = \min\{1, d'_{Y}(y_{1}, y_{2})\}, \quad y_{1}, y_{2} \in Y,\]
for an arbitrary compatible metric \(d'_{Y}\). We can then equip \( \LL^0(X,Y) \) with the metric
\( \tilde d_Y \), defined by
\begin{equation}\label{eq:d-tilde}
  \tilde d_Y(f,g) = \int_X d_Y(f(x), g(x)) \, d\mu(x).
\end{equation}
The following properties of convergence in measure and \( \tilde d_Y \) are well-known.
\begin{lemma}\label{lem:basic-facts-on-L0}
  Let \(d_Y\) be a compatible bounded metric on \(Y\). The following properties hold:
  \begin{enumerate}
  \item\label{item:compatible-metric-L0} The metric \(\tilde d_Y\) is compatible with the topology
    of convergence in measure.
  \item\label{item:cv-in-measure-vs-pointwise} A sequence of functions \((f_n)_{n}\) converges in
    measure to \(f\) if and only if every subsequence of \((f_n)_{n}\) has a further subsequence
    that converges pointwise to \(f\).
  \end{enumerate}
\end{lemma}

\begin{proof}
  \eqref{item:compatible-metric-L0} We employ the neighborhood basis established in
  Lemma~\ref{lem:L0-nbhd-basis}. For any \(\epsilon > 0\), Markov's inequality yields
  \[
    \epsilon \cdot \mu(\{x \in X : d_Y(f(x), f_0(x)) \geq \epsilon\}) \leq \int_X d_Y(f(x), f_0(x))
    \, d\mu(x) = \tilde d_Y(f, f_0),
  \]
  demonstrating that the topology induced by \(\tilde d_Y\) refines the topology of convergence in
  measure.

  Conversely, let \(K > 0\) be a bound on \(d_Y\). If
  \[
    \mu(\{x \in X : d_Y(f(x), f_0(x)) \geq \frac{\epsilon}{K}\}) < \frac{\epsilon}{K},
  \]
  then \(\tilde d_Y(f, f_0) < \frac{\epsilon}{K} + \epsilon\). This shows that the topology of
  convergence in measure refines the topology induced by \(\tilde d_Y\). Consequently, the two
  topologies coincide.

  \eqref{item:cv-in-measure-vs-pointwise} We begin with the direct implication. Assume that
  \(f_n \to f\) in measure and consider a subsequence \((f_{n_k})_k\). By passing to a further
  subsequence (still denoted \((f_{n_k})_k\) for simplicity), Lemma~\ref{lem:L0-nbhd-basis} lets us
  assume that for all \(k \in \N\), the set
  \[
    A_k = \{x \in X : d_Y(f_{n_k}(x), f(x)) \geq 2^{-k}\}
  \]
  has measure less than \(2^{-k}\). By the Borel--Cantelli lemma, almost every \(x \in X\) belongs
  to only finitely many \(A_k\). This implies that \(f_{n_k}\) converges pointwise to \(f\) almost
  surely.

  For the converse, assume that every subsequence of \((f_n)_n\) admits a further subsequence that
  converges pointwise to \(f\). Note that \(\tilde d_Y(f_n, f) \to 0\) holds if and only if every
  subsequence \((f_{n_{k}})_{k}\) of \((f_n)_n\) has a further subsequence \((f_{n_{k_{i}}})_i\)
  such that \(\tilde d_Y(f_{n_{k_{i}}}, f) \to 0\). Thus, it suffices to show that if \((f_n)_{n}\)
  converges to \(f\) pointwise, then also \(\tilde d_Y(f_n, f) \to 0\). This follows directly from
  the Lebesgue's dominated convergence theorem, in view of the boundedness of \(d_Y\).
\end{proof}

We finish this section with the following lemma.

\begin{lemma}\label{lem:continuity-on-L0}
  Let \(G\) be a Polish group acting continuously on a Polish space \(X\), and let \(\mu\) be a
  Borel probability measure on \(X\). Then the map
  \begin{displaymath}
    \begin{aligned}
      \Phi: \LL^{0}(X,G) &\to \LL^{0}(X,X)\\
      \LL^{0}(X,G) \ni f  &\mapsto (x\mapsto f(x)\cdot x)\\
    \end{aligned}
  \end{displaymath}
  is continuous.
\end{lemma}
\begin{proof}
  We use the pointwise characterization of convergence in measure, as stated in
  item~\eqref{item:cv-in-measure-vs-pointwise} of Lemma~\ref{lem:basic-facts-on-L0}. Suppose
  \(f_n \to f\) in measure, where \(f_n, f \in \LL^0(X,G)\). To show that \(\Phi(f_n) \to \Phi(f)\)
  in measure, consider an arbitrary subsequence \((\Phi(f_{n_k}))_k\).

  Since \(f_n \to f\) in measure, there exists a further subsequence \((f_{n_{k_i}})_i\) that
  converges pointwise to \(f\). By the continuity of the action, we have
  \(f_{n_{k_i}}(x) \cdot x \to f(x) \cdot x\) for all \(x \in X\). This means
  \((\Phi(f_{n_{k_i}}))_i\) converges pointwise to \(\Phi(f)\), as required.
\end{proof}

\section{\texorpdfstring{\( \LL^{1} \)}{L1} spaces of pointed metric spaces}\label{sec:l1-spaces-with-metric-values}

We now restrict our attention to integrable maps. To this end, we require the target space to be a
\textbf{Polish pointed metric space}, which we define as a separable complete metric space
\( (Y, d_Y) \) equipped with a distinguished basepoint \( e \in Y \).

\begin{definition}\label{def:pointed-L1}
  Let \( (X, \mu) \) be a standard probability space, and let \( (Y, d_Y, e) \) be a Polish pointed
  metric space. We define the \textbf{\( e \)-pointed \( \LL^1 \) space}\index{L1@\(\LL^{1}\)!space}
  \( \LL^{1}_{e}(X,\mu, Y) \), often denoted by \( \LL^{1}_{e}(X, Y) \) for brevity, as the pointed
  metric space consisting of all measurable functions \( f: X \to Y \) satisfying
  \[
    \int_X d_Y(e, f(x)) \, d\mu(x) < +\infty,
  \]
  equipped with the metric
  \[
    \tilde{d}_Y(f_1, f_2) = \int_X d_Y\bigl(f_1(x), f_2(x)\bigr) \, d\mu(x),
  \]
  and with the constant function \( \hat{e}: x \mapsto e \) as its basepoint. The finiteness of the
  integral in the definition of \(\tilde{d}_{Y}\) follows from the triangle inequality, using
  \( \hat{e} \) as an intermediate point.
\end{definition}

\begin{remark}
  To simplify notation, we omit explicit reference to the metric \( d_Y \) in the notation for the
  \( \LL^1 \) space \( \LL^{1}_{e}(X, Y) \). However, it is important to note that the definition of
  this space fundamentally depends on the choice of \( d_Y \).
\end{remark}

\begin{proposition}\label{prop:l1-is-Polish-space}
  Let \( (X, \mu) \) be a standard probability space and \( (Y,d_{Y},e) \) be a Polish pointed
  metric space.  Then \( (\LL^{1}_{e}(X,Y),\tilde{d}_{Y}) \) is a Polish metric space.
\end{proposition}

\begin{proof}
  The proof follows the classical argument establishing that
  \( (\LL^{1}(X, \mathbb{R}), \tilde{d}_{\mathbb{R}}) \) is a Polish metric space. To prove
  completeness, consider a Cauchy sequence \( (f_{n})_{n} \) in \( \LL^{1}_e(X, Y) \). Without loss
  of generality, we may assume that \( \tilde{d}_{Y}(f_{n}, f_{n+1}) < 2^{-n} \),
  \( n \in \mathbb{N} \). Define the sets
  \[ A_{n} = \{x \in X : d_{Y}(f_{n}(x), f_{n+1}(x)) \ge 1/n^{2} \}, \quad n \ge 1.\]
  By Markov's inequality, \( \mu(A_{n}) \le n^{2} 2^{-n} \), and thus
  \( \sum_{n} \mu(A_{n}) < \infty \). The Borel--Cantelli lemma implies that \( (f_n(x))_{n} \) is
  pointwise Cauchy for almost every \( x \in X \). Since \( (Y, d_{Y}) \) is complete, the pointwise
  limit \(f(x) = \lim_{n}f_{n}(x)\) exists almost surely.

  Define functions \( h_{n}, h : X \to \mathbb{R}^{\ge 0} \) by
  \[ h_{n}(x) = \sum_{i < n}d_{Y}(f_{i}(x), f_{i+1}(x)), \quad h(x) = \sum_{i \in \mathbb{N}}
    d_{Y}(f_{i}(x), f_{i+1}(x)) = \lim_{n \to \infty} h_{n}(x), \]
  and note that \( h \in \LL^{1}(X, \mathbb{R}) \) by Fatou's lemma. Finally, we conclude that
  \begin{displaymath}
    \begin{aligned}
      \tilde{d}_{Y}(f_{n}, f)
      &= \int_{X}d_{Y}(f_{n}(x),f(x))\, d\mu(x) \le
        \int_{X} \sum_{k=n}^{\infty} d_{Y}(f_{k}(x), f_{k+1}(x))\, d\mu(x) \\
      &= \int_{X} (h(x) - h_{n}(x))\, d\mu(x) \to 0,
    \end{aligned}
  \end{displaymath}
  where the last convergence follows from Lebesgue's dominated convergence theorem.

  To establish separability, let \( D \subseteq Y \) be a countable dense set. The subspace of maps
  taking values in \( D \) is \( \tilde{d}_{Y} \)-dense (in fact, dense even in the much stronger
  \emph{sup} metric). The set of functions taking only finitely many values in \( D \) remains
  dense. By further restricting to functions measurable with respect to a dense countable subalgebra
  of the measure algebra on \( X \), we obtain a countable dense subset of \( \LL^{1}_{e}(X, Y) \).
\end{proof}
Note that \( \LL^{1}_{e}(X,Y) \) is a subset of \( \LL^{0}(X,Y) \). If \(d_Y\) is bounded, the
integrability condition becomes trivial, and we have \( \LL^1_e(X,Y) = \LL^0(X,Y)\). Combining
item~\eqref{item:compatible-metric-L0} from Lemma~\ref{lem:basic-facts-on-L0} with the preceding
proposition, we obtain the following well-known result.

\begin{proposition}\label{prop:polish-topo-cv-measure} If \( Y \) is a Polish space, then the
  topology of convergence in measure on \(\LL^0(X,Y)\) is Polish,
  and it is induced by the \( \LL^1 \) metric \( \tilde d_Y \) for any bounded compatible
  metric \( d_Y \) on \( Y \). \qed
\end{proposition}

More generally, we can establish a relationship between the \(\LL^1\) topology and the topology of
convergence in measure as follows.

\begin{proposition}\label{prop:L1-in-L0} Let \((Y,d_Y,e)\) be a possibly unbounded Polish pointed
  metric space. The inclusion map \(\LL^1_{e}(X,Y)\hookrightarrow \LL^0(X,Y)\) is continuous.
\end{proposition}
\begin{proof}
  Define \(d^b_Y(y_{1}, y_{2}) = \min\{d_Y(y_{1},y_{2}),1\}\), \(y_{1}, y_{2} \in Y\), to be the
  bounded complete metric on \(Y\) obtained by capping \(d_{Y}\). By
  Lemma~\ref{lem:basic-facts-on-L0}, \(\tilde{d}^b_Y\) induces the topology of convergence in
  measure on \(\LL^0(X,Y)\). Since \(d^b_Y \leq d_Y\), one also has
  \( \tilde{d}^b_Y \leq \tilde{d}_Y\). Thus, the inclusion map
  \[
    (\LL^1_{e}(X,Y),\tilde{d}_Y) \hookrightarrow (\LL^0(X,Y),\tilde{d}^b_Y)
  \]
  is \(1\)-Lipschitz and, in particular, continuous.
\end{proof}

\begin{remark}\label{rem:L1-in-L0-may-not-be-embedding}
The inclusion \(\LL^1_{e}(X,Y) \hookrightarrow \LL^0(X,Y)\) is not, in general, an embedding. To see
this, suppose \(d_{Y}\) is unbounded. Then there exist elements \(y_{n} \in Y\) such that
\(d_{Y}(y_{n},e) \ge 2^{n}\) for all \(n\). Partition the space \(X\) into disjoint sets
\(\bigsqcup_{n=1}^{\infty} A_{n}\), where \(\mu(A_{n}) = 2^{-n}\). Define functions \(f_{n}\) for
\(n \ge 1\) as follows:
\begin{displaymath}
  f_{n}(x) =
  \begin{cases}
    y_{n} & \text{if } x \in A_{n}, \\
    e     & \text{otherwise}.
  \end{cases}
\end{displaymath}
In the notation of Proposition~\ref{prop:L1-in-L0}, we have
\(\tilde{d}_{Y}^{b}(f_{n},\hat{e}) = \mu(A_{n}) \to 0\) as \(n \to \infty\), which implies
\(f_{n} \to \hat{e}\) in \(\LL^{0}\). However, since \(\tilde{d}_{Y}(f_{n},\hat{e}) \ge 1\) for all
\(n \ge 1\), this convergence does not hold in \(\LL^{1}\).
\end{remark}

The group of measure-preserving automorphisms \( \Aut(X, \mu) \) acts naturally on
\( \LL^{1}_{e}(X, Y) \) by composition, i.e., \( (T \cdot f)(x) = f(T^{-1}x) \). This action fixes
the basepoint \(\hat e\). Moreover, each automorphism acts by an isometry, as it preserves the
measure \(\mu\).

\begin{proposition}\label{prop:aut-acts-continuously}
  Let \( (X, \mu) \) be a standard probability space, and let \( (Y, d_{Y}, e) \) be a Polish
  pointed metric space. The action of \( \Aut(X, \mu) \) on \( \LL_{e}^{1}(X, Y) \) is continuous.
\end{proposition}

\begin{proof}
  The argument follows a similar approach to that in~\cite[Prop.~2.9(1)]{MR3464151}. Consider
  sequences \( T_{n} \to T \) and \( f_{n} \to f \). We need to show that
  \( T_{n} \cdot f_{n} \to T \cdot f \). Since the action is by isometries, we have
  \[
    \tilde{d}_{Y}(T_{n}\cdot f_{n}, T \cdot f) = \tilde{d}_{Y}(f_{n}, T_{n}^{-1} T \cdot f ) \le
    \tilde{d}_{Y}(f_{n}, f) + \tilde{d}_{Y}(f, T_{n}^{-1} T \cdot f).
  \]
  Thus, it suffices to prove that for any \( f \in \LL^{1}_{e}(X, Y) \) and any convergent sequence
  of automorphisms \( T_{n} \to T \), the term \( \tilde{d}_{Y}(f, T_{n}^{-1} T \cdot f) \) tends to
  0 as \( n \to \infty \).

  To establish this, it is enough to verify the claim for functions that take only finitely many
  values, as such functions are dense in \( \LL^{1}_{e}(X, Y) \). Suppose \( f \) is a step function
  defined over a partition \( X = \bigsqcup_{i=1}^{m}A_{i} \). The convergence \( T_{n} \to T \)
  implies \( \mu(T_{n}^{-1}T(A_{i}) \triangle A_{i}) \to 0 \) for all \( 1 \le i \le m \), which
  readily gives \( \tilde{d}_{Y}(f, T_{n}^{-1} T \cdot f) \to 0 \).
\end{proof}

In what follows, we identify \(\Aut(X,\mu)\) with a subset of \(\LL^0(X,X)\).

\begin{proposition}\label{prop:Aut-embeds}
  Let \(\tau\) be a Polish topology on a standard probability space \((X,\mu)\) compatible with its
  Borel structure. The inclusion map
  \[
    \Aut(X,\mu) \hookrightarrow \LL^{0}(X,X)
  \]
  is a topological embedding when \(\Aut(X,\mu)\) is equipped with the weak topology and
  \(\LL^{0}(X,X)\) is equipped with the topology of convergence in measure associated with \(\tau\).
\end{proposition}

\begin{proof}
  Fix a bounded complete metric \(d\) compatible with \(\tau\). By
  Lemma~\ref{lem:basic-facts-on-L0}, \(\tilde{d}\) induces the topology of convergence in measure on
  \(\LL^{0}(X,X)\).  The topological group \(\Aut(X,\mu)\) acts on \(\LL^0(X,X)\) by
  \(T \cdot f = f \circ T^{-1}\) and this action is continuous by
  Proposition~\ref{prop:aut-acts-continuously}.  Furthermore, for every \(T \in \Aut(X,\mu)\), we
  have \(T^{-1} \cdot \mathrm{id}_X = T \in \LL^0(X,X)\), which shows that the inclusion map is
  continuous.

  The metric \(\tilde{d}\) induces a right-invariant metric on \(\Aut(X,\mu)\). To prove that the
  inclusion map is a topological embedding, it suffices to show that if
  \(\tilde{d}(T_n, \mathrm{id}_X) \to 0\), then \(T_n \to \mathrm{id}_X\) weakly in
  \(\Aut(X,\mu)\). By Lemma~\ref{lem:basic-facts-on-L0}, \(\tilde{d}(T_n, \mathrm{id}_X) \to 0\)
  implies \(T_n \to \mathrm{id}_X\) in measure. In particular, for every \(\tau\)-open set
  \(U \subseteq X\), the definition of convergence in measure yields
  \[
    \liminf_{n \to \infty} \mu(T_n^{-1}(U) \cap U) \geq \mu(\mathrm{id}_X^{-1}(U) \cap U) = \mu(U).
  \]
  We conclude that
  \[\mu(U) \ge \limsup_{n \to \infty} \mu(T_{n}^{-1}(U) \cap U) \ge \liminf_{n \to \infty}
    \mu(T_n^{-1}(U) \cap U) \geq \mu(U),\]
  and therefore \(\lim_{n}\mu(T_{n}^{-1}(U) \cap U) = \mu(U) \).  Since each \(T_n\) is
  measure-preserving, it follows that \(\lim_{n} \mu(T_n^{-1}(U) \bigtriangleup U) = 0\) for every
  \(U \in \tau\). By the regularity of \(\mu\), this implies \(T_n \to \mathrm{id}_X\) weakly, as
  required.
\end{proof}

Let us now return to \(\LL^1\) spaces associated with possibly unbounded pointed metric spaces
\((Y, d_Y, e)\). When \(Y\) is a Polish group, there is a natural choice for \(e\), namely the
identity element of the group, which we also denote by \(e\). In this case, we simplify the notation
further and write \(\LL^1(X, Y)\).

Recall that a \textbf{Polish normed group}\index{Polish group!normed} is a Polish group equipped
with a compatible norm (see Appendix~\ref{sec:norms-on-groups}). In particular, if
\((G, \snorm{\cdot})\) is a Polish normed group, there is a canonical choice of a compatible
complete metric on \(G\), namely
\[
  d_G(u, v) = \frac{\snorm{u^{-1}v} + \snorm{vu^{-1}}}{2}.
\]
The corresponding space \(\LL^1(X, G)\) is Polish by
Proposition~\ref{prop:l1-is-Polish-space}. Furthermore, it forms a Polish group under pointwise
operations, as we now demonstrate.

\begin{proposition}\label{prop:l1-is-a-Polish-group}
  Let \( (G, \snorm{\cdot}) \) be a Polish normed group. The space \( \LL^{1}(X, G) \) is a Polish
  normed group under the pointwise operations,
  \[ (f \cdot g)(x) = f(x) g(x), \quad f^{-1}(x) = f(x)^{-1},\]
  and the norm \( \snorm{f}_{1}^{\LL^{1}(X, G)} = \int_{X} \snorm{f(x)} \, d\mu(x) \).
\end{proposition}

\begin{proof}
  The space \(\LL^{1}(X, G)\) consists of all measurable functions \(f : X \to G\) with finite norm,
  \(\snorm{f}_{1}^{\LL^{1}(X, G)} < \infty\). Using the properties of the norm \(\snorm{\cdot}\) on
  \(G\), we have
  \begin{displaymath}
    \begin{aligned}
      \snorm{f \cdot g}_{1}^{\LL^{1}(X, G)}
      &= \int_{X} \snorm{f(x) g(x)} \, d\mu(x) \\
      &\le \int_{X} \bigl(\snorm{f(x)} + \snorm{g(x)}\bigr) \, d\mu(x) \\
      &= \snorm{f}_{1}^{\LL^{1}(X, G)} + \snorm{g}_{1}^{\LL^{1}(X, G)}, \\
      \snorm{f^{-1}}_{1}^{\LL^{1}(X, G)}
      &= \int_{X} \snorm{f(x)^{-1}} \, d\mu(x) \\
      &= \int_{X} \snorm{f(x)} \, d\mu(x) = \snorm{f}_{1}^{\LL^{1}(X, G)}. \\
    \end{aligned}
  \end{displaymath}
  Thus, \(\LL^{1}(X, G)\) is closed under the group operations, and
  \(\snorm{\,\cdot\,}_{1}^{\LL^{1}(X, G)}\) defines a group norm on it.

  To verify the continuity of the group operations, it suffices to show that for any
  \(g \in \LL^{1}(X, G)\) and any sequence \(f_{n} \in \LL^{1}(X, G)\), \(n \in \mathbb{N}\),
  converging to the identity function \(\hat{e}\) (i.e.,
  \(\snorm{f_{n}}_{1}^{\LL^{1}(X, G)} \to 0\)), there exists a subsequence \((f_{n_{k}})_{k}\) such
  that
  \[\snorm{g \cdot f_{n_{k}} \cdot g^{-1}}_{1}^{\LL^{1}(X, G)} \to 0 \textrm{ as } k \to \infty,\]
  (see, for example,~\cite[Thm~3.4 and Lem.~3.5]{MR2743093}).

  By Proposition~\ref{prop:L1-in-L0} and the fact that convergence in measure implies pointwise
  convergence of a subsequence (see item~\eqref{item:cv-in-measure-vs-pointwise} of
  Lemma~\ref{lem:basic-facts-on-L0}), there exists a subsequence \((f_{n_{k}})_{k}\) such that
  \(f_{n_{k}}(x) \to e\) for almost all \(x \in X\). By passing to a further subsequence, we may
  assume without loss of generality that
  \(\sum_{k} \snorm{f_{n_{k}}}_{1}^{\LL^{1}(X, G)} < +\infty\). The function
  \(M(x) = \sum_{k} \snorm{f_{n_{k}}(x)}\) belongs to \(\LL^{1}(X, G)\), and for all
  \(k \in \mathbb{N}\) and \(x \in X\),
  \[
    \snorm{g(x) f_{n_{k}}(x) g(x)^{-1}} \le 2\snorm{g(x)} + \snorm{f_{n_{k}}(x)} \le 2\snorm{g(x)} +
    M(x).
  \]
  The continuity of the group operations on \(G\) ensures that
  \(g \cdot f_{n_{k}} \cdot g^{-1} \to \hat{e}\) pointwise. It remains to apply Lebesgue's dominated
  convergence theorem and conclude that
  \(\snorm{g \cdot f_{n_{k}} \cdot g^{-1}}_{1}^{\LL^{1}(X, G)} \to 0\), as required.
\end{proof}

\begin{remark}
  If the chosen compatible norm on \(G\) is bounded, then \(\LL^1(X,G) = \LL^0(X,G)\), and the
  topology under consideration coincides with the topology of convergence in measure by
  Lemma~\ref{lem:basic-facts-on-L0}. Consequently, we recover the well-known fact that
  \(\LL^0(X,G)\) is a Polish group when equipped with the topology of convergence in measure.
\end{remark}


\chapter{Hopf decomposition}\label{sec:hopf-decomposition-appendix}

An important tool in the theory of invertible non-singular transformations on \( \sigma \)-finite
measure spaces is the Hopf decomposition, which partitions the phase space into the so-called
\emph{dissipative} and \emph{recurrent} parts, reflecting different dynamics of the transformation.
In this appendix, we recall the relevant definitions and indicate what happens for
measure-preserving transformations of a \( \sigma \)-finite space.  The reader may
consult~\cite[Sec.~1.3]{MR797411} for further details on the following definitions.

\begin{definition}\label{def:dissipative-recurrent}
  Let \( S \) be an invertible non-singular transformation of a \( \sigma \)-finite measure space
  \( (\Omega, \lambda) \). A measurable set \( A \subseteq \Omega \) is said to be:
  \begin{itemize}
  \item \textbf{wandering} if \( A \cap S^{k}(A) = \varnothing \) for all \( k \ge 1 \);
  \item \textbf{recurrent} if \( A \subseteq \bigcup_{k \ge 1}S^{k}(A) \);
  \item \textbf{infinitely recurrent} if
    \( A \subseteq \bigcap_{n \ge 1} \bigcup_{k \ge n} S^{k}(A) \).
  \end{itemize}
  The inclusions above are understood to hold up to a null set.  The transformation \( S \) is:
  \begin{itemize}
  \item \textbf{dissipative}\index{Transformation!dissipative}\index{Dissipative transformation} if the phase space \( \Omega \) is a countable union of wandering sets;
  \item \textbf{conservative}\index{Transformation!conservative}\index{Conservative transformation} if there are no wandering sets of positive measure;
  \item \textbf{recurrent} if every set of positive measure is recurrent;
  \item \textbf{infinitely recurrent} if every set of positive measure is infinitely recurrent.
  \end{itemize}
\end{definition}

It turns out that the properties of being conservative, recurrent, and infinitely recurrent are all
mutually equivalent.
\begin{proposition}\label{prop:equivalence-conservative-recurrent}
  Let \( S \) be an invertible non-singular transformation of a \( \sigma \)-finite measure space
  \( (\Omega, \lambda) \).  The following are equivalent:
  \begin{enumerate}
  \item \( S \) is conservative;
  \item \( S \) is recurrent;
  \item \( S \) is infinitely recurrent.
  \end{enumerate}
\end{proposition}

Among the properties introduced in Definition~\ref{def:dissipative-recurrent}, only recurrence and
dissipativity are therefore different.  In fact, any non-singular transformation admits a canonical
decomposition, known as the Hopf decomposition, into these two types of action.

\begin{proposition}[Hopf decomposition]\label{prop:hopf-decomposition}
  Let \( S \) be an invertible non-singular transformation of a \( \sigma \)-finite measure space
  \( (\Omega, \lambda) \). There exists an \( S \)-invariant partition \( \Omega = D \sqcup C \)
  such that \( S\restriction_{D} \) is dissipative and \( S\restriction_{C} \) is recurrent
  (equivalently, conservative). Moreover, if \( \Omega = D' \sqcup C' \) is another partition with
  this property then \( \lambda(D \triangle D') = 0 \) and \( \lambda(C \triangle C') = 0\).
\end{proposition}

We also note the following consequence of dissipativity in case the measure is preserved.

\begin{proposition}\label{prop:dissipative-property}
  Let \( S \) be an invertible measure-preserving transformation of a \( \sigma \)-finite measure
  space \( (\Omega, \lambda) \) and let \( \Omega = D \sqcup C \) be its Hopf decomposition. For
  every set \( A \subseteq \Omega \) of finite measure, almost every point in \( D \) eventually
  escapes \( A \): \[ \forall^{\lambda}x \in D\ \exists N\ \forall n \ge N\ T^{n}x \not \in A. \]
\end{proposition}

\begin{proof}
  We may as well assume \(D=\Omega\). Let \(A\subseteq \Omega\) have finite measure.  Let \(Q\) be a
  wandering set whose translates cover \(\Omega\)
  as provided by~\cite[Prop.~1.1.2]{aaronsonIntroductionInfiniteErgodic1997}.
  Consider the map \(\Phi:Q\times \Z\to \Omega\) which
  maps \((x,n)\) to \(T^n(x)\), and observe that \(\Phi\) is measure-preserving if we endow \(Q\times\Z\)
  with the product of the measure induced by \(\lambda\) on \(Q\) and the counting measure on \(\Z\).

  If the set of \(x \in Q\) for which \(S^n(x) \in A\) for infinitely many \(n \in \mathbb{N}\) has
  positive measure, then, by Fubini's theorem, the set \(A\) must have infinite measure, which
  contradicts the assumption. The same reasoning applies to any \(S\)-translate of \(Q\). Since
  these translates cover \(\Omega\), the proof is complete.
\end{proof}


\chapter{Disintegration of measure}\label{sec:disintegration-measure}

Let \( \mathcal{R} \) be a smooth measurable equivalence relation on a standard Lebesgue space
\( (X, \mu) \), and let \( \pi : X \to Y \) be a measurable reduction to the identity relation on
some standard Lebesgue space \( (Y, \nu) \), \( \pi(x) = \pi(y) \) if and only if
\( x \mathcal{R} y \). Suppose that \( \nu \) is a \(\sigma\)-finite measure on \( Y \) that is
equivalent to the push-forward \( \pi_{*}\mu \). A \textbf{disintegration} of \( \mu \) relative to
\( (\pi, \nu) \) is a collection of measures \( (\mu_{y})_{y \in Y} \) on \( X \) such that for all
Borel sets \( A \subseteq X \):
\begin{enumerate}
\item\label{item:measure-on-fiber} \(\mu_{y}(X \setminus \pi^{-1}(y)) = 0 \) for \( \nu \)-almost
  all \( y \in Y \);
\item\label{item:measurability-condition} the map \( Y \ni y \mapsto \mu_{y}(A) \in \mathbb{R} \) is
  measurable;
\item\label{item:disintegration-formula} \( \mu(A) = \int_{Y} \mu_{y}(A)\, d\nu(y) \).
\end{enumerate}

A theorem of D.~Maharam~\cite{MR36817} asserts that \( \mu \) can be disintegrated over any
\( (\pi, \nu) \) as above. In fact, the existence of a disintegration can be proved in a
considerably more general setup (see, for example, D.~H.~Fremlin~\cite[Thm.~452I]{MR2462372}), but
in the framework of standard Lebesgue spaces, disintegration is essentially unique. While the
context of our work is purely ergodic-theoretical, we note that the disintegration result holds in
the descriptive set-theoretical setting as well, as discussed in~\cite{MR786682}
and~\cite{MR1012985}.  Without striving for generality, we formulate here a specific version that
suits our needs.

\begin{theorem}[Disintegration of Measure\index{Disintegration of measure}]\label{thm:disintegration-of-measure}
  Let \( (X,\mu) \) be a standard Lebesgue space, \( (Y, \nu) \) be a \( \sigma \)-finite standard
  Lebesgue space, and let \( \pi : X \to Y \) be a measurable function. If \( \pi_{*}\mu \) is
  equivalent to \( \nu \), then there exists a disintegration \( (\mu_{y})_{y \in Y} \) of \( \mu \)
  over \( (\pi, \nu) \). Moreover, such a disintegration is essentially unique in the sense that if
  \( (\mu'_{y})_{y \in Y} \) is another disintegration, then \( \mu_{y} = \mu_{y}' \) for
  \( \nu \)-almost all \( y \in Y \).
\end{theorem}

\begin{remark}\label{rem:radon-nykodim-disintegration}
  It is more common to formulate the disintegration theorem with the assumption that
  \( \pi_{*}\mu = \nu \), when one can additionally ensure that \( \mu_{y}(X) = \mu(X) \) for
  \( \nu \)-almost all \( y \). Weakening the equality \( \pi_{*}\mu = \nu \) to mere equivalence is
  a simple consequence, for if \( (\mu_{y})_{y \in Y} \) is a disintegration of \( \mu \) over
  \( (\pi, \pi_{*}\mu) \), then
  \( \bigl(\frac{d\pi_{*}\mu}{d\nu}(y) \cdot \mu_{y}\bigr)_{y \in Y} \) is a disintegration of
  \( \mu \) over \( (\pi, \nu) \).
\end{remark}

Let \( X_{a} \subseteq X \) be the set of atoms of the disintegration, i.e.,
\[ X_{a} = \{ x \in X: \mu_{y}(x) > 0 \textrm{ for some } y \in Y \} ,\]
and let \( F \) be the equivalence relation on \( X_{a} \), where two atoms within the same fiber
are equivalent whenever they have the same measure: \( x_{1} F x_{2} \) if and only if
\( \mu_{\pi(x_{1})}(x_{1}) = \mu_{\pi(x_{2})}(x_{2}) \) and \( \pi(x_{1}) = \pi(x_{2}) \). The
equivalence relation \( F \) is measurable and has finite classes \( \mu \)-almost surely. Let
\( X_{n} \), \( n \ge 1 \), be the union of \( F \)-equivalence classes of size exactly \( n \),
thus \( X_{a} = \bigsqcup_{n \ge 1} X_{n} \). Set also \( X_{0} = X \setminus X_{a} \) to be the
atomless part of the disintegration and let \( \mathcal{R}_{n} \) denote the restriction of
\( \mathcal{R} \) onto \( X_{n} \).

Consider the group \( \fgr{\mathcal{R}} \le \Aut(X, \mu) \) of measure-preserving bijections \(T\)
such that \( x \mathcal{R} Tx \) holds \( \mu \)-almost surely.  Every \( T \in \fgr{\mathcal{R}} \)
preserves \( \nu \)-almost all measures \( \mu_{y} \), since \( (T_{*}\mu_{y})_{y \in Y} \) is a
disintegration of \( T_{*}\mu = \mu \), which has to coincide with \( (\mu_{y})_{y \in Y} \) by the
uniqueness of the disintegration. In particular, the partition
\( X = \bigsqcup_{n \in \mathbb{N}} X_{n} \) is invariant under the full group
\( \fgr{\mathcal{R}} \), and for any \( T \in \fgr{\mathcal{R}} \), the restriction
\( T\restriction_{X_{n}} \in \fgr{\mathcal{R}_{n}} \) for every \( n \in \mathbb{N} \). Conversely,
for a sequence \( T_{n} \in \fgr{\mathcal{R}_{n}} \), \( n \in \mathbb{N} \), one has
\( T = \bigsqcup_{n} T_{n} \in \fgr{\mathcal{R}} \). We therefore have an isomorphism of (abstract)
groups \( \fgr{\mathcal{R}} \cong \prod_{n \in \mathbb{N}} \fgr{\mathcal{R}_{n}} \).

The groups \( \fgr{\mathcal{R}_{n}} \) can be described quite explicitly. First, consider the case
\( n \ge 1 \); thus \( X_{n} \subseteq X_{a} \). All equivalence classes of the restriction of
\( F \) onto \( X_{n} \) have size \( n \). Let \( Y_{n} \subseteq X_{n} \) be a measurable
transversal, i.e., a measurable set intersecting every \( F \)-class in a single point, and let
\( \nu_{n} = \mu\restriction_{Y_{n}} \). Every \( T \in \fgr{\mathcal{R}_{n}} \) produces a
permutation of \( \mu \)-almost every \( F \)-class, so we can view it as an element of
\( \LL^{0}(Y_{n}, \nu_{n}, \mathfrak{S}_{n}) \), where \( \mathfrak{S}_{n} \) is the group of
permutations of an \( n \)-element set. This identification works in both directions and produces an
isomorphism \( \fgr{\mathcal{R}_{n}} \cong \LL^{0}(Y_{n}, \nu_{n}, \mathfrak{S}_{n}) \). Note also
that all \( \nu_{n} \) are atomless if so is \( \mu \). We allow for \( \mu(X_{n}) = 0 \), in which
case \( \LL^{0}(Y_{n}, \nu_{n}, \mathfrak{S}_{n}) \) is the trivial group.

Let us now return to the equivalence relation
\( \mathcal{R}_{0} = \mathcal{R} \cap X_{0} \times X_{0} \), and recall that the measures
\( \mu_{y}\restriction_{X_{0}} \) are atomless. Let \( Y_{0} = \{ y : \mu_{y}(X_{0}) > 0 \} \) be
the encoding of fibers with non-trivial atomless components, and put
\( \nu_{0} = \nu\restriction_{Y_{0}} \). In particular, for every \( y \in Y_{0} \), the space
\( (X_{0}, \mu_{y}) \) is isomorphic to the interval \( [0, \mu_{y}(X_{0})] \) endowed with the
Lebesgue measure. In fact, one can select such isomorphisms in a measurable way across all
\( y \in Y_{0} \). More precisely, there is a measurable isomorphism
\[ \psi : X_{0} \to \{ (y, r) \in Y_{0} \times \mathbb{R} : 0 \le r \le \mu_{y}(X_{0}) \} \]
such that for all \( y \in Y_{0} \),
\begin{itemize}
\item \( \psi(\pi^{-1}(y) \cap X_{0}) = \{ y \} \times [0, \mu_{y}(X_{0})] \);
\item \( \psi_{*}(\mu_{y}\restriction_{X_{0}}) \) coincides with the Lebesgue measure on
  \( \{ y \} \times [0, \mu_{y}(X_{0})] \).
\end{itemize}
The reader may find further details in~\cite[Thm.~2.3]{MR1012985}, where the same construction is
discussed in a more refined setting of Borel disintegrations.

Using the isomorphism \( \psi \), each \( \pi^{-1}(y) \cap X_{0} \), \( y \in Y_{0} \), is
identified with \( [0,\mu_{y}(X_{0})] \). Since every \( T \in \fgr{\mathcal{R}_{0}} \) preserves
\( \nu \)-almost every \( \mu_{y} \), we may rescale these intervals and view any
\( T \in \fgr{\mathcal{R}_{0}} \) as an element of
\( \LL^{0}(Y_{0}, \nu_{0}, \Aut([0,1], \lambda)) \). Conversely, every
\( f \in \LL^{0}(Y_{0}, \nu_{0}, \Aut([0,1], \lambda)) \) gives rise to
\( T_{f} \in \fgr{\mathcal{R}_{0}} \) via the notationally convoluted but natural
\[ T_{f}(x) = \psi^{-1}\bigl(\pi(x), \bigl(f(\pi(x)) \cdot \proj_{2}(\psi(x))/
  \mu_{\pi(x)}(X_{0})\bigr) \mu_{\pi(x)}(X_{0})\bigr), \]
which, in plain words, simply applies \( f(\pi(x)) \) upon the corresponding fiber identified with
\( [0,1] \) using \( \psi \). This map is an isomorphism between the groups
\( \fgr{\mathcal{R}_{0}} \) and \( \LL^{0}(Y_{0}, \nu_{0}, \Aut([0,1], \lambda)) \).

Let us say that \( \mathcal{R} \) has \textbf{atomless classes} if \( \mu_{y} \) is atomless
\( \nu \)-almost surely or, equivalently, if \( \mu(X_{a}) = 0 \) in the notation above. We may
summarize the discussion so far into the following proposition.

\begin{proposition}\label{prop:full-group-smooth-L0}
  Let \( \mathcal{R} \) be a smooth measurable equivalence relation on a standard Lebesgue space
  \( (X, \mu) \). There are (possibly empty) standard Lebesgue spaces \( (Y_{n}, \nu_{n}) \),
  \( n \in \mathbb{N} \), such that the full group \( \fgr{\mathcal{R}} \le \Aut(X, \mu) \) is
  (abstractly) isomorphic to
  \[ \LL^{0}(Y_{0}, \nu_{0}, \Aut([0,1], \lambda)) \times \prod_{n \ge 1} \LL^{0}(Y_{n}, \nu_{n},
    \mathfrak{S}_{n}), \]
  where \( \mathfrak{S}_{n} \) is the group of permutations of an \( n \)-element set. If \( \mu \)
  is atomless, then so are the spaces \( (Y_{n}, \nu_{n}) \), \( n \ge 1 \). If \( \mathcal{R} \)
  has atomless classes, then all \( (Y_{n}, \nu_{n}) \), \( n \ge 1 \), are negligible and
  \( \fgr{\mathcal{R}} \) is isomorphic to \( \LL^{0}(Y_{0}, \nu_{0}, \Aut([0,1], \lambda)) \).
\end{proposition}

We can further refine the product in Proposition~\ref{prop:full-group-smooth-L0} by decomposing the
spaces \( (Y_{n}, \nu_{n}) \) into individual atoms and the atomless remainders.  This relies on the
following general result.

\begin{proposition}\label{prop:l0-partition-of-base}
  Let \( (Z, \nu_{Z}) \) be a standard Lebesgue space, and let \( G \) be a Polish group. For any
  finite or countably infinite measurable partition \( Z = \bigsqcup_{n \in I} Z_{n} \), there
  exists an isomorphism of topological groups between \( \LL^{0}(Z, \nu_{Z}, G) \) and
  \( \prod_{n \in I} \LL^{0}(Z_{n}, \nu_{Z,n}, G) \), where \( \nu_{Z,n} \) is the restriction of
  \( \nu_{Z} \) onto \( Z_{n} \).
\end{proposition}
\begin{proof}
  Consider the map \(\Phi\) that assigns to each \( f \in \LL^{0}(Z, \nu_{Z}, G) \) the sequence of
  its restrictions \( f\restriction_{Z_{n}} \in \LL^{0}(Z_{n}, \nu_{Z,n}, G) \), where \(n \in
  I\). It is straightforward to verify that \(\Phi\) is a group isomorphism, and its continuity
  follows directly from the definition of convergence in measure. Automatic continuity implies that
  \(\Phi\) is a homeomorphism, as it is a Borel group isomorphism between Polish groups (see
  \cite[Sec.~1.6]{MR1425877}). Alternatively, the continuity of \(\Phi^{-1}\) can easily be checked
  directly.
\end{proof}

Applying Proposition~\ref{prop:l0-partition-of-base} to the partition of \( (Z, \nu_{Z}) \) into the
atomless part \( Z_{0} \) and individual atoms \( Z_{k} = \{ z_{k} \} \) (if any), and noting that
for a singleton \( Z_{k} \) the group \( \LL^{0}(Z_{k}, \nu_{Z,k}, G) \) is naturally isomorphic to
\( G \), we get the following corollary.

\begin{corollary}\label{cor:l0-split-base-atoms}
  Let \( (Z, \nu_{Z}) \) be a standard Lebesgue space, and let \( G \) be a Polish group. Let
  \( Z_{a} \subseteq Z \) be the set of atoms of \( Z \), and let \( Z_{0} = Z \setminus Z_{a} \) be
  the atomless part. The group \( \LL^{0}(Z, \nu_{Z}, G) \) is isomorphic to
  \( \LL^{0}(Z_{0}, \nu_{Z} \restriction_{Z_{0}}, G) \times G^{|Z_{a}|} \).
\end{corollary}

Combining the above discussion with Proposition~\ref{prop:full-group-smooth-L0}, we obtain a very
concrete representation for \( \fgr{\eqr} \). In the formulation below, \( G^{0} \) is understood to
be the trivial group.

\begin{proposition}\label{prop:full-group-smooth-representation}
  Let \( \mathcal{R} \) be a smooth measurable equivalence relation on a standard Lebesgue space
  \( (X, \mu) \). There exist cardinals \( \kappa_{n} \le \aleph_{0} \) and
  \( \epsilon_{n} \in \{ 0,1 \} \) such that
  \begin{multline*}
    \fgr{\eqr} \cong \LL^{0}([0,1], \lambda, \Aut([0,1], \lambda))^{\epsilon_{0}} \times
    \Aut([0,1], \lambda)^{\kappa_{0}} \\ \times
    \Bigl( \prod_{n \ge 1} \LL^{0}([0,1], \lambda, \mathfrak{S}_{n})^{\epsilon_{n}} \times
    \mathfrak{S}_{n}^{\kappa_{n}}\Bigr).
  \end{multline*}

  If \( \mu \) is atomless, then \( \kappa_{n} = 0 \) for all \( n \ge 1 \); if \( \eqr \) has
  atomless classes, then \( \epsilon_{n} = 0 \) for all \( n \ge 1 \).
\end{proposition}

So far, we have considered \( \fgr{\mathcal{R}} \) as an abstract group. This is because neither of
the two natural topologies on \( \Aut(X, \mu) \) interacts well with the full group
construction---\( \fgr{\mathcal{R}} \) is generally not closed in the weak topology, and is not
separable in the uniform topology whenever \( \mu(X_{0}) > 0 \). Nonetheless, the isomorphism given
in Proposition~\ref{prop:full-group-smooth-L0} shows that there is a natural Polish topology on
\( \fgr{\mathcal{R}} \), which arises when we view the groups
\( \LL^{0}(Y_{0}, \nu_{0}, \Aut([0,1], \lambda)) \) and
\( \LL^{0}(Y_{n}, \nu_{n}, \mathfrak{S}_{n}) \) as Polish groups in the topology of convergence in
measure. It is with respect to this topology that we formulate
Proposition~\ref{prop:full-group-smooth-generated-periodic}.

\begin{proposition}\label{prop:full-group-smooth-generated-periodic}
  Let \( \mathcal{R} \) be a smooth measurable equivalence relation on a standard Lebesgue space
  \( (X, \mu) \). The set of periodic elements is dense in \( \fgr{\mathcal{R}} \).
\end{proposition}

\begin{proof}
  Rokhlin's Lemma implies that any \( T \in \fgr{\mathcal{R}} \) can be approximated in the uniform
  topology by periodic elements from \( \fgr{T} \subseteq \fgr{\mathcal{R}} \). Since the uniform
  topology is stronger than the Polish topology on \( \fgr{\mathcal{R}} \), the proposition follows.
\end{proof}



\chapter{Actions of locally compact Polish groups}
\label{chap:appendix-actions-lcsc}

In this chapter of the appendix, we collect some well-known facts related to the actions of locally
compact Polish groups. This exposition is provided for the reader's convenience and completeness. We
recall that, by a result of G.~W.~Mackey~\cite{MR143874}, any Boolean measure-preserving action of a
locally compact Polish group can be realized as a spatial Borel action. Thus, we may switch to
pointwise formulations whenever convenient for the exposition.

\section{Ergodic decomposition}
\label{sec:ergod-decomp}

Let \( G \acts X \) be a measure-preserving action of a locally compact Polish group on a standard
probability space \( (X, \mu) \).  The space \( \mathcal{E} = \einv(G \acts X) \) of invariant
ergodic probability measures of this action possesses the structure of a standard Borel space. The
Ergodic Decomposition theorem\index{Ergodic decomposition} of
V.~S.~Varadarajan~\cite[Thm.~4.2]{MR159923} asserts that there exist an essentially unique Borel
\( G \)-invariant surjection \( X \ni x \mapsto \nu_{x} \in \mathcal{E} \) and a probability measure
\( p \) on \( \mathcal{E} \) such that \( \mu = \int_{\mathcal{E}}\nu\, dp(\nu) \).  This equality
holds in the sense that \( \mu(A) = \int_{\mathcal{E}}\nu(A)\, dp(\nu) \) for all Borel
\( A \subseteq X \).

There is a one-to-one correspondence between measurable \( G \)-invariant functions
\( h : X \to \mathbb{R} \) and measurable functions \( \tilde{h} : \mathcal{E} \to \mathbb{R} \),
given by \( \tilde{h}(\nu_{x}) = h(x) \). For measures \( \mu \) and \( p \) as above, this
correspondence gives an isometric isomorphism between \( \LL^{1}(\mathcal{E}, \mathbb{R}) \) and the
subspace of \( \LL^{1}(X, \mathbb{R}) \) consisting of \( G \)-invariant functions.

\section{Tessellations}
\label{sec:tessellations}
An important feature of locally compact group actions is the fact that they all admit Lebesgue measurable
cross-sections. This was proved by J.~Feldman, P.~Hahn, and C.~Moore in~\cite{MR492061}, whereas a
Borel version of the result was obtained by A.~S.~Kechris in~\cite{MR1176624}.

\begin{definition}
  \label{def:cross-section}
  Let \( G \acts X \) be a Borel action of a locally compact Polish group. A
  \textbf{cross-section}\index{Cross-section} is a Borel set \( \mathcal{C} \subseteq X \) which is
  both
  \begin{itemize}
  \item a \textbf{complete section} for \(\mathcal{R}_{G}\): it intersects every orbit of the
    action; and
  \item \textbf{lacunary}\index{Cross-section!lacunary}: there is a neighborhood of the identity \( 1_{G} \in U \subseteq G \)
  such that \(\mathcal C\) is \( U \)-lacunary, namely one has
    \( U \cdot c \cap U \cdot c' = \varnothing \) for all distinct \( c, c' \in \mathcal{C} \).
  \end{itemize}
  A cross-section \( \mathcal{C} \) is \( K \)-\textbf{cocompact}\index{Cross-section!cocompact},
  where \( K \subseteq G \) is a compact set, if \( K \cdot \mathcal{C} = X \); a cross-section is
  \textbf{cocompact} if it is \( K \)-cocompact for some compact \( K \subseteq G \).
\end{definition}
Any action \( G \acts X \) admits a \( K \)-cocompact cross-section, whenever \( K \subseteq G \) is
a compact neighborhood of the identity (see~\cite[Thm.~2.4]{MR3681992}). We also recall the
following well-known lemma on the possibility of partitioning a cross-section into pieces with a
prescribed lacunarity parameter.

\begin{lemma}
  \label{lem:lacunary-partition}
  Let \( G \acts X \) be a Borel action of a locally compact Polish group, and let \( \mathcal{C} \)
  be a cross-section for the action. For any compact neighborhood of the identity
  \( V \subseteq G \), there exists a finite Borel partition
  \( \mathcal{C} = \bigsqcup_{i} \mathcal{C}_{i} \) such that each \( \mathcal{C}_{i} \) is
  \( V \)-lacunary.
\end{lemma}

\begin{proof}
  Set \( K = (V \cup V^{-1})^{2} \), and let \( U \subseteq G \) be a compact neighborhood of the
  identity small enough for \( \mathcal{C} \) to be \( U \)-lacunary. Define a binary relation
  \( \mathcal{G} \) on \( \mathcal{C} \) by declaring \( (c, c') \in \mathcal{G} \) whenever
  \( c \in K \cdot c' \) and \( c \ne c' \). Note that \( \mathcal{G} \) is symmetric since
  \( K \) is. We view \( \mathcal{G} \) as a Borel graph on \( \mathcal{C} \) and claim that it is
  locally finite. More specifically, if \( \lambda \) is a right Haar measure, then the degree of
  each \( c \in \mathcal{C} \) is at most
  \( \bigl\lfloor \frac{\lambda(U\cdot K)}{\lambda(U)} \bigr\rfloor - 1 \).

  Indeed, let \( c_{0}, \ldots, c_{N} \in \mathcal{C} \) be distinct elements such that
  \( c_{i} \in K \cdot c_{0} \) for all \( i \le N \); in particular
  \( (c_{i},c_{0}) \in \mathcal{G} \) for \( i \ge 1 \). Let \( k_{i} \in K \) be such that
  \( k_{i} \cdot c_{0} = c_{i} \). The lacunarity of \( \mathcal{C} \) asserts that the sets
  \( U\cdot c_{i} = Uk_{i} \cdot c_{0} \) are supposed to be pairwise disjoint, which necessitates
  \( Uk_{i} \) to be pairwise disjoint for \( 0 \le i \le N \). Clearly, \( Uk_{i} \subseteq UK \)
  since \( k_{i} \in K \). Using the right-invariance of \(\lambda\), we have
  \( \lambda(UK) \ge \lambda\bigl(\bigsqcup_{i \le N} Uk_{i}\bigr) = (N+1)\lambda(U) \), and thus
  \( N+1 \le \frac{\lambda(UK)}{\lambda(U)} \), as claimed.

  We may now use~\cite[Prop.~4.6]{MR1667145} to deduce the existence of a finite partition
  \( \mathcal{C} = \bigsqcup_{i} \mathcal{C}_{i} \) such that no two points in \( \mathcal{C}_{i} \)
  are adjacent. In other words, if \( c, c' \in \mathcal{C}_{i} \) are distinct, then
  \( c \not \in K\cdot c' \), and therefore \( V\cdot c \cap V \cdot c' = \varnothing \), which
  shows that each \( \mathcal{C}_{i} \) is \( V \)-lacunary.
\end{proof}

Every cross-section \( \mathcal{C} \) gives rise to a smooth subrelation of \( \mathcal{R}_{G} \) by
associating to \( x \in X \) ``the closest point'' of \( \mathcal{C} \) in the same orbit. Such a
subrelation is known as the Voronoi tessellation. For the purposes of
Chapter~\ref{chap:derived-l1-full-group}, we need a slightly more abstract concept of a
tessellation, which may not correspond to Voronoi domains. While far from being the most general,
the following treatment is sufficient for our needs.

\begin{definition}
  \label{def:tessellation}
  Let \( G \acts X \) be a Borel action of a locally compact Polish group on a standard Borel space
  and let \( \mathcal{C} \subseteq X \) be a cross-section. A
  \textbf{tessellation}\index{Tessellation} over \( \mathcal{C} \) is a Borel set
  \( \mathcal{W} \subseteq \mathcal{C} \times X \) such that

  \begin{enumerate}
  \item all fibers \( \mathcal{W}_{c} = \{ x \in X : (c,x) \in \mathcal{W} \} \) are pairwise
    disjoint for \( c \in \mathcal{C} \);
  \item\label{item:refines-R_G} for all \( c \in \mathcal{C} \), elements of \( \mathcal{W}_{c} \)
    are \( \mathcal{R}_{G} \)-equivalent to \( c \), i.e.,
    \( \{c\} \times \mathcal{W}_{c} \subseteq \mathcal{R}_{G} \);
  \item\label{item:fibers-cover-phase-space} fibers cover the phase space,
    \( X = \bigsqcup_{c \in \mathcal{C}}\mathcal{W}_{c} \).
  \end{enumerate}

  A tessellation \(\mathcal{W}\) over \(\mathcal{C}\) is called \textbf{\(N\)-lacunary} for an open
  subset \(N \subseteq G\) if, for every \(c \in \mathcal{C}\), the inclusion
  \(N \cdot c \subseteq \mathcal{W}_c\) holds. It is said to be \textbf{\(K\)-cocompact}, where
  \(K \subseteq G\) is a compact subset, if \(\mathcal{W}_c \subseteq K \cdot c\) for all
  \(c \in \mathcal{C}\). We say that \(\mathcal{W}\) is \textbf{cocompact} if it is \(K\)-cocompact
  for some compact subset \(K \subseteq G\).
\end{definition}

Any tessellation \( \mathcal{W} \) can be viewed as a (flipped) graph of a function, since for any
\( x \in X \), there is a unique \( c \in \mathcal{C} \) such that \( (c, x) \in \mathcal{W} \). We
denote such \( c \) by \( \pi_{\mathcal{W}}(x) \), which produces a Borel map
\( \pi_{\mathcal{W}} : X \to \mathcal{C} \). There is a natural equivalence relation
\( \mathcal{R}_{\mathcal{W}} \) associated with the tessellation. Namely, \( x_{1} \) and
\( x_{2} \) are \( \mathcal{R}_{\mathcal{W}} \)-equivalent whenever they belong to the same fiber,
i.e., \( \pi_{\mathcal{W}}(x_{1}) = \pi_{\mathcal{W}}(x_{2}) \). In view of
item~\eqref{item:refines-R_G}, \( \mathcal{R}_{\mathcal{W}} \subseteq \mathcal{R}_{G} \), and
moreover, every \( \mathcal{R}_{G} \)-class consists of countably many
\( \mathcal{R}_{\mathcal{W}} \)-classes.

\section{Voronoi tessellations}

Voronoi tessellations provide a specific method for constructing tessellations over a given
cross-section. Suppose that the group \( G \) is endowed with a compatible \emph{proper} norm
\( \snorm{\cdot} \). Let \( D : \mathcal{R}_{G} \to \mathbb{R}^{\ge 0} \) be the associated metric
on the orbits of the action (as in Section~\ref{sec:l1-full-groups}), and let
\( \preceq_{\mathcal{C}} \) be a Borel linear order on \( \mathcal{C} \). The \textbf{Voronoi
  tessellation}\index{Tessellation!Voronoi}\index{Voronoi tesselation} over the cross-section \( \mathcal{C} \) relative to a
proper norm \( \norm{\cdot} \) is the set
\( \mathcal{V}_{\mathcal{C}} \subseteq \mathcal{C} \times X \) defined by
\begin{displaymath}
  \begin{aligned}
    \mathcal{V}_{\mathcal{C}} = \bigl\{(c,x) \in \mathcal{C} \times X\ :\
    &c \mathcal{R}_{G} x \textrm{
      and for all } c' \in \mathcal{C} \textrm{ such that } c' \mathcal{R}_{G} x, \textrm{ either } \\
    & D(c,x) <  D(c',x) \textrm{ or } \\
    &( D(c,x) = D(c',x) \textrm{ and } c \preceq_{\mathcal{C}} c')\bigr\}.
  \end{aligned}
\end{displaymath}
The properness of the norm ensures that for each \( x \in X \), there are only finitely many
candidates \( c \) that minimize \( D(c,x) \), and hence each \( x \in X \) is associated with a
unique \( c \in \mathcal{C} \). Here is a basic application of Voronoi tessellations.

\begin{proposition}\label{prop:union-of-smooth}
  Let \( G \) be a locally compact Polish group acting in a Borel manner on a standard Borel space
  \( X \).  There exists a sequence of cocompact tessellations \( (\mathcal{V}_{n})_{n\in\N} \) such
  that \(\mathcal{R}_{G} = \bigcup_{n} \mathcal{R}_{\mathcal{V}_{n}}\).
\end{proposition}
\begin{remark}
  It is essential that in the statement above, the equivalence relations
  \(\mathcal R_{\mathcal V_n}\) are \emph{not} required to be nested. For a related statement where
  these relations are indeed nested and the acting group is amenable, the reader is referred to
  Lemma~\ref{lem:exhaustive-bounded-smooth-atomless}.
\end{remark}
\begin{proof}[Proof of Proposition~\ref{prop:union-of-smooth}]
  Let \( \mathcal{C} \) be a cocompact cross-section, and let \(\norm\cdot\) be a proper norm on
  \( G \) inducing the metric \( D \) on the orbits. For each \( n \in \mathbb{N} \), define
  \( U_{n} \) to be the open ball of radius \( n \) centered at \( 1_{G} \). We denote the Voronoi
  tessellation over \( \mathcal{C} \) by \( \mathcal{V}_{\mathcal{C}} \).

  Lemma~\ref{lem:lacunary-partition} lets us pick a finite sequence of cocompact cross-sections
  \( \mathcal{C}^{n}_{1}, \ldots, \mathcal{C}^{n}_{k_{n}} \) such that each
  \( \mathcal{C}^{n}_{i} \) is \( U_{n} \)-lacunary and
  \( \mathcal{C} = \bigsqcup_{i=1}^{k_{n}}\mathcal{C}^{n}_{i} \).  Notably, the \(U_{n}\)-lacunarity
  condition ensures that for distinct \(c, c' \in \mathcal{C}^{n}_{i}\), we have \(D(c,c')\geq n\).
  Let \( \mathcal{V}_{i}^{n} \) denote the Voronoi tessellation over \( \mathcal{C}_{i}^{n} \).

  We claim that the tessellations \(\mathcal{C}_{i}^{n}\) constitute the desired sequence. Let
  \( x,y \in X \) be such that \( (x,y) \in \mathcal{R}_{G} \).  Set
  \(c=\pi_{\mathcal{V}_{\mathcal{C}}}(x)\).  If \(n\) is so large that \(D(x,c) < n\) and
  \(D(y,c) < n \), then \(\pi_{\mathcal{V}_{i}^{2n}}(y) = c = \pi_{\mathcal{V}_{i}^{2n}}(x)\) for
  the index \(i\) satisfying \(c \in \mathcal{C}_{i}^{2n}\).  We conclude that
  \((x,y)\in\mathcal{R}_{\mathcal{V}_{i}^{2n}}\), and the claim follows.
\end{proof}

For the sake of Chapter~\ref{chap:derived-l1-full-group}, we also need a definition of the Voronoi
tessellation for norms that may not be proper. The set \( \mathcal{V}_{\mathcal{C}} \) specified as
above may, in this case, fail to satisfy item~\eqref{item:fibers-cover-phase-space} of the
definition of a tessellation.  For some \( x \in X \), there may be infinitely many
\( c \in \mathcal{C} \) that minimize \( D(c,x) \), none of which are
\( \preceq_{\mathcal{C}} \)-minimal. We therefore need a different way to resolve the points on the
``boundary'' between the regions, which can be achieved, for example, by delegating this task to a
proper norm.

\begin{definition}
  \label{def:voronoi-tessellation-non-proper}
  Let \( \norm{\cdot} \) be a compatible norm on \( G \) and let \( \mathcal{C} \) be a
  cross-section. Pick a compatible proper norm \( \norm{\cdot}' \) on \( G \) and a Borel linear
  order \( \preceq_{\mathcal{C}} \) on \( \mathcal{C} \). Let \( D \) and \( D' \) be the metrics on
  the orbits of the action associated with the norms \( \norm{\cdot} \) and \( \norm{\cdot}' \),
  respectively.  The \textbf{Voronoi tessellation}\index{Tessellation!Voronoi}\index{Voronoi tesselation} over the
  cross-section \( \mathcal{C} \) relative to the norm \( \norm{\cdot} \) is the set
  \( \mathcal{V}_{\mathcal{C}} \subseteq \mathcal{C} \times X \) defined by
  \begin{displaymath}
    \begin{aligned}
      \mathcal{V}_{\mathcal{C}} = \bigl\{(c,x) \in\
      &\mathcal{C} \times X\ :\ c \mathcal{R}_{G}x \textrm{ and for all } c' \in \mathcal{C}
        \textrm{ such that } c' \mathcal{R}_{G} x \textrm{ either } \\
      & D(c,x) < D(c',x) \textrm{ or } \\
      &( D(c,x) = D(c',x) \textrm{ and } D'(c,x) < D'(c', x)) \textrm{ or } \\
      &( D(c,x) = D(c',x) \textrm{ and } D'(c,x) = D'(c', x) \textrm{ and }
        c \preceq_{\mathcal{C}} c')\bigr\}.
    \end{aligned}
  \end{displaymath}
\end{definition}

The definition of the Voronoi tessellation does depend on the choice of the norm \( \norm{\cdot}' \)
and the linear order \( \preceq_{\mathcal{C}} \) on the cross-section, but its key properties remain
the same regardless of these choices. We therefore often do not explicitly specify which
\( \norm{\cdot}' \) and \( \preceq_{\mathcal{C}} \) are picked. Note also that if the cross-section
is cocompact, then every region of the Voronoi tessellation is bounded, i.e.,
\( \sup_{x \in X} D(x, \pi_{\mathcal{V}_{\mathcal{C}}}(x)) < +\infty \).

Our goal is to show that the equivalence relations \( \mathcal{R}_{\mathcal{W}} \) are atomless in
the sense of Section~\ref{sec:disintegration-measure}, provided that each orbit of the action is
uncountable. To this end, we first need the following lemma.

\begin{lemma}
  \label{lem:countable-section-measure-zero}
  Let \( G \) be a locally compact Polish group acting on a standard Lebesgue space \( (X, \mu) \)
  by measure-preserving automorphisms. Suppose that almost every orbit of the action is
  uncountable. If \( \mathcal{A} \subseteq X \) is a measurable set such that the intersection of
  \( \mathcal{A} \) with almost every orbit is countable, then \( \mu(\mathcal{A}) = 0 \).
\end{lemma}

\begin{proof}
  Pick a proper norm \( \norm{\cdot} \) on \( G \).  Let \( \mathcal{C} \) be a cross-section for
  the action, and let \( B_{2r} \subseteq G \) be an open ball around the identity of sufficiently
  small radius \( 2r > 0 \) such that \( B_{2r} \cdot c \cap B_{2r} \cdot c' = \varnothing \)
  whenever \( c, c' \in \mathcal{C} \) are distinct.  Let \( \mathcal{V}_{\mathcal{C}} \) be the
  Voronoi tessellation over \( \mathcal{C} \) relative to \( \norm{\cdot} \). Note that
  \( B_{2r}\cdot c \) is fully contained in the \( \mathcal{R}_{\mathcal{V}_{\mathcal{C}}} \)-class
  of \( c \).

  We claim that it is enough to consider the case when \( \mathcal{A} \) intersects each
  \( \mathcal{R}_{\mathcal{V}_{\mathcal{C}}} \)-class in at most one point. Indeed, the restriction
  of \( \mathcal{R}_{\mathcal{V}_{\mathcal{C}}} \) onto \( \mathcal{A} \) is a smooth countable
  equivalence relation, so one can write
  \( \mathcal{A} = \bigsqcup_{n \in \mathbb{N}} \mathcal{A}_{n}' \), where each
  \( \mathcal{A}_{n}' \) intersects each \( \mathcal{R}_{\mathcal{V}_{\mathcal{C}}} \)-class in at
  most one point. To simplify notation, we assume that \( \mathcal{A} \) already possesses this
  property.

  Let \( Y = B_{r} \cdot \mathcal{C} \), and let \( (g_{n})_{n\in \mathbb{N}} \) be a countable
  dense subset of~\( G \).  We define the function \( \gamma : X \to \mathbb{N} \) by
  \[ \gamma(x) = \min\{ n \in \mathbb{N}: x\, \mathcal{R}_{\mathcal{V}_{\mathcal{C}}}\, g_{n}x
    \textrm{ and } g_{n}x \in Y \}. \]
  Let \( \mathcal{A}_{n} = \mathcal{A} \cap \gamma^{-1}(n) \), and note that the sets
  \( \mathcal{A}_{n} \) partition \( \mathcal{A} \). It is therefore enough to show that
  \( \mu(\mathcal{A}_{n}) = 0 \) for any \( n \in \mathbb{N} \). Pick \( n_{0} \in \mathbb{N} \).
  The action is measure-preserving, and therefore
  \( \mu(\mathcal{A}_{n_{0}}) = \mu(g_{n_{0}}\mathcal{A}_{n_{0}}) \). Set
  \( \mathcal{B}_{0} = g_{n_{0}}\mathcal{A}_{n_{0}} \) and note that for any
  \( x \in \mathcal{B}_{0} \) and \( g \in B_{r} \subseteq G \), one has
  \( gx \mathcal{R}_{\mathcal{V}_{\mathcal{C}}} x \). If the action were free, we could easily
  conclude that \( \mu(\mathcal{B}_{0}) = 0 \), since the sets \( g\mathcal{B}_{0} \),
  \( g \in B_{r} \), would be pairwise disjoint. In general, we need to exhibit a little more care
  and construct a countable family of pairwise disjoint sets \( \mathcal{B}_{n} \) as follows. For
  \( x \in \mathcal{B}_{0} \), let
  \[ \tau_{n}(x) = \min \bigl\{ m \in \mathbb{N}: x \mathcal{R}_{\mathcal{V}_{\mathcal{C}}} g_{m}x
    \textrm{ and } g_{m}x \not \in \bigcup_{k \le n} \mathcal{B}_{n} \bigr\}. \]
  The value \( \tau_{n}(x) \) is well-defined because the stabilizer of \( x \) is closed and must
  be nowhere dense in \( B_{r} \) due to the orbit \( G \cdot x \) being uncountable. Put
  \( \mathcal{B}_{n+1} = \{ g_{\tau_{n}(x)}x : x \in \mathcal{B}_{0} \} \) and note that
  \( \mu(\mathcal{B}_{n}) = \mu(\mathcal{B}_{0}) \). We get a pairwise disjoint infinite family of
  sets \( \mathcal{B}_{n} \), all having the same measure. Since \( \mu \) is finite, we conclude
  that \( \mu(\mathcal{B}_{0}) = 0 \), and the lemma follows.
\end{proof}

\begin{corollary}
  \label{cor:tessellation-atomless}
  Let \( G \) be a locally compact Polish group acting on a standard Lebesgue space \( (X, \mu) \)
  by measure-preserving automorphisms.  Let \( \mathcal{C} \) be a cross-section for the action, and
  let \( \mathcal{W} \subseteq \mathcal{C} \times X \) be a tessellation. If \( \mu \)-almost every
  orbit of \( G \) is uncountable, then \( \mathcal{R}_{\mathcal{W}} \) is atomless.
\end{corollary}

\begin{proof}
  Consider the disintegration \( (\mu_{c})_{c \in \mathcal{C}} \) of \( \mathcal{R}_{\mathcal{W}} \)
  relative to \( (\pi_{\mathcal{W}}, \nu) \), where \( \pi_{\mathcal{W}} : X \to \mathcal{C} \) and
  \( \nu = (\pi_{\mathcal{W}})_{*}\mu \). Let \( X_{a} \subseteq X \) be the set of atoms of the
  disintegration. Since \( \nu \)-almost every \( \mu_{c} \) is finite, the fibers
  \( \pi^{-1}_{\mathcal{W}}(c) \) have countably many atoms. Since every tessellation has only
  countably many tiles within each orbit, we conclude that \( X_{a} \) has a countable intersection
  with almost every orbit of the action. Lemma~\ref{lem:countable-section-measure-zero} applies and
  shows that \( \mu(X_{a}) = 0 \).  Hence, \( \mathcal{R}_{\mathcal{W}} \) is atomless as required.
\end{proof}

Consider the full group \( \fgr{\mathcal{R}_{\mathcal{W}}} \), which, by
Proposition~\ref{prop:full-group-smooth-L0} and Corollary~\ref{cor:tessellation-atomless}, is
isomorphic to \( \LL^{0}(Y, \nu, \Aut([0,1], \lambda)) \) for some standard Lebesgue space
\( (Y, \nu) \). This full group can naturally be viewed as a subgroup of
\( \fgr{\mathcal{R}_{G}} \), and the topology induced on \( \fgr{\mathcal{R}_{\mathcal{W}}} \) from
the full group \( \fgr{\mathcal{R}_{G}} \) coincides with the topology of convergence in measure on
\( \LL^{0}(Y, \nu, \Aut([0,1], \lambda)) \) (see Section~3 of~\cite{MR3464151}). We therefore have
the following corollary.

\begin{corollary}
  \label{cor:maharam-identification-with-l0}
  Let \( G \) be a locally compact Polish group acting on a standard Lebesgue space \( (X, \mu) \)
  by measure-preserving automorphisms. Let \( \mathcal{C} \) be a cross-section for the action, let
  \( \mathcal{W} \subseteq \mathcal{C} \times X \) be a tessellation, and let
  \( \pi_{\mathcal{W}} : X \to \mathcal{C} \) be the corresponding reduction.  If \( \mu \)-almost
  every orbit of \( G \) is uncountable, then the subgroup
  \( \fgr{\mathcal{R}_{\mathcal{W}}} \le \fgr{\mathcal{R}_{G}} \) is isomorphic as a topological
  group to \( \LL^{0}(\mathcal{C}, (\pi_{\mathcal{W}})_{*}\mu, \Aut([0,1], \lambda)) \). If moreover
  all orbits of the action have measure zero, then \( (\pi_{\mathcal{W}})_{*}\mu \) is non-atomic,
  and \( \fgr{\mathcal{R}_{\mathcal{W}}} \) is isomorphic to
  \( \LL^{0}([0,1], \lambda, \Aut([0,1],\lambda)) \).
\end{corollary}


\chapter{Conditional measures}\label{chap:conditional-measures-appendix}

The ergodic decomposition theorem, as formulated in Section~\ref{sec:ergod-decomp}, is not available
for general probability measure-preserving actions of Polish groups.  Conditional measures provide a
useful framework to remedy this.  As before, \(\Aut(X,\mu)\) stands for the group of
measure-preserving automorphisms of a standard probability space.  It is more useful, however, to
view \(\Aut(X,\mu)\) as the group of measure-preserving automorphisms of the \textbf{measure
  algebra}\index{MAlg@\(\MAlg\)} \(\MAlg(X,\mu)\) of \((X,\mu)\), i.e., the Boolean algebra of
equivalence classes of Borel subsets of \(X\), identified up to measure zero.  The measure algebra
is endowed with a natural metric \(d_\mu\) given by \( d_{\mu}(A,B) = \mu (A \bigtriangleup B) \).
The completeness of \( (\MAlg(X,\mu),d_\mu) \) follows directly from its natural isometric
identification with \( (\LL^1(X,\mu,\Z/2\Z),\tilde{d}_{Y}) \), where \(\Z/2\Z\) is endowed with the
discrete metric \(d_{Y}\), and the metric \(\tilde{d}_{Y}\) is given in Definition~\ref{def:pointed-L1}.

\begin{proposition}\label{prop:malg-complete}
  The metric space \( (\MAlg(X,\mu),d_{\mu}) \) is complete.\qed
\end{proposition}

Note that closed subalgebras of \(\MAlg(X,\mu)\) are in a one-to-one correspondence with complete
(in the measure-theoretical sense) \(\sigma\)-subalgebras of the \(\sigma\)-algebra of Lebesgue
measurable sets.

\section{Conditional expectations}

We give a concise overview of how conditional expectations can be defined without the need for
disintegration.

Let \( M\) be a closed subalgebra of \(\MAlg(X,\mu)\) and let \(\LL^2( M,\mu)\) denote the
\(\LL^{2}\) space of real-valued \(M\)-measurable functions.  Note that \(\LL^{2}(M, \mu)\) is a
closed subspace of \(\LL^{2}(X,\mu) = \LL^{2}(\MAlg(X,\mu),\mu)\). The \(M\)-\textbf{conditional
  expectation}\index{Conditional expectation} is the orthogonal projection
\(\mathbb E_M:\LL^2(X,\mu)\to \LL^2(M,\mu)\).  It is uniquely defined by the condition
\begin{equation}\label{eq:cond-exp}
  \int_X fg \, d\mu= \int_X\mathbb E_M(f)g\, d\mu \quad
  \textrm{for all \(f\in \LL^2(X,\mu)\) and all \(g\in \LL^2(M,\mu)\)}.
\end{equation}
Due to the density of step functions in \(\LL^2(M,\mu)\), the conditional expectation can equivalently
be defined as the linear contraction \(\LL^2(X,\mu)\to \LL^2(M,\mu)\) satisfying
\begin{equation}\label{eq:cond-exp-on-chara}
  \int_{A} f \, d\mu=\int_{A} \mathbb{E}_{M}(f)\, d\mu \quad \textrm{for all \(A\in M\) and all
    \(f \in \LL^{2}(X,\mu)\)}.
\end{equation}

Positive functions are precisely those that yield a non-negative dot product with any characteristic
function. By allowing \(g\) in Eq.~\eqref{eq:cond-exp} to vary over the set of all characteristic
functions of subsets in \(M\), we can see that the conditional expectation \(\mathbb{E}_{M}\)
preserves positivity.

\begin{proposition}\label{prop:cond-exp-positivity-preserving}
  If \(f \in \LL^{2}(X,\mu)\) is non-negative, \(f\geq 0\), then \(\mathbb E_M(f)\geq 0\).
\end{proposition}

While we defined conditional expectations as operators on \(\LL^{2}(X,\mu)\), their domain can be
extended to all of \(\LL^{1}(X,\mu)\), making \(\mathbb{E}_{M}\) a contraction from
\(\LL^{1}(X,\mu) \) to \(\LL^{1}(M,\mu)\).  This is justified by the following proposition.

\begin{proposition}
  The conditional expectation \(\mathbb{E}_{M} : \LL^{2}(X,\mu) \to \LL^{2}(M,\mu)\) is a
  contraction when the domain and the range are endowed with the \(\LL^1\) norms.
\end{proposition}

\begin{proof}
  If \(f\in\LL^2(X,\mu)\) is non-negative, \(f\geq 0\), then Eq.~\eqref{eq:cond-exp} yields
  \[
    \norm f_1=\int_X f\, d\mu = \int_Xf \cdot 1\, d\mu = \int_X\mathbb E_M(f)\cdot 1\, d\mu =
    \int_X\mathbb E_M(f) \, d\mu.
  \]
  Since \(\mathbb E_M(f) \geq 0\) by Proposition~\ref{prop:cond-exp-positivity-preserving}, we
  conclude that \(\norm{\mathbb E_M(f)}_1=\norm{f}_1\) for all non-negative \(f\in\LL^2(X,\mu)\).

  An arbitrary \(f \in \LL^{2}(X,\mu)\) can be written as the difference \(f^{+} - f^{-}\) of
  non-negative functions \(f^{+} = \max\{f,0\}\) and \(f^{-} = \max\{-f,0\}\). Note that
  \(f^{+}, f^{-} \in \LL^{2}(X,\mu)\) and \(\norm{f^+}_1+\norm{f^-}_1=\norm{f}_1\).  We therefore
  have
    \begin{displaymath}
      \begin{aligned}
        \snorm{\mathbb{E}_{M}(f)}_{1}
        &= \snorm{\mathbb{E}_{M}(f^+-f^-)}_{1} \leq
           \snorm{\mathbb{E}_{M}(f^+)}_1+\snorm{\mathbb{E}_{M}(f^-)}_{1} \\
        &= \snorm{f^+}_1+\snorm{f^-}_1 = \snorm{f}_1,
      \end{aligned}
    \end{displaymath}
  showing that \(\mathbb{E}_{M}\) is a contraction in the \(\LL^{1}\) norm.
\end{proof}

\begin{remark}
  By the previous proposition, \(\mathbb{E}_{M}\) admits a (necessarily unique) extension to a contraction
  \[
    \mathbb{E}_{M} : \LL^{1}(X,\mu)\to \LL^{1} (M,\mu).
  \]
  Moreover, since every non-negative integrable function can be written as an increasing limit of
  bounded non-negative functions, the analog of
  Proposition~\ref{prop:cond-exp-positivity-preserving} continues to hold for
  \(f\in\LL^{1}(X,\mu)\).
\end{remark}

\section{Conditional measures}\label{sec:conditional-measures}

Throughout this section, we let \(\chi_A: X\to \{0,1\}\) denote the characteristic function of
\(A \subseteq X\).

\begin{definition}
  Let \(M\) be a closed subalgebra of \(\MAlg(X,\mu)\).  The \textbf{\(M\)-conditional
    measure}\index{Conditional measure} of \(A \in \MAlg(X,\mu)\), denoted by \(\mu_M(A)\), is the
  conditional expectation of the characteristic function of \(A\), i.e.,
  \(\mu_M(A) = \mathbb E_M(\chi_A)\).
\end{definition}

In particular, the conditional measure \(\mu_M(A)\) is an \(M\)-measurable function. It enjoys the
following natural properties.

\begin{proposition}\label{prop:cond-measure}
  Let \(M\subseteq \MAlg(X,\mu)\) be a closed subalgebra. The following properties hold for all
  \(A \in \MAlg(X,\mu)\):
  \begin{enumerate}
  \item \(\mu_M(\varnothing)=0\) and \(\mu_M(X)=1\), where \(0\) and \(1\) denote the constant maps;
  \item \(\mu_M(A)\) takes values in \([0,1]\) and \( \int_X\mu_M(A)=\mu(A)\);
  \item\label{item:additivity} \(\mu_M\) is \(\sigma\)-additive: if \(A = \bigsqcup_{n}A_n\),
    \(A_{n} \in \MAlg(X,\mu)\), is a partition, then
    \[\mu_M(A)=\sum_{n\in\N} \mu_M(A_n),\]
    where the convergence holds in \(\LL^1(M,\mu)\);
  \item\label{item:preserve-cond-meas} if \(T\in\Aut(X,\mu)\) fixes every element of \(M\), then
    \( \mu_M(A)=\mu_M(T(A))\).
  \end{enumerate}
\end{proposition}

\begin{proof}
  The first item is clear from the fact that both \(\varnothing\) and \(X\) belong to \(M\), so
  their characteristic functions are fixed by \(\mathbb E_M\).  The second item follows from the
  first and positivity of the conditional expectation; the equality is a direct consequence of
  Eq.~\eqref{eq:cond-exp-on-chara}.  The third one is a consequence of the \(\LL^1\) continuity of
  \(\mathbb E_M\) and its linearity, noting that \(\chi_A=\sum_n \chi_{A_n}\) in \(\LL^1(M,\mu)\).

  Finally, the last item follows from the uniqueness of conditional expectation given by
  Eq.~\eqref{eq:cond-exp-on-chara}. Indeed, if an automorphism \(T\) fixes every element of \(M\),
  then
  \[
    \int_{B} f\circ T\inv \, d\mu = \int_{T(B)} \mkern-20mu f\, d\mu =\int_B f\, d\mu \quad \textrm{ for all
      \(B \in \MAlg(X,\mu)\)},
  \]
  so \(\mathbb E_M(f\circ T\inv)=\mathbb E_M(f)\). Taking \(f=\chi_{A}\) for \(A\in\MAlg(X,\mu)\), we
  conclude that \(\mu_M(T(A))=\mu_M(A)\).
\end{proof}

\section{Conditional measures and full groups}\label{sec:cond-meas-full-groups}

Conditional measures, as defined in Section~\ref{sec:conditional-measures}, are associated with
closed subalgebras of \(\MAlg(X,\mu)\).  Each subgroup \(\mathbb{G} \le \Aut(X,\mu)\) gives rise to
the subalgebra of \(\mathbb{G}\)-invariant sets, and we may therefore associate a conditional
measure with the group \(\mathbb{G}\) itself.

\begin{definition}
  Let \(\mathbb G\) be a subgroup of \(\Aut(X,\mu)\).  The closed \textbf{subalgebra of
    \(\mathbb{G}\)-invariant sets} is denoted by \( M_{\mathbb G}\) and consists of all
  \(A\in\MAlg(X,\mu)\) such that \(gA = A\) for all \(g \in \mathbb{G}\).
\end{definition}

By definition, \(\mathbb{G} \leq \Aut(X,\mu)\) is \textbf{ergodic} if
\( M_{\mathbb G}=\{\varnothing,X\}\). Since \(\{\varnothing, X\}\)-measurable functions are
constants, the \(M_{\mathbb G}\)-conditional measure corresponds to the measure \(\mu\) when
\(\mathbb{G}\) is ergodic.  The following lemma is an easy consequence of the definitions of the
full group generated by a subgroup (Section~\ref{sec:full-and-finfull}) and the weak topology on
\(\Aut(X,\mu)\).

\begin{lemma}\label{lem:conditional-measure-full-and-dense}
  Let \(\mathbb G\leq\Aut(X,\mu)\) be a group.
  \begin{enumerate}
  \item If \([\mathbb{G}]\) is the full group generated by \(\mathbb{G}\), then
    \(M_{\mathbb{G}}=M_{[\mathbb{G}]}\).
  \item If \(\Gamma \le \mathbb{G}\) is dense in the weak topology, then
    \(M_\Gamma=M_{\mathbb{G}}\).
  \end{enumerate}
\end{lemma}

Given a subgroup \(\mathbb G\leq \Aut(X,\mu)\), we denote the \(M_{\mathbb{G}}\)-conditional measure
simply by \(\mu_{\mathbb G}\).

Recall that a \textbf{partial measure-preserving automorphism} of \((X,\mu)\) is a
measure-preserving bijection \(\varphi : \dom \varphi \to \rng \varphi\) between measurable subsets
of \(X\), called the domain and the range of \(\varphi\), respectively. The \textbf{pseudo full
  group} generated by a group \(\Gamma \le \Aut(X,\mu)\) is denoted by \(\pafgr{\Gamma}\) and
consists of all partial automorphisms \(\varphi: \dom \varphi \to \rng \varphi \) for which there
exists a partition \(\dom \varphi = \bigsqcup_{n}A_n\) and elements \(\gamma_{n} \in \Gamma\) such
that \(\varphi\restriction_{A_{n}} = \gamma_{n}\restriction_{A_{n}}\) for all \(n\).  Elements of
\(\pafgr{\Gamma}\) automatically preserve the conditional measure \(\mu_\Gamma\) in the sense that
if \(A \subseteq \dom \varphi\), then \(\mu_{\Gamma}(\varphi(A)) = \mu_{\Gamma}(A)\).  Indeed,
\begin{displaymath}
  \begin{aligned}
    \mu_{\Gamma}(\varphi(A))
    &= \mu_{\Gamma}\Bigr(\varphi\Bigl(\bigsqcup_{n} (A \cap A_{n})\Bigr)\Bigr) =
      \mu_{\Gamma}\Bigr(\bigsqcup_{n} \varphi(A \cap A_{n})\Bigr) \\
    \because \textrm{Prop.~\ref{prop:cond-measure}\eqref{item:additivity}}
    &= \sum_{n}\mu_{\Gamma}(\varphi(A \cap A_{n})) = \sum_{n}\mu_{\Gamma}(\gamma_{n}(A \cap
      A_{n}))\\
    \because \textrm{Prop.~\ref{prop:cond-measure}\eqref{item:preserve-cond-meas}}
    &= \sum_{n}\mu_{\Gamma}(A \cap A_{n}) = \mu_{\Gamma}(A). \\
  \end{aligned}
\end{displaymath}

\begin{lemma}\label{lem:partial-full-group-match}
  Let \( \mathbb{G} \leq\Aut(X,\mu)\) be a group.  For all \(A,B\in\MAlg(X,\mu)\) satisfying
  \(\mu_{\mathbb{G}}(A)=\mu_{\mathbb{G}}(B)\), there exists an element
  \(\varphi \in \pafgr{\mathbb{G}}\) such that \(\dom \varphi=A\) and \(\rng\varphi = B\).
\end{lemma}
\begin{proof}
  Let \(\Gamma = \{\gamma_n : n \in \mathbb{N}\}\) be a countable weakly dense subgroup of
  \(\mathbb{G}\). It follows from Lemma~\ref{lem:conditional-measure-full-and-dense} that
  \(\mu_{\Gamma}(A) = \mu_{\mathbb{G}}(A) = \mu_{\mathbb{G}}(B) = \mu_{\Gamma}(B)\), and it is
  evident that \(\pafgr{\Gamma} \le \pafgr{\mathbb{G}}\).

  We inductively define sequences \((A_n)_{n}\) and \((B_n)_{n}\) of subsets of \(A\) and \(B\),
  respectively, starting with \(A_0 = A \cap \gamma_0^{-1} B\) and \(B_0 = \gamma_0 A_0\), and for
  \(n \geq 1\), we set
  \[
    A_{n} = \Bigl(A \setminus \bigcup_{m < n} A_m\Bigr) \cap \gamma_{n}^{-1}\Bigl(B \setminus
    \bigcup_{m < n} B_m\Bigr) \quad \text{ and } \quad B_n = \gamma_n A_n.
  \]
  By construction, the sets \(A_n\) are pairwise disjoint subsets of \(A\), each set satisfies
  \(\gamma_n A_n = B_n\), and the sets \(B_n\) are pairwise disjoint subsets of \(B\). We assert
  that \(\varphi = \bigsqcup_n (\gamma_n \restriction_{A_n})\) is the desired element of
  \(\pafgr{\mathbb{G}}\).

  Suppose, towards a contradiction, that either \(\dom \varphi\neq A\) or \(\rng\varphi \neq B\).
  Since \(\Gamma\) preserves \(\mu_\Gamma\) and \(\mu_\Gamma(A) = \mu_\Gamma(B)\), the sets
  \(A\setminus \dom \varphi\) and \(B\setminus\rng\varphi\) have the same \(M_{\Gamma}\)-conditional
  measure, which is not constantly equal to zero.  The set
  \[\tilde{A} = \bigcup_{\gamma\in\Gamma} \gamma (A\setminus \dom\varphi)\] is \(\Gamma\)-invariant
  and non-zero. Its conditional measure is therefore the characteristic function
  \(\chi_{\tilde{A}}\), which must be greater than or equal to
  \(\mu_{\Gamma}(A\setminus\dom\varphi)=\mu_{\Gamma}(B\setminus\rng\varphi)\).  We conclude that
  \(B\setminus \rng\varphi\subseteq \bigcup_{\gamma\in\Gamma} \gamma (A\setminus\dom\varphi)\).  In
  particular, there is the first \(n\in\N\) such that
  \((A\setminus \dom\varphi)\cap\gamma_n^{-1} (B\setminus\rng\varphi)\) is non-zero. By
  construction, this set should be a subset of \(A_n\), yielding the desired contradiction.
\end{proof}

\begin{proposition}\label{prop:construct-involutions}
  Let \(\mathbb G\) be a full subgroup of \(\Aut(X,\mu)\).  The following conditions are equivalent
  for all \(A,B\in\MAlg(X,\mu)\):
  \begin{enumerate}
  \item\label{item:same-cond-meas} \(\mu_{\mathbb G}(A)=\mu_{\mathbb G}(B)\);
  \item\label{item:same-fg-orbit} there is \(T\in\mathbb G\) such that \(T(A)=B\);
  \item\label{item:same-fg-orbit-with-small-support} there is an involution \(T\in\mathbb G\) such
    that \(T(A)=B\) and \(\supp T= A\bigtriangleup B\).
  \end{enumerate}
\end{proposition}
\begin{proof}
  The implication~\eqref{item:same-fg-orbit}\(\Rightarrow\)\eqref{item:same-cond-meas} is a direct
  consequence of the definition of \(M_{\mathbb G}\) along with item~\eqref{item:preserve-cond-meas}
  of
  Proposition~\ref{prop:cond-measure}.
  Also,~\eqref{item:same-fg-orbit-with-small-support}\(\Rightarrow\)\eqref{item:same-fg-orbit} is
  evident.

  We now prove the
  implication~\eqref{item:same-cond-meas}\(\Rightarrow\)\eqref{item:same-fg-orbit-with-small-support}.
  The assumption \(\mu_{\mathbb{G}}(A) = \mu_{\mathbb{G}}(A)\) guarantees that
  \(\mu_{\mathbb{G}}(A\setminus B) = \mu_{\mathbb{G}}(B\setminus A)\).
  Lemma~\ref{lem:partial-full-group-match} applies and produces an element
  \(\varphi \in \pafgr{\mathbb{G}}\) such that \(\varphi(A\setminus B) = B \setminus A\).  The
  involution \(\varphi\sqcup \varphi\inv \sqcup \mathrm{id}_{X\setminus(A \bigtriangleup B )}\)
  meets the required conditions.
\end{proof}

\section{Aperiodicity}\label{sec:aperiodicity}

A countable subgroup \(\Gamma \leq \Aut(X,\mu)\) is called
\textbf{aperiodic}\index{Transformation group!aperiodic}\index{Aperiodic!subgroup of \(\Aut(X,\mu)\)}
if almost all the orbits of some
(equivalently, any) realization of its action on \((X,\mu)\) are infinite.  The so-called Maharam's
lemma provides a characterization of aperiodicity in a purely measure-algebraic way.  We begin by
formulating a variant of the standard marker lemma for countable Borel equivalence relations (see,
for instance,~\cite[Lemma~6.7]{kechrisTopicsOrbitEquivalence2004}).

\begin{lemma}\label{lem:marker-lemma}
  Let \(\Gamma\acts X\) be a Borel action of a countable group on a standard Borel space \(X\).  For
  every Borel set \(C\subseteq X\), there is a decreasing sequence \((C_n)_{n}\) of Borel subsets of
  \(C\) such that \(C \subseteq \Gamma \cdot C_n\) for each \(n\), and the set \(\bigcap_{n} C_n\)
  intersects the \(\Gamma\)-orbit of every \(x\in X\) in at most one point. Furthermore, if all
  orbits of \(\Gamma\) are infinite, the sets \(C_{n}\) can be chosen to have an empty intersection,
  \(\bigcap_{n} C_{n} = \varnothing\).
\end{lemma}

The following result is essentially due to H.~Dye~\cite{dyeGroupsMeasurePreserving1959}, where it is
called Maharam's lemma.

\begin{theorem}[Maharam's lemma\index{Maharam's lemma}]\label{thm:maharam-lemma}
  Let \(\Gamma\leq\Aut(X,\mu)\) be a countable subgroup.  The following are equivalent:
  \begin{enumerate}
  \item\label{item:Gamma-aperiodic} \(\Gamma\) is aperiodic;
  \item\label{item:Maharam-lemma}for any \(A\in\MAlg(X,\mu)\) and any \(M_\Gamma\)-measurable
    function \(f:X\to[0,1]\) satisfying \(f\leq \mu_\Gamma(A)\), there is \(B\subseteq A\),
    \(B \in \MAlg(X,\mu)\), such that \(\mu_\Gamma(B)=f\).
  \end{enumerate}
\end{theorem}
\begin{proof}
  Let us start with the simpler
  implication~\eqref{item:Maharam-lemma}\(\Rightarrow\)\eqref{item:Gamma-aperiodic}, which is proved
  using a contrapositive argument.  Assume that~\eqref{item:Gamma-aperiodic} fails, meaning that
  \(\Gamma\) is not aperiodic.  Let \(n \in \N\) be an integer such that the \(\Gamma\)-invariant
  set \(X_n = \{x\in X : \abs{\Gamma\cdot x} = n \}\) has non-zero measure.  We may assume that
  \(X\) bears a Borel total order (for instance, by identifying \(X\) with \([0,1]\)). Let
  \(A = \{x\in X_n : x=\max \{ \Gamma\cdot x\} \}\) be the set of maximal points of the
  \(n\)-element \(\Gamma\)-orbits and set \(\varphi\), \(\dom \varphi = X_{n} \setminus A\), to be
  the element of the pseudo full group \(\pafgr{\Gamma}\) that takes every \(x\in X_n\setminus A\)
  to its \(<\)-successor in the orbit \(\Gamma\cdot x\).  Given any \(B\subseteq A\), the set
  \(\bigsqcup_{k=0}^{n-1} \varphi^{-k}(B)\) is \(\Gamma\)-invariant, hence
  \(\mu_\Gamma(\bigsqcup_{k=0}^{n-1}\varphi^{-k}(B))\) takes values in \(\{0,1\}\).  Also
  \[
    \mu_\Gamma\big(\bigsqcup_{k=0}^{n-1}\varphi^{-k}(B)\big)
    =\sum_{k=0}^{n-1}\mu_\Gamma\big(\varphi^{-k}(B)\big)=n \mu_\Gamma(B),
  \]
  where the last equality is a consequence of Proposition~\ref{prop:cond-measure}.  We conclude that
  \(\mu_\Gamma(B)\) necessarily takes values in \(\{0,\frac 1n\}\), which
  contradicts~\eqref{item:Maharam-lemma}.

  We now assume that \(\Gamma\) is aperiodic and prove the direct
  implication~\eqref{item:Gamma-aperiodic}\(\Rightarrow\)\eqref{item:Maharam-lemma}.  The argument
  is based on the following claim.

  \begin{claim*}
    For every \(C\in\MAlg(X,\mu)\), for every \(M_\Gamma\)-measurable not almost surely zero
    \(f:X\to[0,1]\) such that \(f\leq \mu_\Gamma(C)\), there is a non-empty \(B\subseteq C\)
    satisfying \(\mu_\Gamma(B)\leq f\).
  \end{claim*}
  \begin{cproof}
    Let \((C_n)_{n}\) be a sequence of subsets of \(C\) given by Lemma~\ref{lem:marker-lemma}.  Note
    that \(\mu_\Gamma(C_n)\to 0\) in \(\LL^1\), since \(\bigcap_n C_n=\varnothing\) and the
    \(C_n\)'s are decreasing. Passing to a subsequence, we may assume that the convergence
    \(\mu_\Gamma(C_n) \to 0\) holds pointwise.  Set
    \( B_{n} = \{x\in C_n : \mu_\Gamma(C_n)(x)\leq f(x)\}\) and note that
    \(\mu_{\Gamma}(B_{n}) \le \mu_{\Gamma}(C_{n})\) and therefore \(\mu_{\Gamma}(B_{n}) \le f\).

    Pointwise convergence \(\mu_\Gamma(C_n)\to 0\) guarantees the existence of an index \(n\) such
    that \(\mu(B_{n})>0\), and so the set \(B = B_{n}\) is as required.
  \end{cproof}
  The conclusion of the theorem now follows from a standard application of Zorn's lemma\footnote{A
    more constructive version of the whole argument can be found
    in~\cite[Prop.~D.1]{lemaitreGroupesPleinsPreservant2014}.}.  The latter provides a
  maximal family \((B_i)_{i\in I}\) of pairwise disjoint positive measure elements of
  \(\MAlg(X,\mu)\) contained in \(A\) and satisfying \(\sum_{i\in I}\mu_\Gamma(B_i)\leq f\).  The
  index set \(I\) has to be countable, and if \(B=\bigsqcup_{i\in I} B_i\) then
  \(\mu_\Gamma(B)=\sum_{i\in I}\mu_\Gamma(B_i)\leq f\). Assume towards a contradiction that
  \(\mu_\Gamma(B)\) is not equal to \(f\) almost everywhere, and use the previous claim to get a non
  null \(B'\subseteq A\setminus B\) with \(\mu_\Gamma(B')\leq f - \mu_\Gamma(B)\), contradicting the
  maximality of \((B_i)_{i\in I}\).  Therefore, \(\mu_\Gamma(B)=f\) as claimed.
\end{proof}

We conclude this appendix with a useful consequence of aperiodicity.  Recall that in
Definition~\ref{def:general-aperiodicity}, we define a potentially uncountable subgroup
\(G \leq \Aut(X, \mu)\) to be aperiodic if it contains an aperiodic countable subgroup.
Furthermore, this implies that \(G\) contains a countable weakly dense aperiodic subgroup.

\begin{lemma}\label{lem:aperiodic-action-has-many-involutions}
  Let \(\mathbb G\leq \Aut(X,\mu)\) be an aperiodic full group. For each set \(B\in\MAlg(X,\mu)\),
  there is an involution \(U\in\mathbb G\) whose support is equal to \(B\).
\end{lemma}
\begin{proof}
  Let \(\Gamma \leq \mathbb{G}\) be a countable weakly dense aperiodic subgroup of
  \(\mathbb{G}\). Then, by weak density, \(M_\Gamma = M_{\mathbb{G}}\) and
  \(\mu_{\Gamma} = \mu_{\mathbb{G}}\). Theorem~\ref{thm:maharam-lemma} thus provides
  \(A \subseteq B\) such that \(\mu_{\mathbb{G}}(A) = \mu_{\mathbb{G}}(B)/2\). We then have
  \[
    \mu_{\mathbb{G}}(B \setminus A) = \mu_{\mathbb{G}}(B) - \mu_{\mathbb{G}}(B)/2 =
    \mu_{\mathbb{G}}(A),
  \]
  and item~\eqref{item:same-fg-orbit-with-small-support} of
  Proposition~\ref{prop:construct-involutions} provides an involution \(T \in \mathbb{G}\)
  satisfying \(T(B \setminus A) = A\) and \(\supp T = (B \setminus A) \bigtriangleup A = B\).
\end{proof}

\begin{remark}
  Lemma~\ref{lem:aperiodic-action-has-many-involutions} in fact characterizes the aperiodicity of
  full groups.  If \(\mathbb G\) is not aperiodic, then there is some \(B\in \MAlg(X,\mu)\) that is
  not the support of any involution.  This is because its \(M_{\mathbb{G}}\)-conditional measure
  cannot be split in half, as shown in the proof of the direct implication in
  Theorem~\ref{thm:maharam-lemma}.
\end{remark}


\backmatter

\bibliographystyle{alphaurl}
\bibliography{references}

\begin{thebibliography}{BdlHV08}

\bibitem[Aar97]{aaronsonIntroductionInfiniteErgodic1997}
Jon Aaronson.
\newblock {\em An Introduction to Infinite Ergodic Theory}, volume~50 of {\em Math. {{Surv}}. {{Monogr}}.}
\newblock American Mathematical Society, Providence, RI, 1997.
\newblock \href {https://doi.org/10.1090/surv/050} {\path{doi:10.1090/surv/050}}.

\bibitem[ADR00]{MR1799683}
C.~Anantharaman-Delaroche and J.~Renault.
\newblock {\em Amenable groupoids}, volume~36 of {\em Monographies de L'Enseignement Math\'{e}matique [Monographs of L'Enseignement Math\'{e}matique]}.
\newblock L'Enseignement Math\'{e}matique, Geneva, 2000.
\newblock With a foreword by Georges Skandalis and Appendix B by E. Germain.

\bibitem[Aus16]{austinBehaviourEntropyBounded2016}
Tim Austin.
\newblock Behaviour of {{Entropy Under Bounded}} and {{Integrable Orbit Equivalence}}.
\newblock {\em Geometric and Functional Analysis}, 26(6):1483--1525, 2016.
\newblock \href {https://doi.org/10.1007/s00039-016-0392-5} {\path{doi:10.1007/s00039-016-0392-5}}.

\bibitem[Ban32]{banachTheorieOperationsLineaires1932}
Stefan Banach.
\newblock {\em {Th\'eorie des op\'erations lin\'eaires}}.
\newblock {Warsaw}, 1932.

\bibitem[BdlHV08]{bekkaKazhdanProperty2008}
Bachir Bekka, Pierre de~la Harpe, and Alain Valette.
\newblock {\em Kazhdan's Property ({{T}})}, volume~11.
\newblock {Cambridge University Press}, {Cambridge}, 2008.
\newblock \href {https://doi.org/10.1017/CBO9780511542749} {\path{doi:10.1017/CBO9780511542749}}.

\bibitem[Bec13]{MR2995370}
Howard Becker.
\newblock Cocycles and continuity.
\newblock {\em Trans. Amer. Math. Soc.}, 365(2):671--719, 2013.
\newblock \href {https://doi.org/10.1090/S0002-9947-2012-05570-X} {\path{doi:10.1090/S0002-9947-2012-05570-X}}.

\bibitem[Bel68]{MR0245756}
R.~M. Belinskaja.
\newblock Partitionings of a {L}ebesgue space into trajectories which may be defined by ergodic automorphisms.
\newblock {\em Funkcional. Anal. i Prilo\v{z}en.}, 2(3):4--16, 1968.

\bibitem[BK96]{MR1425877}
Howard Becker and Alexander~S. Kechris.
\newblock {\em The descriptive set theory of {P}olish group actions}, volume 232 of {\em London Mathematical Society Lecture Note Series}.
\newblock Cambridge University Press, Cambridge, 1996.
\newblock URL: \url{https://doi-org.ezproxy.lib.utexas.edu/10.1017/CBO9780511735264}, \href {https://doi.org/10.1017/CBO9780511735264} {\path{doi:10.1017/CBO9780511735264}}.

\bibitem[BO10]{MR2743093}
Nicholas~H. Bingham and Adam~J. Ostaszewski.
\newblock Normed versus topological groups: dichotomy and duality.
\newblock {\em Dissertationes Math.}, 472:138, 2010.
\newblock \href {https://doi.org/10.4064/dm472-0-1} {\path{doi:10.4064/dm472-0-1}}.

\bibitem[CFW81]{MR662736}
Alain Connes, Jacob Feldman, and Benjamin Weiss.
\newblock An amenable equivalence relation is generated by a single transformation.
\newblock {\em Ergodic Theory Dynam. Systems}, 1(4):431--450 (1982), 1981.
\newblock \href {https://doi.org/10.1017/s014338570000136x} {\path{doi:10.1017/s014338570000136x}}.

\bibitem[CJMT23]{2201.06662}
Alessandro Carderi, Matthieu Joseph, Fran{\c c}ois~Le Ma{\^i}tre, and Romain Tessera.
\newblock {Belinskaya's theorem is optimal}, 2023.
\newblock \href {https://doi.org/10.4064/fm266-4-2023} {\path{doi:10.4064/fm266-4-2023}}.

\bibitem[CLM16]{MR3464151}
Alessandro Carderi and Fran\c{c}ois Le~Ma\^{\i}tre.
\newblock More {P}olish full groups.
\newblock {\em Topology Appl.}, 202:80--105, 2016.
\newblock \href {https://doi.org/10.1016/j.topol.2015.12.065} {\path{doi:10.1016/j.topol.2015.12.065}}.

\bibitem[CLM18]{MR3748570}
Alessandro Carderi and Fran\c{c}ois Le~Ma\^{\i}tre.
\newblock Orbit full groups for locally compact groups.
\newblock {\em Trans. Amer. Math. Soc.}, 370(4):2321--2349, 2018.
\newblock \href {https://doi.org/10.1090/tran/6985} {\path{doi:10.1090/tran/6985}}.

\bibitem[DHU08]{drosteFullGroupsMeasurepreserving2008}
Manfred Droste, W.~Charles Holland, and Georg Ulbrich.
\newblock On full groups of measure-preserving and ergodic transformations with uncountable cofinalities.
\newblock {\em Bulletin of the London Mathematical Society}, 40(3):463--472, 2008.
\newblock \href {https://doi.org/10.1112/blms/bdn028} {\path{doi:10.1112/blms/bdn028}}.

\bibitem[Dye59]{dyeGroupsMeasurePreserving1959}
Henry~A. Dye.
\newblock On {{Groups}} of {{Measure Preserving Transformations}}. {{I}}.
\newblock {\em American Journal of Mathematics}, 81(1):119--159, 1959.
\newblock \href {https://doi.org/10.2307/2372852} {\path{doi:10.2307/2372852}}.

\bibitem[FHM78]{MR492061}
Jacob Feldman, Peter Hahn, and Calvin~C. Moore.
\newblock Orbit structure and countable sections for actions of continuous groups.
\newblock {\em Adv. in Math.}, 28(3):186--230, 1978.
\newblock \href {https://doi.org/10.1016/0001-8708(78)90114-7} {\path{doi:10.1016/0001-8708(78)90114-7}}.

\bibitem[Fol07]{Folland}
Gerald~B. Folland.
\newblock {\em Real Analysis: Modern Techniques and Their Applications}.
\newblock Wiley, 2nd edition, 5 2007.

\bibitem[Fre04]{MR2459668}
David~H. Fremlin.
\newblock {\em Measure theory. {V}ol. 3}.
\newblock Torres Fremlin, Colchester, 2004.
\newblock Measure algebras, Corrected second printing of the 2002 original.

\bibitem[Fre06]{MR2462372}
David~H. Fremlin.
\newblock {\em Measure theory. {V}ol. 4}.
\newblock Torres Fremlin, Colchester, 2006.
\newblock Topological measure spaces. Part I, II, Corrected second printing of the 2003 original.

\bibitem[Fri70]{friedmanIntroductionErgodicTheory1970}
Nathaniel~A. Friedman.
\newblock {\em Introduction to {{Ergodic Theory}}}.
\newblock Van Nostrand Reinhold, 1970.

\bibitem[GK25]{MR4850604}
Marlies Gerber and Philipp Kunde.
\newblock Non-classifiability of ergodic flows up to time change.
\newblock {\em Invent. Math.}, 239(2):527--619, 2025.
\newblock \href {https://doi.org/10.1007/s00222-024-01312-x} {\path{doi:10.1007/s00222-024-01312-x}}.

\bibitem[GM89]{MR1012985}
Siegfried Graf and R.~Daniel Mauldin.
\newblock A classification of disintegrations of measures.
\newblock In {\em Measure and measurable dynamics ({R}ochester, {NY}, 1987)}, volume~94 of {\em Contemp. Math.}, pages 147--158. Amer. Math. Soc., Providence, RI, 1989.
\newblock \href {https://doi.org/10.1090/conm/094/1012985} {\path{doi:10.1090/conm/094/1012985}}.

\bibitem[GP02]{MR1891002}
Thierry Giordano and Vladimir Pestov.
\newblock Some extremely amenable groups.
\newblock {\em C. R. Math. Acad. Sci. Paris}, 334(4):273--278, 2002.
\newblock \href {https://doi.org/10.1016/S1631-073X(02)02218-5} {\path{doi:10.1016/S1631-073X(02)02218-5}}.

\bibitem[GPS99]{MR1710743}
Thierry Giordano, Ian~F. Putnam, and Christian~F. Skau.
\newblock Full groups of {C}antor minimal systems.
\newblock {\em Israel J. Math.}, 111:285--320, 1999.
\newblock \href {https://doi.org/10.1007/BF02810689} {\path{doi:10.1007/BF02810689}}.

\bibitem[GTW05]{glasnerautomorphismgroupGaussian2005}
Eli Glasner, Boris Tsirelson, and Benjamin Weiss.
\newblock The automorphism group of the {{Gaussian}} measure cannot act pointwise.
\newblock {\em Israel Journal of Mathematics}, 148(1):305--329, December 2005.
\newblock \href {https://doi.org/10.1007/BF02775441} {\path{doi:10.1007/BF02775441}}.

\bibitem[GW05]{glasnerSpatialNonspatialActions2005}
E.~Glasner and B.~Weiss.
\newblock Spatial and non-spatial actions of {{Polish}} groups.
\newblock {\em Ergodic Theory and Dynamical Systems}, 25(05):1521, 2005.
\newblock \href {https://doi.org/10.1017/S0143385705000052} {\path{doi:10.1017/S0143385705000052}}.

\bibitem[Hal17]{halmosLecturesErgodicTheory2017}
Paul~R. Halmos.
\newblock {\em Lectures on Ergodic Theory}.
\newblock {Dover Publications}, {Mineola, NY}, 2017.

\bibitem[Kat75]{katokTimeChangeMonotone1975}
Anatole~B. Katok.
\newblock Time change, monotone equivalence, and standard dynamical systems.
\newblock {\em Doklady Akademii Nauk SSSR}, 223(4):789--792, 1975.

\bibitem[Kat77]{katokMonotoneEquivalenceErgodic1977}
Anatole~B. Katok.
\newblock Monotone equivalence in ergodic theory.
\newblock {\em Izvestiya Akademii Nauk SSSR. Seriya Matematicheskaya}, 41(1):104--157, 231, 1977.

\bibitem[Kec92]{MR1176624}
Alexander~S. Kechris.
\newblock Countable sections for locally compact group actions.
\newblock {\em Ergodic Theory Dynam. Systems}, 12(2):283--295, 1992.
\newblock URL: \url{https://doi-org.ezproxy.lib.utexas.edu/10.1017/S0143385700006751}, \href {https://doi.org/10.1017/S0143385700006751} {\path{doi:10.1017/S0143385700006751}}.

\bibitem[Kec95]{kechris_classical_1995}
Alexander~S. Kechris.
\newblock {\em Classical descriptive set theory}, volume 156 of {\em Graduate Texts in Mathematics}.
\newblock Springer-Verlag, New York, 1995.
\newblock \href {https://doi.org/10.1007/978-1-4612-4190-4} {\path{doi:10.1007/978-1-4612-4190-4}}.

\bibitem[Kec10]{MR2583950}
Alexander~S. Kechris.
\newblock {\em Global aspects of ergodic group actions}, volume 160 of {\em Mathematical Surveys and Monographs}.
\newblock American Mathematical Society, Providence, RI, 2010.
\newblock \href {https://doi.org/10.1090/surv/160} {\path{doi:10.1090/surv/160}}.

\bibitem[KM04]{kechrisTopicsOrbitEquivalence2004}
Alexander~S. Kechris and Benjamin~D. Miller.
\newblock {\em Topics in orbit equivalence}, volume 1852 of {\em Lecture Notes in Mathematics}.
\newblock Springer-Verlag, Berlin, 2004.
\newblock \href {https://doi.org/10.1007/b99421} {\path{doi:10.1007/b99421}}.

\bibitem[KPV15]{kyedBettiNumbersLocally2015}
David Kyed, Henrik Petersen, and Stefaan Vaes.
\newblock {{L}}{$^2$}-{{Betti}} numbers of locally compact groups and their cross section equivalence relations.
\newblock {\em Trans. Amer. Math. Soc.}, 367(7):4917--4956, 2015.
\newblock \href {https://doi.org/10.1090/S0002-9947-2015-06449-6} {\path{doi:10.1090/S0002-9947-2015-06449-6}}.

\bibitem[Kre85]{MR797411}
Ulrich Krengel.
\newblock {\em Ergodic theorems}, volume~6 of {\em De Gruyter Studies in Mathematics}.
\newblock Walter de Gruyter \& Co., Berlin, 1985.
\newblock With a supplement by Antoine Brunel.
\newblock URL: \url{https://doi-org.ezproxy.lib.utexas.edu/10.1515/9783110844641}, \href {https://doi.org/10.1515/9783110844641} {\path{doi:10.1515/9783110844641}}.

\bibitem[KST99]{MR1667145}
Alexander~S. Kechris, S\l{}awomir Solecki, and Stevo Todorcevic.
\newblock Borel chromatic numbers.
\newblock {\em Adv. Math.}, 141(1):1--44, 1999.
\newblock \href {https://doi.org/10.1006/aima.1998.1771} {\path{doi:10.1006/aima.1998.1771}}.

\bibitem[Kun23]{kundeAnticlassificationResultsWeakly2023}
Philipp Kunde.
\newblock Anti-classification results for weakly mixing diffeomorphisms, March 2023.
\newblock \href {https://arxiv.org/abs/arXiv:2303.12900} {\path{arXiv:arXiv:2303.12900}}, \href {https://doi.org/10.48550/arXiv.2303.12900} {\path{doi:10.48550/arXiv.2303.12900}}.

\bibitem[LM14]{lemaitreGroupesPleinsPreservant2014}
Fran{\c c}ois Le~Ma{\^i}tre.
\newblock {\em {Sur les groupes pleins pr\'eservant une mesure de probabilit\'e}}.
\newblock PhD thesis, ENS Lyon, 2014.

\bibitem[LM16]{MR3568978}
Fran\c{c}ois Le~Ma\^{\i}tre.
\newblock On full groups of non-ergodic probability-measure-preserving equivalence relations.
\newblock {\em Ergodic Theory Dynam. Systems}, 36(7):2218--2245, 2016.
\newblock \href {https://doi.org/10.1017/etds.2015.20} {\path{doi:10.1017/etds.2015.20}}.

\bibitem[LM18]{MR3810253}
Fran\c{c}ois Le~Ma\^{\i}tre.
\newblock On a measurable analogue of small topological full groups.
\newblock {\em Adv. Math.}, 332:235--286, 2018.
\newblock \href {https://doi.org/10.1016/j.aim.2018.05.008} {\path{doi:10.1016/j.aim.2018.05.008}}.

\bibitem[LM21]{MR4398251}
Fran\c{c}ois Le~Ma\^{\i}tre.
\newblock On a measurable analogue of small topological full groups {II}.
\newblock {\em Ann. Inst. Fourier (Grenoble)}, 71(5):1885--1927, 2021.
\newblock \href {https://doi.org/10.5802/aif.3443} {\path{doi:10.5802/aif.3443}}.

\bibitem[Mac62]{MR143874}
George~W. Mackey.
\newblock Point realizations of transformation groups.
\newblock {\em Illinois J. Math.}, 6:327--335, 1962.
\newblock URL: \url{http://projecteuclid.org/euclid.ijm/1255632330}.

\bibitem[Mah50]{MR36817}
Dorothy Maharam.
\newblock Decompositions of measure algebras and spaces.
\newblock {\em Trans. Amer. Math. Soc.}, 69:142--160, 1950.
\newblock \href {https://doi.org/10.2307/1990602} {\path{doi:10.2307/1990602}}.

\bibitem[Mah84]{MR786682}
Dorothy Maharam.
\newblock On the planar representation of a measurable subfield.
\newblock In {\em Measure theory, {O}berwolfach 1983 ({O}berwolfach, 1983)}, volume 1089 of {\em Lecture Notes in Math.}, pages 47--57. Springer, Berlin, 1984.
\newblock \href {https://doi.org/10.1007/BFb0072599} {\path{doi:10.1007/BFb0072599}}.

\bibitem[Mil77]{MR440519}
Douglas~E. Miller.
\newblock On the measurability of orbits in {B}orel actions.
\newblock {\em Proc. Amer. Math. Soc.}, 63(1):165--170, 1977.
\newblock \href {https://doi.org/10.2307/2041088} {\path{doi:10.2307/2041088}}.

\bibitem[Mil04]{millerFullGroupsClassification2004}
Benjamin~D. Miller.
\newblock {\em Full Groups, Classification, and Equivalence Relations}.
\newblock PhD thesis, University of California, {Berkeley}, 2004.

\bibitem[Nac65]{MR0175995}
Leopoldo Nachbin.
\newblock {\em The {H}aar integral}.
\newblock D. Van Nostrand Co., Inc., Princeton, N.J.-Toronto-London, 1965.

\bibitem[Nek19]{MR3904185}
Volodymyr Nekrashevych.
\newblock Simple groups of dynamical origin.
\newblock {\em Ergodic Theory Dynam. Systems}, 39(3):707--732, 2019.
\newblock \href {https://doi.org/10.1017/etds.2017.47} {\path{doi:10.1017/etds.2017.47}}.

\bibitem[ORW82]{ornsteinEquivalenceMeasurePreserving1982}
Donald~S. Ornstein, Daniel~J. Rudolph, and Benjamin Weiss.
\newblock Equivalence of measure preserving transformations.
\newblock {\em Memoirs of the American Mathematical Society}, 37(262):xii+116, 1982.
\newblock \href {https://doi.org/10.1090/memo/0262} {\path{doi:10.1090/memo/0262}}.

\bibitem[PS17]{MR3711882}
Vladimir~G. Pestov and Friedrich~Martin Schneider.
\newblock On amenability and groups of measurable maps.
\newblock {\em J. Funct. Anal.}, 273(12):3859--3874, 2017.
\newblock \href {https://doi.org/10.1016/j.jfa.2017.09.011} {\path{doi:10.1016/j.jfa.2017.09.011}}.

\bibitem[RF17]{Royden}
Halsey Royden and Patrick Fitzpatrick.
\newblock {\em Real Analysis}.
\newblock Pearson Modern Classics for Advanced Mathematics. Pearson, 4th edition, 2 2017.

\bibitem[Ros22]{MR4327092}
Christian Rosendal.
\newblock {\em Coarse geometry of topological groups}, volume 223 of {\em Cambridge Tracts in Mathematics}.
\newblock Cambridge University Press, Cambridge, 2022.
\newblock \href {https://doi.org/10.1017/9781108903547} {\path{doi:10.1017/9781108903547}}.

\bibitem[RS98]{robertsonNegativedefinitekernels1998}
Guyan Robertson and Tim Steger.
\newblock Negative definite kernels and a dynamical characterization of property ({{T}}) for countable groups.
\newblock {\em Ergodic Theory and Dynamical Systems}, 18(1):247--253, February 1998.
\newblock \href {https://doi.org/10.1017/S0143385798100342} {\path{doi:10.1017/S0143385798100342}}.

\bibitem[Rud76]{rudolphTwovaluedStepCoding1976}
Daniel Rudolph.
\newblock A two-valued step coding for ergodic flows.
\newblock {\em Mathematische Zeitschrift}, 150(3):201--220, 1976.
\newblock \href {https://doi.org/10.1007/BF01221147} {\path{doi:10.1007/BF01221147}}.

\bibitem[Ryz85]{ryzhikovRepresentationTransformationsPreserving1985}
V.~V. Ryzhikov.
\newblock Representation of transformations preserving the lebesgue measure in the form of periodic transformations.
\newblock {\em Mathematical notes of the Academy of Sciences of the USSR}, 38(6):978--981, 1985.
\newblock \href {https://doi.org/10.1007/BF01157015} {\path{doi:10.1007/BF01157015}}.

\bibitem[Slu17]{MR3681992}
Konstantin Slutsky.
\newblock Lebesgue orbit equivalence of multidimensional {B}orel flows: a picturebook of tilings.
\newblock {\em Ergodic Theory Dynam. Systems}, 37(6):1966--1996, 2017.
\newblock \href {https://doi.org/10.1017/etds.2015.119} {\path{doi:10.1017/etds.2015.119}}.

\bibitem[Slu19]{slutskyRegularCrossSections2019}
Konstantin Slutsky.
\newblock Regular cross sections of {{Borel}} flows.
\newblock {\em Journal of the European Mathematical Society}, 21(7):1985--2050, 2019.
\newblock \href {https://doi.org/10.4171/jems/879} {\path{doi:10.4171/jems/879}}.

\bibitem[Str74]{MR348037}
Raimond~A. Struble.
\newblock Metrics in locally compact groups.
\newblock {\em Compositio Math.}, 28:217--222, 1974.

\bibitem[Var63]{MR159923}
Veeravalli~S. Varadarajan.
\newblock Groups of automorphisms of {B}orel spaces.
\newblock {\em Trans. Amer. Math. Soc.}, 109:191--220, 1963.
\newblock \href {https://doi.org/10.2307/1993903} {\path{doi:10.2307/1993903}}.

\end{thebibliography}

\printindex

\end{document}